\documentclass[a4paper,12pt,oneside]{book}
\usepackage{amsmath,amsthm,amsfonts,amssymb,makeidx,mathscinet}
\usepackage{hyperref}
\usepackage[T1]{fontenc}
\usepackage{graphicx}
\usepackage[lmargin=4cm, rmargin=2.5cm, tmargin=2.5cm, bmargin=2.5cm]{geometry}
\usepackage[all]{xy}
\usepackage{setspace}
\newdir{ >}{{}*!/-5pt/@{>}}
\DeclareMathOperator{\colim}{colim}
\DeclareMathOperator{\im}{Im}
\DeclareMathOperator{\coker}{coker}
\DeclareMathOperator{\id}{Id}
\DeclareMathOperator{\ID}{ID}
\DeclareMathOperator{\leftmod}{--mod}
\DeclareMathOperator{\Mod}{-Mod}
\DeclareMathOperator{\Ext}{Ext}
\DeclareMathOperator{\Hom}{Hom}
\DeclareMathOperator{\Weyl}{Weyl-G-Sheaf}
\DeclareMathOperator{\Gspectra}{G-Spectra}
\DeclareMathOperator{\Ho}{Ho}
\DeclareMathOperator{\mackey}{Mackey}
\DeclareMathOperator{\rank}{Rank}

\DeclareMathOperator{\ch}{Ch}
\DeclareMathOperator{\sspan}{Span}
\DeclareMathOperator{\orb}{Orb}
\DeclareMathOperator{\ab}{Ab}
\DeclareMathOperator{\weyl}{Weyl}
\DeclareMathOperator{\gheaf}{G-Sheaf}
\DeclareMathOperator{\cb}{CB}
\DeclareMathOperator{\fun}{Fun}
\DeclareMathOperator{\derive}{Derived}
\DeclareMathOperator{\gset}{-Set}
\newcommand{\Grpf}{\text{Grp}^f}

\DeclareMathOperator{\stabgx}{stab_G}
\newtheorem{theorem}{Theorem}[section]
\newtheorem{proposition}[theorem]{Proposition}
\newtheorem{conjecture}[theorem]{Conjecture}
\newtheorem{lemma}[theorem]{Lemma}
\newtheorem{corollary}[theorem]{Corollary}
\theoremstyle{definition}
\newtheorem{example}[theorem]{Example}
\newtheorem{definition}[theorem]{Definition}
\newtheorem{remark}[theorem]{Remark}
\newtheorem{construction}[theorem]{Construction}
\makeindex
\begin{document}
\linespread{1.25}
\bibliographystyle{alpha}
\begin{titlepage}
\centering
\includegraphics[scale=0.3]{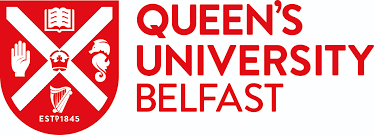}\par
{\Huge Queen's University Belfast\par}
\vspace{1cm}
{\Large School of Mathematics and Physics\par}
\vspace{2cm}
{\Large Doctor of Philosophy in Pure Mathematics\par}
\vspace{1.5cm}
{\Large \bfseries Rational $G$-spectra for profinite $G$\par}
\vspace{1cm}
{\Large \itshape Danny Sugrue MSci\par
\vfill 
supervised by\par 
Dr. David Barnes
\vfill
{\large \today\par}}
\end{titlepage}
\tableofcontents
\chapter*{Acknowledgements}
I would like to thank my supervisor Dr David Barnes for his generous support throughout the course of my PhD. I have learned a lot and have thoroughly enjoyed my experience under his supervision. I would also like to thank my second supervisor Dr Martin Mathieu for helpful conversations we have had, as well as everyone who has participated in teaching me throughout my time at QUB. I would like to express gratitude also to Prof John Greenlees for his generosity in providing valuable feedback during discussions we have had. I wish to also thank my fellow PhD students for their support and comradery which greatly helped me and even made this experience even more special. Finally I would like to thank my family and friends for their support which made everything possible.

\chapter*{Abstract}
In this thesis we will investigate rational $G$-spectra for a profinite group $G$. We will provide an algebraic model for this model category whose injective dimension can be calculated in terms of the Cantor-Bendixson rank of the space of closed subgroups of $G$, denoted $SG$. The algebraic model we consider is chain complexes of Weyl-$G$-sheaves of $\mathbb{Q}$-modules over the space $SG$. The key step in proving that this is an algebraic model for $G$-spectra is in proving that the category of rational $G$-Mackey functors is equivalent to Weyl-$G$-sheaves. In addition to the fact that this sheaf description utilises the topology of $G$ and the closed subgroups of $G$ in a more explicit way than Mackey functors do, we can also calculate the injective dimension. In the final part of the thesis we will see that the injective dimension of the category of Weyl-$G$-sheaves can be calculated in terms of the Cantor-Bendixson rank of $SG$, hence giving the injective dimension of the category of Mackey functors via the earlier equivalence.
\chapter*{Introduction}
\section*{$G$-spaces, $G$-cohomology theories and $G$-spectra} 
If $G$ is a topological group, a $G$-space is a topological space which admits a continuous action of $G$. An interesting area of study is the homotopy theory of spaces with a continuous group action. A highly successful method of studying $G$-spaces is to use equivariant cohomology theories, a generalisation of cohomology theories.

In order to define formally a rational $G$-equivariant cohomology theory we first need to define the representation ring of a compact Lie group $G$ with respect to a $G$-universe $U$, denoted $RO(G,U)$. The following construction and definition is given in May \cite[Definition 13.1.1]{alaska}. This uses the definition of a universe as seen in Definition \ref{Univdefn}. 
\begin{definition}
A \index{$G$-universe} \textbf{$G$-universe} $U$ is a countably infinite direct sum $\underset{n\in\mathbb{N}}{\bigoplus} U^{\prime}$ of a real $G$-inner product space $U^{\prime}$ satisfying the following:
\begin{itemize}
\item the one dimensional trivial $G$-representation is contained in $U^{\prime}$,
\item $U$ is topologised as the union of all finite dimensional $G$-subspaces of $U$ (each with the norm topology),
\item the $G$-action on all finite dimensional $G$-subspaces $V$ of $U$ factors through a compact Lie group quotient of $G$.
\end{itemize}  
\end{definition}
\begin{construction}
If $G$ is a profinite group and $U$ is a $G$-universe, we set $RO(G,U)$ to be the category whose objects are the finite dimensional \\$G$-representations embeddable in $U$ and whose morphisms are the $G$-linear isometric isomorphisms. We say that two such morphisms from $U$ to $V$ are homotopic if their associated based $G$-maps, $S^U\rightarrow S^V$ (between one point compactification), are stably homotopic. The notation $hRO(G,U)$ represents the homotopy category.

For $W$ a finite dimensional $G$-representation, let $\Sigma^W$ denote $S^W\wedge (-)$ and $S^W$ the one point compactification of $W$. We have a functor:
\begin{align*}
\Sigma^W:RO(G,U)\times hG\mathcal{T}&\rightarrow RO(G,U)\times hG\mathcal{T}\\(V,X)&\mapsto \left(V\oplus W,\Sigma^WX\right)
\end{align*}
where $hG\mathcal{T}$ represents the homotopy category of based $G$-spaces.
\end{construction}  
We are now in a position to define\index{$RO(G,U)$-graded cohomology theories} $RO(G,U)$-graded cohomology theories.
\begin{definition}\label{cohomdefn}
An $RO(G,U)$-graded cohomology theory is a functor:
\begin{align*}
E_G^*:hRO(G,U)\times hG\mathcal{T}^{op}&\rightarrow {Ab}\\(V,X)&\mapsto E^V_G(X)
\end{align*}
on objects, where morphisms are assigned similarly. We also have natural isomorphisms $\sigma^W:E^*_G\rightarrow E^*_G\circ \Sigma^W$ written as:
\begin{align*}
\sigma^W:E^V_G(X)\rightarrow E^{V\bigoplus W}_G(\Sigma^WX),
\end{align*}
satisfying the following axioms.
\begin{enumerate}
\item For each representation $V$, the functor $E^V_G$ is exact on cofiber sequences and sends wedges to products. 
\item If $\alpha:W\rightarrow W^{\prime}$ is a morphism in $RO(G,U)$, then the following diagram commutes:
\begin{align*}
\xymatrix{E^V_G(X)\ar[r]^{\sigma^{W}}\ar[d]_{\sigma^{W^{\prime}}}&E^{V\bigoplus W}_G\left(\Sigma^WX\right)\ar[d]^{E_G^{\id\bigoplus\alpha}(\id)}\\E_G^{V\bigoplus W^{\prime}}\left(\Sigma^{W^{\prime}}X\right)\ar[r]_{\left(\Sigma^{\alpha}\id\right)^*}&E_G^{V\bigoplus W^{\prime}}\left(\Sigma^WX\right).}
\end{align*}
\item The map $\sigma^0$ is the identity and the morphisms $\sigma$ are transitive in the sense that the following diagram commutes for each pair of representations $(W,Z)$:
\begin{align*}
\xymatrix{E_G^V(X)\ar[dr]_{\sigma^{W\bigoplus Z}}\ar[rr]^{\sigma^W}&&E_G^{V\bigoplus W}\left(\Sigma^WX\right)\ar[dl]^{\sigma^Z}\\&E_G^{V\bigoplus W\bigoplus Z}\left(\Sigma^{W\bigoplus Z}X\right)}.
\end{align*}
\end{enumerate}
We extend the theory defined above to formal differences, $V\ominus W$, for any pair of representations by defining:
\begin{align*}
E_G^{V\ominus W}(X)=E_G^V\left(\Sigma^WX\right).
\end{align*}
\end{definition}
Notice that the symbol $\ominus$ is used to avoid confusion with either the orthogonal complement or difference of $G$-representations. We view a formal difference $V\ominus W$ as an element of $hRO(G,U)\times hRO(G,U)^{op}$. This means that for every $G$-space $X$, $E^*_G(X)$ defines a functor from this category into the category of abelian groups. 

We say that $V\ominus W$ is equivalent to $V^{\prime}\ominus W^{\prime}$ if there exists a $G$-linear isometric isomorphism:
\begin{align*}
\alpha:V\oplus W^{\prime}\rightarrow V^{\prime}\oplus W.
\end{align*}
We define the representation group for $G$ relative to $U$, denoted $\mathfrak{RO}(G,U)$, by passing to equivalence classes of formal differences.
In particular, for every $G$-space $X$, a cohomology theory defines a functor from $\mathfrak{RO}(G,U)$ into the category of abelian groups. Viewing cohomology theories as being graded in this manner requires us to define a linear map of the form:
\begin{align*}
E_G^V\left(\Sigma^WX\right)\rightarrow E_G^{V^{\prime}}\left(\Sigma^{W^{\prime}}X\right),
\end{align*}
which is given by the unlabelled arrow in the following diagram of isomorphisms:
\begin{align*}
\xymatrix{E_G^V\left(\Sigma^WX\right)\ar[d]\ar[r]^(.4){\sigma^{W^{\prime}}}&E_G^{V\bigoplus W^{\prime}}\left(\Sigma^{W\bigoplus W^{\prime}}X\right)\ar[d]^{E^{\alpha}_G\left(\Sigma^{\tau}\id\right)}\\E_G^{V^{\prime}}\left(\Sigma^{W^{\prime}}X\right)\ar[r]_(.4){\sigma^W}&E_G^{V^{\prime}\bigoplus W}\left(\Sigma^{W^{\prime}\bigoplus W}X\right),}
\end{align*}
where $\tau$ is the transposition isomorphism.

$G$-equivariant cohomology theories are important in algebraic topology since they are a homotopy invariant for topological $G$-spaces. One such example is the Borel construction on a non-equivariant cohomology theory $F^*$ we define a $G$-equivariant cohomology theory $F_G^*$ as follows. If $X$ is any $G$-space we define:
\begin{align*}
F_G^*(X)=F^*(EG_+\underset{G}{\wedge}X),
\end{align*}
where $F^*$ is a reduced cohomology theory.

Other examples include equivariant $K$-theory and equivariant cobordism. We can see from the following example that non-equivariant cohomology theories are insufficient for the study of $G$-spaces. 
\begin{example}
Let $G$ be the cyclic group of order $2$. For each $n\in\mathbb{N}$ we can give $S^n$ the antipodal action and denote it by $S_a^n$. Non-equivariantly, 
\begin{align*}
S^{\infty}_a=\underset{k\in\mathbb{N}}{\colim}S^k_a\simeq \ast
\end{align*}
with maps the equatorial inclusions. In particular for any $n\in\mathbb{N}$,
\begin{align*}
\pi_n\left(S^{\infty}\right)&=\pi_n\left(\underset{k\in\mathbb{N}}{\colim} S^k\right)=\underset{k\in\mathbb{N}}{\colim}\,\pi_n\left(S^{n+k}\right)=0.
\end{align*}
Since $S^{\infty}_a$ has a free action of $G$ (i.e. no fixed points) and is a $G$-CW complex, we can write $S^{\infty}_a=EG$ as:
\begin{align*}
S_a^{\infty}/G=EG/G=BG=\mathbb{R}P^{\infty}.
\end{align*}
Thus a non-equivariant cohomology theory $F^*$ can not distinguish $S^{\infty}$ from $S^0$. On the other hand we have $F^*_G(S^{\infty}_+)\cong F^*(\mathbb{R}P^{\infty}_+)$.
\end{example}
\begin{example}
Let $G$ be the circle group $S^1$ and $S\left(\mathbb{C}^n\right)$ be the unit circle in $\mathbb{C}^n$. In particular $S\left(\mathbb{C}\right)=S^1$. In this case $EG=\underset{n\in \mathbb{N}}{\colim}S\left(\mathbb{C}^n\right)$. The action on $S(\mathbb{C}^n)$ by $S^1$ is given by:
\begin{align*}
\lambda\left(z_1,z_2,\ldots,z_n \right)=\left(\lambda z_1,\lambda z_2,\ldots,\lambda z_n\right)
\end{align*}
which is also free. We know that $EG/G=\mathbb{C}P^{\infty}$ which has non-trivial homotopy groups. On the other hand $EG$ is homotopically trivial. So once again, a non-equivariant cohomology theory $F^*$ can not distinguish $EG_+$ from $S^0$. Explicitly, $F^*(EG_+)\cong F^*(*_+)\cong F^*(S^0)$ and $F^*_G(EG_+)\cong F^*(\mathbb{C}P^{\infty}_+)$.
\end{example}
As can be seen, the definition of $G$-cohomology theories is quite complicated. In order to study $G$-cohomology theories, it would be useful to have a description with some notion of homotopy. This definition comes from $G$-spectra, a generalisation of the concept of a spectrum.

A $G$-spectrum $X$ consists of a based $G$-space $X(V)$ for each finite dimensional $G$-representation embeddable in $U$, with structure maps:
\begin{align*}
\Sigma^W X(V)\to X(V\oplus W). 
\end{align*}
When $G$ is the trivial group we recover the definition of non-equivariant spectra. The formal definition can be found in Chapter \ref{chp3}. Equivariant cohomology theories are related to equivariant spectra via the following construction and theorem.
 
A $G$-spectrum $E$ determines a rational $G$-equivariant cohomology theory $E^*_G$, by:
\begin{align*}
E^{V\ominus W}_G(A)=\left[\Sigma^{\infty}(\Sigma^WA),\Sigma^VE\right]_*
\end{align*}  
where $\Sigma^{\infty}$ is the suspension spectrum on a based $G$-space $A$. We say $E$ represents $E^*_G$. All $G$-cohomology theories come from $G$ in this way.
\begin{theorem}
If $G$ is a compact Lie group, every $RO(G;U)$-graded cohomology theory on based $G$-spaces is represented by an $G$-spectrum indexed on $U$.
\end{theorem}
In the non-equivariant setting, the development of spectra made the study of cohomology theories easier. In particular, the suspension spectrum gives a functor from spaces to spectra, so we may study spaces by studying their associated suspension spectra. Equivariantly the same is true, so we now study $G$-spectra.

The homotopy theory of $G$-spectra is extremely complicated, even more so than non-equivariant spectra. One source of difficulty is the torsion which arises in the homotopy groups of spheres, which is incalculable even when $G$ is trivial. It is for this reason that working rationally becomes advantageous, since rationalising removes torsion leaving a good amount of interesting equivariant behaviour.

The formal definition of rational $G$-spectra is given in Chapter \ref{chp3}. For now it is sufficient to think of rational $G$-spectra as a particular model structure on the category of $G$-spectra. It can be checked that rational $G$-spectra correspond to $G$-equivariant cohomology theories which take values in rational vector spaces.

We note here that there is no version of Brown representability for profinite groups in the literature. However, it is expected that an analogous statement holds.  

\section*{Classifying rational $G$-spectra}
In the case when $G$ is a finite group, rational $G$-spectra are well understood. In particular the homotopy classes of maps $[\Sigma^{\infty}G/H_+,\Sigma^{\infty}G/K_+]_*$ in the category of $G$-spectra are concentrated in degree zero. Consequently this means that the category of $G$-spectra are Quillen equivalent to chain complexes of rational $G$-Mackey functors as seen in Greenlees and May, \cite{Tate}.
\begin{theorem}
If $G$ is a finite discrete group then there is an equivalence of categories:
\begin{align*}
\derive\left(\mackey_{\mathbb{Q}}\left(G\right)\right)\cong \Ho\left(\Gspectra_{\mathbb{Q}}\right).
\end{align*}
\end{theorem}  
In \cite{Tate} there is a classification of $G$-Mackey functors for $G$ a finite discrete group. This is given by the following theorem:
\begin{theorem}
If $G$ is a finite discrete group then there is an equivalence of categories:
\begin{align*}
\mackey_{\mathbb{Q}}\left(G\right)\cong \underset{H\in SG/G}{\bigoplus}\mathbb{Q}\left[N_G(H)/H\right]\Mod
\end{align*}
where $SG/G$ are the conjugacy classes of subgroups of $G$. Consequently we have the following equivalence:
\begin{align*}
\derive\left(\underset{H\in SG/G}{\bigoplus}\mathbb{Q}\left[N_G(H)/H\right]\Mod \right)\cong \Ho\left(\Gspectra_{\mathbb{Q}}\right).
\end{align*}
We say that $A(G)=\underset{H\in SG/G}{\bigoplus}\mathbb{Q}\left[N_G(H)/H\right]\Mod$ is the abelian model for rational $G$-spectra when $G$ is finite and $\ch(A(G))$ is the algebraic model.
\end{theorem}
Since the case of $G$ a finite discrete group is well understood, the next cases to consider are compact topological groups. The main examples of these are the profinite groups and compact Lie groups.

The motivation behind this project is to extend a conjecture by Greenlees \cite{conjecture} for compact Lie groups to profinite groups. The conjecture is given as follows:
\begin{conjecture}
If $G$ is a compact Lie group then there is an abelian category $A(G)$ such that the homotopy category of rational $G$-spectra is equivalent to the derived category of $A(G)$. Furthermore the injective dimension of $A(G)$ is equal to the rank of $G$.
\end{conjecture}
The reference \cite{conjecture} completes the conjecture for tori, and Barnes and K\k{e}dziorek, \cite{BarnesO2} and \cite{MagSO3}, prove it in the cases when $G=O(2)$ and $G=SO(3)$ respectively. In the last two cases a part of the abelian model for $G$-spectra is determined by sheaf categories. This structure appears in the characterisation of the profinite case presented in this thesis. This work will help build up our understanding of how sheaf structures can be used to model the homotopy theory of $G$-spectra more generally. 

Profinite groups are an important class of topological group and they occur naturally in mathematics as Galois groups of Galois field extensions. We know from Fausk \cite[Appendix 1]{Fausk} that a compact topological group is an inverse limit of compact Lie groups. We also know that the finite compact Lie groups are precisely the finite discrete groups. Therefore the characterisation of profinite groups (if we begin with the compact Hausdorff and totally disconnected group definition) as inverse limits of finite discrete groups can be taken as a special case of \cite[Appendix 1]{Fausk}. If we can use the inverse limit structure of a profinite group $G$ to assemble the algebraic model for $G=\underset{i}{\lim}G_i$ from the algebraic models for the $G_i$, then a future possibility is open: to calculate algebraic models for compact topological groups by assembling known algebraic models for compact Lie groups.

When $G=\mathbb{Z}_p$, an algebraic model for rational $G$-spectra is given in Barnes \cite{BarZp}, so the challenge is to try and generalise this for arbitrary profinite groups. The case when $G$ is profinite is more approachable than the general compact Lie group case as the homotopy classes of morphisms of $G$-spectra 
\begin{align*}
[\Sigma^{\infty}G/H_+,\Sigma^{\infty}G/K_+]_*, 
\end{align*}
are concentrated in degree zero when $G$ is profinite, as in the finite case. 

In addition to simplifying the homotopy theory of $G$-spectra, working rationally simplifies the algebra. Specifically, the rational Burnside ring has an easier algebraic structure as there are many more non-trivial multiplicative idempotents than in the integral case.
\section*{Main Results}
We now give a break down of the contents of this thesis.
\subsection*{Rational $G$-spectra for profinite $G$}
In the first chapter of the thesis we look at some of the underlying theory that is needed for the rest of the project including the theory of model categories, sheaf theory and homological algebra. For the second chapter we define the category of rational $G$-Mackey functors when $G$ is profinite. We then define the Burnside ring and give an alternative construction which is very useful in developing the later theory. In particular, we will see that the Burnside ring interacts with any Mackey functor in a way which allows us to attribute a collection of sheaves to any Mackey functor. In the third chapter we will focus on demonstrating that the category of chain complexes of rational $G$-Mackey functors provide an algebraic model for rational $G$-spectra. A key part of this chapter will involve giving an alternative definition of Mackey functors in terms of the full subcategory of the homotopy category of rational $G$-spectra, called the orbit category $\mathfrak{O}_G$. Specifically, we will see that $G$-Mackey functors can be defined as contravariant additive functors of the form:
\begin{align*}
M:\pi_0\left(\mathfrak{O}_G\right)\rightarrow \mathbb{Q}\text{-Mod}.
\end{align*} 
The key piece of information which we use here is that for profinite groups the homotopy classes of morphisms of rational $G$-spectra $[\Sigma^{\infty}G/H_+,\Sigma^{\infty}G/K_+]_*$, are concentrated in degree zero. This alternative definition of a Mackey functor for $G$ has the consequence that homotopy groups of rational $G$-spectra are graded Mackey functors: if $X$ is a $G$-spectrum, then for any open subgroup $H$ of $G$ we define a Mackey functor by:
\begin{align*}
\pi_*(X)(G/H)=\pi^H_*(X)=\left[\Sigma^{\infty}G/H_+,X\right]_*.
\end{align*}
The rest of the thesis will focus on understanding Mackey functors since characterising these will provide an alternative description of rational $G$-spectra, as seen in Theorem \ref{equicohomchar}:
\begin{theorem}
If $G$ is a profinite group then there is an equivalence of categories:
\begin{align*}
\derive\left(\mackey_{\mathbb{Q}}\left(G\right)\right)\cong \Ho\left(\Gspectra_{\mathbb{Q}}\right).
\end{align*}
\end{theorem}
This will follow from a Quillen equivalence between rational $G$-spectra and chain complexes of rational $G$-Mackey functors, which we will see exists from Schwede and Shipley \cite{StableModel}.
\subsection*{Characterising $G$-Mackey functors}
For the fourth chapter, we will focus on defining $G$-equivariant sheaves of $\mathbb{Q}$-modules over a profinite $G$-space $X$. There are two definitions and we will sketch the equivalence between them. In the situation when $X$ is the space of closed subgroups of $G$, denoted $SG$, we will define the category of Weyl-$G$-sheaves of $\mathbb{Q}$-modules over $SG$. Much of the information on $G$-sheaves is known in the literature, however the assumption that $G$ is a profinite group and $X$ a profinite space allows us to prove some key lemmas regarding the properties of both $G$-sheaves and Weyl-$G$-sheaves of $\mathbb{Q}$-modules. These lemmas will be essential to proving the desired characterisation of Mackey functors. Finally we also look at limits and colimits in the category of $G$-sheaves with particular attention to infinite products. Understanding how to do infinite products will be especially important when calculating the injective dimension of the category. In the fifth chapter we construct a correspondence between Weyl-$G$-sheaves of $\mathbb{Q}$-modules over $SG$ and $G$-Mackey functors in both directions. For chapter six we show that these correspondences give us equivalences of categories. By doing this we show that rational $G$-Spectra are algebraically modelled by chain complexes of Weyl-$G$-sheaves of $\mathbb{Q}$-modules. This is summarised in the following theorem: 
\begin{theorem}
If $G$ is a profinite group then there is an equivalence of categories:
\begin{align*}
\mackey_{\mathbb{Q}}\left(G\right)\cong \Weyl_{\mathbb{Q}}\left(SG\right).
\end{align*}
Consequently we have the following characterisation:
\begin{align*}
\derive \left(\Weyl_{\mathbb{Q}}(SG)\right) \cong \Ho \left(\Gspectra_{\mathbb{Q}}\right).
\end{align*}
\end{theorem}
This is given in Theorem \ref{equivalencemain}. The most substantial challenge in proving this theorem comes from constructing the correspondence from Mackey functors to Weyl-$G$-sheaves, which requires careful analysis of the action of Burnside ring idempotents on Mackey functors.

We then show that the inclusion functor from the category of Weyl-$G$-sheaves over $SG$ into the category of $G$-sheaves over $SG$ (as a full subcategory) is a left adjoint. The corresponding right adjoint is explicitly defined using the proceeding equivalence. This adjunction is proven in Proposition \ref{Weyladjunct}. We then take the opportunity to define infinite products of Weyl-$G$-sheaves of $\mathbb{Q}$-modules using both the equivalence and the right adjoint.
\subsection*{Injective dimension}
Since we have constructed the abelian category modelling rational $G$-spectra we can work towards verifying the injective dimension aspect of the Greenlees conjecture. In the seventh chapter we shift our focus in this direction and show that the injective dimension of sheaves of $\mathbb{Q}$-modules over a profinite space $X$ is determined by the Cantor-Bendixson rank of $X$. We do this by constructing an injective resolution of any sheaf by considering the Godement resolution. We calculate $\text{Ext}$ groups by resolving with respect to the skyscraper sheaf of a point of maximal Cantor-Bendixson height, if one exists. We look at some examples before giving a general proof for more general profinite spaces. The following is Theorem \ref{Summary}.
\begin{theorem}
If $X$ is a space with empty perfect hull and finite Cantor-Bendixson rank, say $n$, then the injective dimension of sheaves of $\mathbb{Q}$-modules over $X$ is equal to $n-1$. If $X$ has infinite Cantor-Bendixson rank then the injective dimension of sheaves of $\mathbb{Q}$-modules over $X$ is infinite. 
\end{theorem}
The final case where $X$ has finite Cantor-Bendixson rank and non-empty perfect hull is discussed and it is conjectured that the injective dimension is infinite in this case. 

Finally, the eighth chapter considers what happens when we work with a $G$-action by considering $G$-sheaves of $\mathbb{Q}$-modules and Weyl-$G$-sheaves over $SG$. We do this by constructing a $G$-equivariant analogue of the Godement resolution and proving that this is an injective resolution. Once this has been set up we show that we can extend the argument from the non-equivariant case to the equivariant one. This completes the calculation when $G$ is a profinite group with $SG$ having infinite Cantor-Bendixson rank or finite if $SG$ has empty perfect hull. This is given in Theorem \ref{GSummary}.

In the non-equivariant version of the calculation we used that $\mathbb{Q}$-modules are injective. In order to extend this equivariantly we use Castellano and Weigel \cite[Proposition 3.7]{Castellano} which proves that if $G$ is profinite then every discrete $\mathbb{Q}\left[G\right]$-module is injective. This formally completes a proof in \cite{BarZp} which left out the verification that $\mathbb{Q}\left[\mathbb{Z}_p\right]$-modules are injective.

Since we have the Mackey functor equivalence from Chapter 6 this allows us to characterise the injective dimension of Mackey functors also. Therefore, we have found an abelian category which models rational $G$-Spectra and we can calculate the injective dimension of this category in terms of the Cantor-Bendixson rank.

We summarise with a diagram which illustrates the characterisation of rational $G$-Spectra for $G$ profinite.
\begin{align*}
\xymatrix{\text{Rational $G$-spectra}\ar[d]^{\simeq}\\ \text{Ch}\left(\text{Fun}_{\text{Ab}}\left(\pi_0(\mathfrak{O}^{\mathbb{Q}}_G),\mathbb{Q}\text{-Mod}\right)\right)\ar[d]^{\cong}\\\text{Ch}\left(\text{Mackey}_{\mathbb{Q}}(G)\right)\ar[d]^{\cong}\\ \text{Ch}\left(\text{Weyl-}G\text{-Sheaf}_{\mathbb{Q}}(SG)\right)}
\end{align*}
where $\simeq$ represents a Quillen equivalence and $\cong$ is an equivalence of categories. We can calculate the injective dimension of chain complexes of Weyl-$G$-sheaves of $\mathbb{Q}$-modules over $SG$ in terms of the Cantor-Bendixson rank of $SG$, if the rank is infinite or if it is finite and $SG$ has trivial perfect hull.
\section*{Future Work}
\subsection*{Brown representability}
In this thesis we will see that the category of Weyl-$G$-sheaves of $\mathbb{Q}$-modules over $SG$ algebraically models $G$-spectra. If we can prove that the Brown representability theorem holds when $G$ is profinite, then this work will prove that all $G$-cohomology theories are determined by this algebraic model. 
\subsection*{Injective dimension of sheaves of $R$-modules}
In the calculation of the injective dimension of sheaves we used the fact that every $\mathbb{Q}$-module is injective. In general this argument extends to sheaves of $R$-modules if $R$ is a ring such that every $R\text{-module}$ is injective. We could investigate whether this result holds for sheaves of $R$-modules more generally. Furthermore, the characterisation of the injective dimension of sheaves in terms of the Cantor-Bendixson rank has one remaining case to consider, namely when $X$ is a non-scattered space with finite Cantor-Bendixson rank. Future work could involve resolving this remaining question. 
\subsection*{Equivariant Adams spectral sequence}
One advantage of this work comes from our understanding of the image of a $G$-spectrum $X$ as a Weyl-$G$-sheaf, $\Theta X$. In passing through the various Quillen equivalences we have to fibrantly and cofibrantly replace the images of the relevant Quillen adjoints from \cite[Theorem 5.1.1]{StableModel} in order to construct $\Theta X$, which can be difficult. To aid calculation, we would like to construct a spectral sequence whose input would be maps in the algebraic model and output maps in the homotopy category of $G$-spectra. The injective dimension calculation of Theorem \ref{Summary} will hopefully indicate that this spectral sequence collapses at some stage in good cases. This is a similar method to that used in \cite{BarZp} using the Adams spectral sequence. Another interesting problem to address could be calculating the Galois cohomology in the algebraic model of $\pi_*(X)$, of a given $G$-spectrum $X$.
\subsection*{Change of group functors}
While we have shown that $G$-spectra and chain complexes of Weyl-$G$-Sheaves are Quillen equivalent we have not considered the change of group functors. If $\iota:H\rightarrow G$ is an inclusion of a closed subgroup $H$ into $G$, we have functors:
\begin{align*}
\iota^*:\text{Weyl-} \,G\text{-Sheaf}_{\mathbb{Q}}(SG)&\rightarrow \text{Weyl-}\,H\text{-Sheaf}_{\mathbb{Q}}(SH) \\ \iota^*:G\text{-spectra}&\rightarrow H\text{-spectra}.
\end{align*}
The first extends to a functor on chain complexes. This gives us a functor between the algebraic models for $G$-spectra and that of $H$-spectra. We can then ask if these functors gives us a commuting square on homotopy categories. Similarly by considering $\varepsilon:G\rightarrow G/N$ we have a functor between the algebraic models for $G/N$-spectra and those for $G$-spectra. We can ask if we have a commuting square on homotopy categories in this case also. 

If we consider the group homomorphism $\iota$ above, we have a pair of adjunctions between the categories of $G$-spectra and $H$-spectra given by:
\begin{align*}
(G\underset{H}{\wedge}(-),\iota^*) \text{ and } (\iota^*,F_H\left(G,-\right)),
\end{align*}
where $\iota^*$ is restriction along $\iota$. We could ask if we get an adjunction between the algebraic models and if they are compatible in the following sense. We could ask if the following square of functors commutes: 
\begin{align*}
\xymatrix{G\text{-spectra}\ar@{~>}[r]\ar[d]^L&\ch\left(\text{Weyl-}G\text{-sheaf}\right)\ar[d]^{L^{\prime}}\\H\text{-spectra}\ar@{~>}[r]&\ch\left(\text{Weyl-}H\text{-sheaf}\right)}
\end{align*}
where $L$ and $L^{\prime}$ are the corresponding left adjoints and the horizontal arrows represent zig-zags of Quillen equivalences. We could equivalently ask if the square with the right adjoints commute.

Similarly we could ask a similar question by considering the quotient group homomorphism $\varepsilon$ as above we have a pair of adjunctions between $G$-spectra and $H$-spectra:
\begin{align*}
((-)/N,\varepsilon^*)\text{ and }(\varepsilon^*,(-)^N),
\end{align*}
where $\varepsilon^*$ is restriction along $\varepsilon$, $(-)/N$ is the orbit functor and $(-)^N$ is the fixed point functor. We could ask which of the adjunctions are compatible with the algebraic models.
\subsection*{Model structures}
As a consequence of Proposition \ref{Weyladjunct}, we have a functor from the category of $G$-equivariant sheaves over $SG$ to the full subcategory of Weyl-$G$-sheaves which is right adjoint to the functor which includes Weyl-$G$-sheaves into the category of $G$-equivariant sheaves. This is useful since we have an alternative way of constructing limits in the category of Weyl-$G$-sheaves. This is done by constructing the limit in the category of $G$-equivariant sheaves and then applying the right adjoint functor described. The other method involves using the equivalence to calculate the limit in the category of Mackey functors. Since calculating products of Weyl-$G$-sheaves involves first calculating products in the larger category of $G$-sheaves, it is more convenient to work in the category of $G$-sheaves despite the fact that it is a much larger category.

The adjunction extends to give an adjoint pair between $G$-equivariant sheaves and $G$-Mackey functors since the latter is equivalent to Weyl-$G$-sheaves. This leaves open the possibility of giving a model structure to the category of chain complexes of $G$-equivariant sheaves by utilising the projective model structure on the category of $G$-Mackey functors (or equivalently Weyl-$G$-sheaves). By the above comment about the calculation of categorical products, understanding the homotopy theory of $G$-spectra in terms of $G$-sheaves would be more convenient than doing so in terms of Weyl-$G$-sheaves. This is especially true if we can approach the model category of chain complexes of $G$-sheaves in a simpler and more localised manner. 

In particular if we start with the model category of chain complexes of $G$-sheaves over $SG$, we can apply right Bousfield localisation from Barnes and Roitzheim \cite[Definition 2.3]{BousfieldBarnes}, Hirschhorn \cite[Theorem 5.1.1]{Hirschhorn}. This means that we can choose a particular subcollection of objects $S$ in the category which generate the model structure in the following sense:
\begin{itemize}
\item  The weak equivalence are given to be the collection of maps $f:A\rightarrow B$ which satisfy that
\begin{align*}
f_*:\text{Map}\left(C,A\right)\rightarrow \text{Map}\left(C,B\right)
\end{align*}
is a weak equivalence of simplicial sets for every $C\in S$.
\item The fibrations remain the same as previously.
\end{itemize}
The main goal is to choose a collection of objects $S$ such that the right Bousfield localisation of the model category of chain complexes of $G$-sheaves satisfies the following:
\begin{align*}
S\text{-cell-}\text{Ch}\left(G\text{-sheaf}_{\mathbb{Q}}(SG)\right)\simeq \text{Ch}\left(\text{Weyl-}\,G\text{-sheaf}_{\mathbb{Q}}(SG)\right).
\end{align*}
We have an idea of what the objects of $S$ should be. If $H$ is a closed subgroup of $G$ we let $\left\lbrace U_{NH}\right\rbrace_{N}$ denote the closed-open neighbourhood basis of $H$ indexed by the open normal subgroups $N$. Since the Burnside ring is given by:
\begin{align*}
C(SG/G,\mathbb{Q})=\left[\mathbb{S},\mathbb{S}\right]^G=\left[\hat{f}\mathbb{S},\hat{f}\mathbb{S}\right]^G,
\end{align*}
we have a multiplicative idempotent given by the characteristic function on $U_{NH}$ which is represented by $\overline{e}_{NH}$ in the right hand side ring. If $N\leq N^{\prime}$ are open normal subgroups of $G$ we have a map of spectra:
\begin{align*}
\xymatrix{G/NH_{+}\ar[r]&G/N^{\prime}H}.
\end{align*}
We define:
\begin{align*}
e_HG/H_+=\underset{N\underset{open}{\unlhd}G}{\text{Hocolim}}\,e_{NH}G/NH_{+}
\end{align*}
and
\begin{align*}
S^{\prime}=\left\lbrace e_HG/H_+ \mid  H\underset{\text{closed}}{\leq}G \right\rbrace.
\end{align*}
We let $S$ be the image of the set $S^{\prime}$ under the above correspondence. A potentially useful consequence is that the algebraic model in terms of $G$-sheaves over $X$ is an abelian category which can be defined for any space $X$, whereas Weyl-$G$-sheaves are only defined for $X=SG$.
\subsection*{Monoidal product}
Another consequence of the right Bousfield localisation approach is that by expressing the homotopy theory in this manner, we open the possibility of giving a monoidal algebraic model for $G$-spectra. The category of rational $G$-spectra has a monoidal product given by the smash product, $\wedge$. If we can define a monoidal product $\otimes$ on $G$-sheaves of $\mathbb{Q}$-modules we could ask if there is a new approach which compares these monoidal categories. Explicitly, this question would be addressed by constructing a symmetric monoidal Quillen equivalence. The difficulty here is that the equivalence between rational $G$-spectra and chain complexes of rational Mackey functors is not well suited to monoidal considerations.

\chapter{Basic Definitions}
In this chapter we will explore some of the key concepts required in this thesis. In the first section we will define a model category and examine some of the key properties of this idea. In particular we will see that the model structure is determined by two out of the three classes of morphisms, Lemma \ref{determine}, and how we can compare different model structures. In the second section we will focus on sheaf theory. We will define a sheaf of $\mathbb{Q}$-modules over a space $X$ as well as a sheaf space over $X$, and we will see how these are related, Theorem \ref{nonequivariantsheafequivdefn}. We will also define the limits and colimits in the category of sheaves, Construction \ref{nonequivsheafconstruct}. In the final section we will study homological algebra. We will define an abelian category and an injective resolution, observing that some abelian categories have an injective resolution for every object, Proposition \ref{enoughinject}. We will define the injective dimension of such an abelian category and see how the concept of a right derived functor can used to calculate this, Theorem \ref{extcharacter}. 
\section{Model Categories}
In this section we aim to understand some useful definitions and theory surrounding model categories.
\begin{definition}
A \textbf{model structure} on a category $\mathfrak{C}$, which has all small limits and  small colimits, consists of three classes of morphisms of $\mathfrak{C}$ called weak equivalences, cofibrations $\rightarrowtail$ and fibrations $\twoheadrightarrow$. They satisfy the following axioms:
\begin{enumerate}
\item The two out of three axiom which says that if $f=g\circ h$ and any two out of three of $f,g$ and $h$ are weak equivalences then the third is also.
\item The three classes are closed under retracts, that is if $f$ belongs to one of the three classes and $g$ is a retract of $f$ then $g$ also belongs to that class.
\item Every cofibration has the left lifting property with respect to acyclic fibrations (fibrations which are also weak equivalences). This means that every possible commuting square of the following form has a diagonal lifting
\begin{center}
$\xymatrix{A\ar[r]\ar@{ >->}_{\iota}[d]&X\ar@{->>}^p[d]\\B\ar@{-->}^{\exists}[ur]\ar[r]&Y}$
\end{center}
where $\iota$ is a cofibration and $p$ an acyclic fibration. Similarly we also require that every acyclic cofibration has the left lifting property with respect to fibrations. 
\item We also require that every morphism $f$ can be factored as $f=p\circ \iota$ where $p$ is a fibration and $\iota$ is an acyclic cofibration. Similarly we require the ability to factor $f$ as a composite of a cofibration with an acyclic fibration. We call this the factorisation axiom.
\end{enumerate}
A \index{model category} \textbf{model category} is a category $\mathfrak{C}$ with all small limits and colimits, paired with a model structure.

\end{definition}
\begin{remark}
Every model category has an initial object and a terminal object by definition. This is because the initial object is the colimit of the empty diagram and the terminal object is the limit of the empty diagram and by assumption all small limits and small colimits exist.

If the unique morphism from the initial object into an object $X$ is a cofibration we say that $X$ is a cofibrant object. Dually we have the same definition for fibrant objects.  
\end{remark}
Given any object $X$ in a model category we can use the factorisation axiom on the unique map from the initial object to $X$ to define a cofibrant replacement of $X$ as follows. We have the following triangle using the factorisation axiom to write the unique morphism as the composite of a cofibration and an acyclic fibration. 
\begin{center}
$\xymatrix{\emptyset\ar@{ >->}[d]\ar[r]&X\\QX\ar@{->>}[ur]_{\simeq}}$
\end{center}
We call $QX$ a \index{cofibrant replacement} cofibrant replacement of $X$ and sometimes denote it by $\hat{c}X$. We can also define a fibrant replacement of $X$ denoted $RX$ or $\hat{f}X$ in a similar fashion.

The following lemma will lead to a useful characterisation of the three classes of a model category.
\begin{lemma}\label{retract}
If $f=p\circ\iota$ is such that $f$ has the left lifting property with respect to $p$, then $f$ is a retract of $\iota$. Dually if $f$ has the right lifting property with respect to $\iota$ then $f$ is a retract of $p$. 
\end{lemma}
\begin{proof}
See \cite[Lemma 1.1.9]{Hovey}.
\end{proof}
Given a model category $\mathfrak{C}$, any two of the three classes (cofibrations, fibrations and weak equivalences) specify the third.
\begin{lemma}\label{determine}
In a model category $\mathfrak{C}$ the cofibrations are precisely the maps which have the left lifting property with respect to acyclic fibrations. Also the fibrations are precisely the maps which have the right lifting property with respect to acyclic cofibrations. 
\end{lemma}
\begin{proof}
This follows on from Lemma \ref{retract}, see \cite[Lemma 1.1.10]{Hovey} for more details.
\end{proof}
\begin{proposition}
The class of weak equivalences is determined by the classes of fibrations and cofibrations.
\end{proposition}
\begin{proof}
The set of maps of the form $p\circ \iota$ where $p$ is an acyclic fibration and $\iota$ is an acyclic cofibration is the set of weak equivalences. This follows directly from the two out of three axiom and the factorisation axiom.
\end{proof}
Therefore we have shown that any two classes of maps determine the third.
\begin{example}
An example of a model structure on the category of topological spaces is the Hurewicz model structure. This is the model structure where the weak equivalences are the homotopy equivalences, the fibrations are the maps which have the right lifting property with respect to all inclusions $\xymatrix{A\ar@{ >->}[r]&A\times [0,1],}$ and the cofibrations are then determined by the left lifting property.
\end{example}
\begin{example}
Another example of a model category can be observed by looking at the category of simplicial sets. A morphism of simplicial sets is a weak equivalence if its geometric realisation is a homotopy equivalence, the cofibrations are the monomorphisms and the fibrations are the Kan fibrations.

A Kan fibration is a map which has the right lifting property with repsect to all horn inclusions i.e with respect to the inclusions $\Lambda_k^n\rightarrow \Delta^n$ for $1\leq k\leq n$ where $\Lambda_k^n$ is the boundary of $\Delta^n$ minus the $k$th face.
\end{example}
The following lemma is a famous result about model categories called Ken Brown's Lemma\index{Ken Brown's Lemma}.
\begin{lemma}\label{KenBrown}
Let $\mathfrak{C}$ be a model category and $\mathfrak{D}$ be a category which has a class of weak equivalences satisfying the two out of three axiom. If $F:\mathfrak{C}\rightarrow \mathfrak{D}$ is a functor which sends acyclic cofibrations between cofibrant objects to weak equivalences, then $F$ preserves all weak equivalences between cofibrant objects. We also have the dual result for fibrant objects.
\end{lemma}
\begin{proof}
See \cite[Lemma 1.1.12]{Hovey}
\end{proof}
We now look at a special type of functor between model categories which is compatible with the model structure.
\begin{definition}
An adjoint pair of functors $F$ and $U$ is called a \textbf{Quillan adjunction}\index{Quillan adjunction} if $F$ preserves cofibrations and acyclic cofibrations between cofibrant objects, and $U$ preserves fibrations and acyclic fibrations between fibrant objects.

Let $\psi$ be the natural isomorphism in the definition of adjunction. Then the adjoint pair of functors above are called a \textbf{Quillen equivalence}\index{Quillen equivalence} if and only if for all cofibrant $X$ in $\mathfrak{C}$ and all fibrant $Y$ in $\mathfrak{D}$ we have that $f:FX\rightarrow Y$ is a weak equivalence if and only if $\psi(f):X\rightarrow UY$ is a weak equivalence.
\end{definition}
\section{Sheaf Theory}
We now define sheaves and presheaves of $\mathbb{Q}$-modules over a topological space $X$. For the next definition, we use the fact that a partially ordered set can be considered a category. The objects are the elements of the set and a morphism from an element $a$ to an element $b$ comes from the relation $a\leq b$.
\begin{definition}
Given a topological space $X$, let $\mathfrak{U}$ be the poset of all open subsets of $X$. Then a \textbf{presheaf}\index{presheaf} of $\mathbb{Q}$-modules over $X$ is a contravariant functor $F:\mathfrak{U}\rightarrow \mathbb{Q}\leftmod$.

A morphism of presheaves is a natural transformation of functors.
\end{definition}
For each $V\subseteq U$ the definition of a presheaf $F$ of $\mathbb{Q}$-modules specifies a restriction map $\varphi_{U,V}:F(U)\rightarrow F(V)$.

\begin{definition}
We define the \textbf{stalk} of a presheaf $F$ at a point $x$ to be the direct limit of $F(U)$ over the direct system of open neighbourhoods $U$ of $x$, i.e. $F_x=\underset{U\backepsilon x}{\colim}\, F(U)$.
\end{definition}
We now define a sheaf of $\mathbb{Q}$-modules over a topological space $X$.
\begin{definition}\label{sheaf}
A \textbf{sheaf}\index{sheaf} of $\mathbb{Q}$-modules over a topological space $X$ is a presheaf $F$ satisfying the following two conditions:
\begin{enumerate}
\item If $\left\lbrace U_{\lambda} \mid \lambda \in \Lambda\right\rbrace$ is an open covering of any open set $U$ and $s_1,s_2 \in F(U)$ with $\varphi_{U,U_{\lambda}}(s_1)=\varphi_{U,U_{\lambda}}(s_2)$ for every $\lambda$, then $s_1=s_2$.
\item Let $\left\lbrace U_{\lambda}\mid {\lambda \in \Lambda}\right\rbrace$ be an open covering of an open set $U$. If a family $(s_{\lambda})_{\lambda\in\Lambda}\in \underset{\lambda\in \Lambda}{\prod}F(U_{\lambda})$ satisfies that for every pair $\mu$ and $\lambda$ we have $\varphi_{U_{\lambda},U_{\mu}\cap U_{\lambda}}(s_{\lambda})=\varphi_{U_{\mu},U_{\mu}\cap U_{\lambda}}(s_{\mu})$, then there is an $s\in F(U)$ such that $\varphi_{U,U_{\lambda}}(s)=s_{\lambda}$.
\end{enumerate}
A morphism of sheaves is a morphism of the underlying presheaves.
\end{definition}
If a presheaf satisfies condition $1$ of the definition then we say that it is a \textbf{monopresheaf}. The following lemma gives an alternative characterisation of a sheaf, as seen in \cite[pp. 15]{Tennison}.
\begin{lemma}
A presheaf $F$ is a \textbf{sheaf} if and only if the following diagram is an equaliser diagram:
\begin{center}
$\xymatrix{F(U)\ar[r]|-h&\underset{\lambda\in\Lambda}{\prod}F(U_{\lambda})\ar@<4pt>[r]|(.4){g}\ar@<-4pt>[r]|(.4){f}&\underset{(\lambda,\mu)\in\Lambda\times\Lambda}{\prod}F(U_{\lambda}\cap U_{\mu})}$,
\end{center}
where $U$ is an open subset and $\left\lbrace U_{\lambda} \mid \lambda \in \Lambda\right\rbrace$ is an open covering of $U$. The map $h$ sends $s$ to $(\varphi_{U,U_{\lambda}}(s))_{\lambda\in\Lambda}$, $f$ is defined by 
\begin{align*}
(s_{\lambda})_{\lambda\in\Lambda}\mapsto(\varphi_{U_{\lambda},U_{\mu}\cap U_{\lambda}}(s_{\lambda}))_{(\lambda,\mu)\in \Lambda\times\Lambda}
\end{align*}
and $g$ is defined by
\begin{align*}
(s_{\lambda})_{\lambda\in\Lambda}\mapsto (\varphi_{U_{\mu},U_{\mu}\cap U_{\lambda}}(s_{\mu}))_{(\lambda,\mu)\in \Lambda\times\Lambda}.
\end{align*}
\end{lemma}
\begin{definition}
If $p:X\rightarrow Y$ is a map of spaces, we say that $p$ is a \textbf{local homeomorphism} if each $x\in X$ has an open neighbourhood $U_x$ such that $p_{|_{U_x}}$ is a homeomorphism onto $p(U_x)$.
\end{definition}
We will give an alternative definition of a sheaf in Theorem \ref{nonequivariantsheafequivdefn}. We now work towards explaining this theorem.
\begin{definition}\label{sheafspacedefn}
A \textbf{sheaf space}\index{sheaf space} over a topological space $X$ is a pair $(E,p)$ where:
\begin{enumerate}
\item $E$ is a topological space and $p:E\rightarrow X$ is a continuous local homeomorphism,
\item for every $x\in X$, the stalk $p^{-1}(x)$ has a $\mathbb{Q}$-module structure whose addition and unit extend to a continuous addition of sections and a zero section.
\end{enumerate}
\end{definition}
To explain what we mean by the second point, if $U$ is an open subset of $X$ then we can define the following set:
\begin{align*}
\Gamma \left(U,E\right)=\left\lbrace s:U\rightarrow E\mid s\,\text{ is continuous}\,\text{and}\,p\circ s=\id \right\rbrace
\end{align*}
which we sometimes denote by $E(U)$. Since each $p^{-1}(x)$ is a $\mathbb{Q}$-module, if $s,t\in E(U)$, we can define a function by point-wise addition $s+t$ which may not be continuous in general. We say the $\mathbb{Q}$-module operations are continuous with respect to $E$ if a function of the form $s+t$ is continuous. It follows that sets of the form $E(U)$ have a $\mathbb{Q}$-module structure.

To a presheaf $F$ we can apply a process called \textbf{sheafification}\index{sheafification} to obtain a sheaf from $F$. We first define the sheaf space over the presheaf $F$ whose underlying set is given by $\underset{x\in X}{\coprod}\,F_x$, and a basis for the topology is given by sets of the form $\left\lbrace s_x\mid x \in U\right\rbrace$ for each open set $U$ and $s \in F(U)$. Here $s_x$ represents the equivalence class of $s$ in $F_x$, called the germ of $s$ at $x$. 

Note that as a requirement from the definition of a sheaf space we must have a local homeomorphism from $LF$ to $X$. The projection $p:LF\rightarrow X$ is a local homeomorphism. We denote the sheafification of $F$ by $\Gamma LF$ and this is defined to be the contravariant functor which is defined by:
\begin{align*}
U \mapsto \Gamma LF(U),
\end{align*}
where
\begin{align*}
\Gamma LF(U)=\left\lbrace \alpha:U\rightarrow LF\mid p\circ\alpha=\id,\,\alpha\,\,\text{is continuous}\right\rbrace.
\end{align*}
This is a sheaf as seen in \cite[Definition 2.4.1]{Tennison}. In particular this shows that every sheaf determines a sheaf space.
\begin{theorem}
If $X$ is a topological space then every sheaf of $\mathbb{Q}$-modules over $X$ determines a sheaf space.
\end{theorem}
\begin{proof}
If $F$ is a sheaf of $\mathbb{Q}$-modules over $X$ then since every sheaf is a presheaf we can apply the sheafification process to $F$ to construct a sheaf space.  
\end{proof}
We will see that if a sheaf space is determined by a sheaf, then we can recover the original sheaf in Lemma \ref{2}. Sheafification satisfies the following universal property:
\begin{theorem}\label{univthm}
If $F$ is a presheaf and $G$ a sheaf then any morphism of \\presheaves $f:F\rightarrow G$ factors uniquely as follows:
\begin{center}
$\xymatrix{F\ar[d]\ar[r]^f&G\\ \Gamma LF\ar@{-->}[ur]^{\exists !}}$
\end{center}
\end{theorem}
\begin{proof}
See \cite[Theorem 2.4.2]{Tennison}.
\end{proof}
The following three results are used to obtain the universality of sheafification. We record them since the results themselves are instructive.
\begin{lemma}\label{2}
If $F$ is a presheaf, then $F$ is a sheaf if and only if $\Theta:F\rightarrow \Gamma LF$ is an isomorphism of sheaves, where:
\begin{align*}
\Theta(U):F(U)&\rightarrow \Gamma LF(U)\\s&\mapsto \hat{s}
\end{align*}
where $\hat{s}(x)=s_x$.
\end{lemma}
\begin{proof}
See \cite[Lemma 2.4.3]{Tennison}.
\end{proof}
\begin{proposition}
If $(E,p)$ is a sheaf space over $X$ and $\Gamma E$ is the sheaf of sections over $E$, then $\Gamma E_x=p^{-1}(x)$.
\end{proposition}
\begin{proof}
\cite[See Proposition 2.3.6]{Tennison}.
\end{proof}
We have the following connection between the stalks of a presheaf and the stalks of the sheafification of the presheaf.
\begin{lemma}
For any presheaf $F$ we have an isomorphism $F_x\cong \Gamma LF_x$ for every $x \in X$.
\end{lemma}
\begin{proof}
See \cite[Lemma 2.4.5]{Tennison}.
\end{proof}
Theorem \ref{univthm} has the following main application:
\begin{theorem}\label{Limsheaf}
Let $U$ be the forgetful functor from the category of sheaves to the category of presheaves. Then the pair of functors $\left(\Gamma L(-),U(-)\right)$ are an adjoint pair.
\end{theorem}
We now define a morphism of sheaf spaces.
\begin{definition}
A morphism of sheaf spaces $f:\left(E,p\right)\rightarrow \left(E^{\prime},p^{\prime}\right)$ is a continuous map of spaces $f:E\rightarrow E^{\prime}$ such that $p=p^{\prime}\circ f$ and that the restriction of $f$ to each stalk is a map of $\mathbb{Q}$-modules.
\end{definition}
\begin{theorem}\label{1}
If $E$ is a sheaf space, then there is an isomorphism of sheaf spaces $E\cong L\Gamma E$.
\end{theorem}
\begin{proof}
See Theorem \cite[Theorem 2.3.10]{Tennison}.
\end{proof}
Consequently we now have a characterisation of sheaves of $\mathbb{Q}$-modules over a space.
\begin{theorem}\label{nonequivariantsheafequivdefn}
The category of sheaves of $\mathbb{Q}$-modules over $X$ is equivalent to the category of sheaf spaces of $\mathbb{Q}$-modules over $X$.
\end{theorem}
\begin{proof}
This is proved by considering Theorem \ref{1} and Lemma \ref{2} . 
\end{proof}
We shall now take some time to understand the construction of limits and colimits in the category of sheaves of $\mathbb{Q}$-modules. We begin by looking at the category of presheaves.
\begin{construction}
Let $F_i$ be a diagram of presheaves of $\mathbb{Q}$-modules for some indexing diagram $I$. We define the limit to be:
\begin{align*}
\left(\underset{I}{\lim}F_i\right)(U)=\underset{I}{\lim}\left(F_i(U)\right),
\end{align*} 
where $U$ is an open subset of $X$. Here we are simply taking the limit in the category of $\mathbb{Q}$-modules at each $U$. The restriction maps are obtained by applying the universal properties as follows:
\begin{align*}
\xymatrix{F_i(U)\ar[r]^{\rho^i_{U,V}}&F_i(V)\\\underset{I}{\lim}F_i(U)\ar[u]\ar@{-->}[r]_{\exists !}&\underset{I}{\lim}F_i(V)\ar[u]}
\end{align*}  
where $V\subseteq U$ are both open subsets of $X$. Colimits are constructed in a similar level-wise manner.
\end{construction}
The proof that these are the categorical limits and colimits follows easily from the fact that they are built using these constructions in in the category of $\mathbb{Q}$-modules.

In order to understand what form limits and colimits take in the category of sheaves, we use the fact that sheaves are a full subcategory of the category of presheaves. The question we ask is; if we have a diagram of sheaves, do the above colimit and limit constructions provide another sheaf? Answering this question involves using the adjoint functors in Theorem \ref{Limsheaf}.
\begin{construction}\label{nonequivsheafconstruct}
Let $F_i$ be a diagram of sheaves over some indexing diagram $I$. The limit in the category of presheaves is a sheaf, and this is the categorical limit in the category of sheaves. For the colimit construction, we obtain the categorical colimit by taking the sheafification of colimit construction in the category of presheaves. 
\end{construction} 
\begin{proof}
The reason that the given colimit structure works is because left adjoint functors preserve colimits and the sheafification functor is left adjoint to the forgetful functor as in Theorem \ref{Limsheaf}. We quickly verify that the limit construction is a sheaf and this will be sufficient since sheaves are a full subcategory of presheaves. We begin with the monopresheaf condition.

Suppose that $(s_i)_{I},(t_i)_I\in \underset{I}{\lim}F_i(U)$ and $\left\lbrace U_{\lambda}\right\rbrace_{\lambda\in\Lambda}$ is an open covering of $U$ satisfying that $\rho_{U,U_{\lambda}}\left((s_i)_I\right)=\rho_{U,U_{\lambda}}\left((t_i)_I\right)$ for each $\lambda$. Then it follows that for each $i\in I$ we have $\rho^i_{U,U_{\lambda}}\left(s_i\right)=\rho^i_{U,U_{\lambda}}\left(t_i\right)$. Applying the monopresheaf condition for each $F_i$ implies that $s_i=t_i$ for every $i\in I$. This gives the required equality.

For the gluing condition, suppose that we have $(s_i^{\lambda})_I\in \underset{I}{\lim}F_i(U_{\lambda})$ for each $\lambda$ satisfying that restrictions to all intersections agree. For each $i\in I$ we use that $F_i$ satisfies the gluing axiom, so that there exists $s_i$ such that $s_i^{\lambda}=\rho^i_{\lambda}(s_i)$ for each $\lambda$. We finish by showing that $(s_i)_{I}$ is an element of $\underset{I}{\lim}F_i(U)$.

We accomplish this by showing that if $\alpha_{ij}:F_j(U)\rightarrow F_i(U)$ is given by the morphism of sheaves in the diagram, then $\alpha_{ij}(s_j)=s_i$. Consider the following diagram:
\begin{align*}
\xymatrix{F_j(U)\ar[d]^{\rho^j_{\lambda}}\ar[r]^{\alpha}&F_i(U)\ar[d]^{\rho^i_{\lambda}}\\F_j(U_{\lambda})\ar[r]^{\alpha}&F_i(U_{\lambda})}
\end{align*}
We know that $(s_i^{\lambda})_I$ is in the limit for each $\lambda$, hence $\alpha(s_j^{\lambda})=s_i^{\lambda}$. We also know that morphisms of sheaves commute with restriction maps so: 
\begin{align*}
\rho_{\lambda}^i\left(\alpha(s_j)\right)=\alpha\left(\rho^j_{\lambda}(s_j)\right)=\alpha\left(s_j^{\lambda}\right)=s_i^{\lambda}=\rho_{\lambda}^i(s_i),
\end{align*}
and this holds for each $\lambda$. Applying the monopresheaf condition gives $\alpha(s_j)=s_i$ as required.
\end{proof}

An application of this construction shows that the category of presheaves over a topological space $X$ has kernels and cokernels, and these are defined objectwise. This means that if $f:F\rightarrow G$ is a morphism of presheaves then $\ker f$ is a presheaf where $(\ker f)(U)$ is given by $\ker f(U)$, with $f(U)$ being the map $f(U):F(U)\rightarrow G(U)$. The same holds for cokernels of morphisms of presheaves.

In general if $f:F\rightarrow G$ is a morphism of sheaves then the presheaf kernel of $f$ is a sheaf (see \cite[Proposition 3.3.3]{Tennison}). However the presheaf cokernel of a morphism of sheaves need not be a sheaf, so to calculate the sheaf cokernel of a morphism of sheaves we must take the sheafification of the presheaf cokernel. Later on when we construct injective resolutions of a sheaf this will be an issue that we will need to overcome.

The following proposition characterises the monomorphisms, epimorphisms and isomorphisms in the category of sheaves.
\begin{proposition}
A morphism of sheaves $f:F\rightarrow G$ is a monomorphism if and only if $f_x$ is injective for every $x \in X$. Equivalently a morphism $f$ is a monomorphism if and only if $f(U)$ is injective for every $U\subseteq X$ open.

A morphism of sheaves $f$ is an epimorphism of sheaves if and only if $f_x$ is surjective for every $x \in X$. A morphism $f$ is an isomorphism if and only if each $f_x$ is bijective.
\end{proposition}
\begin{proof}
For monomorphisms see \cite[Theorem 3.3.5]{Tennison}, for epimorphisms see \cite[Theorem 3.4.8]{Tennison} and for the isomorphisms see \cite[Theorem 3.4.10]{Tennison}.
\end{proof}
Note that only the characterisation of monomorphisms says anything about the maps $f(U)$ for $U\subseteq X$, which further underlines why we need to be more careful when calculating cokernels of sheaf morphisms.
\section{Injective Resolutions}
In this section we will look at the injective dimension of abelian categories and how to calculate them.
\begin{definition}
An \index{abelian category}\textbf{abelian category} is a category $\mathfrak{C}$ satisfying the following conditions:
\begin{enumerate}
\item For every pair of elements $c_1,c_2 \in \mathfrak{C}$, $\text{Hom}(c_1,c_2)$ has an abelian group and the composition of morphisms is bilinear.
\item $\mathfrak{C}$ has binary coproducts.
\item $\mathfrak{C}$ has a zero object.
\item $\mathfrak{C}$ has kernels and cokernels.
\item Every monomorphism in $\mathfrak{C}$ is a kernel of another morphism, and every epimorphism is a cokernel of another morphism.
\end{enumerate}
\end{definition}
\begin{theorem}
The category of sheaves of $\mathbb{Q}$-modules over $X$ is an abelian category.
\end{theorem}
\begin{proof}
See \cite[Theorem 3.5.5]{Tennison}.
\end{proof}
We now look at a property of abelian categories called injective dimension. We start off with the definition of an injective object.
\begin{definition}
An \textbf{injective object}\index{injective object} in a category $\mathfrak{C}$ is an object $I$ such that for every monomorphism $f:A\rightarrow B$ and every $g:A\rightarrow I$ there exists $h:B\rightarrow I$ such that the following diagram commutes:
\begin{center}
$\xymatrix{A\ar[r]^g\ar[d]^f &I\\B\ar@{-->}[ur]_{\exists h}}$
\end{center}
A category has enough injectives if every object has a monomorphism into an injective object.
\end{definition}
A useful property of this concept is that arbitrary products of injective objects are injective.
\begin{definition}
In an abelian category $\mathfrak{C}$, an \textbf{injective resolution}\index{injective resolution} for an object $X$ of $\mathfrak{C}$ is an exact sequence
\begin{center}
$\xymatrix{0\ar[r]&X\ar[r]^{\epsilon}&I_0\ar[r]^{f_0}&I_1\ar[r]^{f_1} &I_2\ar[r]^{f_2}&\ldots}$
\end{center} 
where each $I_k$ is an injective object of $\mathfrak{C}$.
\end{definition}
The following is a useful property of abelian categories with enough injectives.
\begin{proposition}\label{enoughinject}
If an abelian category $\mathfrak{C}$ has enough injectives, then every object in $\mathfrak{C}$ has an injective resolution.
\end{proposition}
\begin{proof}
Let $X$ be an object in an abelian category $\mathfrak{C}$ which has enough injectives. Then by the definition of enough injectives we have a monomorphism $\varepsilon:X\rightarrow I_0$ where $I_0$ is injective, and hence a short exact sequence of the form:
\begin{center}
$\xymatrix{0\ar[r]&X\ar[r]^{\varepsilon}&I_0\ar[r]^p&\coker\varepsilon\ar[r]&0}$
\end{center}
We again use the definition of enough injectives to obtain a monomorphism \\$\delta^{\prime}_0:\coker\varepsilon\rightarrow I_1$, so we have a diagram of the form:
\begin{center}
$\xymatrix{0\ar[r]&X\ar[r]^{\varepsilon}&I_0\ar[d]\ar@{-->}[r]^{\delta_0}&I_1\\&&\coker\varepsilon\ar[ur]_{\delta_0^{\prime}}}$
\end{center}
where $\delta_0=\delta_0^{\prime}\circ p$. We then build up the rest of the resolution inductively following the same procedure as before. The exactness property is easy to verify.
\end{proof}
\begin{definition}
The \textbf{injective dimension}\index{injective dimension} of an object $X$ (denoted by $\text{ID}(X)$) in an abelian category $\mathfrak{C}$ is the minimum positive integer $n$ (if it exists) such that there is an injective resolution of the form
\begin{center}
$\xymatrix{0\ar[r]&X\ar[r]^{\epsilon}&I_0\ar[r]^{f_0}&I_1\ar[r]^{f_1}&\ldots\ar[r]^{f_{n-1}} &I_n\ar[r]&0}$.
\end{center}
The injective dimension is infinite if no minimum integer exists. See \cite[Definition 4.1.1]{Weibel}.
\end{definition}
\begin{definition}
For an abelian category $\mathfrak{C}$, the injective dimension of $\mathfrak{C}$ is defined to be $\sup \left\lbrace \text{ID}(X)\mid X\in \mathfrak{C}\right\rbrace$. See \cite[Definition 4.1.1, Corollary 10.7.5]{Weibel}.
\end{definition} 
The injective dimension of an abelian category is a measure of the complexity of the category. We now consider some examples.  
\begin{example}
If $R$ is a field then every object in the category of $R$-Modules is injective. Then every module has injective resolution of length zero. It follows that the injective dimension of this category is equal to zero. This holds, for example, for $R=\mathbb{Q}$.
\end{example}
\begin{example}
If $G$ is a profinite group, then the category of discrete $\mathbb{Q}[G]$-modules satisfies that every object is injective. Therefore the injective dimension of this category is also zero. See \cite[Proposition 3.7]{Castellano}.
\end{example}
We will see more substantial examples in Section \ref{Sheafexample}. A morphism of chain complexes $C$ and $D$ over $\mathfrak{C}$ is a level-wise morphism of objects in $\mathfrak{C}$ which commute with differentials. We now make the following useful definition.
\begin{definition}
Let $\mathfrak{C}$ be an abelian category, and let $C$ and $D$ be chain complexes over $\mathfrak{C}$. If $f$ and $g$ are chain maps from $C$ to $D$, we define a \textbf{chain homotopy}\index{chain homotopy} between $f$ and $g$ to be a collection of maps in $\mathfrak{C}$:
\begin{align*}
s_n:C_n\rightarrow D_{n+1}
\end{align*}
such that $f-g=sd+ds$, where $d$ is the differential.
\end{definition}
\begin{construction}\label{Derived}
Let $F:\mathfrak{C}\rightarrow \mathfrak{D}$ be a left exact functor between abelian categories where $\mathfrak{C}$ has enough injectives. We define the right derived functor $RF^i$ for $i\geq 0$ as follows. We can construct an injective resolution $A\rightarrow I_*$ for any object $A$ of $\mathfrak{C}$. Applying the functor $F$ and omitting the first term gives a complex:
\begin{align*}
\xymatrix{0\ar[r]&F(I_0)\ar[r]^d&F(I_1)\ar[r]^d&F(I_2)\ar[r]^d&\ldots\ar[r]^d&F(I_k)\ar[r]^d&F(I_{k+1})\ar[r]^d&\ldots}
\end{align*}
We define $R^iF(A)$ to be $H^i(F(I_*))$.
\end{construction}
\begin{remark}
Since $F$ is left exact it follows that the following sequence is exact:
\begin{align*}
\xymatrix{0\ar[r]&F(A)\ar[r]&F(I_0)\ar[r]&F(I_1)}.
\end{align*}
Therefore $R^0F(A)\cong F(A)$.
\end{remark}
\begin{example}
Let $A$ be any object in an abelian category $\mathfrak{C}$. Then we have a functor $\text{Hom}(A,-)$ which is left exact. We define $\text{Ext}$ group functor by:
\begin{align*}
\text{Ext}^i(A,B)=R^i\text{Hom}(A,-)(B).
\end{align*} 
\end{example}
\begin{theorem}
Let $\mathfrak{C}$ be an abelian category with enough injectives and $A$ be any object of $\mathfrak{C}$. Then injective resolutions for $A$ are unique upto chain homotopy. 
\end{theorem}
\begin{proof}
This follows from \cite[Theorem 2.3.7]{Weibel}.
\end{proof}
In particular this means that calculating cohomology of an object is independent of the choice of injective resolution used. It is for this reason that Construction \ref{Derived} is independent of the choice of injective resolution used. The concept of an $\text{Ext}$ group functor has the following useful characterisation.
\begin{theorem}\label{extcharacter}
Let $B$ be an object in an abelian category $\mathfrak{C}$. Then the injective dimension of $B$ is less than or equal to a natural number $n$ if and only if $\Ext^i(A,B)=0$ for every object $A$ and $i> n$. The injective dimension equals $n$ if in addition, there exists some $A$ such that $\Ext^n(A,B)\neq 0$.
\end{theorem}
\begin{proof}
See \cite[Corollary 10.7.5, Exercise 10.7.2 and Lemma 4.1.8]{Weibel}.
\end{proof}
This result is formally used in the proof of Theorem \ref{Summary}.
\chapter{Characterising Mackey Functors}
In this chapter we will be interested in the category of rational Mackey functors of profinite groups. For the first section we will be searching for useful chacterisations of Mackey functors. We will also see an inflation functor between Mackey functor categories, Proposition \ref{inflatemackey}. In the second section we will define the Burnside ring of a profinite group and explore the various characterisations of this ring. In particular, we will define the space of closed subgroups of $G$ and see how this determines the Burnside ring, Theorem \ref{burnchar} and Corollary \ref{Burnsidecolim}. In the final section we will see how the Burnside ring interacts with an arbitrary Mackey functor, Proposition \ref{mackeyact}, and consequently how a Mackey functor determines a collection of sheaves, Corollary \ref{MHsheaf}. We will characterise the multiplicative idempotents of the Burnside ring, Proposition \ref{idempotentform}, and show how these interact with the restriction, conjugation and induction maps of a Mackey functor, Proposition \ref{restrictprop} and \ref{inductprop}. 
\section{Mackey Functors over a Profinite Group}
In this section we will look at some equivalent definitions of Mackey functors. For the following we let $\Grpf$ denote the collection of all closed subgroups of a profinite group $G$ of finite index (or equivalently the collection of all open subgroups of $G$). The following definition is from \cite[Definition 2.2.12]{Thiel}.
\begin{definition}\label{defM1}
Let $G$ be a profinite group, then a rational \textbf{$G$-Mackey functor}\index{Mackey functor} $M$ is a collection of data:
\begin{itemize}
\item A collection of $\mathbb{Q}$-modules $M(H)$ for each subgroup $H$ in $\Grpf$.
\item For each $K\leq H$ with $K,H\in \Grpf$ we have the following maps of $\mathbb{Q}$-modules; a \textbf{restriction} map
\begin{align*}
R^H_K:M(H)\rightarrow M(K),
\end{align*}
an \textbf{induction} map
\begin{align*}
I^H_K:M(K)\rightarrow M(H),
\end{align*}
and a map called \textbf{conjugation} for each $g \in G$ given by:
\begin{align*}
C^H_g:M(H)\rightarrow M(gHg^{-1}).
\end{align*}
These maps satisfy the following conditions:
\end{itemize}

\begin{enumerate}
\item $R^H_H=\id_{M(H)}=I^H_H$ for each $H\leq G$, and $C^H_h=\id_{M(H)}$ for each $h\in H$.
\item If $L\leq K\leq H$ with all three subgroups in $\Grpf$, then $I^H_L=I^H_K\circ I^K_L$, $R^H_L=R^K_L\circ R^H_K$ and for $g,h\in G$ we have $C_{gh}=C_g\circ C_h$. These are the \textbf{transitivity condition} and \textbf{associativity condition} respectively.
\item If $g\in G$ and $K\leq H$ then:
\begin{align*}
R^{gHg^{-1}}_{gKg^{-1}}\circ C_g=C_g\circ R^H_K \,\text{ and }\,I^{gHg^{-1}}_{gKg^{-1}}\circ C_g=C_g\circ I^H_K. 
\end{align*}
This is the \textbf{equivariance condition}.
\item If $K,L\leq H$ with all three in $\Grpf$, then: 
\begin{align*}
R^H_K\circ I^H_L=\sum_{x\in[K\diagup H \diagdown L]} I^K_{K\cap xLx^{-1}}\circ C_x\circ R^L_{L\cap x^{-1}Kx}.
\end{align*}
This condition is called the \textbf{Mackey axiom}.
\end{enumerate}
\end{definition}
We now look at the following example.
\begin{example}
Consider the profinite group given by the $p$-adic integers, $\mathbb{Z}_p$. The open subgroups in this case are of the form $p^k\mathbb{Z}_p$ for $k\geq 0$. We define a Mackey functor by the following assignment:
\begin{align*}
p^k\mathbb{Z}_p\mapsto \mathbb{Q}
\end{align*}
where all of the restriction maps and conjugation maps are taken to be the identity. We finish by defining the induction maps.

We begin by observing that an open subgroup $p^l\mathbb{Z}_p$ is contained inside another $p^k\mathbb{Z}_p$ if and only if $l\geq k$. We define induction maps as follows:
\begin{align*}
I^{p^k\mathbb{Z}_p}_{p^l\mathbb{Z}_p}=\,\text{Multiplication by}\,p^{l-k}.
\end{align*} 
\end{example}
\begin{definition}
Let $G$ be a profinite group. We let $G$-Set denote the category of finite $G$-sets whose group action is continuous. The morphisms in this category are the equivariant $G$-maps. 
\end{definition}
\begin{lemma}
If $X$ is a finite $G$-set then $\stabgx(x)$ is open for each $x\in X$.
\end{lemma}
\begin{proof}
We know that the action map:
\begin{align*}
G\times\left\lbrace x\right\rbrace\rightarrow X
\end{align*}
is continuous. Since the set $\left\lbrace x\right\rbrace$ is open and closed in $X$ it follows that the preimage is open and closed in $G$. This preimage is precisely $\stabgx(x)$ as required. 
\end{proof}
\begin{remark}
For more general $G$-spaces this holds if $x$ is an isolated point. Also we can use that closed subgroups of finite index are open to deduce this when $x$ has finite orbit. This is because $|Gx|=|G/\stabgx(x)|$.
\end{remark}
When we consider finite $G$-sets it is important to note that we are always considering the discrete topology. 
\begin{corollary} 
If $X$ is an element of $G\gset$ as described above, then $X$ is a finite disjoint union of orbits $G/H$ where $H$ is an open subgroup of $G$.
\end{corollary} 
\begin{proof}
We begin by defining an equivalence relation on $X$ where $x\sim y$ if and only if $y=gx$, so that $X=\underset{1\leq i\leq n}{\coprod}Gx_i$. We are using that $X$ is finite to deduce that there are only finitely many orbits. Using the previous lemma we have that $X=\underset{1\leq i\leq n}{\coprod}G/\stabgx(x_i)$ and that each $\text{stab}_G(x_i)$ is open as required.
\end{proof}
In particular this shows that $G$-Set is generated by the orbits $G/H$ where $H$ is an open subgroup, in the sense that they provide an additive basis for $G$-Set. Similar to studying $G$-Mackey functors when $G$ is finite, we have the following alternative definition which we will prove later is equivalent to the first. This is from \cite[Definition 1]{Linder}.
\begin{definition}\label{defM2}
Let $\mathfrak{C}$ be a category with pullbacks and finite coproducts, and $\mathfrak{D}$ be an abelian category. A categorical Mackey functor\index{Mackey functor} is a bifunctor
\begin{center}
$M=(M_*,M^*):\mathfrak{C}\rightarrow \mathfrak{D}$
\end{center}
satisfying two conditions. By bifunctor we mean a covariant functor $M_*$ and a contravariant functor $M^*$ which agree on objects, so we denote by $M(X)$ the common value.

The two conditions are:
\begin{enumerate}
\item For every pullback diagram in $\mathfrak{C}$: 
\begin{align*}
\xymatrix{X_1\ar[r]^{\alpha}\ar[d]^{\beta}&X_2\ar[d]^{\gamma}\\X_3\ar[r]^{\delta}&X_4}
\end{align*}
we have the equality:
\begin{align*}
M^*(\delta)\circ M_*(\gamma)=M_*(\beta)\circ M^*(\alpha).
\end{align*}
\item The following coproduct diagram $\xymatrix{X\ar[r]&X\coprod Y&Y\ar[l]}$ defines an isomorphism $M(X\coprod Y)\cong M(X)\bigoplus M(Y)$ via both $M_*$ and $M^*$.
\end{enumerate}
\end{definition}
A Mackey functor of $\mathbb{Q}$-modules for $G$ is the above definition restricted to the case $\mathfrak{C}=G$-Set and $\mathfrak{D}=\mathbb{Q}$-Mod. We will prove this equivalence but in order to do this we first need to look at a couple of lemmas.
\begin{lemma}\label{pullbackorbit}
In the category $G\gset,$ consider any pullback diagram:
\begin{align*}
\xymatrix{\Omega\ar[r]\ar[d]&A\ar[d]\\C\ar[r]&B.}
\end{align*}
Then $\Omega$ is given by a disjoint union of pullbacks of diagrams whose objects are orbits.
\end{lemma}
\begin{proof}
We know that we can write $A=\underset{1\leq i\leq n_a}{\coprod}G/K_i$, $B=\underset{1\leq j\leq n_b}{\coprod}G/H_j$ and $C=\underset{1\leq k\leq n_c}{\coprod}G/L_k$. This description of the $G$-sets shows that our diagram boils down to:
\begin{center}
$\xymatrix{\Omega\ar[r]\ar[d]&\underset{1\leq i\leq n_a}{\coprod}G/K_i\ar[d]\\\underset{1\leq k\leq n_c}{\coprod}G/L_k\ar[r]&\underset{1\leq j\leq n_b}{\coprod}G/H_j}$
\end{center}
We know that $\Omega=\left\lbrace (a,c)\in A\times C\mid P_C(c)=P_A(a)\right\rbrace$, where $P_A,P_C$ are the relevant projections. First note that if there exists a $G/H_j$ in the decomposition of $B$ which has no element from the decomposition of $A$ and of $C$ mapping into it, then $G/H_j$ does not factor into the pullback by definition. Therefore we can assume that each $G/H_j$ is mapped into by both $A$ and $C$.

For each $G/H_j$ we can pick out the orbits in the decomposition of $A$ and $C$ which both map into $G/H_j$, so we have $\underset{1\leq i\leq n_a^j}{\coprod}G/K_i\subseteq A$ and $\underset{1\leq k\leq n_c^j}{\coprod}G/L_k\subseteq C$, and hence for each $1\leq j\leq n_b$ we have:
\begin{center}
$\xymatrix{\Omega^j\ar[r]\ar[d]&\underset{1\leq i\leq n^j_a}{\coprod}G/K_i\ar[d]\\\underset{1\leq k\leq n^j_c}{\coprod}G/L_k\ar[r]&G/H_j}$
\end{center}
By definition $\Omega^j=\left\lbrace (xK_i,yL_k)\mid xH_j=yH_j\right\rbrace$. Now look at the pullback diagram:
\begin{center}
$\xymatrix{\Omega_{i,k}^j\ar[r]\ar[d]&G/K_i\ar[d]\\G/L_k\ar[r]&G/H_j}$
\end{center}
where $1\leq j\leq n_b$ and $i,k$ are the indices of subgroups $K_i$ and $L_k$ which contribute to $\Omega^j$ as defined above. We set $P=\underset{1\leq j\leq n_b}{\coprod}\underset{i,k}{\coprod}\Omega^j_{i,k}$.

We define a map $\psi:P\rightarrow \Omega$ by $(aK_i,cL_k)\mapsto (aK_i,cL_k)$. This is well defined, injective and surjective. This is also $G$-equivariant since both the cartesian product and pullback construction have $G$-action given by the diagonal. This is obviously a homeomorphism since both the domain and codomain are finite discrete sets.
\end{proof}
\begin{lemma}\label{lemmaPB}
Given a pullback diagram in $G\gset$:
\begin{align*}
\xymatrix{P\ar[r]\ar[d]&G/K\ar[d]\\G/H\ar[r]&G/J}
\end{align*}
where $H,K\leq J$, then $P\cong G\times_J\left(J/H\times J/K\right)$.
\end{lemma}
\begin{proof}
Firstly if $(g_1H,g_2K)$ belongs to $P$, then by definition $g_1J=g_2J$ so $g_1^{-1}g_2\in J$. It then follows that $g_1j=g_2$, so that $(g_1H,g_2K)=(g_1H,g_1jK)$. Therefore $P=\left\lbrace (gH,gjK)\mid g \in G,j\in J\right\rbrace$.

We now define a map
\begin{align*}
\psi:G\times_J\left(J/H\times J/K\right)&\rightarrow P,\\
(g,(j_1H,j_2K))&\mapsto (gj_1H,gj_2K)=(gj_1H,gj_1(j_1^{-1}j_2)K).
\end{align*}

To see that it is well defined observe the following: 
\begin{align*}
\psi(g,(j_1H,j_2K))=(gj_1H,gj_2K)=(gj_1H,gj_1(j_1^{-1}j_2)K)=\psi(gj_1,(eH,j_1^{-1}j_2K)).
\end{align*}
To see surjectivity, take any $(gH,gjK)$ in $P$, then $(g,(eH,jK))$ is in the preimage. For injectivity, suppose we have $(g_1,(j_1H,j_2K))$ and $(g_2,(a_1H,a_2K))$ with $(g_1j_1H,g_1j_2K)=(g_2a_1H,g_2a_2K)$. It follows from this equality that $g_1j_1=g_2a_1h$ for some $h\in H$, hence $g_1=g_2a_1hj_1^{-1}$. From this we show that $a_1hj_1^{-1}j_2K=a_2K$.

We know by assumption that $j_2^{-1}g_1^{-1}g_2a_2$ belongs to $K$ and therefore equals some $k\in K$. Substituting the value for $g_1$ stated earlier we get that:
\begin{align*}
k=j_2^{-1}g_1^{-1}g_2a_2=j_2^{-1}(j_1h^{-1}a_1^{-1}g_2^{-1})g_2a_2=j_2^{-1}j_1h^{-1}a_1^{-1}a_2
\end{align*}
hence $j_2^{-1}j_1h^{-1}a_1^{-1}a_2\in K$, proving that $a_1hj_1^{-1}j_2K=a_2K$. Therefore we have the following:
\begin{align*}
(g_1,(j_1H,j_2K))&=(g_2a_1hj_1^{-1},(j_1H,j_2K))=(g_2a_1h,(eH,j_1^{-1}j_2K))\\&=(g_2,(a_1H,a_1hj_1^{-1}j_2K))=(g_2,(a_1H,a_2K)).
\end{align*}
This proves injectivity of $\psi$. This map is also $G$-equivariant since: 
\begin{align*}
a\psi(g,(j_1H,j_2K))=a(gj_1H,gj_2K)=(agj_1H,agj_2K)=\psi(ag,(j_1H,j_2K))
\end{align*}
for each $a\in G$. The map is clearly a homeomorphism since it is between finite discrete spaces.
\end{proof}
We now prove the equivalence of the two definitions of Mackey functors for profinite groups provided above. This is also well known in the case where $G$ is finite.
\begin{theorem}
If $G$ is a profinite group then the definition of Mackey functor in Definition \ref{defM1} is equivalent to Definition \ref{defM2}.
\end{theorem}
\begin{proof}
If we begin with a Mackey assignment $N:\text{Grp}(G)^f\rightarrow \mathbb{Q}$ as in Definition \ref{defM1}, we will construct a bifunctor $\tilde{N}$ from $G$-Set to $\mathbb{Q}$-Mod satisfying the conditions of Definition \ref{defM2}.

We define $\tilde{N}$ on $G/H$ for $H$ an open subgroup of $G$ by $\tilde{N}(G/H)=N(H)$. This is sufficient since the collection of all such $G/H$ form an additive basis for $G$-set. If we take the projection $pr:G/H\rightarrow G/L$ where $H\subseteq L$ and $L$ is an open subgroup, we can set $\tilde{N}_*(pr)=I^L_H$ and $\tilde{N}^*(\iota)=R^L_H$. Given conjugation $C_g:G/H\rightarrow G/gHg^{-1}$ we can also define:
\begin{align*}
\tilde{N}(C_g):\tilde{N}(G/H)\rightarrow \tilde{N}(G/gHg^{-1})
\end{align*}
using the conjugation maps provided for $N$ by Definition \ref{defM1}. These assignments are functorial as a direct consequence of the relations satisfied by $N$ in Definition \ref{defM1}. The additivity condition of Definition \ref{defM2} is satisfied by construction, since $\tilde{N}$ is defined additively on an additive basis. Therefore we need only show that $\tilde{N}$ satisfies the pullback condition in Definition \ref{defM2}.

To do this we first show that if we have a pullback diagram of the form:
\begin{align*}
\xymatrix{P\ar[r]\ar[d]&G/H\ar[d]\\G/K\ar[r]&G/J}
\end{align*}
where $H,K\leq J$ and the two arrows are the canonical projections, then $P=\underset{x\in [H\diagdown J\diagup K]}{\coprod}G/\left(H\cap xKx^{-1}\right)$. To see this observe that the following is a pullback diagram where the arrows are the canonical projections:
\begin{center}
$\xymatrix{J/K\times J/H\ar[r]\ar[d]&J/H\ar[d]\\J/K\ar[r]&J/J=\left\lbrace 0\right\rbrace}$
\end{center}
Next observe that we have an isomorphism of $J$-sets given by 
\begin{align*}
\psi:J/H\times J/K&\rightarrow \underset{x\in [H\diagdown J\diagup K]}{\coprod}J/\left(H\cap xKx^{-1}\right)\\(j_1H,j_2K)&\mapsto j_2\left(H\cap xKx^{-1}\right)
\end{align*}
where $x=j_1^{-1}j_2$. The inverse is given by:
\begin{align*}
\phi:\underset{x\in [H\diagdown J\diagup K]}{\coprod}J/\left(H\cap xKx^{-1}\right)&\rightarrow J/H\times J/K\\j\left(H\cap xKx^{-1}\right)&\mapsto (jx^{-1}H,jK).
\end{align*}
To see that this is a $J$ map, if $j \in J$ then:
\begin{align*}
(jj_1H,jj_2K)\mapsto jj_2\left(H\cap {j_1}^{-1}j_2K{j_2}^{-1}j_1\right)
\end{align*}
which is equal to $j\left(j_2\left(H\cap {j_1}^{-1}j_2K{j_2}^{-1}j_1\right)\right)=j\psi(j_1H,j_2K)$. The fact that this is a homeomorphism is clear from the fact that both domain and codomain are finite discrete spaces. We then apply Lemma \ref{lemmaPB} to deduce that 
\begin{align*}
P&\cong G\times _J \left(J/H\times J/K\right)\cong G\times_J\left(\underset{x\in [H\setminus J/K]}{\coprod}J/\left(H\cap xKx^{-1}\right)\right)\\&\cong \underset{x\in [H\diagdown J\diagup K]}{\coprod}G/\left(H\cap xKx^{-1}\right).
\end{align*}
So we now know that the pullback diagram can be written as follows:
\begin{center}
$\xymatrix{\underset{x\in [H\diagdown J\diagup K]}{\coprod}G/\left(H\cap xKx^{-1}\right)\ar[r]^(0.7){\alpha}\ar[d]^{\beta}&G/H\ar[d]^{\gamma}\\G/K\ar[r]^{\delta}&G/J}$
\end{center}
We need to show that the following commutes:
\begin{center}
$\xymatrix{\underset{x\in [H\diagdown J\diagup K]}{\coprod}\tilde{N}(G/\left(H\cap xKx^{-1}\right))\ar[d]^{\beta_*}&\tilde{N}(G/H)\ar[l]^(0.3){\alpha^*}\ar[d]^{\gamma_*}\\ \tilde{N}(G/K)&\tilde{N}(G/J)\ar[l]^{\delta^*}}$
\end{center}
Observing that the maps in the square are obtained from restriction, induction and conjugation maps we apply the Mackey axiom to get the desired result. The correspondence in the opposite direction is the same construction in reverse.  
\end{proof}
A further classification of Mackey functors comes from \cite{Linder}. Before we make this precise we first make clear the following construction. 
\begin{construction}\label{spanconstruct}
If $\mathfrak{C}$ is a category with pullbacks and finite coproducts, we define the category of spans\index{category of spans} over $\mathfrak{C}$ denoted $\text{Span}(\mathfrak{C})$. This category has objects consisting of the objects of $\mathfrak{C}$. For the morphisms, we consider spans of the form:
\begin{align*}
\xymatrix{X&Z\ar[l]^{\alpha}\ar[r]^{\beta}&Y},
\end{align*}
where $\alpha$ and $\beta$ are morphisms in $\mathfrak{C}$. We say that two such spans are equivalent if there exists an isomorphism in $\mathfrak{C}$:
\begin{align*}
\phi:Z\rightarrow Z^{\prime}
\end{align*}
such that the following diagram commutes:
\begin{align*}
\xymatrix{&Z\ar[dd]^{\phi}\ar[dl]_{\alpha}\ar[dr]^{\beta}\\X&&Y\\&Z^{\prime}\ar[ul]^{\alpha^{\prime}}\ar[ur]_{\beta^{\prime}}}
\end{align*}
We consider the equivalence classes of spans from $X$ to $Y$. For composition of a span from $X$ to $Y$ with a span from $Y$ to $V$, consider the following:
\begin{align*}
\xymatrix{&Z\ar[dl]_{\alpha}\ar[dr]^{\beta}&&W\ar[dl]_{\gamma}\ar[dr]^{\zeta}\\X&&Y&&V}
\end{align*}
We calculate the pullback $P$ of the diagram:
\begin{align*}
\xymatrix{Z\ar[dr]_{\beta}&&W\ar[dl]^{\gamma}\\&Y}
\end{align*}
This gives a span $X\leftarrow P\rightarrow V$. Furthermore we also have an addition on the set of equivalence classes of spans from $X$ to $Y$. The addition is given by the following assignment:
\begin{align*}
\left(\xymatrix{X&V\ar[r]_{\beta}\ar[l]^{\alpha}&Y},\xymatrix{X&Z\ar[l]^{\gamma}\ar[r]_{\zeta}&Y}\right)\mapsto \xymatrix{X&V\coprod Z\ar[l]^{\alpha \coprod \gamma}\ar[r]_{\beta\coprod\zeta}&Y}
\end{align*}
Since the set of equivalence classes of spans from $X$ to $Y$ have an addition, we can apply the Grothendieck group completion to obtain a group from this data. We call this $\text{Hom}(X,Y)$, and these are the morphisms in the category $\text{Span}(\mathfrak{C})$.
\end{construction}
We will see in Lemma \ref{topmacklem3} that morphisms of spans over finite $G$-sets, when $G$ is profinite, can be calculated in terms of spans over finite $G/N$-sets for $N\unlhd G$ open. We can therefore see that the above construction is well defined for profinite groups using \cite[pp. 1868-1870]{Webb}. We will provide a proof that finite categorical coproducts and dually finite  products exist in this category.
\begin{proposition}
If $G$ is a profinite group, then finite categorical coproducts and products exists in the span category of finite $G$-sets.
\end{proposition}
\begin{proof}
We prove this by showing this for the orbits and the more general case will be similar. Given $G/H$ and $G/K$, we will show that $G/H\coprod G/K$ is the coproduct with map $\iota_0$ given by:
\begin{align*}
\left[\xymatrix{G/H& G/H\ar[l]^{Id}\ar[r]^(0.4){\overline{\iota_0}}&G/H\coprod G/K}\right]
\end{align*} 
where $\overline{\iota_0}$ is the coproduct inclusion in finite $G$-sets. We have similar maps for $G/K$. Suppose we have diagrams:
\begin{align*}
\xymatrix{G/H\ar[d]^{\iota_0}\ar[r]^a&X&&G/K\ar[d]^{\iota_1}\ar[r]^b&X\\G/H\coprod G/K&&&G/H\coprod G/K}
\end{align*}
where $a$ and $b$ are given by:
\begin{align*}
\left[\xymatrix{G/H& Y\ar[l]_(0.4){\overline{a}}\ar[r]^(0.4){c}&X}\right]\,\,\,\text{and }\,\left[\xymatrix{G/K& Z\ar[l]_(0.4){\overline{b}}\ar[r]^(0.4){d}&X}\right]
\end{align*}
respectively. We define a map from $G/H\coprod G/K$ to $X$ in the span category by:
\begin{align*}
\left[\xymatrix{G/H\coprod G/K& Y\coprod Z\ar[l]_(0.4){\overline{a}\coprod \overline{b}}\ar[r]^(0.5){c\coprod d}&X}\right].
\end{align*}
We show that the diagram commutes by considering the following diagram:
\begin{align*}
\xymatrix{&G/H\ar[dl]_{Id}\ar[dr]^{\overline{\iota_0}}&&Y\coprod Z\ar[dl]_{\overline{a}\coprod \overline{b}}\ar[dr]^{c+d}\\G/H&&G/H\coprod G/K&&X}
\end{align*}
where $c+d$ refers the map in finite $G$-sets induced by the universal property of colimits. The pullback which determines the composition is the set of pairs $(gH,y)\in G/H\prod Y$ such that $c(y)=gH$. This set is homeomorphic in $G$-Set to $Y$ since we have a $G$-equivariant bijection between finite discrete spaces. The projections must be the inclusion $Y\rightarrow Y\coprod Z$ and the map $c$ from $Y$ to $G/H$. It follows that the span given by the composition is:
\begin{align*}
\left[\xymatrix{G/H& Y\ar[l]_(0.4){\overline{a}}\ar[r]^(0.4){c}&X}\right].
\end{align*} 
A similar argument gives the analogous result for $G/K$. For uniqueness suppose we have a span:
\begin{align*}
\left[\xymatrix{G/H\coprod G/K& A\ar[l]_(0.3){e}\ar[r]^(0.5){f}&X}\right]
\end{align*}
which commutes in the above diagram. Working out the composition gives us the following diagram:
\begin{align*}
\xymatrix{&G/H\ar[dl]_{Id}\ar[dr]^{\overline{\iota_0}}&&A\ar[dl]_{e}\ar[dr]^{f}\\G/H&&G/H\coprod G/K&&X.}
\end{align*}
Using the proof of Lemma \ref{pullbackorbit} we can see that when calculating the pullback we need only concentrate on the orbits in the decomposition of $A$ which map into $G/H$. Therefore we have:
\begin{align*}
\xymatrix{&G/H\ar[dl]_{Id}\ar[dr]^{\overline{\iota_0}}&&\underset{1\leq i\leq n}{\coprod}G/L_i\ar[dl]_{e}\ar[dr]^{f}\\G/H&&G/H\coprod G/K&&X.}
\end{align*}
The pullback which determines the composition is the set of pairs $(gH,xL_i)\in G/H\prod \left(\underset{1\leq i\leq n}{\coprod}G/L_i\right)$ such that $e(xL_i)=gH$. This set is equivalent in $G$-Set to $\underset{1\leq i\leq n}{\coprod}G/L_i$ by the same argument as above. The projections must be the inclusion:
\begin{align*}
\underset{1\leq i\leq n}{\coprod}G/L_i\rightarrow A
\end{align*}
and the map $e$ from $\underset{1\leq i\leq n}{\coprod}G/L_i$ to $G/H$. It follows that the span given by the composition is:
\begin{align*}
\left[\xymatrix{G/H& \underset{1\leq i\leq n}{\coprod}G/L_i\ar[l]_(0.6){e}\ar[r]^(0.6){f}&X}\right].
\end{align*} 
By assumption this is equal to the span from $G/H$ to $X$ with $Y$ in the middle, so $Y$ and $\underset{1\leq i\leq n}{\coprod}G/L_i$ must be isomorphic as $G$-sets with their corresponding spans commuting. A similar argument gives the analogous result for $G/K$, by showing that $Z$ is isomorphic to the complement in $A$ of $\underset{1\leq i\leq n}{\coprod}G/L_i$ and their corresponding spans commute. This proves that the following two spans are equal:
\begin{align*}
\left[\xymatrix{G/H\coprod G/K& A\ar[l]_(0.3){e}\ar[r]^(0.5){f}&X}\right]
\end{align*}
and
\begin{align*}
\left[\xymatrix{G/H\coprod G/K& Y\coprod Z\ar[l]_(0.4){\overline{a}\coprod \overline{b}}\ar[r]^(0.5){c+d}&X}\right].
\end{align*}
For the finite products of $G/H$ and $G/K$ we consider $G/H\coprod G/K$ with the following map:
\begin{align*}
\left[\xymatrix{G/H\coprod G/K& G/H\ar[l]^(0.3){\overline{\iota_0}}\ar[r]^(0.5){Id}&G/H}\right]
\end{align*}
and the proof works out the same way.
\end{proof}

We now provide a characterisation of rational $G$-Mackey functors in terms of $\text{Span}\left(G\text{-set}\right)$ as seen in Construction \ref{spanconstruct}. In this theorem, a finite product preserving functor is a functor:
\begin{align*}
F:\mathfrak{C}\rightarrow \mathfrak{D}
\end{align*} 
such that if $\underset{1\leq j\leq n}{\prod}X_i$ is the product in $\mathfrak{C}$ then $F\left(\underset{1\leq j\leq n}{\prod}X_j\right)$ is the product in $\mathfrak{D}$. Furthermore, the product projections are of the form $F(p_j)$ where $p_j$ are the product projections of $\underset{1\leq j\leq n}{\prod}X_i$ in $\mathfrak{C}$.
\begin{theorem}\label{Catdef}
Let $\mathfrak{C}$ be a category with pullbacks and finite coproducts and $\mathfrak{D}$ be an abelian category. Then the category of Mackey functors from $\mathfrak{C}$ to $\mathfrak{D}$ is canonically isomorphic to the category of all finite product preserving functors from $\text{Span}(\mathfrak{C})$ to $\mathfrak{D}$. 

In particular if $\mathfrak{C}=G\gset^f$ and $\mathfrak{D}=\mathbb{Q}\Mod$, this applies to \index{Mackey functor}Mackey functors over profinite groups.
\end{theorem}
See \cite[Theorem 4]{Linder} for details.
\begin{construction}\label{Mackquot}
Let $M$ be a Mackey functor over a profinite group $G$ and $N$ any open normal subgroup of $G$, then we have a Mackey functor $\overline{M}$ defined for $G/N$ which takes value $M(G/K)$ on subgroup $K/N$ of $G/N$. The restriction, induction and conjugation maps are the same as those for $M$.
\end{construction}
To see that the construction $\overline{M}$ is a Mackey functor we first establish two group theoretic lemmas.
\begin{lemma}
Let $G$ be a profinite group with an open normal subgroup $N$. If $H$ and $K$ are open subgroups containing $N$ then we have a bijection from $\left[H\diagdown G\diagup K\right]$ to $\left[\left(H/N\right)\diagdown\left(G/N\right)\diagup\left(K/N\right)\right]$.
\end{lemma}
\begin{proof}
We define a map 
\begin{align*}
\theta:\left[H\diagdown G\diagup K\right]&\rightarrow \left[\left(H/N\right)\diagdown \left(G/N\right)\diagup \left(K/N\right)\right]\\ HxK&\mapsto \left(H/N\right)xN\left(K/N\right) 
\end{align*}
To see that this is well defined, if $HxK=HyK$ then $x=hyk$ for $h$ in $H$ and $k$ in $K$. Then $xN=\left(hyk\right)N$ which is equal to be $hNyNkN$ so $\theta(HxK)=\theta(HyK)$.

For surjectivity, given $\left(H/N\right)xN\left(K/N\right)$ in the codomain then we have that $\theta(HxK)$ maps to it.

For injectivity, suppose that $HxK$ and $HyK$ satisfy:
\begin{align*}
\left(H/N\right)xN\left(K/N\right)=\left(H/N\right)yN\left(K/N\right).
\end{align*}
Then there is $hN\in H/N$ and $kN\in K/N$ with $yN=hNxNkN$. But this is equal to $hxkN$ so $y=kxkn$ for some $n\in N$. We now use that $N$ is contained in $K$ to deduce that $kn\in K$ and therefore $HyK=HxK$.
\end{proof}
We will use the following elementary fact.
\begin{lemma}
If $K,H\supseteq N$, then $\left(K/N\right)\cap \left(H/N\right)=(H\cap K)/N$.
\end{lemma}

\begin{proposition}\label{inflatemackey}
The construction $\overline{M}$ is a Mackey functor over $G/N$.
\end{proposition}
\begin{proof}
We show that each of the required conditions hold.
\begin{enumerate}
\item Given $K/N$, then $R^{K/N}_{K/N}=R^K_K=\id_{M(G/K)}$ since it is true for $M$. The induction map satisfies the condition for the same reason. Given $kN\in K/N$, then $C_{kN}(s)$ for $s\in M(G/K)$ is defined to be $C_k(s)$ which equals $s$ since it does for $M$. 
\item Given $H/N\geq K/N\geq L/N$, then $R^{K/N}_{L/N}R^{H/N}_{K/N}$ is defined to be $R^K_LR^H_K$ which equals $R^H_L$ since this is true for $M$. The statement for the induction and conjugation maps are similar.
\item Given $gN\in G/N$ and $H/N\geq K/N$. Then $C_{gN}R^{H/N}_{K/N}$ is by definition $C_gR^H_K$, which equals $R^{gHg^{-1}}_{gKg^{-1}}C_g$. This is the definition of $R^{gHg^{-1}/N}_{gKg^{-1}/N}C_{gN}$. The corresponding statement for the induction maps are similar.
\item Let $H/N$ and $K/N$ be subgroups of $G/N$, then we have that:
\begin{align*}
R^{G/N}_{H/N}I^{G/N}_{K/N}&=R^G_HI^G_K=\sum_{H\setminus G/ K}I^H_{H\cap xKx^{-1}}C_xR^K_{K\cap x^{-1}Hx}\\&=\sum_{\left(H/N\right)\setminus \left(G/N\right)/ \left(K/N\right)}I^H_{H\cap xKx^{-1}}C_xR^K_{K\cap x^{-1}Hx}\\&=\sum_{\left(H/N\right)\setminus \left(G/N\right)/ \left(K/N\right)}I^{H/N}_{H/N\cap xKx^{-1}/N}C_xR^{K/N}_{K/N\cap x^{-1}Hx/N}
\end{align*}
Here we use the definition of the restriction, induction and conjugation maps of $\overline{M}$, as well as the proceeding two lemmas.
\end{enumerate}
\end{proof}
\section{The Burnisde Ring of a Profinite Group}
For finite discrete groups there is a construction called the Burnside ring and this is an example of a Mackey functor. Furthermore the Burnside ring interacts with more general Mackey functors in a useful way. We will see this construction is also relevant for profinite groups.

In order to understand what the Burnside ring is, we begin by considering the set of all finite $G$-sets. We put an addition on this set by considering disjoint union of two finite $G$-sets. We can also define a multiplication by considering the cartesian product of two finite $G$-sets. These operations satisfy the associativity and distributivity laws similar to the finite case. The only issue preventing this construction being a ring is that we do not necessarily have additive inverses. As with the finite case we can get around this using the Grothendieck construction which introduces formal differences.
\begin{definition}
If $G$ is a profinite group we can construct the \textbf{Burnside ring}\index{Burnside ring} for $G$, denoted $A(G)$, as the Grothendieck ring of finite $G$-sets.
\end{definition}  
\begin{remark}
$A(G)$ has an additive basis given by the elements $\left[G/H\right]$.
\end{remark}
Since we are working rationally, from here on $A(G)$ shall denote $A(G)\otimes \mathbb{Q}$. When $G$ is a finite group there are equivalent definitions of the Burnside ring, we now examine which definitions extend to the case where $G$ is a profinite group.  We begin by giving the construction for the space of closed subgroups of a profinite group $G$.
\begin{construction}\label{spaceclosedsubgroup}
Let $G$ be a profinite group, then we can define the space of closed subgroups\index{space of closed subgroups} $SG$ to be the space whose underlying set is the collection of all closed subgroups. This comes with an action map:
\begin{align*}
G\times SG&\rightarrow SG\\(g,K)&\mapsto gKg^{-1}.
\end{align*}
For the topology, consider:
\begin{align*}
O(N,J)=\left\lbrace K\in SG\mid NK=J\right\rbrace
\end{align*}
where $N\unlhd_O G$, $J\leq_O G$ and $N\leq J$. Then we can give $SG$ the topology generated by the sets of the form $O(N,J)$. That is, we consider the topology where open sets are precisely the unions of finite intersections of all sets of this form. With this topology we can show that the above action is continuous.

\end{construction}
\begin{remark}\label{ONJ}
The sets $O(N,J)$ are $J$-invariant since if $A\in O(N,J)$ and $j\in J$ then $jAj^{-1}N=jANj^{-1}=jJj^{-1}=J$.
\end{remark}
We now show that the space of closed subgroups of a profinite group $G$ provide us with an example of a $G$-space.
\begin{proposition}
If $G$ is a profinite group then the action on $SG$ given by:
\begin{align*}
G\times SG&\rightarrow SG\\(g,K)&\mapsto gKg^{-1}
\end{align*}
is continuous.
\end{proposition}
\begin{proof}
Let $\psi$ denote the action map and take any basic open set $O(N,J)$ and pair $(g,K)\in\psi^{-1}\left(O(N,J)\right)$. We claim that $V=Jg\times O(N,g^{-1}Jg)$ is contained in the preimage and $(g,K)$ belongs to $V$. 

To see this, if $(jg,A)\in V$ then $jgAg^{-1}j^{-1}N=jgANg^{-1}j^{-1}=J$ as required. Also $(g,K)$ belongs to $V$ since $K$ belongs to $O(N,g^{-1}Jg)$.
\end{proof}
\begin{remark}
Notice that if $G$ is a finite discrete group then $\left\lbrace e\right\rbrace$ is an open normal subgroup of $G$ and any subgroup $H$ is an open subgroup of $G$. Therefore $O(\left\lbrace e\right\rbrace,H)=\left\lbrace H\right\rbrace$ is open in $SG$ proving that $SG$ is a discrete space. 

Furthermore we can look at the conjugacy classes of closed subgroups of $G$. This is topologised as $SG/G$ using the quotient topology.
\end{remark}
We will now see what the construction $SG$ looks like when $G$ is the $p$-adic integers.
\begin{proposition}\label{padicspace}
The space $S\left(\mathbb{Z}_p\right)$ is homeomorphic to:
\begin{align*}
P=\left\lbrace \frac{1}{n}\mid n\in\mathbb{N}\right\rbrace\bigcup\left\lbrace 0\right\rbrace, 
\end{align*}
which has the subspace topology with respect to $\mathbb{R}$.
\end{proposition}
\begin{proof}
First notice that there is a bijection of sets given by $p^k\mathbb{Z}_p\mapsto \frac{1}{k+1}$ and $e\mapsto 0$. We next observe that points of the form $p^k\mathbb{Z}_p$ and $\frac{1}{k+1}$ are isolated in their respective spaces. This is clear for $P$ but for $S\left(\mathbb{Z}_p\right)$ we can see the following:
\begin{align*}
\left\lbrace p^k\mathbb{Z}_p\right\rbrace=O\left(p^{k+1}\mathbb{Z},p^k\mathbb{Z}_p\right).
\end{align*}
To see that this is true observe that $p^k\mathbb{Z}_p$ clearly belongs to this set. On the other hand, if $A$ satisfies that $Ap^{k+1}\mathbb{Z}_p=p^k\mathbb{Z}_p$ then $A\leq p^k\mathbb{Z}_p$. However $A\leq p^{k+1}\mathbb{Z}_p$ has to be false since if it were true then we would have the following contradiction:
\begin{align*}
p^k\mathbb{Z}_p=Ap^{k+1}\mathbb{Z}_p=p^{k+1}\mathbb{Z}_p.
\end{align*}
This shows that $A$ must equal $p^k\mathbb{Z}_p$. A similar argument shows that there is a one to one correspondence between the neighbourhood basis of $0$ and that of $e$. We can observe that:
\begin{align*}
O\left(p^k\mathbb{Z}_p,p^k\mathbb{Z}_p\right)=\left\lbrace e\right\rbrace \bigcup \left\lbrace p^n\mathbb{Z}_p\mid n\geq k\right\rbrace,
\end{align*}
which corresponds to the typical open neighbourhood of $0$ in $P$ of the form:
\begin{align*}
\left\lbrace 0\right\rbrace \bigcup \left\lbrace \frac{1}{n}\mid n\geq k+1\right\rbrace.
\end{align*}  
\end{proof}
For $G$ a profinite group we have a more informative way of describing $SG$ which the following proposition will illustrate. Since $G$ is a profinite group, we can write $G$ as an inverse limit $\underset{\leftarrow}{\lim}\,G/N$ where $N$ ranges over all open normal subgroups of $G$. Consequently, if $N$ is an open normal subgroup of $G$ then we have the following map:
\begin{align*}
p_N:SG&\rightarrow S(G/N)\\K&\mapsto NK/N.
\end{align*} 
This map is useful since we know it is a continuous map of spaces where the codomain is a finite discrete space. This is true since we have the following equality:
\begin{align*}
p_N^{-1}\left(\left\lbrace NK/N\right\rbrace\right)=O(N,NK).
\end{align*} 
\begin{proposition}\label{SGprofinite}
If $G$ is a profinite group then we have a homeomorphism $SG\cong \underset{\leftarrow}{\text{lim}}\,S(G/N)$. Hence $SG$ is a compact, Hausdorff and totally disconnected space.
\end{proposition}
\begin{proof}
We know from \cite[Corollary 1.1.8, Proposition 2.2.1]{Ribes} that each closed subgroup $H$ of $G$ is of the form $\underset{\leftarrow}{\lim}\,HN/N$, where $N$ ranges over all open normal subgroups of finite index of $G$. Therefore we need only show that the topology described above coincides with the profinite topology. 

Take a subbasis set $O(N,J)$ which is non-empty. Then we know that there exists a closed subgroup of $G$ namely $K$ with $NK=J$ so we can write $O(N,J)$ as $O(N,NK)$. We next observe that:
\begin{align*}
p_N^{-1}\left(\left\lbrace KN/N\right\rbrace\right)=O(N,NK)=O(N,J)
\end{align*}
so the subbasis given above is another way of writing the profinite topology. 
\end{proof}
If $N$ is an open normal subgroup of $G$ and $K$ is a closed subgroup of $G$, then all finite intersections of sets of the form $O(N,NK)$ give a closed-open basis for the topology. We also have another basis for this topology as the following proposition will show.
\begin{proposition}\label{basis}
The sets of the form $O(N,NK)$ give a basis of compact-open sets for the profinite topology on $SG$. 
\end{proposition}
\begin{proof}
Clearly all sets of the form $O(N,NK)$ cover $SG$ so it is left to show that if $A\in O(N_1,N_1K)\cap O(N_2,N_2L)$ we can find a set of the form $O(N,NA)$ satisfying $A\in O(N,NA)\subseteq O(N_1,N_1K)\cap O(N_2,N_2L)$.

To start with set $N=N_1\cap N_2$, then clearly $A\in O(N,NA)$. Suppose $B\in O(N,NA)$, then we have that:
\begin{align*}
BN_1=BNN_1=ANN_1=AN_1=N_1K
\end{align*}
and so $B\in O(N_1,N_1K)$. Similarly $B\in O(N_2,N_2L)$ also, and hence $A$ belongs to $O(N,NA)\subseteq O(N_1,N_1K)\cap O(N_2,N_2L)$ as required.

In the construction of $G$ as an inverse limit we know that if $N$ is open and normal, then $G/N$ is a discrete group and hence $S(G/N)$ is a finite discrete space. Continuity of the projection maps tells us that the subbasis sets are also closed, and because the space is Hausdorff this further implies compactness. 
\end{proof}
The following proposition will give a more localised description of the topology of the profinite space $SG$. In order to understand this we first define the core of a subgroup of $G$. The idea is that any subgroup determines a normal subgroup contained inside it.
\begin{definition}
Let $G$ be a group and $H$ be any subgroup of $G$. We define the \textbf{core of $H$}\index{core of $H$} with respect to $G$ to be:
\begin{align*}
\text{Core}_G(H)=\underset{g\in G}{\bigcap}gHg^{-1}.
\end{align*}
\end{definition}
We sometimes denote the core of $H$ by $H^C$ when the ambient group is understood. Suppose that $G$ is a profinite group and $H$ is an open subgroup of $G$. Since $H$ is of finite index it therefore follows that $H^C$ is a finite intersection of open subgroups and is therefore open. 
\begin{remark}
If $G$ is a profinite group and $J\leq G$ is an open subgroup, then $J$ as a point in the topological space $SG$ defined above has a neighbourhood $U$ where each element in $U$ is a subgroup of $J$. To see that any such $J$ has a neighbourhood of this form take $U=O(J^C,J)$. This holds for any normal subgroup contained in $J$ replaced with $J^C$. 
\end{remark}
\begin{proposition}\label{nbhdbasis}
If $K\leq G$ is closed and if $G$ is profinite, then a neighbourhood basis for $K$ is given by:
\begin{align*}
\left\lbrace O(N,NK)\mid N\unlhd G,\text{open}\right\rbrace.
\end{align*} 
If $H$ is open in $G$ then a neighbourhood basis of $H$ is given by:
\begin{align*}
\left\lbrace O(N,H)\mid N\unlhd G,\text{ open and }N\leq H\right\rbrace.
\end{align*}
\end{proposition}
\begin{proof}
The statement for $K$ closed holds because the sets $O(N,NK)$ form a basis for $SG$ with the profinite topology. For $H$ open we prove that the set $\left\lbrace O(N,H)\mid N\unlhd H,G,\text{open}\right\rbrace$ is cofinal in $\left\lbrace O(N,NH)\mid N\unlhd G,\text{open}\right\rbrace$. 

Given an open set $O(N,NH)$, set $N^{\prime}=N\cap H^c$ where $H^c$ is the core of $H$. A similar argument as in the proof of Proposition \ref{basis} shows that $O(N^{\prime},N^{\prime}H)=O(N^{\prime},H)$ is contained in $O(N,NH)$ as required.
\end{proof}
Notice that we are interested in the core of $H$ since we want an open normal subgroup contained in $H$. This argument will hold if we replace thecore with any open normal subgroup of $G$ contained in $H$.

We can now provide a characterisation for the Burnside ring\index{Burnside ring} of a profinite group $G$.
\begin{theorem}\label{burnchar}
Let $G$ be a profinite group, then we have an isomorphism of rings $A(G)\cong C(SG/G,\mathbb{Q})$, where the ring structure is given by:
\begin{align*}
(f+g)(x)=f(x)+g(x)\, \text{and}\, (fg)(x)=f(x)g(x).
\end{align*}
\end{theorem}
\begin{proof}
This proof is available in \cite[pp. B8]{Dress}. The isomorphism maps $[G/H]$ to the function which sends a conjugacy class of  closed subgroup $K$ to $| (G/H)^K |$.
\end{proof}
We also have the following useful proposition involving profinite spaces. A profinite space is a compact, Hausdorff and totally disconnected space, or equivalently an inverse limit of finite discrete spaces. These are important for our purposes since by Proposition \ref{SGprofinite} the space $SG$ is a profinite space.
\begin{proposition}\label{commlim}
Let $X$ be a profinite space, where $X=\underset{\leftarrow}{\lim}\,X_i$ for an inverse system of finite discrete spaces $X_i$. Then $C(X,\mathbb{Q})$ and $\underset{\rightarrow}{\colim}\,C(X_i,\mathbb{Q})$ are isomorphic as rings.
\end{proposition}
\begin{proof}
We define a map: 
\begin{align*}
\phi:\underset{\rightarrow}{\colim\,}C(X_i,\mathbb{Q})&\rightarrow C(X,\mathbb{Q})\\ \left[ f_i\right]&\mapsto f_i\circ p_i,
\end{align*}
where $p_i:X\rightarrow X_i$ is the projection which comes from the definition of limit. First note, using \cite[Corollary 1.1.8]{Ribes} since $X$ is an inverse limit of finite discrete spaces which are in particular compact and Hausdorff, we can assume without loss of generality that each $p_i$ is surjective.

Next we justify that this map is well defined. Suppose we have $f_j:X_j\rightarrow \mathbb{Q}$ and $f_k:X_k\rightarrow \mathbb{Q}$ which are equivalent by the colimit relation. This means that there exists $i\geq j,k$ such that $f_j\circ p_{ij}=f_k\circ p_{ik}$, where $p$ is the morphism in the limit diagram for $X$. It then follows that:
\begin{align*}
f_j\circ p_j=f_j\circ p_{ij}\circ p_i=f_k\circ p_{ik}\circ p_i=f_k\circ p_k
\end{align*}  
as required.

Next we show that $\phi$ is injective, so suppose we have $\left[f_i\right]$ and $\left[ g_j\right]$ with $f_i\circ p_i=g_j\circ p_j$. Since the inverse limit is over a directed set we can choose an index $\gamma$ which satisfies $\gamma\geq i,j$ and so we have the following diagram:
\begin{center}
$\xymatrix{&X\ar_{p_i}[dl]\ar^{p_j}[dr]\ar^{p_{\gamma}}[dd]\\X_i&&X_j\\&X_{\gamma}\ar_{p_{\gamma ,j}}[ur]\ar^{p_{\gamma ,i}}[ul]}$
\end{center}
where $p_{\gamma ,i}$ is the map from $X_{\gamma}$ to $X_i$ from the definition of inverse system. It follows that $f_i\circ p_i=(f_i\circ p_{\gamma ,i})\circ p_{\gamma}$ and $g_j\circ p_j=(g_j\circ p_{\gamma ,j})\circ p_{\gamma}$ so we can assume without loss of generality that $i=\gamma =j$, so that $f_{\gamma}\circ p_{\gamma}=g_{\gamma}\circ p_{\gamma}$. By the above fact that each $p_i$ is surjective and hence an epimorphism, we have that $f_{\gamma}=g_{\gamma}$. In particular this proves injectivity.

In order to see that $\phi$ is surjective, take any $f:X\rightarrow \mathbb{Q}$. Since $X$ is a profinite space and hence compact we have that $f(X)\subseteq \mathbb{Q}$ is a compact subspace where $\mathbb{Q}$ has the discrete topology. This then implies that $f(X)$ must be finite as well as discrete enabling us to apply \cite[Proposition 1.1.16 (a)]{Ribes} to factor $f:X\rightarrow f(X)$ through $X_i$ for some $i\in I$ via some $f_i:X_i\rightarrow f(X)$, proving surjectivity.

For the algebraic structure, starting with addition if we take $\left[f_i\right]+\left[g_j\right]$, there is a $\gamma\geq i,j$ since the set $I$ is directed, and hence $\left[f_i\right]+\left[g_j\right]=\left[f_{\gamma}+g_{\gamma}\right]$. This has image $\left(f_{\gamma}+g_{\gamma}\right)\circ p_{\gamma}=f_{\gamma}\circ p_{\gamma}+g_{\gamma}\circ p_{\gamma}$ under $\phi$. On the other hand $\phi(\left[f_i\right])+\phi(\left[g_j\right])=f_i\circ p_i + g_j\circ p_j$. By the above argument we can use the directed property to deduce that this equals $f_{\gamma}\circ p_{\gamma}+g_{\gamma}\circ p_{\gamma}$ as required. Showing that $\phi$ is multiplicative is done similarly. 
\end{proof}
This proposition has the following interesting corollary.
\begin{corollary}\label{Burnsidecolim}
If $G$ is a profinite group with $G=\underset{\leftarrow}{\lim}\, G/N$, then $A(G)\cong \underset{\rightarrow}{\colim}A(G/N)$.
\end{corollary}
\begin{proof}
This follows directly from Proposition \ref{commlim} combined with the following diagram:
\begin{center}
$\xymatrix{&A(G)\ar[d]^{\cong}\ar[r]&\colim A(G/N)\ar[d]^{\cong}\\C(\underset{\leftarrow}{\lim} S(G/N),\mathbb{Q})\ar[r]^(0.6){=}&C(SG,\mathbb{Q})\ar[r]^(0.4){\cong}&\underset{\rightarrow}{\colim} C(S(G/N),\mathbb{Q})}$
\end{center}
\end{proof}
\begin{example}
We can define a Mackey functor $\underline{A}$ by considering the assignment $H\mapsto A(H)$. If $K\leq H$ then we have a restriction map from $A(H)$ to $A(K)$ which assigns a $H$-set $X$ to the restriction of $X$ with respect to $K$. For the induction map we assign a $K$-set $Y$ to $H\underset{K}{\times}Y$. If we take an orbit $H/L$ in $A(H)$ we have an element of $A(gHg^{-1})$ by considering $gHg^{-1}/gLg^{-1}$.
\end{example}
We will see that we can characterise the Burnside ring of $G$ in terms of homotopy theory of spheres in Theorem \ref{Burnsidehtpy}.
\section{Burnside Ring Idempotents}
The characterisation of the Burnside ring of a profinite group $G$ as $C(SG/G,\mathbb{Q})$ is significant because we can now characterise its multiplicative idempotents. These idempotents are very important for this theory, since it is through these that we can define an action of the Burnside ring on an arbitrary Mackey functor. This action is significant since it is through this that we can define a sheaf structure from the information given by a Mackey functor, see Construction \ref{Sheaf_mack}. 

When calculating explicitly what the multiplicative idempotents are when $G$ is finite, we need to consider the poset of subgroups of $G$. It turns out from \cite[Section 3]{Yosh} that we can calculate these idempotents, but to understand these calculations we need to understand the Moebius function which relates to posets. A chain of elements in a poset $(\mathfrak{I},\leq)$ of length $n$ is a sequence:
\begin{align*}
A_1\leq A_2\leq \ldots \leq A_n.
\end{align*}
When $G$ is a finite discrete group and $H\leq G$, the idempotent $e_H\in A(G)$ can be written as a sum of additive basis elements as follows:
\begin{align*}
e_H=\sum_{D\leq H}\frac{|D|}{|N_G(H)|}\mu(D,H)\left[G/D\right]
\end{align*} 
where $\mu$ is the Moebius function. In this context the Moebius function where $D\leq K$, is given as follows:
\begin{align*}
\mu(D,K)=\sum_i(-1)^ic_i
\end{align*}
where the sum ranges over all lengths of strictly increasing chains from $D$ to $K$, and $c_i$ is the number of chains from $D$ to $K$ of length $i$. For more information see \cite[Proposition 1]{Rota} and \cite[section 2]{Gluck}.

Using the continuous function definition of the Burnside ring we view $\left[G/D\right]$ as the continuous function which sends a subgroup $A$ to $|\left(G/D\right)^A|$. This means that $e_H$, which is the characteristic function for $H$, can be written as the function $A\mapsto \sum_{D\leq K}\frac{|D|}{|N_G(K)|}\mu(D,K)\left[G/D\right](A)$.

By Theorem \ref{burnchar}, an element of $A(G)$ is a multiplicative idempotent\index{multiplicative idempotent of Burnside rings} if and only if its corresponding element of $C(SG/G,\mathbb{Q})$ is. By definition of the point-wise multiplication of the latter, this restricts multiplicative idempotents of the Burnside ring to being the functions which take values $0$ or $1$ at each point. Since $\mathbb{Q}$ has the discrete topology, such idempotents are determined by the compact-open basis in the following manner:

Let $e^G_{O(N,NK)}$ represent the function which sends $x$ to $1$ if it belongs to the $G$-saturation of $O(N,NK)$ and $0$ otherwise. The $G$-saturation of a subset $U$ of a $G$-space $X$ is the set $p^{-1}(p(U))$, where:
\begin{align*}
p:X\rightarrow X/G.
\end{align*}
When the ambient group is understood we simply write $e_{O(N,NK)}$.

We prove the following lemma which will help with the idempotent calculations.
\begin{lemma}\label{non3iso}
Let $N$ be an open normal subgroup of $G$ and $A$ be any closed subgroup. Then we have the following homeomorphism of spaces:
\begin{align*}
G/NA\cong (G/N)/(NA/N).
\end{align*}
\end{lemma}
\begin{proof}
We will construct a bijection and the fact that both the domain and codomain are discrete will give the homeomorphism.
\begin{align*}
\psi:G/NA&\rightarrow (G/N)/(NA/N)\\gNA&\mapsto (gN)NA/N.
\end{align*}
This clearly well defined. For surjectivity, given any $(gN)NA/N$ we know that $gNA$ is in the preimage. For injectivity, if $(g_1N)NA/N=(g_2N)NA/N$ then by definition there exists $na\in NA$ such that $g_1N=(g_2N)naN$. Therefore $g_1N=g_2naN$, and since $N\leq NA$ it follows that $g_1NA=g_2NA$.
\end{proof}
In this proof it was tempting to apply the third isomorphism theorem, however $NA$ may not have been normal.
\begin{proposition}\label{idempotentform}
For $G$ a profinite group and $O(N,NK)$ a basis subset of the profinite space $SG$ we have the following equality: 
\begin{align*}
e^G_{O(N,NK)}=\sum_{NA\leq NK}\frac{|NA|}{|N_G(NK)|}\mu(NA,NK)[G/NA]P_N.
\end{align*}
\end{proposition}
\begin{proof}
First note that the sum is finite since we are interested in the open subgroups and these have finite index and strictly contain $N$. If $G$ is profinite, by Corollary \ref{Burnsidecolim} we have that $e_{O(N,NK)}$ (the characteristic function for the open subset $O(N,NK)$) is given by $e_{NK/N}\circ P_N$. Using the above formula this boils down to:
\begin{align*}
\underset{NA\leq NK}{\sum}\frac{|NA|}{|N_G(NK)|}\mu(NA,NK)\left[\left(G/N\right)/\left(NA/N\right)\right]\circ P_N
\end{align*} 
First observe that $G/NA$ is $N$-fixed since for any coset $gNA$ we have that $g^{-1}Ng=N\leq NA$ always holds. Therefore for any $J$ we have $|\left(G/NA\right)^J|=|\left(G/NA\right)^{JN/N}|$. 

Next note that by Lemma \ref{non3iso} we have that $G/NA\cong \left(G/N\right)/\left(NA/N\right)$, so that $|\left(G/NA\right)^{JN}|=|\left(\left(G/N\right)/\left(NA/N\right)\right)^{JN/N}|$. Consequently $e_{O(N,NK)}$ can be written as the map:
\begin{align*}
L\mapsto \sum_{NA\leq NK}\frac{|NA|}{|N_G(NK)|}\mu(NA,NK)[G/NA](NL)
\end{align*}
which is $1$ for $L$ belonging to the $G$-saturation of $O(N,NK)$ and $0$ otherwise. 
\end{proof}
We then notice the following observation relating the rational coefficients in the above formula.
\begin{proposition}\label{rationcoeff}
Let $K\leq J\leq G$ be open subgroups of $G$ and $a$ be an element of $N_G(J)$. If we set $q_K=\frac{|K|}{|N_G(J)|}\mu(K,J)$, then we have $q_K=q_{aKa^{-1}}$.
\end{proposition}
\begin{proof}
Clearly $A\leq J$ if and only if $aAa^{-1}\leq J$ for $a\in N_G(J)$, so we know that $aKa^{-1}\leq J$. We also know that $|K|=|aKa^{-1}|$ so we need only show that $\mu(K,J)=\mu(aKa^{-1},J)$. But this follows from the fact that:
\begin{align*}
K=A_0<A_1<\ldots<A_n=J
\end{align*}
is a strictly increasing chain if and only if the following sequence is:
\begin{align*}
aKa^{-1}=aA_0a^{-1}<aA_1a^{-1}<\ldots<A_n=J. 
\end{align*}
It follows that $\mu(aKa^{-1},J)=\mu(K,J)$, and hence $q_{aKa^{-1}}=q_K$.
\end{proof}

Note that if we want to work with the basis consisting of finite intersections of sets of he form $O(N,NK)$, we can use the fact that $e_{\cap_{1\leq i\leq n}O(N_i,N_iK_i)}=\prod_{1\leq i\leq n}e_{O(N_i,N_iK_i)}$.

In \cite[Proposition 8]{Structure} we can see that in the case where $G$ is a finite group, the Burnside ring interacts with Mackey functors over $G$. We will now see how the equivalent result applies when $G$ is profinite.
\begin{proposition}\label{mackeyact}
If $M$ is a Mackey functor on a profinite group $G$ and $H$ an open subgroup of $G$, then each $M(H)$ has an $A(H)$-module structure. 
\end{proposition}
\begin{proof}
To prove that $M(H)$ has an $A(H)$-module structure we need to give an action of $A(H)$ on $M(H)$ and prove that this action interacts with the abelian group structure of $M(H)$ in the desired ways.

If $X\in A(H)$ we will use $X$ to construct a map $X:M(H)\rightarrow M(H)$. By above work we know we can write $X=\underset{1\leq j\leq n}{\coprod}a_jH/K_j$ for $a_1,\ldots,a_n\in\mathbb{Q}$. We define the map $X$ as follows:
\begin{align*}
M(H)&\rightarrow \bigoplus_{1\leq j\leq n}M(K_j)\rightarrow M(H)\\
m&\mapsto \sum_{1\leq j\leq n}a_jR^H_{K_j}(m)\mapsto \sum_{1\leq j\leq n}a_jI_{K_j}^HR^H_{K_j}(m)
\end{align*}  
We finish by checking the conditions.
\begin{enumerate}
\item We first need to show that if $r\in A(H)$ and $m$ and $n$ belong $M(H)$ then $r(m+n)=rm+rn$. First note that the action is made up of sums of morphisms of the form $I^H_K\circ R^H_K$. Since these are additive we have:
\begin{align*}
I^H_K\circ R^H_K(m+n)=I^H_K\circ R^H_K(m)+I^H_K\circ R^H_K(n),
\end{align*}
and hence the action defined above is additive.
\item We next justify that if $r,s \in A(H)$ and $m\in A(H)$ then:
\begin{align*}
(r+s)m=rm+sm. 
\end{align*}
This follows immediately from the definition of the action. 
\item We next have to show that the unit element $1$ of $A(H)$ satisfies that $1(m)=m$ for each $m \in M(H)$. In our case this means 
\begin{align*}
H/H(m)=I^H_H\circ R^H_H(m)=\id_{M(H)}(m)=m
\end{align*}
as required.
\item Finally we have to show that the action interacts well with the ring multiplication in $A(H)$. This means that if $r$ and $s$ belongs to $A(H)$ and $m$ to $M(H)$, then $rs(m)=r(s(m))$. In this case this means we need to show that $H/K(H/L(m))=(H/K\times H/L))(m)$ for $m \in A(H)$. 

Firstly starting with $H/K(H/L(m))$, this satisfies: 
\begin{align*}
H/K(H/L(m))=H/K(I^H_L\circ R^H_L(m))=I^H_K\circ R^H_K\circ I^H_L\circ R^H_L(m).
\end{align*}
We now use the Mackey axiom on $I^H_K\circ (R^H_K\circ I^H_L)\circ R^H_L(m)$ as follows:
\begin{align*}
&I^H_K\circ\left(\sum_{y\in\left[K\diagdown H\diagup L\right]}I^K_{K\cap yLy^{-1}}\circ C_y\circ R^L_{y^{-1}Ky\cap L}\right)\circ R^H_L(m)\\&=\sum_{y\in\left[K\diagdown H\diagup L\right]}I^H_{K\cap yLy^{-1}}\circ C_y\circ R^H_{y^{-1}Ky\cap L}(m)\\&=\sum_{y\in\left[K\diagdown H\diagup L\right]}I^H_{K\cap yLy^{-1}}\circ R^H_{yLy^{-1}\cap K}(m)
\end{align*}
On the other hand if we start with $\left(H/K\times H/L\right)(m)$ we first observe that $H/K\times H/L\cong \underset{x\in \left[K\diagdown H\diagup L\right]}{\coprod} H/\left(K\bigcap yLy^{-1}\right)$. This satisfies:
\begin{align*}
H/K\times H/L(m)&=\underset{x\in \left[K\diagdown H\diagup L\right]}{\coprod} H/\left(K\cap yLy^{-1}\right)(m)\\&=\sum_{y\in\left[K\diagdown H\diagup L\right]}I^H_{K\cap yLy^{-1}}\circ R^H_{K\cap yLy^{-1}}(m).
\end{align*}
Thus proving the desired equality.
\end{enumerate}
\end{proof}
We now look at a more general result which will have useful applications later on.
\begin{proposition}\label{actsheaf}
Suppose $X$ is a profinite space with an open-closed basis $\mathfrak{B}$ and $Y$ is a $C(X,\mathbb{Q})$-module, then $Y$ determines a sheaf over $X$.
\end{proposition}
\begin{proof}
Firstly, if $U\in \mathfrak{B}$ we know that the characteristic map $e_U:X\rightarrow \mathbb{Q}$ (defined by $e_U(x)=1$ if $x\in U$ or $0$ otherwise) is continuous and hence an element of $C(X,\mathbb{Q})$. This is similar to the discussion on the Burnside ring idempotents. If we want to define a sheaf on a space $X$ then it is sufficient to define it on a basis of $X$, so in particular we can define what it does on elements of $\mathfrak{B}$.

We define $F(U)$ to be $e_U Y$ where $U\in \mathfrak{B}$ and $e_U Y$ is the image of $Y$ under the action of $e_U$ defined above. We now look at the restriction maps. If $U\subseteq V$ then the complement of $U$ in $V$ denoted $U^c$ is both open and closed in $X$. Therefore $e_{U^c}$ is an element of $C(X,\mathbb{Q})$ and so acts on $Y$, satisfying $e_V=e_U+e_{U^c}$. It follows that $e_V Y=e_U Y\bigoplus e_{U^c}Y$.

We define the restriction map to be the canonical projection onto the $e_U Y$ component. To see that this definition gives us a presheaf, observe that if $W\subseteq U\subseteq V$ then we have the following:
\begin{align*}
e_UY=e_WY\bigoplus e_{W^c}Y
\end{align*}
\begin{align*}
e_V Y=e_{U^c}Y\bigoplus e_U Y=e_{U^c}Y\bigoplus e_{W^c}Y\bigoplus e_WY
\end{align*}
To see that the restriction maps satisfy transitivity we observe the following diagram always commutes:
\begin{center}
$\xymatrix{&e_{U^c}Y\bigoplus e_{W^c}Y\bigoplus e_WY\ar[dl]\ar[dr]\\e_WY&&e_WY\bigoplus e_{W^c}Y\ar[ll]}$
\end{center}
Since $Y$ is defined on the basis $\mathfrak{B}$ we then have that this is defined on any open subset of $X$ as follows:

If $U=\bigcup U_{\lambda}$ is a union of open closed subsets in the basis $\mathfrak{B}$, we then define the sheaf at $U$ to be the following equaliser:
\begin{center}
$\xymatrix{F(U)\ar[r]&\prod F(U_{\lambda}) \ar@< 2pt>[r]^-b
											\ar@<-2pt>[r]_-c&\prod_{(\lambda,\mu)\in\Lambda^2}F(U_{\lambda}\cap U_{\mu})}$,
\end{center}
where $b$ sends $(s_{\lambda})_{\lambda\in\Lambda}$ to $(pr^{U_{\lambda}}_{U_{\lambda}\cap U_{\mu}}(s_{\lambda}))_{(\lambda,\mu)\in \Lambda^2}$ and $pr^{U_{\lambda}}_{U_{\lambda}\cap U_{\mu}}$ is the projection defined above, and $c$ does the same but reverses $\lambda$ and $\mu$. 

We next show that this presheaf which we denote by $F$ is a sheaf. Since any open subset $U$ is evaluated using the equaliser diagram involving the closed open basis, it is sufficient to show that the sheaf conditions hold for $U\in \mathfrak{B}$ and any open covering of $U$ with elements in $\mathfrak{B}$ say $\left\lbrace U_i\mid i\in I\right\rbrace$.

Since $U$ is closed in a profinite space (and hence a Hausdorff space) we must have that the covering has a finite subcover say $\left\lbrace U_i\mid 1\leq i\leq n\right\rbrace$. Finite intersections of open-closed subsets are open-closed, as are complements of open-closed subsets in open-closed subsets. Therefore we can assume that this open covering is a partition involving closed-open subsets where restriction is given by the direct sum projection.

We first show that if $s_1$,$s_2\in F(U)$ with $s_1|_{U_i}=s_2|_{U_i}$ for $1\leq i\leq n$ then $s_1=s_2$. Since $U$ is a finite partition we have that $s_1=\sum_{1\leq i\leq n} \overline{s_1|_{U_i}}$ and similar for $s_2$, where $\overline{s}$ denotes extension by $0$ of $s$. Since $s_1|_{U_i}=s_2|_{U_i}$ for each $1\leq i\leq n$ we have:
\begin{align*}
s_1=\underset{1\leq i\leq n}{\sum} \overline{s_1|_{U_i}}=\underset{1\leq i\leq n}{\sum} \overline{s_2|_{U_i}}=s_2
\end{align*}
as required. 

For the second condition if we have a family of elements $(s_i)_{i\in I}$ in $\underset{i\in I}{\prod}F(U_i)$ with $pr^i_{(i,j)}(s_i)=pr^j_{(i,j)}(s_j)$, we need to find an element $s\in F(U)$ with $pr^U_i(s)=s_i$. By earlier arguments we can rewrite the open covering as a finite partition of open closed subsets $\left\lbrace U_i\mid 1\leq i\leq n\right\rbrace$. Therefore we can write $s=\sum_{1\leq i\leq n} s_i$.  
\end{proof}
Proposition \ref{actsheaf} is useful because we can apply Proposition \ref{mackeyact} to deduce the following Corollary.
\begin{corollary}\label{MHsheaf}
If $M$ is a Mackey functor over a profinite group $G$, then each open subgroup $H$ of $G$ satisfies that $M(H)$ is a sheaf over $SH/H$, the space of $H$-conjugacy classes of closed subgroups of $H$.
\end{corollary}
\begin{proof}
By Proposition \ref{mackeyact} we have that $M(H)$ is an $A(H)$-module, and by Theorem \ref{burnchar} we know that $M(H)$ is a $C(SH/H,\mathbb{Q})$-module. It follows from Proposition \ref{actsheaf} that $M(H)$ is a sheaf over $SH/H$.
\end{proof}
This corollary is useful since it is used in Construction \ref{Sheaf_mack} to construct a Weyl-$G$-sheaf which will be defined in Definition \ref{Weyldefn}. We can also determine how idempotents of the Burnside ring interact with induction and restriction maps of Mackey functors. We begin exploring this by proving the following lemmas:
\begin{lemma}\label{idemack1}
Let $K\leq H\leq G$ where $K$ is closed and $H$ open in $G$. Then for $A$ closed in $G$ and $N$ open and normal in $G$ we have:
\begin{align*}
\sum_{H\diagdown G\diagup NA} \left[H/H\cap xNAx^{-1}\right](K)=\left[G/NA\right](K).
\end{align*}
\end{lemma}
\begin{proof}
By definition of how these functions are defined we have that this boils down to proving that:
\begin{align*}
|\left(G/NA\right)^K|=\underset{H\diagdown G\diagup NA}{\sum} |\left(H/H\cap xNAx^{-1}\right)^K|.
\end{align*}
If $g\in \left(G/NA\right)^K$ then $g^{-1}Kg\leq NA$. By assumption we therefore have that $K\leq H\cap gNAg^{-1}$. Also note that the right hand side sum represents $G/NA$ written as a disjoint union of $H$-orbits, and so we can find $x\in H\diagdown G\diagup NA$ and $h\in H$ with $hxNA=gNA$. Therefore $h\left(H\cap xNAx^{-1}\right)$ corresponds to $gNA$, and this is $K$-fixed since $hxNA=gNA$ is.

On the other hand given $h\left(H\cap xNAx^{-1}\right)$ which is $K$-fixed, the inverse map is given by sending this to $hxNA$. This is $K$-fixed in $G/NA$ since it follows from our assumption that $(hx)^{-1}Khx\leq NA$. Therefore we have a bijection:
\begin{align*}
\left(G/NA\right)^K\cong \sum_{H\diagdown G\diagup NA} \left(H/H\cap xNAx^{-1}\right)^K
\end{align*}
as required.
\end{proof}
\begin{lemma}\label{idemack2}
If $H\leq G$ is an open subgroup of $G$ and $D\leq G$ such that $D$ is not a subgroup of $H$, then:
\begin{align*}
\sum_{H\setminus G/NA} \left[H/H\cap xNAx^{-1}\right](D)=0.
\end{align*}
\end{lemma}
\begin{proof}
Suppose $h(H\cap xNAx^{-1})$ is $D$-fixed, then 
\begin{align*}
h^{-1}Dh\leq H\cap xNAx^{-1}\leq H
\end{align*}
which contradicts the assumption that $D$ is not a subgroup of $H$. Hence each $\left[H/H\cap xNAx^{-1}\right](D)=0$ for each of the summands, which gives the result. 
\end{proof}
\begin{lemma}\label{idemack3}
Let $\alpha$ denote the following:
\begin{align*}
\sum_{NA\leq NK}\frac{|NA|}{|N_G(NK)|}\mu(NA,NK)\sum_{H\setminus G/NA}\left[H/H\cap xNAx^{-1}\right]. 
\end{align*}
Then $NK$ is not $G$-subconjugate to $H$ if and only if no $G$-conjugate of $NK$ belongs to the support of $\alpha$.
\end{lemma}
\begin{proof}
First suppose that $NK$ is not $G$-subconjugate to $H$. Therefore each $gNKg^{-1}$ is not a subgroup of $H$. If we consider $\alpha(gNKg^{-1})$, we can apply Lemma \ref{idemack2} to deduce that $\alpha(gNKg^{-1})=0$.

Conversely if $xNKx^{-1}\leq H$ then $\alpha(xNKx^{-1})=e_{O(N,NK)}(xNKx^{-1})=1$, which is non-zero. 
\end{proof}
More generally, if $D$ is a closed subgroup of $G$, then there are two possibilities. If $D$ is not a subgroup of $H$, then:
\begin{align*}
\sum_{H\diagdown G\diagup NA}\left[H/H\cap xNAx^{-1}\right](D)=0
\end{align*}
by Lemma \ref{idemack2} and hence $\alpha(D)=0$. The other possibility is that $D\leq H$. If $D$ doesn't belong to $e_{O(N,NK)}^G$ we have by Lemma \ref{idemack1} that:
\begin{align*}
\alpha(D)&=\sum_{NA\leq NK}\frac{|NA|}{|N_G(NK)|}\mu(NA,NK)\left[G/NA\right](D)\\&=e_{O(N,NK)}(D)=0.
\end{align*}
On the other hand if we consider the possibility of $D$ belonging to the support of the idempotent then we will observe that $\alpha(D)=e^G_{O(N,NK)}(D)=1$. Consider the collection:
\begin{align*}
\left\lbrace O(H\cap N,(H\cap N) D) \mid D\in O^G(N,NK)\cap SH\right\rbrace,
\end{align*}
where $H\cap N$ is normal in $H$ and $O^G(N,NK)$ represents the $G$-saturation of $O(N,NK)$. This is an open covering of compact set $SH\cap O^G(N,NK)$ which we know is compact since it is the intersection of two compact sets. We know that $O^G(N,NK)$ is compact since it is a finite union of compact sets of the form $O(N,gNKg^{-1})$, where we are using that $NK$ has finite index. Similar to the proof of Proposition \ref{actsheaf} we can write the above cover as a finite collection of open-closed sets which are pairwise disjoint. Therefore the entire collection can be written as a closed-open set $V$. 

Considering $\alpha$, we can see that we have a set of the form $O(N,NxNKx^{-1})$ for each $xNKx^{-1}$ a subgroup of $H$ contained in the support. However if $G$ is an infinite profinite group then we also have the extra set $V$ contained in the support as described above. This contains subgroups in $O(N,gNKg^{-1})$ which belong to $SH$ and where $gNKg^{-1}$ is not a subgroup of $H$. If $G$ is finite then we do not require this addition since $K$ is an isolated point and we only deal with conjugates of this element.

Consider the set:
\begin{align*}
G\mathfrak{I}^H_{NK}=\left\lbrace xNKx^{-1}\mid x\in H\diagdown G\diagup N_G(NK)\,\text{and}\,  xNKx^{-1}\leq H\right\rbrace.
\end{align*}
We denote this set by $\mathfrak{I}^H_{NK}$ when the ambient group is understood. This is an important indexing set and we have an alternate way of writing this set.
\begin{proposition}\label{indexbij}
There is a bijection between the set $G\mathfrak{I}^H_{NK}$ and: 
\begin{align*}
\left\lbrace (D)_{H}\in SH/H\mid D\in (NK)_G\in SG/G\right\rbrace.
\end{align*}
\end{proposition}
\begin{proof}
We shall label the set we wish to show is equivalent to $G\mathfrak{I}^H_{NK}$ by $A$ and then define a map as follows:
\begin{align*}
G\mathfrak{I}^H_{NK}&\rightarrow A\\ gNKg^{-1}&\mapsto\left(gNKg^{-1}\right)_H 
\end{align*}
To see that this is injective, suppose that $xNKx^{-1}$ and $yNKy^{-1}$ are $H$-conjugate by some element $h$ in $H$. Then $xNKx^{-1}=hyNKy^{-1}h^{-1}$ and therefore $x^{-1}hy$ is an element of $N_G(NK)$. It follows that $xN_G(NK)=hyN_G(NK)$ and so $x$ and $y$ are equivalent in $H\diagdown G\diagup N_G(NK)$. Surjectivity is immediate from the definition of $A$.
\end{proof}
\begin{lemma}\label{idemack4}
The map $\alpha$ from Lemma \ref{idemack3} is equal to:
\begin{align*}
\underset{xNKx^{-1}\in\mathfrak{I}^H_{NK}}{\sum}e_{O(N,xNKx^{-1})}+e_V, 
\end{align*}
as described above. The final $e_V$ term is not present for $G$ finite.
\end{lemma}
\begin{proof}
If $G\mathfrak{I}^H_{NK}=\emptyset$, then both sides of the equality are $e_V$ by Lemma \ref{idemack3} and proceeding paragraph. For $G\mathfrak{I}^H_{NK}\neq \emptyset$, take $xNKx^{-1}\in G\mathfrak{I}^H_{NK}$ so that $xNKx^{-1}\leq H$. If $D\leq xNKx^{-1}$ is a subgroup, then by Lemma \ref{idemack1} $\alpha(D)=e^G_{O(N,NK)}(D)=0$, if $D$ does not belong to the support. Lemma \ref{idemack1} also says that: 
\begin{align*}
\alpha(D)=e^G_{O(N,NK)}(D)=1,
\end{align*}
if $D$ does belong to the support. This coincides with $\underset{xNKx^{-1}\in G\mathfrak{I}^H_{NK}}{\sum}e_{O(N,xNKx^{-1})}+e_V$. The finite case is the same except without the $e_V$ term.
\end{proof}
These lemmas yield the following Propositions.
\begin{proposition}\label{restrictprop}
Let $M$ be a Mackey functor for a profinite group $G$ and $s\in M(G/G)$, then
\begin{align*}
R^G_H(e_{O(N,NK)}(s))=\underset{xNKx^{-1}\in G\mathfrak{I}^H_{NK}}{\sum}e_{O(N,xNKx^{-1})}R^G_H(s)+e_VR^G_H(s),
\end{align*}
where $e_V$ term is not present in the case where $G$ is finite.
\end{proposition}
\begin{proof}
By Proposition \ref{idempotentform} this amounts to showing that 
\begin{align*}
R^G_H\left(\underset{NA\leq NK}{\sum}\frac{|NA|}{|N_G(NK)|}\mu(NA,NK)[G/NA](s)\right)
\end{align*}
is equal to:
\begin{align*}
\underset{xNKx^{-1}\in G\mathfrak{I}^H_{NK}}{\sum}e_{O(N,xNKx^{-1})}R^G_H(s)+e_VR^G_H(s).
\end{align*}
Using the definition of the Burnside action on the Mackey functor this gives the following:
\begin{align*}
&\sum_{NA\leq NK}\frac{|NA|}{|N_G(NK)|}\mu(NA,NK)R^G_HI^G_{NA}R^G_{NA}(s)\\&=\sum_{NA\leq NK}\frac{|NA|}{|N_G(NK)|}\mu(NA,NK)\sum_{H\setminus G/NA}I^H_{H\cap xNAx^{-1}}R^{xNAx^{-1}}_{xNAx^{-1}\cap H}R^G_{xNAx^{-1}}(s)\\
&=\sum_{NA\leq NK}\frac{|NA|}{|N_G(NK)|}\mu(NA,NK)\sum_{H\setminus G/NA}I^H_{H\cap xNAx^{-1}}R^{H}_{xNAx^{-1}\cap H}R^G_{H}(s)
\end{align*}
where we use the Macky axiom, the equivariance condition and the fact that $s$ is $G$-fixed. But this is by definition:
\begin{align*}
\sum_{NA\leq NK}\frac{|NA|}{|N_G(NK)|}\mu(NA,NK)\sum_{H\setminus G/NA}\left[H/H\cap xNAx^{-1}\right](R^G_H(s))
\end{align*}
which is $\alpha (R^G_H(s))$. By Lemma \ref{idemack4} this is the same as: 
\begin{align*}
\underset{xNKx^{-1}\in G\mathfrak{I}^H_{NK}}{\sum}e_{O(N,xNKx^{-1})}(R^G_H(s))+e_VR^G_H(s). 
\end{align*}
The proof in the finite case is the same, except without the $e_V$ term.
\end{proof}
\begin{proposition}\label{inductprop}
Let $M$ be a Mackey functor for a profinite group $G$ and $s\in M(G/H)$ then:
\begin{align*}
e_{O(N,NK)}I^G_H(s)=I^G_H\left(\underset{xNKx^{-1}\in G\mathfrak{I}^H_{NK}}{\sum}e_{O(N,xNKx^{-1})}(s)+e_Vs\right),
\end{align*}
where the $e_V$ term doesn't appear when $G$ is finite.
\end{proposition}
\begin{proof}
By definition we have the following:
\begin{align*}
e_{O(N,NK)}I^G_H(s)&=\underset{NA\leq NK}{\sum}q_{NA}\mu(NA,NK)I^G_{NA}R^G_{NA}I^G_H(s)\\
&=I^G_H\left(\underset{NA\leq NK}{\sum}q_{NA}\mu(NA,NK)\underset{H\setminus G/NA}{\sum}\left[H/H\cap xNAx^{-1}\right](s)\right)\\
&=I^G_H\left(\underset{xNKx^{-1}\in G\mathfrak{I}^H_{NK}}{\sum}e_{O(N,xNKx^{-1})}(s)+e_Vs\right)
\end{align*}
where $q_{NA}=\frac{|NA|}{|N_G(NK)|}$. Here we apply the Mackey axiom, the fact that $I^G_H$ has $G$-fixed image and Lemma \ref{idemack4} since the formula given of the function $\alpha$ appears.
\end{proof}
\begin{corollary}\label{burncommute}
Let $M$ be a Mackey functor for a profinite group $G$, $s$ an element of $M(G/G)$ then:
\begin{align*}
e_{O(N,NK)}R^G_{NK}(s)+e_VR^G_{NK}(s)=R^G_{NK}(e_{O(N,NK)}(s)),
\end{align*}
where $e_V$ doesn't appear in the finite case. We also have a similar result for the induction map.
\end{corollary}
\begin{proof}
By Proposition \ref{restrictprop} we have:
\begin{align*}
R^G_{NK}(e_{O(N,NK)}(s))=\sum_{xNKx^{-1}\in G\mathfrak{I}^{NK}_{NK}}e_{O(N,xNKx^{-1})}R^G_{NK}(s)+e_VR^G_{NK}(s). 
\end{align*}
But by Proposition \ref{indexbij} we have:
\begin{align*}
G\mathfrak{I}^{NK}_{NK}=\left\lbrace (D)_{NK}\in SNK/NK\mid D\in (NK)_G\in SG/G\right\rbrace, 
\end{align*}
which is the set $\left\lbrace (NK)_{NK}\right\rbrace$. Hence the above sum is just $e_{O(N,NK)}R^G_{NK}(s)+e_VR^G_{NK}(s)$. The proof for the induction map is similar.
\end{proof}
This result is useful, especially in the finite case since it tells us that the action of the Burnside ring idempotent commutes with the restriction map. In the infinite case we almost have this except for the additional $e_V$ term. Since in the more general profinite case we will be interested in sheaves, we can restrict from $e_VR^H_G(s)+e_{O(N,NK)}R^G_H(s)$ to $e_{O(N,NK)}R^G_H(s)$ by projection in order to study the stalk of a point in $O(N,NK)$.
\begin{corollary}\label{burn0}
Let $M$ be a Mackey functor for a profinite group $G$, $s\in M(G/G)$, $t\in M(G/H)$, and suppose $NK$ is not $G$-subconjugate to $H$, then the following is true:
\begin{enumerate}
\item $R^G_H(e_{O(N,NK)}(s))=e_VR^G_H(s)$,
\item $e_{O(N,NK)}I^G_H(t)=I^G_H(e_Vs)$.
\end{enumerate}
If $G$ is finite the right hand side terms are all zero. This means in particular that the stalk of any $G$-conjugate of $NK$ is zero since none of them belong to the support.
\end{corollary}
\begin{proof}
By assumption we have $G\mathfrak{I}^H_{NK}=\emptyset$, so the result follows from Propositions \ref{restrictprop} and \ref{inductprop}. To see that an element not in $V$, say $L$, has zero stalk on $I^G_H(e_Vs)$ consider a open neighbourhood $O(N^{\prime},N^{\prime}L)$ which is disjoint from $V$. Evaluate $e_{O(N^{\prime},N^{\prime}L)}I^G_H(e_Vs)$, which we have already shown to be $I^G_H(e_We_Vs)$. Recall that $W$ is of the form $SH\cap O(N^{\prime},N^{\prime}L)$ which must be disjoint from $V$ by definition. Since $V$ and $W$ have disjoint support $e_We_Vs$ must evaluate to zero.
\end{proof}
Notice that the above propositions hold if we begin with an element $s\in M(G/H)$ and $NK,L\leq H$. We look at the exact same proof except replace $G\mathfrak{I}^L_{NK}$ with $H\mathfrak{I}^L_{NK}$.

\chapter{Rational G-Spectra}\label{chp3}
In this chapter we wil characterise the homotopy category of rational $G$-spectra. We begin the first section by defining orthogonal $G$-spectra for $G$ a profinite group and exploring the theory surrounding this category which is based on \cite{Fausk} and \cite{MandelMayequivspectra}. We will see how to calculate homotopy classes of morphisms between suspension spectra, Corollary \ref{Ghomotpy} and Proposition \ref{tomDieck}. We will see how to rationalise this model structure. In Theorem \ref{GSpChar} we will see that the category of rational $G$-spectra is Quillen equivalent to a category $\text{Ch}\left(\text{Fun}_{\text{Ab}}\left(\pi_0(\mathfrak{O}^{\mathbb{Q}}_G),\mathbb{Q}\text{-Mod}\right)\right)$. In the second section we will see that there is an equivalence between chain complexes of rational Mackey functors and \\$\text{Ch}\left(\text{Fun}_{\text{Ab}}\left(\pi_0(\mathfrak{O}^{\mathbb{Q}}_G),\mathbb{Q}\text{-Mod}\right)\right)$, Corollary \ref{mackeymodelequi}. We will conclude that chain complexes of rational $G$-Mackey functors algebraically model $G$-spectra.
\section{Characterisation of rational G-spectra}
In order to define a $G$-spectrum, we begin by defining a $G$-universe as seen in \cite[Definition 3.1]{Fausk}.
\begin{definition}\label{Univdefn}
Let $G$ be a compact Hausdorff group. A \textbf{$G$-universe}\index{$G$-universe} $U$ is a countably infinite direct sum $\underset{n\in\mathbb{N}}{\bigoplus} U^{\prime}$ of a real $G$-inner product space $U^{\prime}$ with the following structure:
\begin{itemize}
\item the one dimensional trivial $G$-representation is contained in $U^{\prime}$,
\item $U$ is a topological space whose topology is given by the union of all finite dimensional $G$-subspaces of $U$ (each with the norm topology),
\item the $G$-action on all finite dimensional $G$-subspaces $V$ of $U$ factors through a compact Lie group quotient of $G$.
\end{itemize}  
\end{definition}
The third condition implies that if $V$ is a finite dimensional subrepresentation of $U$, then the action of $G$ on $V$ factors through a quotient by an open normal subgroup. If each finite dimensional orthogonal $G$-representation is isomorphic to a $G$-subspace of $U$, then we say that $U$ is complete. The following definitions show how to use the idea of a universe to construct an indexing category, as seen in \cite[Definitions 3.5, 3.6]{Fausk}.

\begin{definition}
Let $U$ be a universe for a compact Hausdorff group $G$. An \textbf{indexing representation}\index{indexing representation} is a finite dimensional $G$-inner product subspace of $U$. If $V$ and $W$ are two indexing representations and $V\subseteq W$, then the orthogonal complement of $V$ in $W$ is denoted by $W-V$. The collection of all real $G$-inner product spaces which are isomorphic to an indexing representation is denoted $\mathcal{V}(U)$, or simply by $\mathcal{V}$ when $U$ is understood.

We define the category $\mathcal{I}^{\mathcal{V}}_G$ to be the topological category whose objects are elements of $\mathcal{V}$ and the morphisms are the linear isometric isomorphisms. The hom-sets have topology given by the compact-open topology. Similarly we also have the $G$-fixed category $G\mathcal{I}^{\mathcal{V}}$ which has the same objects and whose morphisms are the $G$-equivariant linear isometric isomorphisms. 
\end{definition}
The following definition takes us closer to defining $G$-spectra, where $G$ is compact Hausdorff.
\begin{definition}
An \textbf{$\mathcal{I}^{\mathcal{V}}_G$-space}\index{$\mathcal{I}^{\mathcal{V}}_G$-space} is a continuous $G$-functor (a functor which induces continuous $G$-maps on hom spaces) from $\mathcal{I}^{\mathcal{V}}_G$ to $\mathcal{T}_G$, the category of $G$-spaces whose morphisms are non-equivariant continuous maps. The $G$-action on the hom sets is given as follows:
\begin{align*}
(g*f)(x)=gf(g^{-1}x),
\end{align*}
where $g\in G$, $f$ and $x$ belong to a hom set and an object of either $\mathcal{T}_G$ or $\mathcal{I}^{\mathcal{V}}_G$ respectively. A morphism of $\mathcal{I}^{\mathcal{V}}_G$-spaces is an enriched natural transformation. We denote this category by $\mathcal{I}^{\mathcal{V}}_G\mathcal{T}$. 

We consider the $G$-fixed category $G\mathcal{I}^{\mathcal{V}}\mathcal{T}$ which has the same objects but whose morphisms induce level-wise $G$-equivariant maps of spaces. That is, the objects are continuous $G$-functors from $\mathcal{I}^{\mathcal{V}}_G$ to $G\mathcal{T}$, the category of $G$-spaces whose morphisms are $G$-equivariant continuous maps. In this case the hom sets of the target category are $G$-fixed.
\end{definition}
\begin{example}
Let $\mathbb{S}^{\mathcal{V}}_G\colon\mathcal{I}^{\mathcal{V}}_G\rightarrow \mathcal{T}_G$ be the functor which assigns $V$ to the one point compactification $S^V$ of $V$. For simplicity we denote this by $\mathbb{S}$. 
\end{example}
The next step in defining $G$-spectra is defining a smash product on $\mathcal{I}^{\mathcal{V}}_G\mathcal{T}$, where $G$ is compact Hausdorff. We reference both \cite{Fausk} and \cite{SchwedeMayShipleyMandell}, where we require our $G$-universe to satisfy that the indexing category of finite dimensional subspaces are closed under direct sum.
\begin{construction}
Let $X$ and $Y$ be two objects of $\mathcal{I}^{\mathcal{V}}_G\mathcal{T}$ and $W$ be an $N$-dimensional $G$-representation of $\mathcal{V}$. For each $n\in\left\lbrace 0,1,2,\ldots, N\right\rbrace$, choose $G$-representations of dimension $n$ in $\mathcal{V}$ of the form $V_n$ and $V^{\prime}_n$. We define a smash product\index{smash product} by:
\begin{align*}
\left(X\wedge Y\right)(W)\cong \overset{N}{\underset{n=0}{\bigvee}}\mathcal{I}^{\mathcal{V}}_G\left(W,V_n\oplus V^{\prime}_{N-n}\right)\underset{O(V_n)\times O(V^{\prime}_{N-n})}{\wedge}X(V_n)\wedge Y(V^{\prime}_{N-n}).
\end{align*}
\end{construction}
This is a closed symmetric monoidal category which has unit given by the functor which assigns the trivial representation to $S^0$ and every other indexing representation to a point. This also gives a closed symmetric tensor category on $G\mathcal{I}^{\mathcal{V}}\mathcal{T}$. 

We can also write this using \cite[Definition 21.4]{SchwedeMayShipleyMandell} as:
\begin{align*}
\left(X\wedge Y\right)(W)=\int^{A,B\in \mathcal{I}^{\mathcal{V}}_G}X(A)\wedge Y(B)\wedge \mathcal{I}^{\mathcal{V}}_G(A\oplus B,W).
\end{align*}
Furthermore with this tensor product $\mathbb{S}$ is a symmetric monoid in $G\mathcal{I}^{\mathcal{V}}\mathcal{T}$. See \cite[pp. 116,117]{Fausk}. We can now define an orthogonal $G$-spectrum, which holds for $G$ compact Hausdorff.
\begin{definition}
An \textbf{orthogonal $G$-spectrum}\index{orthogonal $G$-spectrum} $X$ is a $\mathcal{I}^{\mathcal{V}}_G$-space together with a left module structure over the symmetric monoid $\mathbb{S}$. We denote this category by $\mathcal{I}^{\mathcal{V}}_G\mathcal{S}$. We also have the $G$-fixed category $G\mathcal{I}^{\mathcal{V}}\mathcal{S}$.
\end{definition}
\begin{remark}\label{GSpStructure}
In particular we have a morphism $\mathbb{S}\wedge X\rightarrow X$, so given any indexing representation $W$ we have a map:
\begin{align*}
\int^{V,W\in \mathcal{I}^{\mathcal{V}}_G}\mathbb{S}(V)\wedge X(W)\wedge \mathcal{I}^{\mathcal{V}}_G(V\oplus W,Z)\rightarrow X(Z).
\end{align*}
We can canonically include $S^V\wedge X(W)$ into the coend in the following way, where we assume that $Z$ is of the form $V\oplus W$:
\begin{align*}
S^V\wedge X(W)&\rightarrow \int^{A,B\in \mathcal{I}^{\mathcal{V}}_G}\mathbb{S}(A)\wedge X(B)\wedge \mathcal{I}^{\mathcal{V}}_G(A\oplus B,Z)\rightarrow X(Z)\\x&\mapsto [x,\text{Id}].
\end{align*}
Here we take $\text{Id}$ in $\mathcal{I}^{\mathcal{V}}_G(A\oplus B,Z)$ for $A=V$ and $B=W$. By canonically mapping $S^V\wedge X(W)$ to the corresponding term of the coend we have a morphism from $S^V\wedge X(W)$ to $X(V\oplus W)$. Using the adjunction $(\Sigma^V,\Omega^V)$ this gives a corresponding adjoint map:
\begin{align*}
X(W)\rightarrow \Omega^VX(V\oplus W).
\end{align*}
\end{remark}
When defining a model structure on the category of $G$-spectra we have to choose a collection of subgroups of $G$ which play a role in the definition. The more general approach is given by \cite[Definition 2.2]{Fausk} but we consider the collection of subgroups given in \cite[Example 2.5]{Fausk}, which in our case boils down to the collection of all open subgroups.

Recall that if $Z$ is a $G$-space and $H$ a subgroup of $G$ then $\pi_n^H(Z)$ is defined to be $\pi_n(Z^H)$. A map of $G$-spaces $f$ is said to be a $\pi_*$-isomorphism if $\pi_n^H(f)$ is an isomorphism for each $n\in\mathbb{N}$ and $H$ open. A map of $G$-spaces $f$ is said to be a fibration if $f^H$ is a Serre fibration of spaces for each open subgroup $H$. We extend this to $G$-spectra in the following definition as seen in \cite[Definition 4.1]{Fausk}.
\begin{remark}\label{picolimmap}
If $X$ is a $G$-spectrum then we have maps by Remark \ref{GSpStructure} of the form: 
\begin{align*}
S^W\wedge X(V)\rightarrow X(V\oplus W)
\end{align*}
for varying $V$ and $W$ finite dimensional subrepresentations. We can then look at the adjoint map $X(V)\rightarrow \Omega^{W}X(V\oplus W)$.

This gives us maps from $\Omega^VX(V)$ to $\Omega^WX(W)$ for $V\subseteq W$. To see this let $Z$ be the orthogonal complement of$V$ in $W$. Then by the previous paragraph we have a map:
\begin{align*}
X(V)\rightarrow \Omega^ZX(Z\oplus V),
\end{align*}
whose codomain equals $\Omega^ZX(W)$. Applying $\Omega^V(-)$ to this map gives:
\begin{align*}
\Omega^VX(V)\rightarrow\Omega^WX(W).
\end{align*}
\end{remark}
\begin{definition}\label{pidefinition}
The $n$th \textbf{homotopy group of an orthogonal $G$-spectrum}\index{homotopy group of an orthogonal $G$-spectrum} $X$ at a subgroup $H$ of $G$ is:
\begin{align*}
\pi_n^H(X)=\underset{V\in\mathcal{V}}{\colim}\pi_n^H(\Omega^VX(V))
\end{align*} 
for $n\geq 0$, and
\begin{align*}
\pi_{-n}^H(X)=\underset{V\supseteq\mathbb{R}^n}{\colim}\pi_0^H(\Omega^{V-\mathbb{R}^n}X(V))
\end{align*}
for $n\geq 0$. The maps in the colimit systems are given as in Remark \ref{picolimmap}. A map of orthogonal $G$-spectra $f\colon X\rightarrow Y$ is a stable equivalence if $\pi_n^H(f)$ is an isomorphism for all $H$ open and each integer $n$.
\end{definition}
We will see in the following theorem a definition of a model structure on orthogonal $G$-spectra, which is found in \cite[Theorem 4.4]{Fausk}. We bear in mind that an unbased weak equivalence of spaces is a map of spaces:
\begin{align*}
f:X\rightarrow Y
\end{align*}
such that $\pi_*(f):\pi_*(X,x)\rightarrow \pi_*(Y,f(x))$ is an isomorphism for all $x\in X$, and such that $\pi_0(f):\pi_0(X)\rightarrow \pi_0(Y)$ is a bijection of sets.
\begin{theorem}
The category $G\mathcal{I}^{\mathcal{V}}\mathcal{S}$ is a compactly generated proper model category such that the weak equivalences are given by the $\pi_*$-isomorphisms as defined in Definition \ref{pidefinition}, the fibrations are given by the maps of spectra $f\colon X\rightarrow Y$ such that for all indexing representations $V$, $f(V)\colon X(V)\rightarrow Y(V)$ is a fibration of $G$-spaces and the canonical map from $X(V)$ to the pullback of the diagram:
\begin{align*}
\xymatrix{&\Omega^WX(V\oplus W)\ar[d]\\Y(V)\ar[r]&\Omega^WY(V\oplus W)}
\end{align*}  
is an unbased weak equivalence of $G$-spaces for all $V,W\in \mathcal{V}$. A map $f$ is an acyclic fibration if and only if each $f(V)$ is an acyclic fibration of $G$-spaces. The cofibrations are determined by the left lifting property.
\end{theorem}
\begin{proof}
See \cite[Theorem 4.4]{Fausk}.
\end{proof}
We denote the morphisms in the homotopy category of rational $G$-spectra from $X$ to $Y$ by $\left[X,Y\right]^G$. Given any $G$-space $X$ we can define a $G$-spectrum using the following suspension functor construction: 
\begin{construction}
Let $V$ be an element of $\mathcal{V}$ and $X$ be a $G$-space, for profinite $G$. Then we have a suspension $G$-spectrum\index{suspension $G$-spectrum} $\Sigma^{\infty}_VX$ which is defined at $W\in \mathcal{V}$ to be:
\begin{align*}
\Sigma^{\infty}_VX(W)=X\wedge O(W)\underset{O(W-V)}{\wedge}S^{W-V}
\end{align*}
when $W\supset V$ and a point otherwise. When $V=0$ we denote this simply by $\Sigma^{\infty}X$.
\end{construction}
This construction is a functor, so if $K\leq H$ we can apply $\Sigma^{\infty}$ to the projection $\pi:G/K\rightarrow G/H$ to obtain a map of spectra:
\begin{align*}
\Sigma^{\infty}G/K_+\rightarrow \Sigma^{\infty}G/H_+.
\end{align*} 
However, we can also construct a transfer map of $G$-spectra\index{transfer map of spectra} in the opposite direction:
\begin{align*}
\tau^H_K\colon\Sigma^{\infty}G/H_+\rightarrow \Sigma^{\infty}G/K_+,
\end{align*} 
by \cite[Construction II.5.1]{LMSMcC}, which we detail below. Let $V$ be a $G$-representation with norm $\Vert - \Vert$. Then we define:
\begin{align*}
D(V)=\left\lbrace v\in V\mid \Vert v\Vert\leq 1\right\rbrace
\end{align*}
and 
\begin{align*}
S(V)=\left\lbrace v\in V\mid \Vert v\Vert=1 \right\rbrace.
\end{align*}
Furthermore we have a $G$-homeomorphism:
\begin{align*}
D(V)/S(V)\cong S^V.
\end{align*}
\begin{construction}\label{transfer}
If $K\leq H$ are open subgroups of profinite $G$ then $H/K$ is a finite discrete space. We can embed $H/K$ in a $H$-representation $V$, for example:
\begin{align*}
H/K&\rightarrow \mathbb{R}[H/K]\\hK &\mapsto hK.
\end{align*}
This map is injective. We can find a sufficiently small real number $\delta>0$ such that the open balls of radius $\delta$ around $hK$ are pairwise disjoint. We therefore have a $H$-map:
\begin{align*}
j\colon H_+\underset{K}{\wedge}D\left(\mathbb{R}[H/K]\right)&\rightarrow \mathbb{R}[H/K]\\ [h,x]&\mapsto h \left( eK  +\delta x\right)
\end{align*}
which is an embedding of the open discs with image the open $\delta$-balls. This yields a $H$-equivariant Thom-Pontryagin collapse map:
\begin{align*}
c\colon S^{\mathbb{R}[H/K]}&\rightarrow \frac{ H_+\underset{K}{\wedge}D\left(\mathbb{R}[H/K]\right)}{H_+\underset{K}{\wedge}S\left(\mathbb{R}[H/K]\right)}\\x&\mapsto j^{-1}(x)\,\text{if}\,x\in \im(j),\,\text{and}\\x&\mapsto \infty\,\text{else}.
\end{align*}
We compose this with a $H$-homeomorphism:
\begin{align*}
\frac{ H_+\underset{K}{\wedge}D\left(\mathbb{R}[H/K]\right)}{H_+\underset{K}{\wedge}S\left(\mathbb{R}[H/K]\right)}\cong H_+\underset{K}{\wedge}S^{\mathbb{R}[H/K]}
\end{align*} 
which stretches the interval by a homeomorphism $[0,1]\cong [0,\infty]$. We therefore have a $H$-equivariant map:
\begin{align*}
S^{\mathbb{R}[H/K]}\rightarrow H\underset{K}{\wedge}S^{\mathbb{R}[H/K]}
\end{align*}
We then apply the functor $\Sigma_V^{\infty}(-)$ with respect to $H$ and $G_+\underset{H}{\wedge}(-)$, where $V=\mathbb{R}[H/K]$, to obtain:
\begin{align*}
\tau^H_K\colon \Sigma^{\infty}G/H_+\rightarrow \Sigma^{\infty}G/K_+. 
\end{align*} 
Here we use that $G_+\underset{H}{\wedge}\Sigma^{\infty}_VS^V$ is weakly equivalent to $\Sigma^{\infty}G/H_+$, and
\begin{align*}
G_+\underset{H}{\wedge}\Sigma^{\infty}_V \left(H_+\underset{K}{\wedge}S^V\right)\simeq G_+\underset{H}{\wedge}\Sigma^{\infty}_VH/K_+\simeq\Sigma^{\infty}G/K_+.
\end{align*}
\end{construction}
The previous construction required $G$ to be profinite in order for the precise details to work. Namely, we use that $H/K$ is finite. The following result will detail how to calculate homotopy classes of maps between suspension spectra. This is recorded in \cite[Corollary 7.2]{Fausk} for compact Hausdorff groups.
\begin{corollary}\label{Ghomotpy}
Let $X$ and $Y$ be two based $G$-spaces where $X$ is a finite cell complex. Then there is a natural isomorphism:
\begin{align*}
\left[\Sigma^{\infty}X,\Sigma^{\infty}Y\right]^{G}\cong\underset{W\in\mathcal{V}}{\colim}\left[X\wedge S^W,Y\wedge S^W\right]^{G\text{-space}}.
\end{align*}
\end{corollary}

The following result tells us how to calculate homotopy groups of suspension spectra. We consider only the case where $G$ is profinite however the reference \cite[Proposition 7.10]{Fausk} deals with more general compact Hausdorff groups.
\begin{proposition}[Tom Dieck Splitting\index{Tom Dieck Splitting}]\label{tomDieck}
If $Y$ is a $G$-space for profinite $G$, then there is an isomorphism of abelian groups:
\begin{align*}
\underset{H}{\bigoplus}\pi_*\left(\Sigma^{\infty}\left(EW_GH_+\underset{W_GH}{\wedge}Y^H\right)\right)\cong \pi_*^G\left(\Sigma^{\infty}Y\right)
\end{align*}
where the sum ranges over all conjugacy classes of open subgroups $H$ of $G$.
\end{proposition} 
In this thesis we will be interested in rational homotopy theory so the next step is to construct the rational model structure on orthogonal $G$-spectra. From now onwards $G$ will be profinite. We begin by constructing a rational sphere spectrum.
\begin{construction}\label{rationalfree}
We start by taking a free resolution of $\mathbb{Q}$ as an abelian group of the form:
\begin{align*}
\xymatrix{0\ar[r]&R\ar[r]^f&F\ar[r]&\mathbb{Q}\ar[r]&0}.
\end{align*}
We can use that free abelian groups are direct sums of copies of $\mathbb{Z}$ to write this as:
\begin{align*}
\xymatrix{0\ar[r]&\underset{i}{\oplus}\mathbb{Z}\ar[r]^f&\underset{j}{\oplus}\mathbb{Z}\ar[r]&\mathbb{Q}\ar[r]&0,}
\end{align*}
with a preferred factor of $\underset{j}{\oplus}\mathbb{Z}$ which maps $1\in \mathbb{Z}$ to $1\in\mathbb{Q}$. If $M$ is any abelian group we can use that $\mathbb{Q}$ is flat to deduce that the following sequence is exact:
\begin{align*}
\xymatrix{0\ar[r]&\underset{i}{\oplus}M\ar[r]^f&\underset{j}{\oplus}M\ar[r]&M\otimes\mathbb{Q}\ar[r]&0}.
\end{align*}
If $H$ is any subgroup of $G$, we can apply this for $M=A(H)$ to deduce that there is an injective map:
\begin{align*}
f\colon \underset{i}{\oplus}A(H)\rightarrow \underset{j}{\oplus}A(H)
\end{align*}
such that $A(H)\otimes\mathbb{Q}\cong \underset{j}{\oplus}A(H)/\underset{i}{\oplus}A(H)$.
\end{construction}
The following is another characterisation of the Burnside ring of a profinite group in terms of homotopy theory of spheres, as seen in \cite[Lemma 7.11]{Fausk}.
\begin{theorem}\label{Burnsidehtpy}
If $G$ is a profinite group, then the Burnside ring\index{Burnside ring} $A(G)$ is equivalent to $\left[\mathbb{S},\mathbb{S}\right]^G$, where $\mathbb{S}$ is given by $\Sigma^{\infty}S^0$.
\end{theorem}
In particular, as a consequence of this theorem we have a ring monomorphism:
\begin{align*}
\mathbb{Z}&\rightarrow A(G)\\n&\mapsto n*
\end{align*}
where $*$ is the one point space and $n*$ is the coproduct of $*$ with itself $n$ times. If $n$ is negative then we consider formal differences. This means that each integer determines a class in $\left[\mathbb{S},\mathbb{S}\right]^G$. The following lemma will relate this algebraic construction to homotopy theory of spheres in a way which helps us to proceed to construct the rational sphere spectrum. We first recall the following definition of a compact object in a stable model category.
\begin{definition}\label{compact}
We say that an object $X$ in a stable model category is said to be \textbf{compact}\index{compact} if for any collection of objects $\left\lbrace A_i\right\rbrace_{i\in I}$ we have an isomorphism:
\begin{align*}
\left[X,\underset{i\in I}{\coprod}A_i\right]\cong\underset{i\in I}{\bigoplus}\left[X,A_i\right].
\end{align*} 
\end{definition} 
\begin{lemma}
If $H$ is a subgroup of $G$ we have isomorphisms:
\begin{align*}
\left[\underset{i}{\bigvee}\mathbb{S},\underset{j}{\bigvee}\mathbb{S}\right]^H&\cong \underset{i}{\prod}\underset{j}{\bigoplus}\left[\mathbb{S},\mathbb{S}\right]^H\cong \underset{i}{\prod}\underset{j}{\bigoplus}A(H)\\&\cong \Hom_{A(H)}\left(\underset{i}{\bigoplus}A(H),\underset{j}{\bigoplus}A(H)\right),
\end{align*}
where the indices $i$ and $j$ are the same as in Construction \ref{rationalfree}.
\end{lemma}
\begin{proof}
We begin with an application of the universal properties of colimits to deduce the following isomorphism:
\begin{align*}
\left[\underset{i}{\bigvee}\mathbb{S},\underset{j}{\bigvee}\mathbb{S}\right]^H\cong\underset{i}{\prod}\left[\mathbb{S},\underset{j}{\bigvee}\mathbb{S}\right]^H.
\end{align*}
We next apply the fact that $\mathbb{S}$ is $H$-compact and the fact that $A(H)\cong \left[\mathbb{S},\mathbb{S}\right]^H$, as in Theorem \ref{Burnsidehtpy}, to deduce the following:
\begin{align*}
\left[\underset{i}{\bigvee}\mathbb{S},\underset{j}{\bigvee}\mathbb{S}\right]^H\cong  \underset{i}{\prod}\underset{j}{\bigoplus}A(H).
\end{align*}
For the final isomorphism we use the fact that for any ring $R$ there is an isomorphism $\text{Hom}_{R\text{-Mod}}(R,M)\cong M$ for any $R$-module $M$, along with the universal properties of colimits.
\end{proof}
This lemma is useful since the monomorphism:
\begin{align*}
f\colon \underset{i}{\oplus}A(H)\rightarrow \underset{j}{\oplus}A(H)
\end{align*}
whose quotient determines the rational Burnside ring of $H$ yields a morphism:
\begin{align*}
g\colon \hat{f}\underset{i}{\bigvee}\mathbb{S}\rightarrow \hat{f}\underset{j}{\bigvee}\mathbb{S}
\end{align*}
whose homotopy class corresponds to $f$. We can now define the rational sphere spectrum.
\begin{definition}\label{ratspheredefn}
The \textbf{rational sphere spectrum}\index{rational sphere spectrum} $S^0\mathbb{Q}$ is defined by the following cofibre sequence:
\begin{align*}
\xymatrix{\hat{f}\underset{i}{\bigvee}\mathbb{S}\ar[r]^g&\hat{f}\underset{j}{\bigvee}\mathbb{S}\ar[r]&S^0\mathbb{Q}}.
\end{align*}
where $g$ is defined above.
\end{definition}
\begin{remark}\label{SQmapspectra}
The definition of the rational sphere spectrum comes with a canonical inclusion $\mathbb{S}\rightarrow S^0\mathbb{Q}$. To see this observe the following:
\begin{align*}
\xymatrix{\mathbb{S}\ar[r]^{\iota}&\underset{j}{\bigvee}\mathbb{S}\ar[r]^Q&\hat{f}\underset{j}{\bigvee}\mathbb{S}\ar[r]&S^0\mathbb{Q}.}
\end{align*}  
The first map is an inclusion of $\mathbb{S}$ into the wedge of the spheres corresponding to the chosen summand, which maps $1\in \mathbb{Z}$ to $1\in\mathbb{Q}$ from Construction \ref{rationalfree}. The map $Q$ is the map into the fibrant replacement of the wedge of spheres, and the final map is the cofibre map in the definition of $S^0\mathbb{Q}$.

If $X$ is a $G$-spectrum, we can smash both sides of this map by $X$ to obtain the following morphism:
\begin{align*}
X\rightarrow X\wedge S^0\mathbb{Q}.
\end{align*}
\end{remark}
 
Note that this definition is independent of the choice of representative $g$. In order to construct the rational model structure on orthogonal $G$-spectra we have to apply Bousfield localisation with respect to the rational sphere spectrum. In order to achieve this we need to understand the following definition.
\begin{definition}
Let $E$ be a cofibrant $G$-spectrum and let $X,Y$ and $Z$ be any $G$-spectra. 
\begin{enumerate}
\item A map $f\colon X\rightarrow Y$ is an \textbf{$E$-equivalence} if $\id_E\wedge f:E\wedge X\rightarrow E\wedge Y$ is a weak equivalence. 
\item $Z$ is \textbf{$E$-local} if $f^*\colon \left[Y,Z\right]^G\rightarrow \left[X,Z\right]^G$ is an isomorphism for all \\$E$-equivalences $f\colon X\rightarrow Y$.
\item An \textbf{$E$-localisation} of $X$ is an $E$-equivalence $\lambda\colon X\rightarrow Y$ from $X$ to an $E$-local object $Y$.
\item $Z$ is \textbf{$E$-acyclic} if the map $*\rightarrow Z$ is an $E$-equivalence.
\end{enumerate}  
\end{definition}
The following defines the Bousfield localisation\index{Bousfield localisation} of a model category.
\begin{definition}
If $E$ is a cofibrant $G$-spectrum, then the category of orthogonal $G$-spectra has an $E$-model structure whose weak equivalences are given by the $E$-equivalences and whose $E$-cofibrations are the same cofibrations as the original model structure. The $E$-fibrant objects are the fibrant $E$-local objects and $E$-fibrant approximation constructs a Bousfield localisation $f_X\colon X\rightarrow \hat{f}_EX$ of $X$ at $E$.  
\end{definition}
See \cite[Theorem 4.1.1]{Hirschhorn} or \cite[Definition 2.3]{BousfieldBarnes}. Since $S^0\mathbb{Q}$ is cofibrant we can consider the $S^0\mathbb{Q}$-model structure denoted $L_{S^0\mathbb{Q}}\mathcal{I}^{\mathcal{V}}_G\mathcal{S}$. The morphisms in the homotopy category are denoted by $[-,-]^G_{\mathbb{Q}}$. We can now define the rational homotopy groups of a $G$-spectrum. The following proposition states that the rational homotopy groups of a $G$-spectrum $X$ are as expected.
\begin{proposition}\label{rationalhomo}
Let $X$ be a $G$-spectrum, then we have an isomorphism \\$[\mathbb{S},X]_{\mathbb{Q}}^H=\pi^H_*(X\wedge S^0\mathbb{Q})\cong \pi^H_*(X)\otimes \mathbb{Q}$.
\end{proposition}
\begin{proof}
Let $H$ be any open subgroup of $G$, then we use the cofibre sequence from Definition \ref{ratspheredefn} to consider the long exact sequence:
\begin{align*}
\xymatrix{\ldots\ar[r]&\pi_n^H\left(X\wedge\hat{f}\underset{i}{\bigvee}\mathbb{S}\right)\ar[r]^{(\id\wedge g)_*}&\pi_n^H\left(X\wedge\hat{f}\underset{j}{\bigvee}\mathbb{S}\right)\ar[r]&\pi_n^H(X\wedge S^0\mathbb{Q})\ar[r]&\ldots}.
\end{align*}
This sequence is equivalent to the following sequences:
\begin{align*}
\xymatrix{\ldots\ar[r]&\pi_n^H\left(\underset{i}{\bigvee}X\right)\ar[r]^{(\id\wedge g)_*}&\pi_n^H\left(\underset{j}{\bigvee}X\right)\ar[r]&\pi_n^H(X\wedge S^0\mathbb{Q})\ar[r]&\ldots,\\
\ldots\ar[r]&\underset{i}{\oplus}\pi_n^H(X)\ar[r]^{\id\otimes g}&\underset{j}{\oplus}\pi_n^H(X)\ar[r]&\pi_n^H(X\wedge S^0\mathbb{Q})\ar[r]&\ldots,\\
\ldots\ar[r]&\pi_n^H(X)\otimes\left(\underset{i}{\oplus}\mathbb{Z}\right)\ar[r]^{\id\otimes g}&\pi_n^H(X)\otimes\left(\underset{j}{\oplus}\mathbb{Z}\right)\ar[r]&\pi_n^H(X\wedge S^0\mathbb{Q})\ar[r]&\ldots}.
\end{align*}
Since $\mathbb{Q}$ is flat the following sequence is exact for any abelian group $M$:
\begin{align*}
\xymatrix{0\ar[r]&M\otimes \left(\underset{i}{\oplus}\mathbb{Z}\right)\ar[r]&M\otimes \left(\underset{j}{\oplus}\mathbb{Z}\right)\ar[r]&M\otimes\mathbb{Q}\ar[r]&0}.
\end{align*}
It follows that $\id\otimes g$ is injective. Using the exactness of the sequences we obtain that $\pi_*^H(X\wedge S^0\mathbb{Q})$ is the cokernel of $\id\otimes g$. We know however that this cokernel is also equal to $\pi_*^H(X)\otimes \mathbb{Q}$. 
\end{proof}
\begin{corollary}
If $A$ is a compact $G$-spectrum then $[A,X]_{\mathbb{Q}}^G=[A,X]^G\otimes\mathbb{Q}$.
\end{corollary}
\begin{proof}
The argument for this is the same as the proof of Proposition \ref{rationalhomo}. As $A$ is compact $[A,\underset{i}{\bigvee}X_i]\cong\underset{i}{\oplus}[A,X_i]$, which was needed in Proposition \ref{rationalhomo}.
\end{proof}
\begin{corollary}\label{corollaryS0Qequiv}
A morphism $f\colon X\rightarrow Y$ is an $S^0\mathbb{Q}$-equivalence if and only if $\pi_*^H(f)\otimes\mathbb{Q}$ is an isomorphism for all $H$ open. Also $\mathbb{S}\rightarrow S^0\mathbb{Q}$ induced by Construction \ref{ratspheredefn} is an $S^0\mathbb{Q}$-equivalence.
\end{corollary}
\begin{proof}
The first statement is a direct consequence of Proposition \ref{rationalhomo}. For the second, suppose $f\colon\mathbb{S}\rightarrow S^0\mathbb{Q}$ is induced by Remark \ref{SQmapspectra}. To see that this is an $S^0\mathbb{Q}$-equivalence, we look at the following map:
\begin{align*}
\pi_*^H(f\wedge\id)\colon \pi_*^H(\mathbb{S}\wedge S^0\mathbb{Q})\rightarrow \pi^H_*(S^0\mathbb{Q}\wedge S^0\mathbb{Q}).
\end{align*}
Applying the fact that $\mathbb{S}$ is the unit and Proposition \ref{rationalhomo} we have:
\begin{align*}
\pi_*^H(f\wedge\id)\colon \pi_*^H(\mathbb{S})\otimes\mathbb{Q}&\rightarrow \pi^H_*(\mathbb{S})\otimes\mathbb{Q}\otimes\mathbb{Q}\\ \left[a,q\right]&\mapsto \left[a,1,q\right]
\end{align*}
which is an isomorphism. Here we use that $\mathbb{Q}\underset{\mathbb{Z}}{\otimes}\mathbb{Q}\cong\mathbb{Q}$.
\end{proof}
\begin{lemma}\label{rathtpylemma}
If $X$ is a $G$-spectrum then the following are equivalent:
\begin{enumerate}
\item $X$ is $S^0\mathbb{Q}$-local.
\item $\pi_*^H(X)$ is a $\mathbb{Q}$-module for every open subgroup $H$.
\item The morphism $X\rightarrow X\wedge S^0\mathbb{Q}$ from Remark \ref{SQmapspectra} is a $\pi_*$-isomorphism.
\end{enumerate}
\end{lemma}
\begin{proof}
We begin by showing that condition one implies condition two. Take any $S^0\mathbb{Q}$-local spectrum $X$. By considering the class of the identity morphism $[\id_{\mathbb{S}}]$ in $\pi_*(\mathbb{S})$, we add $[\id_{\mathbb{S}}]$ to itself $n$ times to obtain a representative $n\colon \mathbb{S}\rightarrow \mathbb{S}$. We consider the morphism:
\begin{align*}
n\wedge \id\colon \mathbb{S}\wedge S^0\mathbb{Q}\rightarrow \mathbb{S}\wedge S^0\mathbb{Q}.
\end{align*}
By Proposition \ref{rationalhomo} we know that this is a $\pi_*$-isomorphism, hence $n$ is an $S^0\mathbb{Q}$-equivalence. By definition of $X$ being $S^0\mathbb{Q}$-local we know that the following map is an isomorphism:
\begin{align*}
n^*\colon \pi_*(X)\rightarrow \pi_*(X).
\end{align*}
We next show that condition two implies condition three. By assumption the map:
\begin{align*}
\pi_*(X)&\rightarrow\pi_*(X)\otimes\mathbb{Q}\\
a&\mapsto [a,1]
\end{align*}
is an isomorphism. From Construction \ref{rationalfree} and Definition \ref{ratspheredefn}, this is the map induced by the morphism in condition three by $\pi_*(-)$ as required. To see that condition three implies condition one, consider the fibrant replacement of $X$ with respect to the rational model structure, $\lambda\colon X\rightarrow f_{\mathbb{Q}}X$, with $S^0\mathbb{Q}$-local codomain. We observe the following square:
\begin{align*}
\xymatrix{X\ar[r]\ar[d]&f_{\mathbb{Q}}X\ar[d]\\X\wedge S^0\mathbb{Q}\ar[r]&f_{\mathbb{Q}}X\wedge S^0\mathbb{Q}}
\end{align*}  
The left hand side arrow is a $\pi_*$-isomorphism by assumption, the bottom arrow is a $\pi_*$-isomorphism by definition of fibrant replacement and the right hand side arrow is given by $\pi_*(\id)\otimes\mathbb{Q}$. Using the condition one implies condition two implication proved earlier, we can apply Corollary \ref{corollaryS0Qequiv} to deduce that the right hand side arrow is a $\pi_*$-isomorphism. We also use the fact that both the source and target have rational homotopy groups. This implies that $X\rightarrow f_{\mathbb{Q}}X$ is a $\pi_*$-isomorphism. Since $X$ is weakly equivalent to an $S^0\mathbb{Q}$-local object we can deduce that $X$ is $S^0\mathbb{Q}$-local, as in \cite[Proposition 3.2.2]{Hirschhorn}.
\end{proof}
\begin{proposition}\label{rationhtpycat}
If $X$ and $Y$ are $G$-spectra then the homotopy classes of maps from $X$ to $Y$ in the homotopy category with respect to the rational model structure, denoted $\left[X,Y\right]^G_{\mathbb{Q}}$, is a $\mathbb{Q}$-module. 
\end{proposition}
\begin{proof}
In order to prove that $\left[X,Y\right]^G_{\mathbb{Q}}$ is a $\mathbb{Q}$-module we will show that for each $n\in\mathbb{Z}$, there is a self map $n^*$ on $\left[X,Y\right]^G_{\mathbb{Q}}$ which is an isomorphism. Let $n\in \mathbb{Z}$ arbitrary and $X$ be a $G$-spectrum. Then we have a morphism $n_X$ in $\left[X,X\right]^G_{\mathbb{Q}}$ given by adding the class of the identity to itself $n$ times. If we can show that $n_X\colon X\rightarrow X$ is an $S^0\mathbb{Q}$-equivalence, then it would follow from Lemma \ref{rathtpylemma} that we have the following isomorphism (since we can assume $Y$ to be $S^0\mathbb{Q}$-fibrant):
\begin{align*}
n_X^*\colon \left[X,Y\right]^G_{\mathbb{Q}}&\rightarrow \left[X,Y\right]^G_{\mathbb{Q}}\\\left[\alpha\right]&\mapsto \left[\alpha\circ n_X\right]
\end{align*}
as required. We need to know that $n_X$ induces an isomorphism of rational homotopy groups, which follows from Corollary \ref{corollaryS0Qequiv}.
\end{proof}
Rational $G$-spectra is related to $G$-equivariant cohomology theories as seen in the introduction. We next characterise the homotopy category of rational $G$-spectra, where we are working strictly with profinite groups. To do this we define the category $\pi_0(\mathfrak{O}^{\mathbb{Q}}_G)$ to be the full subcategory of the homotopy category of rational $G$-spectra, whose objects are the orbit spectra $\Sigma^{\infty}G/H_{+}$ for $H$ an open subgroup of $G$. We will denote the objects in this category by $G/H_{+}$ and denote by $\pi_0(\mathfrak{O}_G)$ the integral version of $\pi_0(\mathfrak{O}^{\mathbb{Q}}_G)$. The next characterisation utilises \cite[Theorem 5.1.1]{StableModel}.
\begin{theorem}\label{GSpChar}
Let $G$ be a profinite group. The category of rational $G$-spectra is Quillen equivalent to the category $\ch\left(\fun_{\text{Ab}}\left(\pi_0(\mathfrak{O}^{\mathbb{Q}}_G),\mathbb{Q}\Mod\right)\right)$, where we consider contravariant functors and $\ch(-)$ represents the category of chain complexes. The last category has the projective model structure.
\end{theorem}
In order to apply \cite[Theorem 5.1.1]{StableModel} we need to prove that the homotopy classes of maps $\left[G/H_{+},G/K_{+}\right]_*^{G}$ are concentrated in degree zero, and that the objects of $\pi_0(\mathfrak{O}_G^{\mathbb{Q}})$ are compact generators of $G\text{-Spectra}_{\mathbb{Q}}$. Using the terminology of \cite{StableModel} this says that the objects of $\pi_0(\mathfrak{O}_G^{\mathbb{Q}})$ are \textbf{tiltors}. We recall Definition \ref{compact} for the definition of a compact object. We now define a generating set of objects of a category.
\begin{definition}
If $\mathcal{T}$ is a triangulated category, a \textbf{localising category}\index{localising category} is a full triangulated subcategory (with triangles inherited from $\mathcal{T}$) which is closed under coproducts. A set of objects $\mathcal{P}$ of $\mathcal{T}$ are called \textbf{generators}\index{generators} if the only localising subcategory of $\mathcal{T}$ containing $\mathcal{P}$ is $\mathcal{T}$. 
\end{definition}
The definition of compact generators can be found in \cite[Definition 2.1.2]{StableModel}. We now observe a lemma which is seen in \cite[Lemma 2.2.1]{StableModel} and provides a characterisation of generators of a category.
\begin{lemma}\label{generatorcharlem}
Let $\mathcal{T}$ be a triangulated category with infinite coproducts and $\mathcal{P}$ be a collection of compact objects. Then the following conditions are equivalent:
\begin{enumerate}
\item $\mathcal{P}$ is a collection of generators,
\item an object $X$ of $\mathcal{T}$ is trivial if and only if $\left[P,X\right]_*=0$ for all $P\in\mathcal{P}$.
\end{enumerate}
\end{lemma}
We begin working towards proving this theorem by understanding the following lemma.
\begin{lemma}\label{tenscom}
There is an equivalence of categories:
\begin{align*}
\pi_0(\mathfrak{O}^{\mathbb{Q}}_G)\cong \pi_0(\mathfrak{O}_G\otimes \mathbb{Q}).
\end{align*}
\end{lemma}
\begin{proof}
To prove this we need to show that:
\begin{align*}
[G/H_+,G/J_+]_{\mathbb{Q}}=[G/H_+,G/J_+]\otimes\mathbb{Q}.
\end{align*}
We know that $[G/H_+,G/J_+]_{\mathbb{Q}}=[G/H_+,G/J_+\wedge S^0\mathbb{Q}]$ by definition. This is also the same as $\pi^H_*(G/J_+\wedge S^0\mathbb{Q})$. Applying Proposition \ref{rationalhomo} gives us that $[G/H_+,G/J_+]_{\mathbb{Q}}$ is equivalent to $\pi_*^H(G/J_+)\otimes\mathbb{Q}$. However the latter can also be written as $[G/H_+,G/J_+]\otimes\mathbb{Q}$, as required.
\end{proof}
\begin{theorem}
The morphisms in $\pi_0(\mathfrak{O}^{\mathbb{Q}}_G)$ are concentrated in degree zero.
\end{theorem}
\begin{proof}
See \cite[Theorem 2.9]{BarZp}.
\end{proof}
\begin{proposition}\label{tilt}
The objects of $\pi_0(\mathfrak{O}_G^{\mathbb{Q}})$ are tiltors for $G\text{-spectra}_{\mathbb{Q}}$. 
\end{proposition}
\begin{proof}
We now prove that $\pi_0(\mathfrak{O}_G^{\mathbb{Q}})$ are tiltors. First notice that the homotopy classes of maps are concentrated in degree zero by Lemma \ref{tenscom}. For compactness, observe that for each $G/H_{+}$ and each $G$-spectrum $X$ we have:
\begin{align*}
\left[G/H_+,X\right]_*^G\cong [\mathbb{S},X]_*^H=\pi^H_*(X).
\end{align*}
We know also that for any collection of $G$-spectra $X_i$ we have:
\begin{align*}
\pi^H_*\left(\underset{i\in I}{\coprod}{X_i}\right)\cong \underset{i\in I}{\bigoplus}\pi^H_*\left(X_i\right),
\end{align*}
as required. To prove that these objects are generators we can apply Lemma \ref{generatorcharlem} since the objects $G/H_+$ are compact. This will show us that these objects are generators if $\left[G/H_+,X\right]_*=0$ for all $G/H_+$ implies that $X$ is trivial. We know that:
\begin{align*}
\left[G/H_+,X\right]_*=\pi_*^H(X).
\end{align*}
From the definition of the model structure on the model category of $G$-spectra we know that $\pi_*^H(X)=0$ for all $H$ open if and only if $X$ is trivial in the homotopy category, as required.
\end{proof}
We can now prove Theorem \ref{GSpChar}.
\begin{proof}
We begin by applying Proposition \ref{tilt} to deduce that $\pi_0(\mathfrak{O}_G^{\mathbb{Q}})$ are tiltors for $G\text{-Spectra}_{\mathbb{Q}}$. Therefore we can apply \cite[Theorem 5.1.1]{StableModel} to $G\text{-spectra}_{\mathbb{Q}}$, which tells us that it is Quillen equivalent to the model category:
\begin{align*}
\text{Ch}\left(\text{Fun}_{\text{Ab}}\left(\pi_0(\mathfrak{O}^{\mathbb{Q}}_{G}),\text{Ab}\right)\right), 
\end{align*}
which are the contravariant additive functors. The model structure on \\$\text{Ch}\left(\text{Fun}_{\text{Ab}}\left(\pi_0(\mathfrak{O}^{\mathbb{Q}}_{G}),\text{Ab}\right)\right)$ is given by the projective model structure, where the weak equivalences are the homology isomorphisms and the fibrations are the level-wise epimorphisms, see \cite[pp. 135]{StableModel}. Observe that in the model category of rational $G$-spectra we have $\left[X,Y\right]_*^G$ is a $\mathbb{Q}$-module by Proposition \ref{rationhtpycat}. Hence $\text{Fun}_{\text{Ab}}\left(\pi_0(\mathfrak{O}^{\mathbb{Q}}_{G}),\text{Ab}\right)$ is equivalent to $\text{Fun}_{\text{Ab}}\left(\pi_0(\mathfrak{O}^{\mathbb{Q}}_{G}),\mathbb{Q}\text{-Mod}\right)$. We can apply Lemma \ref{tenscom} in the following way. If we have a functor:
\begin{align*}
F:\pi_0(\mathfrak{O}^{\mathbb{Q}}_{G})\rightarrow \text{Ab},
\end{align*}
then we can see that $F(a)$ must belong to $\mathbb{Q}$-mod for any $a\in \pi_0(\mathfrak{O}^{\mathbb{Q}}_{G})$. Since Lemma \ref{tenscom} says that $\pi_0(\mathfrak{O}^{\mathbb{Q}}_{G})(a,a)\cong \pi_0(\mathfrak{O}_{G})(a,a)\otimes\mathbb{Q}$, we have a morphism $q_a=q\otimes\id$ on $a$ for any $q\in \mathbb{Q}$. Therefore $F(q_a)$ is an isomorphism on $F(a)$ for every $q\in\mathbb{Q}$, which defines the $\mathbb{Q}$-module structure on $F(a)$.

We know that the category of rational $G$-spectra is Quillen equivalent to the model category $\text{Ch}\left(\text{Fun}_{\text{Ab}}\left(\pi_0(\mathfrak{O}^{\mathbb{Q}}_{G}),\mathbb{Q}\text{-Mod}\right)\right)$.
\end{proof}
\section{Mackey functor characterisation}
In order to further characterise the category of rational $G$-spectra we will prove that the category of rational $G$-Mackey functors are equivalent to the category $\text{Fun}_{\text{Ab}}\left(\pi_0(\mathfrak{O}^{\mathbb{Q}}_G),\mathbb{Q}\text{-Mod}\right)$ of contravariant additive functors. We will use \cite[Proposition 5.9.6]{LMSMcC} which gives the corresponding result when $G$ is finite. $G$ will be assumed to be profinite throughout.

We denote by $\text{Orb}(G)$ the category of orbits $G/H$ for $H$ an open subgroup of $G$ and morphisms the $G$-equivariant maps. Notice these boil down to compositions of projections and conjugations. We begin by establishing a couple of useful lemmas which combine to generalise the argument given in the case when $G$ is finite to the more general profinite case. During these lemmas $G$ shall be assumed to be profinite.
\begin{lemma}\label{topmacklem1}
Let $U$ be a $G$-universe and $V\subseteq U$ a finite dimensional subrepresentation. Then there is an open normal subgroup $N$ of $G$ such that $V^N=V$. Consequently $U=\underset{N\underset{open}{\unlhd} G}{\lim}U^N$, where the diagram is given by the inclusions of fixed points.
\end{lemma}
\begin{proof}
The first statement, that the action of $G$ on a finite subrepresentation $V$ factors through an open normal subgroup, follows from the third condition of Definition \ref{Univdefn}. This says that the action on $V$ factors through a compact Lie quotient $G/N$. However, we can apply the fact that $G$ is profinite to deduce that the only way $G/N$ can be both compact Lie and totally disconnected is if $G/N$ is a finite discrete group. Therefore $N$ is open normal and $V^N=V$. Clearly $\underset{N\underset{open}{\unlhd} G}{\lim}U^N\subseteq U$, for the other direction if $x\in U$ then we can choose a finite subrepresentation $V$ containing $x$. An application of the first part of the lemma gives an open normal subgroup $N$ of $G$ such that $V^N=V$. Therefore $x\in V^N\subseteq U^N$ and hence is represented in the limit.
\end{proof}
The next lemma is a key part of the main theorem of this chapter. For this lemma we will show that the $G$-homotopy classes of morphisms from $G/H_{+}$ to $G/J_{+}$ are determined by each of the $G/N$-homotopy classes of morphisms from $(G/N)/(H/N)_{+}$ to $(G/N)/(J/N)_{+}$ for $N$ open and normal in $G$ and contained in $H\cap J$. We use Corollary \ref{Ghomotpy} throughout, which tells us that:
\begin{align*}
\left[G/H_+,G/J_+\right]_*^{G}\cong\underset{V\in\mathcal{V}}{\colim}\left[S^V\wedge G/H_+,S^V\wedge G/J_+\right]_*^{G\text{-Space}}.
\end{align*}

If $N$ is an open normal subgroup of $G$ and $\varepsilon\colon G\rightarrow G/N$ is the quotient map we define the inflation functor:
\begin{align*}
\varepsilon^*\colon G/N\text{-space}\rightarrow G\text{-space}
\end{align*}
where if $X$ is a $G/N\text{-space}$ then $\varepsilon^*(X)$ has underlying set $X$ and the action is defined by $g*x=\varepsilon(g)x$ for $g\in G$ and $x\in X$.

Notice that if we consider the $G/N$-universe $U^N$ then each finite dimensional $W\subseteq U^N$ is of the form $V^N$ for $V\subseteq U$ finite dimensional. This is clear since by using the inflation functor we can see that $W\subseteq U$ is an indexing representation for $G$ and $(\varepsilon^*W)^N=W$.

Observe that if $N_1\leq N_2$ are both open normal subgroups of $G$ then we have a morphism:
\begin{align*}
\xymatrix{\left[(G/N_2)/(H/N_2)_{+},(G/N_2)/(J/N_2)_{+}\right]^{G/N_2}\ar[d]\\
\left[(G/N_1)/(H/N_1)_{+},(G/N_1)/(J/N_1)_{+}\right]^{G/N_1}}
\end{align*}
since if $V^{N_2}\subseteq U^{N_2}$ is a finite subrepresentation we can use that $N_1\leq N_2$ to deduce that $\left({V^{N_2}}\right)^{N_1}=V^{N_2}$, and hence that $V^{N_2}\subseteq U^{N_1}$ is an indexing representation for $U^{N_1}$ using the inflation functors. We can therefore view a representative $f\colon G/H_{+}\wedge S^{V^{N_2}}\rightarrow G/J_{+}\wedge S^{V^{N_2}}$ of a class in 
\begin{align*}
\left[(G/N_2)/(H/N_2)_{+},(G/N_2)/(J/N_2)_{+}\right]^{G/N_2}
\end{align*}
as a representative of a class in 
\begin{align*}
\left[(G/N_1)/(H/N_1)_{+},(G/N_1)/(J/N_1)_{+}\right]^{G/N_1}.
\end{align*}
\begin{lemma}\label{topmacklem2}
There is an isomorphism of $\mathbb{Q}$-modules: 
\begin{align*}
\underset{N\underset{open}{\unlhd} G}{\colim}\left[(G/N)/(H/N)_{+},(G/N)/(J/N)_{+}\right]^{G/N}\cong\left[G/H_{+},G/J_{+}\right]^G
\end{align*}
for $H$ and $J$ open subgroups of $G$. 
\end{lemma}
\begin{proof}
Take $[f]\in \left[G/H_{+},G/J_{+}\right]^G$ where $f\colon G/H_{+}\wedge S^V\rightarrow G/J_{+}\wedge S^V$ is a representative and $S^V$ is the one point compactification of $V$. By Lemma \ref{topmacklem1} we can choose an open normal subgroup $N$ of $G$ such that $V=V^N$. Since we can always replace $N$ by a smaller open normal subgroup contained in $H\cap J$, we can assume that $N\leq H\cap J$. Combining this with Lemma \ref{non3iso} we can view $f$ as a map of the form:
\begin{align*}
f\colon(G/N)/(H/N)_{+}\wedge S^{V^N}\rightarrow (G/N)/(J/N)_{+}\wedge S^{V^N}.
\end{align*}
For the other direction take a representative in
\begin{align*}
\underset{N\underset{open}{\unlhd} G}{\colim}\left[(G/N)/(H/N)_{+},(G/N)/(J/N)_{+}\right]^{G/N}
\end{align*}
of the form:
\begin{align*}
h\colon(G/N)/(H/N)_{+}\wedge S^{V^N}\rightarrow (G/N)/(J/N)_{+}\wedge S^{V^N}.
\end{align*}
Notice that since $N$ is normal we know that $V^N$ is a $G$-representation by consideration of the inflation functor. Therefore as $V^N\subseteq V\subseteq U$, we know that $V^N$ is a finite subrepresentation of $U$, so if $W=V^N$ then the map
\begin{align*}
h\colon G/H_{+}\wedge S^{W}\rightarrow G/J_{+}\wedge S^{W}
\end{align*}
is represented in $\left[G/H_{+},G/J_{+}\right]^G$.
\end{proof} 
The next lemma is proven similarly, it refers to the span category used in Theorem \ref{Catdef}. Namely, if we take a span of the form:
\begin{align*}
\xymatrix{G/H&G/L\ar[r]^a\ar[l]^b&G/J},
\end{align*}
we can choose an open normal subgroup $N$ contained in $H,L$ and $J$, and apply Lemma \ref{non3iso} to obtain:
\begin{align*}
\xymatrix{\left(G/N\right)/\left(H/N\right)&\left(G/N\right)/\left(L/N\right)\ar[r]^a\ar[l]^b&\left(G/N\right)/\left(J/N\right)}.
\end{align*}
\begin{lemma}\label{topmacklem3}
If $H$ and $J$ are open subgroups of $G$ profinite, then there is an isomorphism:
\begin{align*}
\sspan_G\left(G/H,G/J\right)\cong \underset{N\underset{open}{\leq}G}{\colim}\,\sspan_{G/N}\left((G/N)/(H/N),(G/N)/(J/N)\right).
\end{align*}
\end{lemma}
We now give a construction for a functor which we will show is an equivalence of categories between $\pi_0(\mathfrak{O}^{\mathbb{Q}}_G)$ and $\text{Span}\left(\text{Orb}(G)\right)$. We briefly include $\Sigma^{\infty}$ into our notation for clarity.
\begin{construction}\label{top_mack}
We define an assignment:
\begin{align*}
\psi\colon\text{Span}\left(\text{Orb}(G)\right)&\rightarrow \pi_0(\mathfrak{O}^{\mathbb{Q}}_G)\\G/H&\mapsto \Sigma^{\infty}G/H_{+}\,\text{on objects}\\\left[\xymatrix{G/H&G/L\ar[r]^{\beta}\ar[l]^{\alpha}&G/J}\right]&\mapsto \Sigma^{\infty}\beta\circ \tau(\alpha)\,\text{on morphisms}
\end{align*}
where $\tau(\alpha)$ is the transfer map construction associated to $\alpha$ as seen in Construction \ref{transfer} from \cite[Construction 2.5.1]{LMSMcC}. 
\end{construction}
In the case where $G$ is finite, the assignment $\psi$ has the following property, as seen in \cite[Proposition 5.9.6]{LMSMcC}.
\begin{theorem}
If $G$ is a finite group then the assignment $\psi$ from Construction \ref{top_mack} yields an equivalence of categories.
\end{theorem} 
If we write $\psi_N$ for the functor for the finite discrete group $G/N$, we have an isomorphism induced by $\psi_N$:
\begin{align*}
\sspan_{G/N}\left((G/N)/(H/N),(G/N)/(J/N)\right)\rightarrow \left[G/H_{+},G/J_{+}\right]^{G/N}
\end{align*}
and we therefore have the following commutative diagram:
\begin{align*}
\xymatrix{\text{Span}_G\left(G/H,G/J\right)\ar[d]^{\cong}\ar[r]^{\psi(-)}&\left[G/H_{+},G/J_{+}\right]^G\\  \underset{N\underset{open}{\leq}G}{\colim}\,\text{Span}_{G/N}\left((G/N)/(H/N),(G/N)/(J/N)\right)\ar[r]_(.6){\cong}&\underset{N\underset{open}{\unlhd} G}{\colim}\left[G/H_{+},G/J_{+}\right]^{G/N}\ar[u]^{\cong}}
\end{align*}
The bottom map exists by applying the universal property of colimits to the following collection of maps:
\begin{align*}
\text{Span}_{G/N}\left((G/N)/(H/N),(G/N)/(J/N)\right)&\rightarrow \underset{N\underset{open}{\unlhd} G}{\colim}\left[G/H_{+},G/J_{+}\right]^{G/N}\\ x&\mapsto [\psi_N(x)]
\end{align*}
We can do this since we have the following commuting square: 
\begin{align*}
\xymatrix{\text{Span}_{G/N}(G/H,G/J)\ar[d]\ar[r]^{\psi_N}&[G/H_+,G/J_+]^{G/N}\ar[d]\\ \text{Span}_{G/N^{\prime}}(G/H,G/J)\ar[r]^{\psi_{N^{\prime}}}&[G/H_+,G/J_+]^{G/N^{\prime}}}
\end{align*}
This is true since the two vertical arrows are inclusions. The functors $\tau(-)$ and $\Sigma^{\infty}(-)$ are applied to the same $G$-maps regardless of which order we apply the functors $\psi_N$. We will see in Proposition \ref{psifunctor} that $\psi$ is a functor. This diagram proves the following proposition.
\begin{proposition}\label{topMiso}
For $H$ and $J$ open subgroups of profinite group $G$ there is an isomorphism:
\begin{align*}
\sspan_G\left(G/H,G/J\right)\rightarrow \left[G/H_{+},G/J_{+}\right]^G
\end{align*}
induced by the assignment $\psi$.
\end{proposition}
\begin{proposition}\label{psifunctor}
The assignment $\psi$ constructed in Construction \ref{top_mack} is a functor.
\end{proposition}
\begin{proof}
Clearly the functor maps identity morphisms in $\text{Span}(\text{Orb}(G))$ to identity morphisms in $\pi_0(\mathfrak{O}^{\mathbb{Q}}_G)$. It is left to verify that this assignment respects composition.

Let
\begin{align*}
x_1=\left[\xymatrix{G/H&G/L_1\ar[r]^{\beta_1}\ar[l]^{\alpha_1}&G/K}\right]
\end{align*}
and 
\begin{align*}
x_2=\left[\xymatrix{G/K&G/L_2\ar[r]^{\beta_2}\ar[l]^{\alpha_2}&G/J}\right]
\end{align*}
We begin by choosing an open normal subgroup $N$ which is smaller than all of the subgroups in the definition of $x_1$ and $x_2$. If $\psi_N$ is the functor known to hold for $G/N$, then $\psi_N(x_1\circ x_2)=\psi_N(x_1)\circ\psi_N(x_2)$. Finally we apply the isomorphisms in Lemma \ref{topmacklem2} and \ref{topmacklem3} and the induced commuting square to obtain $\psi(x_1\circ x_2)=\psi(x_1)\circ\psi(x_2)$.
\end{proof}
Since we now know that $\psi$ is a functor we have the following theorem.
\begin{theorem}\label{orb-span}
The functor $\psi$ constructed in Construction \ref{top_mack} is an equivalence of categories.
\end{theorem}
\begin{proof}
We need to show that $\psi$ is essentially surjective and that $\psi$ induces an isomorphism on morphism sets. This last follows from Proposition \ref{topMiso}. For the first if we take $\Sigma^{\infty}G/H_{+}$, then this is in the image of $G/H$ with respect to $\psi$ as required.
\end{proof}
This leads to the following corollary which provides an alternative definition of a Mackey functor.
\begin{corollary}
The category of contravariant functors from $\pi_0(\mathfrak{O}^{\mathbb{Q}}_G)$ to $\mathbb{Q}\Mod$ is equivalent to the category of contravariant functors from $\sspan(\orb(G))$ to $\mathbb{Q}\Mod$.\index{Mackey functor}
\end{corollary}
\begin{proof}
This follows since the domain categories are equivalent by Theorem \ref{orb-span}.
\end{proof}
We also have the following consequences of Theorem \ref{orb-span}.
\begin{corollary}\label{mackeymodelequi}
There is an equivalence of categories:
\begin{align*}
\fun_{\ab}\left(\pi_0(\mathfrak{O}^{\mathbb{Q}}_G),\mathbb{Q}\Mod\right)\cong \mackey_{\mathbb{Q}}\left(G\right).
\end{align*}
\end{corollary}
A consequence of this particular corollary is that if $X$ is a $G$-spectrum then we have a graded $G$-Mackey functor\index{Mackey functor}:
\begin{align*}
\pi_0(\mathfrak{O}^{\mathbb{Q}}_G)&\rightarrow \mathbb{Q}\Mod\\G/H_{+}&\mapsto \pi_*^H\left(X\right)\cong [G/H_+,X]^G_*.
\end{align*}
More generally, a rational $G$-spectrum $E$ and a $G$-space $X$ determine a rational $G$-Mackey functor as follows:
\begin{align*}
\pi_0(\mathfrak{O}^{\mathbb{Q}}_G)&\rightarrow \mathbb{Q}\text{-Mod}\\G/H_{+}&\mapsto \left[G/H_{+}\wedge X,E\right]_*^G
\end{align*}
We may extend this equivalence to chain complexes. 
\begin{corollary}
There is an equivalence of categories:
\begin{align*}
\fun_{\ch\left(\ab\right)}(\pi_0(\mathfrak{O}^{\mathbb{Q}}_G),\ch\left(\mathbb{Q}\Mod\right))\cong \fun_{\ch\left(\ab\right)}(\sspan\left(\orb(G)\right),\ch\left(\mathbb{Q}\Mod\right)).
\end{align*}
Furthermore there is an equivalence:
\begin{align*}
\fun_{\ch\left(\ab\right)}(\sspan\left(\orb(G)\right),\ch\left(\mathbb{Q}\Mod\right))&\cong \ch\left(\fun_{\ab}(\sspan\left(\orb(G)\right),\mathbb{Q}\Mod\right)\\&\cong \ch\left(\mackey_{\mathbb{Q}}(G)\right).
\end{align*}
\end{corollary}
\begin{proof}
The first equivalence follows from Theorem \ref{orb-span} and the second follows from the definition of the two categories.
\end{proof}

As a consequence of this chapter we have the following diagram which illustrates the characterisation of rational $G$-Spectra for $G$ profinite.
\begin{align}\label{diagram}
\xymatrix{\text{Rational $G$-spectra}\ar[d]^{\simeq}\\ \text{Ch}\left(\text{Fun}_{\text{Ab}}\left(\pi_0(\mathfrak{O}^{\mathbb{Q}}_G),\mathbb{Q}\text{-Mod}\right)\right)\ar[d]^{\cong}\\ \text{Ch}\left(\text{Fun}_{\text{Ab}}\left(\text{Span}\left(\text{Orb}(G)\right),\mathbb{Q}\text{-Mod}\right) \right)\ar[d]^{=}\\ \text{Ch}\left(\text{Mackey}_{\mathbb{Q}}(G)\right)}
\end{align}
where $\simeq$ represents a Quillen equivalence and $\cong$ denotes an equivalence of categories. The final equality refers to the characterisation of rational Mackey functors in Theorem \ref{Catdef}.
\begin{remark}\label{MackModel}
In light of Diagram \ref{diagram} we arrive at the conclusion that \\$\text{Ch}\left(\text{Mackey}_{\mathbb{Q}}(G)\right)$ is Quillen equivalent to $G\text{-Spectra}_{\mathbb{Q}}$. The model structure on $\text{Ch}\left(\text{Mackey}_{\mathbb{Q}}(G)\right)$ which achieves this is the projective model structure. Namely, the weak equivalences are the homology isomorphisms and the fibrations are the level-wise epimorphisms. See \cite[Theorem 5.1.1]{StableModel}.
\end{remark}
This diagram delivers the following theorem.
\begin{theorem}\label{equicohomchar}
If $G$ is a profinite group then the category of rational $G$-spectra is Quillen equivalent to $\ch\left(\mackey_{\mathbb{Q}}\left(G\right)\right)$ with the projective model structure.
\end{theorem}

\chapter{G-Equivariant Sheaves}
In this chapter we will study $G$-equivariant sheaves of $\mathbb{Q}$-modules over $G$-spaces for profinite $G$. We will begin by understanding the properties of this category, proving that it is an abelian category (Proposition \ref{Gsheafabelian}), and constructing finite limits (Construction \ref{limconstruct}), infinite products (Construction \ref{profinfprod}), and arbitrary colimits (Construction \ref{sheafcolim}). We will also define and consider some important subcategories of $G$-equivariant sheaves when $X=SG$, including a subcategory which we shall eventually show is equivalent to rational $G$-Mackey functors called Weyl-$G$-sheaves. We will look at some useful properties of $G$-equivariant sheaves over a profinite $G$-space $X$, since we will see that both $G$ and $X$ having a closed-open basis adds more theory which we will use to our advantage (Proposition \ref{Weylequi}). Finally we will set up adjunctions involving $G$-equivariant sheaves which we shall use to calculate the injective dimension of Weyl-$G$-sheaves. 
\section{Category of G-equivariant sheaves}
In the first section we will define $G$-equivariant sheaves and Weyl-$G$-sheaves of $\mathbb{Q}$-modules for profinite $G$. We will also define the morphisms between these objects showing that we have well defined categories.
\subsection{Defining the Category of G-sheaves over a G-space}
We begin with the definition of a $G$-equivariant sheaf over a profinite $G$-space $X$ where $G$ is a profinite group. Consider the $G$-action map:
\begin{align*}
\alpha\colon G\times X&\rightarrow X\\(g,x)&\mapsto gx, 
\end{align*}
as well as the projection map:
\begin{align*}
\pi\colon G\times X&\rightarrow X\\(g,x)&\mapsto x.
\end{align*}
We also consider the following projection map:
\begin{align*}
\text{proj}:G\times G\times X&\rightarrow X\\(q_1,g_2,x)&\mapsto x,
\end{align*}
as well as the following multiplication map:
\begin{align*}
\text{multi}:G\times G\times X&\rightarrow X\\(g_1,g_2,x)&\mapsto g_1g_2x.
\end{align*}
\begin{definition}\label{Defn1}
A $G$-\textbf{equivariant sheaf}\index{$G$-equivariant sheaf} of $\mathbb{Q}$-modules over $X$ is sheaf of $\mathbb{Q}$-modules $F$ over $X$ together with an isomorphism $\rho:\pi^*(F)\rightarrow \alpha^*(F)$ such that $\rho$ is compatible with the group structure. This means that the two maps induced by $\rho$ from $\text{proj}^*(F)$ to $\text{multi}^*(F)$ are equal:
\begin{align*}
\xymatrix{\text{proj}^*(F)\ar@< 2pt>[r]^-a
					\ar@<-2pt>[r]_-b &\text{multi}^*(F)},
\end{align*}
where $a$ is given by $\rho\circ\left(\id\times\rho\right)$ and $b$ is given by $\rho\circ\left(m\circ\id\right)$ where $m$ is group multiplication.
\end{definition}
From Definition \ref{Defn1} the map $\rho:\pi^*(F)\rightarrow \alpha^*(F)$ induces maps on stalks $(g,x)$ as follows:
\begin{align*}
\rho_{(g,x)}:\pi^*(F)_{(g,x)}\rightarrow \alpha^*(F)_{(g,x)}
\end{align*} 
From the definition of the functors $\pi^*$ and $\alpha^*$ we can see that:
\begin{align*}
\pi^*(F)_{(g,x)}&=\left\lbrace (g,x,a)\mid a\in F_x \right\rbrace\cong F_x\,\text{and}\\\alpha^*(F)_{(g,x)}&=\left\lbrace (g,x,a)\mid a\in F_{gx} \right\rbrace\cong F_{gx}.
\end{align*}
where the isomorphisms are given by:
\begin{align*}
\iota_g:F_x&\rightarrow \left\lbrace (g,x)\right\rbrace\times F_x\\s_x&\mapsto ((g,x),s_x) 
\end{align*}
and
\begin{align*}
Pr_g:\left\lbrace (g,x)\right\rbrace\times F_{gx}&\rightarrow F_{gx}\\((g,x),t_{gx})&\mapsto t_{gx}.
\end{align*}
Consequently the map of sheaves $\rho$ induces the following maps of $\mathbb{Q}$-modules:
\begin{align*}
\rho_{(g,x)}:F_{x}\rightarrow F_{gx}.
\end{align*}
Furthermore if $g\in \text{stab}_G(x)$ then $\rho_{(g,x)}$ induces an action on the stalk $F_x$ as the following proposition will show.
\begin{proposition}\label{Unital}
Consider $(e,x)\in G\times X$. Then the map induced by $\rho$ on stalks:
\begin{align*}
\rho_{(e,x)}:F_x\rightarrow F_x
\end{align*}
is the identity.
\end{proposition}
\begin{proof}
If $((e,x),s_x)\in \left\lbrace(e,x)\right\rbrace\times F_x$ then we know that: 
\begin{align*}
\rho_{(e,x)}((e,x),s_x)=\rho_{(e*e,x)}((e,x),s_x)=\rho_{(e,x)}(\rho_{(e,x)}((e,x),s_x))
\end{align*}
as a consequence of the compatibility of $\rho$ with the group structure of $G$ seen in Definition \ref{Defn1}. We now apply the general fact that if $f:A\rightarrow A$ is an isomorphism of sets or abelian groups with $f^2(a)=f(a)$ then $f=\id$.
\end{proof} 
This proposition shows that this particular definition of $G$-equivariant sheaf has a unitality property. 
\begin{corollary}\label{Coract}
If $F$ is a $G$-equivariant sheaf over a $G$-space $X$ and $x\in X$ then $F_x$ has a $\stabgx(x)$-action. 
\end{corollary}
\begin{proof}
The associativity of the map is given by the definition of $\rho$ from Definition \ref{Defn1}. The unitality is given by Proposition \ref{Unital}.  
\end{proof}
There is a second definition of $G$-equivariant sheaf which we will show to be equivalent to the one already stated.

\begin{definition}\label{defn2}
A \textbf{$G$-equivariant sheaf}\index{$G$-equivariant sheaf} of $\mathbb{Q}$-modules over $X$ is a sheaf $(E,p)$ where:
\begin{enumerate}
\item $E$ is a topological space with a continuous $G$-action.
\item $p$ is a continuous $G$-equivariant map $p:E\rightarrow X$.
\item $p$ is a local homeomorphism of spaces.
\item Each map $g:E_x\rightarrow E_{gx}$ is a map of $\mathbb{Q}$-modules for every $x\in X,g\in G$. 
\end{enumerate}
\end{definition}
Note that the existence of the map $g$ from point $4$ follows immediately from points $1$ and $2$. The existence of a map $g$ with domain $E_x$, for all $x$, follows from the $G$-action from point $1$. The fact that this map has codomain $E_{gx}$ follows as a result of the $G$-equivariance of the map $p$ from point $2$. We now prove that the two definitions provided are equivalent.
\begin{theorem}
The definitions of $G$-equivariant sheaves of $\mathbb{Q}$-modules of Definition \ref{Defn1} and Definition \ref{defn2} are equivalent.
\end{theorem}
\begin{proof}
Starting off with Definition \ref{Defn1}. If $LF$ denotes the sheaf space of $F$ we define $E$ to be the topological space $LF$ and $p$ to be the projection for $LF$. By definition we have:
\begin{align*}
L\pi^*(F)=\left\lbrace ((g,x),s_x)\mid (g,x)\in G\times X, \, s_x\in F_x\right\rbrace
\end{align*}
and
\begin{align*}
L\alpha^*(F)=\left\lbrace  ((g,x),s_{gx})\mid (g,x)\in G\times X,\, s_{gx}\in F_{gx}\right\rbrace.
\end{align*}
For the $G$-action on $E$ we define the maps $g$ to be the map $Pr_g\circ \rho_{(g,x)}\circ \iota_g$ where $Pr_g$ and $\iota_g$ are seen in the paragraph proceeding Proposition \ref{Unital}. We therefore have a possible $G$-action on $E$:
\begin{align*}
\psi:G\times E&\rightarrow E\\(g,e)&\mapsto g(e).
\end{align*}
This is a $G$-action since we already observed unitality in Proposition \ref{Unital} and associativity follows from the fact that $\rho$ is compatible with the group structure of $G$. This is similar to Corollary \ref{Coract}. 

We know that $p$ is $G$-equivariant and the fact that each $g$ is a map of $\mathbb{Q}$-modules follows directly from the construction of $g$ from the map of sheaves of $\mathbb{Q}$-modules $\rho$. Continuity of $p$ and that $p$ is a local homeomorphism follows from the fact that it is satisfied for $LF$. Finally we need to verify the continuity of the constructed $G$-action.

Consider $s(U)\subseteq LF$ for $U$ an open neighbourhood in $LF$ and $s\in F(U)$. Take $(g,t_y)\in \psi^{-1}(s(U))$ for $t\in F(V^{\prime})$ and some $V^{\prime}\subseteq X$ open. Then $g(t_y)=s_{gy}\in s(U)$ and hence $gy\in U$. We know that the $G$-action map for $X$ denoted $\phi$ is continuous, and since $(g,y)\in \phi^{-1}(U)$ there exists $W\times V\subseteq G\times X$ open with $(g,y)\in W\times V\subseteq \phi^{-1}(U)$. We can assume $V\subseteq V^{\prime}$ since if not $V^{\prime}\cap V$ satisfies $\phi(W\times \left(V\cap V^{\prime}\right))\subseteq U$.

Let $A=\left(W\times V\times s(U)\right)\cap L\alpha^*(F)$ which is open in $L\alpha^*(F)$ and containing $((g,y),s_{gy})$. By continuity of $\rho$ we have that $\rho^{-1}(A)$ is open and contains $((g,y),t_y)$. Therefore there exists $U_1\subseteq G$ an open neighbourhood of $g$ and $U_2\subseteq X$ an open neighbourhood of $y$ with:
\begin{align*}
((g,y),t_y)\in \left(U_1\times U_2\times t(U_2)\right)\cap L\pi^*(F)\subseteq \rho^{-1}(A).
\end{align*}
Consider the open subset $U_1\times t(U_2)$ and take any point $(a,t_z)$. We know that $(a,z,t_z)\in L\pi^*(F)$ so therefore $\rho_{(a,z)}((a,z),t_z)\in A$ and hence by the definition of $A$ we have:
\begin{align*}
a(t_z)=Pr_a\rho_{(a,z)}\iota_a(t_z)=Pr_a\rho_{(a,z)}((a,z),t_z)\in s(U). 
\end{align*}
Therefore:
\begin{align*}
(g,t_y)\in U_1\times t(U_2) \subseteq \psi^{-1}(s(U))
\end{align*} 
is an open neighbourhood.

In the other direction suppose we are starting with Definition \ref{defn2}. Given $(E,p)$ we begin constructing the homeomorphism of sheaf spaces $\rho$ as follows:
\begin{align*}
\rho:L\pi^*(E)&\rightarrow L\alpha^*(E)\\((g,x),s_x)&\mapsto ((g,x),g(s_x)).
\end{align*}
This map is compatible with the group structure as required in Definition \ref{Defn1} as a result of the fact that we are using a $G$-action, this is similar to Corollary \ref{Coract}. For continuity take an arbitrary open set of the form $\left(W\times V\times s(U)\right)\cap L\alpha^*(E)$ which we denote by $A$. If $((g,x),t_x)\in \rho^{-1}(A)$ then $((g,x),g(t_x))\in A$ and hence $gx\in U$ and $((g,x),g(t_x))=((g,x),s_{gx})$. It follows that $(g,t_x)\in \psi^{-1}(s(U))$ where $\psi$ is the action map for $E$. By continuity of $\psi$ there exists $U_1$ open in $G$ and $U_2$ open in $X$ with:
\begin{align*}
(g,t_x)\in U_1\times t(U_2)\subseteq \psi^{-1}(s(U))
\end{align*}
and so:
\begin{align*}
((g,x),t_x)\in \left[U_1\times U_2\times t(U_2)\right]\cap L\pi^*(E)\subseteq \rho^{-1}(A)
\end{align*}
as required. Therefore $\rho$ is a map of sheaves satisfying the compatibility with the group structure. We finally show that this is an isomorphism on each stalk and hence an isomorphism of sheaves. This follows from the fact that each $g$ is an isomorphism with inverse ${g^{-1}}$. 
\end{proof}
Since the definition of a $G$-sheaf of $\mathbb{Q}$-modules over a $G$-space $X$ specifies that each $x\in X$ must satisfy that the stalk over $x$ is a $\text{stab}_G(x)$-module, we have the following definition.
\begin{definition}
A $G$-sheaf of $\mathbb{Q}$-modules $E$ over a $G$-space $X$ is said to be \textbf{stalk-wise fixed}\index{stalk-wise fixed} if the action of $\text{stab}_G(x)$ on $E_x$ is trivial for each $x\in X$. 
\end{definition}
We recall the definition of the space of closed subgroups given in Construction \ref{spaceclosedsubgroup}.
\begin{remark}
In the case where $X=SG$, a $G$-sheaf of $\mathbb{Q}$-modules $(E,p)$ over $X$ satisfies that $p^{-1}(K)$ has a $N_G(K)$-action since $N_G(K)$ fixes $K$ and $p$ is $G$-equivariant. 
\end{remark}
\begin{example}
If we let $X$ denote the trivial one point $G$-space, then a $G$-sheaf of $\mathbb{Q}$-modules is equivalent to a discrete $\mathbb{Q}[G]$-module.
\end{example}
We will see other examples after we have developed enough theory in Examples \ref{transitivesheafex}, \ref{constantsheafex} and \ref{Gskyscraper}.
\begin{definition}\label{Weyldefn}
In the case where $X=SG$, we define a \textbf{$\text{Weyl-}G$-sheaf}\index{$\text{Weyl-}G$-sheaf} of $\mathbb{Q}$-modules over $X$ to be a $G$-sheaf of $\mathbb{Q}$-modules $F$ such that for each $K\in SG$ we have that $F_K$ is $K$-fixed and hence a $W_G(K)$-module.
\end{definition}
We now move on to consider morphisms of $G$-sheaves. We give the definition of morhpisms separately for each of the two equivalent definitions of $G$-sheaves given above and show that they too coincide. We start with a morphism for Definition \ref{Defn1}.
\begin{definition}
Let $E$ and $E^{\prime}$ be $G$-equivariant sheaves of $\mathbb{Q}$-modules with corresponding homeomorphisms $\rho$ and $\rho^{\prime}$. A morphism of $G$-sheaves of $\mathbb{Q}$-modules $f:E\rightarrow E^{\prime}$ is a morphism of sheaves such that the following square commutes:
\begin{align*}
\xymatrix{\pi^*(E)\ar[d]^{\rho}\ar[r]^{\pi^*(f)}&\pi^*(E^{\prime})\ar[d]^{\rho^{\prime}}\\ \alpha^*(F)\ar[r]^{\alpha^*(f)}&\alpha^*(E^{\prime})}
\end{align*}
\end{definition}
The following is a morphism with respect to Definition \ref{defn2}.
\begin{definition}
If $(E,P)$ and $(E^{\prime},p^{\prime})$ are $G$-equivariant sheaves of $\mathbb{Q}$-modules then a morphism of $G$-equivariant sheaves of $\mathbb{Q}$-modules is a map of sheaf spaces
\begin{align*}
f:(E,p)\rightarrow (E^{\prime},p^{\prime})
\end{align*}
which is $G$-equivariant, and which induce a map of $\mathbb{Q}$-modules on stalks.
\end{definition}
\begin{theorem}
The morphisms of $G$-equivariant sheaves with respect to Definition \ref{Defn1} are equivalent to those morphisms with respect to Definition \ref{defn2}.
\end{theorem}
\begin{proof}
Suppose we have a map of sheaves $f:E\rightarrow E^{\prime}$ with respect to Definition \ref{Defn1} such that the following commutes: 
\begin{align*}
\xymatrix{\pi^*(E)\ar[d]^{\rho}\ar[r]^{\pi^*(f)}&\pi^*(E^{\prime})\ar[d]^{\rho^{\prime}}\\ \alpha^*(E)\ar[r]^{\alpha^*(f)}&\alpha^*(E^{\prime})}
\end{align*}
Then for each $s_x\in E_x$ we have:
\begin{align*}
f(g(s_x))&=f(Pr_g(\rho_{(g,x)}((g,x),s_x)))=Pr_g(\alpha^*(f)(\rho_{(g,x)}((g,x),s_x)))\\&=Pr_g(\rho^{\prime}_{(g,x)}\pi^*(f)(((g,x),s_x)))=Pr_g(\rho^{\prime}_{(g,x)}((g,x),f(s_x)))=g(f(s_x)).
\end{align*}
On the other hand take a $G$-equivariant map of $G$-sheaf spaces from $(E,p)$ to $(E^{\prime},p^{\prime})$. If we take any $((g,x),s_x)\in L\pi^*(E)$ we have:
\begin{align*}
\alpha^*(f)(\rho((g,x),s_x))&=\alpha^*(f)(\rho_{(g,x)}((g,x),s_x))=((g,x),f(g(s_x)))\\&=((g,x),g(f(s_x)))=\rho^{\prime}(\pi^*(f)((g,x),s_x))
\end{align*}
as required.
\end{proof}
We can now define the following two categories.
\begin{definition}
We define the category of $G$-sheaves over a topological $G$-space $X$ by setting the objects to be the $G$-equivariant sheaves of $\mathbb{Q}$-modules over $X$ and the morphisms to be the maps of $G$-equivariant sheaves corresponding to either Definition \ref{Defn1} or \ref{defn2}. We let $\text{Sh}_{\mathbb{Q}}(X)$ denote this category.

We define the category of Weyl-$G$-sheaves to be the full subcategory of \\$\text{Sh}_{\mathbb{Q}}(SG)$ consisting of the collection of objects given by the set Weyl-$G$-sheaves over $SG$. We shall denote this category by $\text{Weyl-}G\text{-Sheaf}_{\mathbb{Q}}(SG)$.
\end{definition}
If $A$ and $B$ are $G$-equivariant sheaves, we let $\text{Sh}_{\mathbb{Q}}(X)(A,B)$ denote the collection of morphisms of $G$-equivariant sheaves from $A$ to $B$. We now define the composition function as follows:
\begin{align*}
\text{Sh}_{\mathbb{Q}}(E,F)\times \text{Sh}_{\mathbb{Q}}(F,H)&\rightarrow \text{Sh}_{\mathbb{Q}}(E,H)\\(\alpha,\beta)&\mapsto \beta\circ\alpha
\end{align*} 
where $\circ$ is interpreted as composition in the category of sheaves. Here we are using that a morphism of $G$-sheaves is in particular a morphism of sheaf spaces with extra conditions. We also note that a composition of $G$-equivariant maps is a $G$-equivariant map.

We also have the identity morphism, as in the category of sheaves over $X$ this clearly corresponds to a $G$-equivariant map. The associativity of the composition function holds since it does in the category of sheaves and this is the category we are technically applying the composition in.

This shows that the category of $G$-sheaves of $\mathbb{Q}$-modules over $X$ is a well defined category. Similarly this holds for the category of Weyl-$G$-sheaves over $SG$.
\section{Abelian category of G-sheaves}
In this section we will verify that the categories of $G$-sheaves and Weyl-$G$-sheaves of $\mathbb{Q}$-modules respectively, are abelian categories.
\subsection{Additivity of Morphisms}
We prove that for any two $G$-equivariant sheaves $E$ and $F$ we have that $\hom(E,F)$ is an abelian group. We begin by defining a suitable operation.
\begin{definition}
Let $(E,p)$ and $(F,p^{\prime})$ be two $G$-equivariant sheaves over $X$. Then we have:
\begin{enumerate}
\item A zero morphism $0:E\rightarrow F$ which is the zero morphism in the category of sheaves and at each stalk $x$ we have $0_x$ is the zero map. This is continuous since it is a morphism of sheaves and it is clearly $G$-equivariant.  
\item If $f,g\in \hom(E,F)$ we define $f+g$ to be the morphism where:
\begin{align*}
f+g:E&\rightarrow F\\s_x&\mapsto f(s_x)+g(s_x)
\end{align*}
where addition is taken in $F_x$.
\item If $f\in \hom(E,F)$ we define $-f$ to be the morphism where:
\begin{align*}
-f:E&\rightarrow F\\s_x&\mapsto -f(s_x)
\end{align*} 
where the negative is taken in $F_x$.
\end{enumerate}
\end{definition}
To prove that both $f+g$ and $-f$ belong to $\hom(E,F)$ we need the following proposition.
\begin{proposition}
For $(E,p)$ a $G$-equivariant sheaf over $X$ we define $E\pi E$ to be the set:
\begin{align*}
\left\lbrace (e,e^{\prime})\in E\times E\mid p(e)=p(e^{\prime})\right\rbrace
\end{align*}
Then the following map is continuous:
\begin{align*}
m:E\pi E&\rightarrow E\\(e,e^{\prime})&\mapsto e-e^{\prime}
\end{align*}
where subtraction is taken in $p^{-1}(p(e))$.
\end{proposition}
\begin{proof}
See \cite[Proposition 2.5.2]{Tennison}.
\end{proof}
\begin{proposition}
If $f,g\in \hom(E,F)$ then $f+g$ and $-f$ are elements of $\hom(E,F)$.
\end{proposition}
\begin{proof}
For continuity we show that the map $f-g$ is continuous. Observe that the map $f-g$ is given by the composition $m\circ \left(f\times g\right)\circ\iota$ where:
\begin{align*}
\iota:E&\rightarrow E\pi E\\e&\mapsto (e,e)
\end{align*}
and $m$ is given as in the previous proposition with $F$ taking the place of $E$. It is a composition of continuous maps. To see that $f+g$ is a map of sheaf spaces we observe that:
\begin{align*}
p^{\prime}((f+g)(s_x))=p^{\prime}(f(s_x)+g(s_x))=x=p(s_x)
\end{align*} 
since addition is taken in $F_x$. Similarly $-f$ is a map of sheaf spaces.

If $a\in G$ recall from Definition \ref{defn2} that $a$ is a map of $\mathbb{Q}$-modules on the stalks of the $G$-sheaves of $\mathbb{Q}$-modules. We therefore have that:
\begin{align*}
a((f+g)(s_x))&=a(f(s_x)+g(s_x))=a(f(s_x))+a(g(s_x))\\&=f(a(s_x))+g(a(s_x))=(f+g)(a(s_x))
\end{align*}
Similarly $-f$ is a morphism of $G$-equivariant sheaves.
\end{proof}
Therefore we have verified that if $(E,p)$ and $(F,p^{\prime})$ are $G$-equivariant sheaves of $\mathbb{Q}$-modules then $\hom(E,F)$ is an abelian group.
\subsection{Finite Limits and Colimits}
We construct finite limits of $G$-equivariant sheaves of $\mathbb{Q}$-modules. To do this we first make a useful remark.
\begin{remark}
In the category of $\mathbb{Q}$-modules it is known that finite limits and filtered colimits commute.  In particular, if $I$ is a finite indexing category and $F^i$ for $i\in I$ is a diagram of sheaves of $\mathbb{Q}$-modules over a space $X$ then:
\begin{align*}
\underset{I}{\lim}\left(F^i_x\right)=\left(\underset{I}{\lim} F^i\right)_x.
\end{align*}
\end{remark}
We can construct finite limits of $G$-equivariant sheaves of $\mathbb{Q}$-modules by considering sheaf spaces.
\begin{construction}\label{limconstruct}
Let $I$ be a finite indexing category and $F^i$ for $i\in I$ be a diagram of $G$-equivariant sheaves of $\mathbb{Q}$-modules. We omit notation for maps in our diagram. Applying the forgetful functor gives a diagram of sheaves. We begin by setting $F_x=\underset{I}{\lim} F^i_x$ for each $x\in X$ and setting $LF=\underset{x\in X}{\coprod}F_x$.

We topologise $LF$ by defining the generating sets for the topology to be those of the form:
\begin{align*}
(s^i)_{i\in I}(U)=\left\lbrace (s^i_x)_{i\in I}\mid x \in U\right\rbrace
\end{align*} 
where $U$ is a basic open subset for $X$ and $(s^i)_{i\in I}\in \underset{I}{\lim}F^i(U)$, and the limit is taken in the category of presheaves.

We define the action by the following map:
\begin{align*}
\psi:G\times LF&\rightarrow LF\\(g,(s^i_x)_{i\in I})&\mapsto (g(s^i_x))_{i\in I}.
\end{align*}
Here we use the $G$-action on each $F^i$.
\end{construction}
\begin{proposition}
The object $F$ explicitly stated in Construction \ref{limconstruct} is a $G$-equivariant sheaf of $\mathbb{Q}$-modules over $X$.
\end{proposition}
\begin{proof}
This is a sheaf of $\mathbb{Q}$-modules since applying the forgetful functor and taking the limit in the category of presheaves is exactly how limits of sheaves are calculated so we need only verify the $G$-action.

Let $g,h\in G$ and $(s^i_x)_{i\in I}\in F_x$. The proposed action map is a unital $G$-action map since:
\begin{align*}
e(s^i_x)_{i\in I}=(e(s^i_x))_{i\in I}=(s^i_x)_{i\in I}
\end{align*}
is true for each $F^i$. The action is associative as:
\begin{align*}
g(h(s^i_x)_{i\in I})=g(h(s^i_x))_{i\in I}=(g(h(s^i_x)))_{i\in I}=((gh)(s^i_x))_{i\in I}=(gh)((s^i_x)_{i\in I}),
\end{align*}
again using that this holds for each $F^i$. This shows that we have a $G$-action map. Similarly $g$ represents a $\mathbb{Q}$-module map since for each $q\in \mathbb{Q}$ we have:
\begin{align*}
g(q(s^i_x)_{i\in I})=g(qs^i_x)_{i\in I}=(g(qs^i_x))_{i\in I}=(q(gs^i_x))_{i\in I}=q(g(s^i_x)_{i\in I}).
\end{align*}
The fact that the induced local homeomorphism $p:LF\rightarrow X$ is $G$-equivariant is immediate from construction. Finally we verify continuity of the $G$-action. Take a set of the form $(s^i)_{i\in I}(U)$ as defined in Construction \ref{limconstruct} and $(g,(t^i_x)_{i\in I})\in \psi^{-1}((s^i)_{i\in I}(U))$. Then $(g,t^i_x)\in \psi^{-1}_i(s^i(U))$ for each $i\in I$ where $\psi_i$ is the action map on the sheaf space for $F^i$.

We know that each $\psi_i$ is continuous so there exists $N_i\unlhd G$ open and $U_i\subseteq X$ open so that:
\begin{align*}
(g,t^i_x)\in gN_i\times t^i(U_i)\subseteq \psi^{-1}_i(s^i(U)).
\end{align*}
We can set $N=\underset{I}{\cap}N_i$ and $V=\underset{I}{\cap}U_i$ so that:
\begin{align*}
(g,(t^i_x)_{i\in I})\in gN\times(t^i)_{i\in I}(V)\subseteq \psi^{-1}((s^i)_{i\in I}(U))
\end{align*}
as required.
\end{proof}
Notice that the continuity of the $G$-action here relies on us dealing with a finite diagram. Given a diagram of $G$-sheaves we have constructed a $G$-sheaf which is potentially the limit, to prove that it is, we need to prove the universality condition.
\begin{proposition}
If $F^i$ is a finite diagram of $G$-sheaves of $\mathbb{Q}$-modules over $X$ then the construction $F$ is universal in the category of $G$-sheaves of $\mathbb{Q}$-modules over $X$.
\end{proposition}
\begin{proof}
If we begin with morphisms of $G$-equivariant sheaves of $\mathbb{Q}$-modules \\$\phi_i:A\rightarrow F^i$, then we can use the fact that every morphism of $G$-equivariant sheaves is a morphism of sheaves to find a unique morphism of sheaves $\phi$ such that the following diagram commutes:
\begin{align*}
\xymatrix{A\ar[rr]^{\phi_i}\ar@{-->}[dr]_{\phi}&&F^i\\&F\ar[ur]^{p_i}}
\end{align*}
It is now sufficient to prove that this morphism $\phi$ of sheaf spaces is $G$-equivariant. To see this observe that if $g\in G$ and $a\in A$, then we have the following:
\begin{align*}
g\phi(a)&=g(p_i\circ\phi(a))_{i\in I}=(g(\phi_i(a)))_{i\in I}\\&=(\phi_i(ga))_{i\in I}=(p_i\circ\phi(ga))_{i\in I}=\phi(ga)
\end{align*}
proving $G$-equivariance of $\phi$ as required.
\end{proof}
We now look at constructing colimits of $G$-sheaves of $\mathbb{Q}$-modules. To do this we first understand the following construction of equivalence relations. 
\begin{construction}\label{colimequiv}
A binary relation $R$ on a set $X$ is a subset of $X\times X$ where we say that for $x,y\in X$ we have $xRy$ if $(x,y)\in R$. Suppose that $R$ is a binary relation containing the set:
\begin{align*}
\left\lbrace(x,x)\in X\times X\mid x\in X\right\rbrace.
\end{align*}
We make explicit the smallest equivalence relation generated by $R$. We first define $R_1$ to be the following set:
\begin{align*}
\left\lbrace (x,y)\in X\times X\mid (x,y)\in R\,\,\text{or}\, (y,x)\in R\right\rbrace
\end{align*}
Then $R_1$ represents a relation which satisfies symmetry. We take $R_2$ to be the set:
\begin{align*}
\left\lbrace(x,y)\in X\times X\mid \exists\, x_0,x_1,x_2,\ldots,x_n\in X\text{with}\,(x_i,x_{i+1})\in R_1,\,\text{for $0\leq i\leq n$}\right\rbrace
\end{align*}
where $x_0=x$ and $x_{n}=y$. The set $R_2$ represents the desired equivalence relation.
\end{construction}
\begin{remark}
If $R$ is a binary relation satisfying that each $x$ in $X$ has an element $y$ such that $xRy$, then the above construction creates an equivalence relation. This is because this condition paired with symmetry and transitivity implies reflexivity.
\end{remark}
Suppose we have a diagram $F^i$ of $G$-sheaves of $\mathbb{Q}$-modules with maps omitted from the notation. As in the construction of limits we apply the forgetful functor to obtain a diagram of sheaves of $\mathbb{Q}$-modules $F^i$. For the following construction keep in mind that $\Gamma$ represents the sheafification functor and $P\underset{I}{\colim}$ represents the colimit in the category of presheaves.
\begin{construction}\label{sheafcolim}
If $F^i$ is a diagram of sheaves of $\mathbb{Q}$-modules, we define $\underset{I}{\colim} F^i=\Gamma P\underset{I}{\colim} F^i$. This means that we consider the sheaf space whose underlying set is given by $\underset{x\in X}{\coprod} \left(\underset{I}{\colim}F^i\right)_x$ where:
\begin{align*}
\left(\underset{I}{\colim} F^i\right)_x&=\left(P\underset{I}{\colim}F^i\right)_x=\underset{I}{\colim}\left({F^i}_x\right)=\underset{I}{\oplus} {F^i}_x/\sim
\end{align*}
where $\sim$ is the smallest equivalence relation generated by the binary relation $x_i\sim \varphi_{ij}(x_i)$, for all $i,j\in I$ and all $x_i\in F^i$. Notice that $\varphi_{ii}=\id_{F^i}$, which implies that this relation is reflexive, allowing us to build an equivalence relation using Construction \ref{colimequiv}.

The set $\underset{x\in X}{\coprod}\left(\underset{I}{\oplus}{F^i}_x/\sim\right)$ is topologised by considering the topology generated by sets of the form:
\begin{align*}
s(U)=\left\lbrace s_x\mid \in U\right\rbrace
\end{align*}
where $s\in P\underset{I}{\colim}F^i(U)$ and $U$ is a basic open subset of $X$.
\end{construction}
Given a diagram of $G$-sheaves of $\mathbb{Q}$-modules over $X$ we know how to construct the colimit sheaf space in the category of sheaves of $\mathbb{Q}$-modules over $X$. We first show that we have a well-defined action of $G$ on the sheaf space, and then show that it is continuous.
\begin{lemma}\label{colimact}
If $F^i$ is a diagram of $G$-sheaves of $\mathbb{Q}$-modules over $X$ then the colimit sheaf space in the category of sheaves of the diagram $F^i$ as given in Construction \ref{sheafcolim} has a $G$-action. 
\end{lemma}
\begin{proof}
First notice that for each $x\in X$ and $i\in I$ we can use the $G$-action on $F^i$ to get a map of $\mathbb{Q}$-modules ${F^i}_x$ to ${F^i}_{gx}$. Using the universal properties of colimits we further obtain a map from $\underset{I}{\oplus}{F^i}_x$ to $\underset{I}{\oplus}{F^i}_{gx}$. This map is compatible with the colimit equivalence relation and we therefore obtain the following map:
\begin{align*}
\psi:G\times \left(\underset{x\in X}{\coprod}\left(\underset{I}{\oplus}{F^i}_x/\sim\right)\right)&\rightarrow \left(\underset{x\in X}{\coprod}\left(\underset{I}{\oplus}{F^i}_x/\sim\right)\right)\\(g,[s_x])&\mapsto [g(s_x)]
\end{align*}
This is a $G$-action since it is constructed using the action of $G$ on each $F^i$.
\end{proof} 
\begin{lemma}\label{colimactcts}
If $F^i$ is a diagram of $G$-sheaves of $\mathbb{Q}$-modules then the $G$-action defined in Lemma \ref{colimact} for the colimit sheaf space displayed in Construction \ref{sheafcolim} is continuous. 
\end{lemma}
\begin{proof}
To see that the $G$-action map is continuous take an open subset of the form:
\begin{align*}
[s](U)=\left\lbrace [s]_x\mid x\in U\right\rbrace
\end{align*} 
for $[s]\in P\underset{I}{\colim} F^i(U)$ and $U$ a basic open subset of $X$. We will prove that $\psi^{-1}([s](U))$ is open so take any element $(g,[t]_{g^{-1}x})$ belonging to it, so that $x\in U$. Therefore we have that $[g(t)]_x=[s]_x$.

It therefore follows that there exists some $V\subseteq X$ open neighbourhood of $x$ with $[s]=[g(t)]$ in $\Gamma P\underset{I}{\colim} F^i(V)$. Since $\left(\Gamma P\underset{I}{\colim} F^i\right)_x=\left(P\underset{I}{\colim}F^i\right)_x$ and since the sheaf space topology is generated by sets of the form $r(W)$ where $r\in \left(P\underset{I}{\colim} F^i\right)(W)$ and $W\subseteq X$ open and basic, we can assume that $[s]=[g(t)]$ in $\left(P\underset{I}{\colim}F^i\right)(V)$.

Therefore $s$ has a representative $\underset{k_i}{\sum} s_{k_i}$ in $\underset{k_i}{\oplus}F_{k_i}(V)$ and $g(t)$ has one $\underset{j_i}{\sum} g(t_{j_i})$ in $\underset{j_i}{\oplus}F_{j_i}(V)$ with $s\sim g(t)$. It follows that $\underset{k_i}{\sum} g^{-1}(s_{k_i})\sim \underset{j_i}{\sum} t_{j_i}$ for $g^{-1}(s_{k_i})\in F^{k_i}(g^{-1}V)$ and $t_{j_i}\in F^{j_i}(g^{-1}V)$. Suppose that $\psi_{k_i}$ is the continuous $G$-action map for the $G$-sheaf $F^{k_i}$.

If $(g,g^{-1}(s_{k_i})_{g^{-1}x})\in \psi_{k_i}^{-1}(s_{k_i}(V))$ (which is open by the continuity of $\psi_{k_i}$) then there exists $W_{k_i}\subseteq G$ open and $U_{k_i}\subseteq X$ open with:
\begin{align*}
(g,g^{-1}(s_{k_i})_{g^{-1}x})\in W_{k_i}\times g^{-1}(s_{k_i})(U_{k_i})\subseteq \psi_{k_i}^{-1}(s_{k_i}(V)).
\end{align*}
Notice that we can assume that $U_{k_i}\subseteq g^{-1}V$ since if it does not we can take $U_{k_i}\cap g^{-1}V$. Set $W=\underset{k_i}{\cap}W_{k_i}$ and $U_1=\underset{k_i}{\cap}U_{k_i}$. Therefore it follows:
\begin{align*}
(g,g^{-1}(s)_{g^{-1}x})\in W\times \underset{k_i}{\sum} g^{-1}(s_{k_i})(U_1)\subseteq \underset{k_i}{\bigoplus}\psi_{k_i}^{-1}\left(\left(\underset{k_i}{\sum}s_{k_i}\right)(V)\right).
\end{align*}
The analogous statement is equally true for the sum represented by $t$. This is because the morphisms of $G$-sheaves are compatible with the $G$-action maps $\psi_k$. Since $g^{-1}(s)\sim t$ as sections over $g^{-1}V$ and hence over $U_1$, we have that $g^{-1}(s)_z\sim t_z$ for every $z\in U_1$. Consequently we have that $\left[g^{-1}(s)\right](U_1)=[t](U_1)$.

By construction of $\psi$ this means that:
\begin{align*}
(g,g^{-1}(s)_{g^{-1}x})\in W\times [g^{-1}(s)](U_1)\subseteq \psi^{-1}([s](V)).
\end{align*}

We have that $W\times [t](U_1)$ is open and containing $(g,[t]_{g^{-1}x})$. This set has image contained in $[s](U)$ since:
\begin{align*}
\psi(W\times [t](U_1))&=\psi(W\times [g^{-1}(s)](U_1))\subseteq [s](V)\subseteq [s](U)
\end{align*}  
This proves that the $G$-action is continuous.
\end{proof}
\begin{lemma}\label{QmodGequi}
Using the action defined in Lemma \ref{colimact} we have that each $g$ is a map of $\mathbb{Q}$-modules. The projection map for the colimit sheaf explicitly stated in Construction \ref{sheafcolim} is $G$-equivariant.
\end{lemma}
\begin{proof}
We know that the action maps $g_i$ from ${F^i}_x$ to ${F^i}_{gx}$ are maps of $\mathbb{Q}$-modules for each $i$. It follows that the map $\underset{I}{\oplus}{g_i}$ is a map of $\mathbb{Q}$-modules. The colimit quotient is taken in the category of $\mathbb{Q}$-modules and this relation is compatible with the $G$-action, so the map $g$ induced by the quotient is a map of $\mathbb{Q}$-modules.

The $G$-equivariance of the projection map for the colimit construction follows from the fact that the $G$-action is constructed from the $G$-action on each $F^i$ and these actions commute with the projections for $F^i$.  
\end{proof}
We have now constructed a $G$-sheaf of $\mathbb{Q}$-modules $F$ given a diagram of $G$-sheaves of $\mathbb{Q}$-modules $F^i$. We now need to verify that this construction is universal.
\begin{proposition}
If $F^i$ is a diagram of $G$-sheaves of $\mathbb{Q}$-modules then the colimit sheaf as seen in Construction \ref{sheafcolim} with $G$-action given by Lemma \ref{colimact} which we denote by $F$, is a $G$-sheaf of $\mathbb{Q}$-modules satisfying the universality condition.   
\end{proposition}
\begin{proof}
The fact that $F$ is a $G$-sheaf of $\mathbb{Q}$-modules follows from Lemmas \ref{colimactcts}, \ref{QmodGequi} and the fact that $F$ is a sheaf over $X$. We now prove universality. Suppose that we have morphisms $\alpha_i:F^i\rightarrow A$ and the colimit maps $p_i:F^i\rightarrow F$. As in the limit case we use that every morphism of $G$-sheaves is in particular a morphism of sheaves so there exists a unique morphism of sheaves making the following triangle commute:
\begin{align*}
\xymatrix{F^i\ar[dr]^{p_i}\ar[rr]^{\alpha_i}&&A\\&F\ar@{-->}[ur]_{\exists !\alpha}}
\end{align*}
We need only prove that $\alpha$ is a morphism of $G$-sheaves. We view these $G$-sheaves as $G$-sheaf spaces, and we will prove that $\alpha$ is a $G$-equivariant map of sheaf spaces. Take $[e]_x\in F_x\subseteq F$, which has representative $\underset{1\leq j\leq n}{\sum}e^{i_j}$ in $\underset{I}{\oplus}{F^i_x}$. Then we have:
\begin{align*}
\alpha(g[e])&=\alpha([ge])=\alpha(\underset{1\leq j\leq n}{\sum}p_{i_j}(ge^{i_j}))=\underset{1\leq j\leq n}{\sum}\alpha (p_{i_j}(ge^{i_j}))\\&=g(\underset{1\leq j\leq n}{\sum}\alpha_{i_j}(e^{i_j}))=g(\underset{1\leq j\leq n}{\sum}\alpha(p_{i_j}(e^{i_j})))\\&=g(\alpha(\underset{1\leq j\leq n}{\sum}p_{i_j}(e^{i_j})))=g\alpha([e])
\end{align*}
as required.
\end{proof}
We finish the construction of limits and colimits by mentioning what happens in the special case where we have a diagram of Weyl-$G$-sheaves. 
\begin{proposition}
Let $X=SG$ for $G$ a profinite group and $F^i$ be a diagram of Weyl-$G$-sheaves. Then the colimit of this diagram in the category of $G$-sheaves is a Weyl-$G$-sheaf. Similarly if $F^i$ is a finite diagram of Weyl-$G$-sheaves then the limit is a Weyl-$G$-sheaf.
\end{proposition}
\begin{proof}
Let $H\in SG$ and $h\in H$. We begin by considering finite limits. If $(s^i_H)_{i \in I}\in \underset{I}{\lim}F^i_H,$ then we have that:
\begin{align*}
h(s^i_H)_{i \in I}=(h(s^i_H))_{i \in I}=(s^i_H)_{i \in I}
\end{align*} 
where we use that each $F^i$ is a Weyl-$G$-sheaf. Similarly if $I$ is any small indexing category, take any $[s_H]\in \underset{I}{\colim}F^i_H$. By definition of colimits in the category of $\mathbb{Q}$-modules, $[s_H]$ has a representative $s_H\in \underset{I}{\bigoplus} F^i_H$, and so $s_H=\underset{1\leq j\leq n}{\sum}s^{i_j}_H$ for $s^{i_j}_H\in F^{i_j}_H$. We therefore have that:
\begin{align*}
h[s_H]=\left[h\left(\sum_{1\leq j\leq n}s^{i_j}_H\right)\right]&=\left[\sum_{1\leq j\leq n}hs^{i_j}_H\right]\\&=\left[\sum_{1\leq j\leq n}s^{i_j}_H\right]=[s_H]
\end{align*}
as required.
\end{proof}
\subsection{Abelian Category}
The next question we ask about the category of $G$-sheaves is whether or not this category is an abelian category. We first verify that finite products and finite coproducts coincide.
\begin{proposition}\label{finiprodcoprod}
Let $I$ be a finite set and $F^i$ a collection of $G$-equivariant sheaves. Then $\underset{I}{\prod}F^i$ is equal to $\underset{I}{\coprod}F^i$.
\end{proposition}
\begin{proof}
From the construction of colimits and finite limits above we know that the underlying sheaf is constructed in the category of sheaves. We know from \cite[Proposition 3.5.3]{Tennison} that finite products and finite coproducts of sheaves coincide hence the only thing left to show is that the $G$-actions coincide.

From Construction \ref{sheafcolim} we know that if $g\in G$ then the map $g$ equals $\underset{I}{\bigoplus}g_i$ in the category of $\mathbb{Q}$-modules, but this equals $\underset{I}{\prod}g_i$ since the category of $\mathbb{Q}$-modules is abelian. 
\end{proof} 
\begin{proposition}
The category of $G$-equivariant sheaves of $\mathbb{Q}$-modules over a space $X$ is an additive category. 
\end{proposition}
\begin{proof}
We have already seen that in the category of $G$-equivariant sheaves over $X$ that each of the hom-sets are an abelian group. By Proposition \ref{finiprodcoprod} we know that finite products and coproducts agree. It is left to verify that the following composition map is bilinear:
\begin{align*}
\hom(E,F)\times \hom(F,H)&\rightarrow \hom(E,H)\\(\alpha,\beta)&\mapsto \beta\circ\alpha 
\end{align*} 
This is true since it holds in the category of sheaves by \cite[3.2.2]{Tennison}.
\end{proof}
This argument works equally for Weyl-$G$-sheaves hence both $G$-sheaves over $X$ and Weyl-$G$-sheaves over $X$ are additive categories.
\begin{proposition}\label{Gsheafabelian}
The categories of $G$-sheaves and Weyl-$G$-sheaves over $X$ are abelian categories.
\end{proposition}
\begin{proof}
These categories are both additive and we know from the previous subsection that finite limits and small colimits exist so in particular all kernels and cokernels exist also. For normality we give a proof for $G$-sheaves over $X$ and the same argument will work for Weyl-$G$-sheaves.

If $f:E\rightarrow F$ is a monomorphism of $G$-sheaves over $X$, then by \cite[Theorem 4.13]{Tennison} $E=\ker\left\lbrace F\rightarrow \coker(f)\right\rbrace$ and $f$ is the inclusion of the kernel. Since $f$ is a morphism of $G$-sheaves so is the map $F\rightarrow \coker(f)$, and therefore $f$ is the inclusion of the kernel of the $G$-sheaf morphism  $F\rightarrow \coker(f)$ as required.

If $p:E\rightarrow F$ is an epimorphism of $G$-sheaves we can again apply \cite[Theorem 4.13]{Tennison} to deduce that $F=\coker\left\lbrace\ker p\rightarrow E\right\rbrace$ and $p$ is the projection onto the cokernel. Since $p$ is a morphism of $G$-sheaves so is the map $\ker p\rightarrow E$, and therefore $p$ is the projection of the cokernel of the $G$-sheaf morphism \\$\ker p\rightarrow E$.
\end{proof}
\section{Useful properties of G-sheaves over profinite spaces}
\sectionmark{Useful properties}
In this section we will see how the assumption that $G$ and $X$ are profinite gives us a useful alternative characterisation of $G$-sheaves and Weyl-$G$-sheaves, Proposition \ref{Weylequi}. Before we examine some useful properties of $G$-sheaves we first take a moment to examine a couple of properties of the $G$-space $SG$.
\begin{lemma}\label{subgroup_open}
If $G$ is a profinite group with open subgroup $H$, then $SH$ is an open-closed subspace of $SG$.
\end{lemma}

\begin{proof}
If $H\leq G$ is an open subgroup, then the core of $H$ denoted $H^C$ is also an open subgroup. We then observe that:
\begin{align*}
SH=O(H^C,H^C)\bigcup \left[\underset{K\in SH}{\bigcup}O(H^C,H^CK)\right].
\end{align*}
To see that this is true, recall that sets of the form $O(N,NK)$ are defined in the paragraph proceeding Remark \ref{ONJ}. We first note that each set in the union is contained in $SH$ as each $K$ and $H^C$ is. Therefore we have one inclusion. For the other direction we observe that each $K$ in $SH$ belongs to $O(H^C,H^CK)$ in the union.

Since in a compact topological group we know that an open subgroup is a closed subgroup of finite index, we must have that $H$ is also closed. In particular $H$ is a profinite group and so $SH$ is a profinite space. We note that the profinite topology on $SH$ coincides with the subspace topology with respect to $SG$. Since $SH$ is a compact subspace of a Hausdorff space $SG$, it is therefore closed. 
\end{proof}
\begin{lemma}\label{restrictequi}
If $U$ is an open compact subset of a $G$-space $X$, then there exists a normal subgroup $N$ of $G$ such that $NU= U$. 
\end{lemma}
\begin{proof}
We begin by finding an open neighbourhood $V$ of $e$ such that $VU\subseteq U$ and then using the fact that the set of all open normal subgroups form a neighbourhood basis for $e$ we can find the required $N$. Let $\psi$ be the continuous action map $\psi:G\times X\rightarrow X$. Then $\psi^{-1}(U)$ is open and by definition of the product topology, for each $x \in U$ we know that $(e,x)\in \psi^{-1}(U)$ implies that there are open neighbourhoods of $e$ and $x$ given by $V_x$ and $U_x$ respectively such that $V_x\times U_x\subseteq \psi^{-1}(U)$. This implies that $V_xU_x\subseteq U$. Using the definition of the basis we can find open normal subgroups $N_x\subseteq V_x$ such that $N_xU_x\subseteq U$.

Since the set $\left\lbrace U_x\mid x\in U_x\right\rbrace$ is an open covering of $U$ we can use the fact that $U$ is compact to obtain a finite subcover $\left\lbrace U_{x_i}\mid 1\leq i\leq n\right\rbrace$. Set $N=\underset{1\leq i\leq n}{\cap}N_{x_i}$, then we can show that $N$ satisfies the required property. Let $n\in N$ and $u\in U$ be any pair of elements. Then $u$ belongs to some $U_{x_i}$, and since $n\in N\subseteq N_{x_i}$ we have that $nu\in U$. 
\end{proof}
\begin{remark}
In particular if $X$ is a $G$-space and $U\subseteq X$ is compact and open, then the subgroup of $G$ defined by:
\begin{align*}
\text{stab}_G(U)=\left\lbrace g\in G\mid gU= U\right\rbrace
\end{align*}
is open since the proceeding lemma guarantees an open normal subgroup is contained in $\text{stab}_G(U)$.
\end{remark}
\begin{example}\label{SGbasisstab}
For the $G$-space $SG$, the basic subsets of the form $O(N,NK)$ where $N$ is open and normal in $G$ and $K$ closed in $G$ satisfy:
\begin{align*}
\text{stab}_G\left(O(N,NK)\right)=N_G(NK).
\end{align*}
\end{example}
\begin{proposition}\label{restProp}
If $F$ is a $G$-sheaf of $\mathbb{Q}$-modules over a profinite $G$-space $X$, then any section $s\in F(U)$ for any $U$ restricts to an $N$-equivariant section over an $N$-invariant domain, for some open normal subgroup $N$ of $G$.
\end{proposition}
\begin{proof}
Given any section $s:U\rightarrow LF$, we can find a compact open subset of the form $V$ contained in $U$ since $X$ has a closed-open basis. Notice that $V$ is a compact subset since it is a closed subset of a Hausdorff space. We first note that $V$ is $N$-invariant for some open normal subgroup $N$ of $G$ by Lemma \ref{restrictequi}.

Another application of Lemma \ref{restrictequi} shows that there is an open normal subgroup $N^{\prime}$ such that $s(V)$ is $N^{\prime}$-invariant. This is because $s(V)$ is open by definition of the topology and compact since it is the image of a compact subset via a continuous function. If we set $\overline{N}=N\cap N^{\prime}$, then this is an open normal subgroup of $G$ such that both $V$ and $s(V)$ are $\overline{N}$-invariant. We now show that $s$ is $\overline{N}$-equivariant, namely that $s(ny)=ns(y)$ for all $y\in V$ and $n\in\overline{N}$.

Since $p$ is the local homeomorphism for the sheaf $F$ we know that $f=p|_{s(V)}$ is the homeomorphism inverse to $s_{|_{V}}$. For $n\in \overline{N}$ and $y\in V$ we have that $f\circ s(ny)=ny$, and we can observe the following:
\begin{align*}
f(ns(y))=nf(s(y))=ny
\end{align*}
where we use the $\overline{N}$-invariance fact stated above and the fact that $f$ is $\overline{N}$-equivariant since $p$ is $G$-equivariant. Using injectivity of $f$ we have that $ns(y)=s(ny)$ as required.
\end{proof}
We consequently have the following useful application.
\begin{corollary}\label{equivariant}
If $F$ is a $G$-sheaf of $\mathbb{Q}$-modules over $SG$, then any section $s\in F(U)$ restricts to an $N$-equivariant section over an $N$-invariant domain, for some open normal subgroup $N$ of $G$.
\end{corollary}

We now seek to establish that an even stronger property holds for a Weyl-$G$-sheaf of $\mathbb{Q}$-modules. We begin with the following proposition about profinite groups.
\begin{proposition}\label{finquot}
If $G$ is an infinite profinite group, $K$ is an infinite closed subgroup of $G$ and $M\unlhd G$ is an open normal subgroup, then $K\cap M$ is non-trivial.
\end{proposition}
\begin{proof}
We know that if $K$ is closed in $G$ then it is a profinite group as a subspace of $G$ and in particular $K$ is also not discrete. We know that $K\cap M$ is open and normal in $K$. If $K\cap M$ were trivial then the trivial subgroup would be open, contradicting that $K$ is not discrete.
\end{proof}
Let $E$ be a $G$-equivariant sheaf of $\mathbb{Q}$-modules over a $G$-space $X$ and $U$ be an open subset of $X$. Then if $s$ is a section over $U$, for each $g\in G$ we can define a section $g*s$ over $gU$ where $(g*s)(x)$ is defined to be $gs(g^{-1}x)$ for every $x\in gU$. 
\begin{remark}\label{commuteaction}
It follows that if $s$ is a section over $U$ for the $G$-equivariant sheaf $E$ over $X$, $g\in G$ and $x\in X$ we have:
\begin{align*}
g(s_x)=(g*s)_{gx}.
\end{align*}
Consequently the $G$-action on the sheaf space can be written down in terms of the action on sections.
\end{remark}
The following proposition relates the action on a section over a $G$-sheaf space by $G$ and the equivariance of sections.
\begin{proposition}\label{equifix}
If $E$ is a $G$-sheaf space over $X$, then a section $s:X\rightarrow E$ is $G$-equivariant if and only if $(g*s)=s$ for each $g\in G$. 
\end{proposition}
\begin{proof}
We know that $(g*s)=s$ for each $g\in G$ if and only if $(g*s)(x)=s(x)$ for every $x\in X$ and $g \in G$. But this happens if and only if $gs(g^{-1}x)=s(x)$ for each $g\in G$ and $x \in X$. This in turn happens if and only if $s(g^{-1}x)=g^{-1}s(x)$ for every $g\in G$ and $x\in X$. This is equivalent to saying that $s$ is $G$-equivariant.  
\end{proof}
We now examine more generally how closed subgroups act on sections whose domain is invariant to the action of the subgroup. In particular, if $E$ is a $G$-sheaf of $\mathbb{Q}$-modules over a profinite $G$-space $X$ and $U$ a compact basic open subset of $X$ then $\Gamma\left(U,E\right)$ has a $\text{stab}_G(U)$-action using the action defined in Proposition \ref{equifix}.
\begin{definition}\label{discreteGmod}
If $M$ is a $\mathbb{Q}$-module with an action of a profinite group $G$, then $M$ is a \textbf{discrete $G$-module}\index{discrete $G$-module} if the action map is continuous with respect to the discrete topology on $M$.
\end{definition}
\begin{remark}\label{discGmod}
Equivalently this definition says that $M$ is a discrete $G$-module if and only if $\text{stab}_G(m)$ is open in $G$ for each $m\in M$. Consequently, a discrete $G$-module $M$ satisfies that every element $m$ has finite orbit. Another equivalent definition is that $M$ is a $G$-module if and only if we have the following equality:
\begin{align*}
M=\underset{H\,\text{open}}{\colim}M^H.
\end{align*}  
\end{remark}
We aim to characterise Weyl-$G$-sheaves over $SG$ by local equivariance properties of the sections since this is used in future chapters, namely Lemma \ref{stalk}. The following proposition takes us in that direction. 
\begin{proposition}\label{sectionquot}
If $E$ is a $G$-sheaf of $\mathbb{Q}$-modules over a profinite $G$-space $X$ and $U$ a compact basic open subset of $X$ then $\Gamma\left(U,E\right)$ is a discrete $\stabgx(U)$-module.
\end{proposition}
\begin{proof}
Let $s\in \Gamma\left(U,E\right)$, then by Proposition \ref{restProp} there exists an open normal subgroup $N$ such that $U$ is $N$-invariant and $s$ is $N$-equivariant. Therefore by Proposition \ref{equifix} the action of $N$ on $s$ is trivial and $N$ has finite index in $\text{stab}_G(U)$. We let $J$ denote $\text{stab}_{\text{stab}_G(U)}(s)$. We know that $N\leq J$ and so it follows that $J$ has finite index in $\text{stab}_G(U)$ since we know that $N$ is of finite index. 
\end{proof} 
We now have the following application to the special case where $X=SG$, bearing in mind Example \ref{SGbasisstab}.
\begin{proposition}\label{equisectweyl}
If $E$ is a $G$-sheaf of $\mathbb{Q}$-modules over a profinite $G$-space $SG$ and $O(N,NK)$ a compact basic open subset of $SG$, then $\Gamma\left(O(N,NK),E\right)$ is a discrete $N_G(NK)$-module.
\end{proposition}
\begin{remark}
It is a consequence of Proposition \ref{sectionquot} that if $K\leq \text{stab}_G(U)$ then $\Gamma\left(U,E\right)$ is also a discrete $K$-module. Therefore in the particular case where $X=SG$ we know that $\Gamma\left(O(N,NK),E\right)$ is also a discrete $N_G(K)$-module and a discrete $K$-module.
\end{remark}
The following proposition provides us with a useful way of characterising stalk-wise fixed sheaves. A specific case of this result will give us precisely the characterisation of Weyl-$G$-sheaves over $SG$ that we will need in Lemma \ref{stalk}.
\begin{proposition}\label{stalkfixchar}
Suppose that $X$ is a profinite $G$-space such that each $x\in X$ has a neighbourhood basis $\mathfrak{B}_x$ such that each $U\in\mathfrak{B}_x$ is $\stabgx(x)$-invariant. Let $E$ be a stalk-wise fixed $G$-sheaf of $\mathbb{Q}$-modules over $X$. Then for each $x\in X$ it follows that each $s_x\in E_x$ is represented by a section:
\begin{align*}
s:U\rightarrow E
\end{align*}
which is $\stabgx(x)$-equivariant for some $U\in \mathfrak{B}_x$.
\end{proposition}
\begin{proof}
We begin with an application of Proposition \ref{sectionquot} to deduce that the action of $\text{stab}_G(x)$ on a representative $s:U\rightarrow E$ factors through a finite quotient, so we can consider finitely many translates say $g_1,g_2,\ldots,g_k$. Since $s_x$ is fixed we know that $(g_j*s)_x=s_x$ for each $1\leq j\leq k$, and so by definition there exists some $V_j\in \mathfrak{B}_x$ with $(g_j*s)_{|_{V_j}}=s_{|_{V_j}}$ for each $j$. We can therefore choose $W\in \mathfrak{B}_x$ with $W\subseteq \underset{1\leq j\leq k}{\bigcap} V_j$, and hence $s_{|_{W}}$ is $\text{stab}_G(x)$-equivariant by construction (since $W$ is known to be $\text{stab}_G(x)$-invariant).
\end{proof}
The following proposition is similar to the proceeding one and is crucial in proving the equivalence of categories between Weyl-$G$-sheaves of $\mathbb{Q}$-modules and rational $G$-Mackey functors in Lemma \ref{stalk}.
\begin{proposition}\label{Weylequi}
If $E$ is a Weyl-$G$-sheaf over $SG$ and $K\in SG$, then for any $s_K\in E_K$ there exists an open subgroup $J$ of $G$ containing $K$ such that $s_K$ is represented by a $J$-equivariant section of the form:
\begin{align*}
s:O(N,NK)\rightarrow E,
\end{align*}
where $J=NK$.
\end{proposition}
\begin{proof}
Take any such $s_K$, then this has a representative $s:O(N^{\prime},N^{\prime}K)\rightarrow E$ and the action of $K$ upon $s$ factors through a finite quotient $K/\left(K\cap M^{\prime}\right)$, for some open normal subgroup $M^{\prime}$ of $G$ by Proposition \ref{equisectweyl}.

Consider the coset representatives $k_1,k_2,\ldots,k_n$. Since $E$ is a Weyl-$G$-sheaf we have that there ${k_i}s_K=s_K$ for each $1\leq i\leq n$, and hence for each $1\leq i\leq n$ there exists $N_i$ open and normal in $G$ so that:
\begin{align*}
{s_|}_{O(N_i,N_iK)}= {(k_i*s)_|}_{O(N_i,N_iK)}.
\end{align*}
Set $L=\underset{1\leq i\leq n}{\bigcap}N_i$, then ${s_|}_{O(L,LK)}={(k*s)_|}_{O(L,LK)}$ for each $k\in K$ and therefore is $K$-equivariant by Proposition \ref{equifix}. By Proposition \ref{equivariant} there exists an open normal subgroup $M$ of $G$ so that ${s_|}_{O(L,LK)}$ is $M$-equivariant and therefore this section must be $J=MK$-equivariant as required.

Note if we set $N=M\cap L$ then we first observe that $O(N,NK)\subseteq O(L,LK)$. Also $O(N,NK)$ is $NK$-invariant since if $A\in O(N,NK)$ and $x\in NK$ then:
\begin{align*}
xAx^{-1}N=xANx^{-1}=xNKx^{-1}=NK
\end{align*}
so $xAx^{-1}\in O(N,NK)$. Furthermore if we restrict $s$ to $O(N,NK)$ then we have that $xs(A)=s(xAx^{-1})$ since $x\in NK\leq MK$.
\end{proof}
\begin{corollary}\label{RFdefn}
If $E$ is a $G$-sheaf of $\mathbb{Q}$-modules over $SG$ and $s_K\in E_K$ is $K$-fixed, then $s_K$ can be represented by a section:
\begin{align*}
s:O(N,NK)\rightarrow E
\end{align*}
which is $NK$-equivariant.
\end{corollary}
Consider $q_H:SH\rightarrow SH/H$ the quotient map for $H$ an open subgroup of a profinite group $G$. For the following proposition, if $V\subseteq SH/H$ is open notice that $q_H^{-1}(V)$ is open in $SH$, which is open in $SG$ by Lemma \ref{subgroup_open}. 
\begin{proposition}\label{sheafact}
If $F$ is a $G$-sheaf over $SG$ and $H\leq G$ open, then the assignment $V\mapsto F(q_H^{-1}(V))^H$ for $V\subseteq SH/H$ open gives a sheaf over $SH/H$.
\end{proposition}
\begin{proof}
We start with $V\subseteq SH/H$ and an open covering $\left\lbrace U_{\lambda}\mid \lambda\in \Lambda\right\rbrace$ of $V$. Suppose $s_1,s_2\in F(q_H^{-1}(V))^H$ such that $\rho^{q_H^{-1}(V)}_{q_H^{-1}(U_{\lambda})}(s_1)=\rho^{q_H^{-1}(V)}_{q_H^{-1}(U_{\lambda})}(s_2)$ for each $\lambda$. Then since these sections are also sections with respect to the sheaf $F$ we can use the sheaf properties of $F$ to deduce that $s_1=s_2$. 

Similarly suppose that we have a family $s_{\lambda}$ such that for each pair $\lambda,\mu\in \Lambda$ we have $\rho^{q_H^{-1}(U_{\lambda})}_{q_H^{-1}(U_{\lambda})\bigcap {q_H^{-1}(U_{\mu})}}(s_{\lambda})=\rho^{{q_H^{-1}(U_{\mu})}}_{q_H^{-1}(U_{\lambda})\bigcap {q_H^{-1}(U_{\mu})}}(s_{\mu})$. It follows from the fact $F$ is a sheaf and that we are dealing with sections over the sheaf $F$ that there exists a section $s$ over $q_H^{-1}(V)$ such that each $s_{\lambda}=\rho^{q_H^{-1}(V)}_{\lambda}(s)$. This is $H$-equivariant since it is built from $H$-equivariant sections glued together. The definition we are using here is given in \cite[Definition 2.1.1]{Tennison}.
\end{proof}
\section{Infinite Products of G-equivariant sheaves}
Previously we have constructed the limit of a finite diagram of $G$-equivariant sheaves and observed that the given proof does not hold for infinite diagrams. We now address what an infinite product of $G$-equivariant sheaves look like.

In the setting where $G$ is finite the argument provided for the products of finite diagrams hold since $G$ has the discrete topology and therefore the continuity issues of the $G$-action disappear. This in particular means that if $G$ is a finite discrete group, then products of $G$-equivariant sheaves are products taken in the category of non-equivariant sheaves with additional structure given by a $G$-action. For more general profinite groups we would like this characterisation of products to hold, however infinite products are more difficult to characterise as the following example illustrates. We recall the definition of a discrete $G$-module from Definition \ref{discreteGmod}.
\begin{example}
If we let $X$ be the one point space and $G=\mathbb{Z}_p$, then $G$-sheaves of $\mathbb{Q}$-modules are given by discrete $\mathbb{Q}[\mathbb{Z}_p]$-modules. If we consider a collection $\mathbb{Q}\left[\mathbb{Z}/p^n\mathbb{Z}\right]$ for each $n\in\mathbb{N}$, then the underlying sheaf product is given by $\underset{n\in\mathbb{N}}{\prod}\mathbb{Q}\left[\mathbb{Z}/p^n\mathbb{Z}\right]$. This is a problem since this product is not a discrete $\mathbb{Z}_p$-module.
\end{example} 
We overcome the discreteness problem by considering the discretisation. Recall from Remark \ref{discGmod} that if $M$ is a $G$-module for $G$ profinite, then the discretisation of $M$ denoted $\text{disc}(M)$ is given by $\underset{H\underset{open}{\leq} G}{\colim}\, M^H$. In order to construct infinite products explicitly we need to define the following category.
\begin{construction}\label{Gsubdisc}
We define the category $G\text{-subdisc}$, whose objects are the discrete $H$-modules $M$ for any open subgroup $H$ of $G$. For the morphisms, if $M$ is a discrete $H$-module and $N$ a discrete $J$-module for $H$ and $J$ open in $G$, then a morphism from $M$ to $N$ is a $H\cap J$-equivariant module map $f:M\rightarrow N$.

If $F^i$ are a collection of $G$-sheaves of $\mathbb{Q}$-modules over a profinite $G$-space $X$, and $\mathfrak{B}_X^{\text{op}}$ represents the category of basic open subsets of $X$ (i.e., the closed-open basis) then we have the following functor:
\begin{align*}
F:\mathfrak{B}_X^{\text{op}}&\rightarrow G\text{-subdisc}\\U&\mapsto \text{disc}\left(\underset{i\in I}{\prod}F^i(U)\right)
\end{align*}
where the discretisation at each $U$ is taken with respect to $\text{stab}_G(U)$. Notice that this could be zero. 
\end{construction}
We now give the explicit construction for the infinite product of $G$-sheaves of $\mathbb{Q}$-modules over a profinite $G$-space $X$.
\begin{construction}\label{profinfprod}
Given a family $F^i$ of $G$-sheaves of $\mathbb{Q}$-modules over $X$ we look at the functor $F$ from Construction \ref{Gsubdisc}. We define the underlying set of the $G$-sheaf space of the product $E$ to be:
\begin{align*}
\underset{x\in X}{\coprod}\underset{U\backepsilon x}{\colim}\,F(U)
\end{align*} 
We topologise this by considering a basis given by sets of the form:
\begin{align*}
s(U)=\left\lbrace s(x)\mid x\in U\right\rbrace
\end{align*}
where $U$ ranges through the basic open subsets of $X$ and $s\in F(U)$. For the $G$-action we define the map as follows:
\begin{align*}
\psi:G\times E&\rightarrow E\\ \left(g,\left[(s^i)_{i\in I}\right]_x\right)&\mapsto \left[(g*s^i)_{i\in I}\right]_{gx}
\end{align*}
where $(s^i)_{i\in I}\in F(U)$ for some open neighbourhood $U$ of $x$ and $\left[..\right]_x$ represents the germ of the section over $x$. Here we are using that if $s\in F(U)$ then it follows that $g*s$ belongs to $F(gU)$. To see this, observe that if $N$ is an open normal subgroup of $G$ such that $s$ is $N$-equivariant, then it follows that $g*s$ is also $N$-equivariant. 

Notice that we have a monomorphism $\iota$ from $E$ into $\underset{i\in I}{\prod}F^i$, hence the product projection maps are of the form $p_i\circ\iota$ where $p_i$ is the product projection for the non-equivariant sheaf product.
\end{construction}
We observe that this is what we need for $F$ to be a $G$-equivariant sheaf in light of Proposition \ref{sectionquot}. This tells us that if $F$ is a $G$-equivariant sheaf and $U$ a closed-open basis element of $X$, then sections over $U$ has to be a discrete $\text{stab}_G(U)$-module. It is for this reason that the product in the category of non-equivariant sheaves fails to deliver a $G$-sheaf by considering the $G$-action induced by the product. Namely, the $\text{stab}_G(U)$-action on sections over $U$ for the non-equivariant product of sheaves is not a discrete action.  
\begin{lemma}
Construction \ref{profinfprod} is the product in the category of $G$-sheaves.
\end{lemma}
\begin{proof}
This is clearly a sheaf space of $\mathbb{Q}$-modules since the sections are constructed using sections of the sheaf of $\mathbb{Q}$-modules given by $\underset{i\in I}{\prod}F^i$ in a manner similar to the sheafification process. We need to verify that the $G$-action on this space is continuous. Take any basic open subset of the form $s(U)$ and take $(g,t_y)\in \psi^{-1}(s(U))$.

It follows that $g(t_y)=s_x$ for some $x\in U$ and so $gy=x$. We know by definition of $F(U)$ that there must exist some open normal subgroup $N$ of $G$ such that $s$ is $N$-equivariant, and so $V=gN\times (g^{-1}*s)(g^{-1}U)$ is an open neighbourhood of $(g,t_y)$. To see that $V\subseteq \psi^{-1}(s(U))$, if $(gn,(g^{-1}*s)(g^{-1}z))$ belongs to $V$ then:
\begin{align*}
gn(g^{-1}*s)(g^{-1}z)=gng^{-1}s(z)=s(gng^{-1}z),
\end{align*}
where the last belongs to $s(U)$ since $U$ is $N$-invariant. 

For the universality condition, if for each $i\in I$ we have morphisms of $G$-sheaves of $\mathbb{Q}$-modules over $X$ of the form $\alpha_i:T\rightarrow F^i$ then we have a unique morphism $\alpha:T\rightarrow \underset{i\in I}{\prod}F^i$ of sheaves of $\mathbb{Q}$-modules. If $U$ is an closed-open basic subset of $X$, then since $T(U)$ is a discrete $\text{stab}_G(U)$-module it follows that $\im\alpha(U)$ is contained in the discrete part of $\underset{i\in I}{\prod}F^i(U)$. This says that $\alpha$ factors as demonstrated by the following diagram:
\begin{align*}
\xymatrix{&E\ar[dr]^{\iota}\\T\ar@{-->}[ur]^{\overline{\alpha}}\ar[rr]^{\alpha}\ar[dr]_{\alpha_i}&&\underset{i\in I}{\prod}F^i\ar[dl]^{p_i}\\&F^i}
\end{align*}  
The uniqueness of the morphism follows immediately from the uniqueness of $\alpha$ since $\iota$ is a monomorphism.
\end{proof}
\section{Useful Adjunctions}
We finish the chapter by proving some useful adjunctions relating to $G$-equivariant sheaves over a $G$-space $X$. We will use these results when constructing injective resolutions of $G$-sheaves and Weyl-$G$-sheaves in Chapter \ref{finalchap}. We begin by proving the following proposition.
\begin{proposition}\label{orbclose}
If $G$ is a profinite group and $X$ a profinite $G$-space, then the $G$-orbit of any point $x\in X$ is closed. 
\end{proposition}
\begin{proof}
Consider the following continuous map:
\begin{align*}
f:G\times X&\rightarrow X\times X\\(g,x)&\mapsto (gx,x).
\end{align*}
Both the domain and codomain are compact and Hausdorff, and therefore closed and compact conditions are equivalent in these spaces. In particular  this implies that preimages of compact spaces are compact and images of closed subsets are closed.

We know that $G\times \left\lbrace x\right\rbrace$ is closed in $G\times X$, so therefore its image $Gx\times \left\lbrace x\right\rbrace$ is closed in $X\times \left\lbrace x\right\rbrace$. We also know that $Gx\times \left\lbrace x\right\rbrace$ is closed in $X\times X$ for the same reason. This therefore leads to the following commutative diagram:
\begin{align*}
\xymatrix{Gx\ar[d]\ar[r]^-{\cong}&Gx\times \left\lbrace x\right\rbrace\ar[d]\ar[dr]\\X\ar[r]^-{\cong}&X\times\left\lbrace x\right\rbrace\ar[r]&X\times X}
\end{align*}
where the two horizontal morphisms in the square are homeomorphisms, and the three arrows in the triangle on the left are closed inclusions, and the remaining arrow on the left is an inclusion. This diagram shows that $Gx$ is the homeomorphic image of a closed subset of a space which is homeomorphic to $X$, hence $Gx$ is closed in $X$.  
\end{proof}
We now look at an example of a $G$-equivariant sheaf which will be useful in our theory. Notice that if $X$ is a $G$-space and $A$ is a $\mathbb{Q}$-module, for each $x\in X$ we can view $A$ as a discrete $\text{stab}_G(x)$-module by giving $A$ the trivial $\text{stab}_G(x)$-action.
\begin{example}\label{constantsheafex}
Let $X$ be a $G$-space, $A$ be a discrete $\mathbb{Q}[G]$-module and $cA$ be the constant sheaf over $X$. We define a continuous $G$-action on the sheaf space of $cA$ by:
\begin{align*}
G\times LcA&\rightarrow LcA\\(g,a_{x})&\mapsto ga_{gx}
\end{align*}
where $a_x$ and $ga_{gx}$ correspond to the stalk $a,ga\in A$ as germs over $x$ and $gx$ respectively.
\end{example}
\begin{proof}
The construction $LcA$ is clearly a sheaf space with $G$-equivariant local homeomorphism. It is left to show that the $G$-action on $LcA$ is continuous. We begin by taking a basic open subset of $LcA$ of the form $a(U)$ for $a\in A$ and some basic open subset $U$ which we can assume to be compact. We know that we have an open subgroup $\text{stab}_G(U)$ since $U$ is compact, and since $A$ is discrete we know that $\text{stab}_G(a)$ is open too. Therefore we can choose an open normal subgroup $N$ satisfying that $N\leq \text{stab}_G(a)\cap \text{stab}_G(U)$. We will show that the open subset $Ng\times g^{-1}a(g^{-1}U)$ of $G\times LcA$ is in the preimage of $a(U)$.

Take any element in the set $(ng,g^{-1}a_{g^{-1}t})$ for $t\in U$ and let $\psi$ denote the action map. Then:
\begin{align*}
\psi(ng,g^{-1}a_{g^{-1}t})=ngg^{-1}a_{ngg^{-1}t}=na_{nt}=a_{nt}.
\end{align*}
The latter belongs to $a(U)$ since $U$ is $N$-invariant.
\end{proof}
In the case where $A$ has the trivial action, the action on this space permutes the stalks and the action of $\text{stab}_G(x)$ on the stalk at $x$ is trivial. Note if $X$ is a transitive space and $x\in X$ arbitrary, then $cA$ is represented by the sheaf space $G\underset{\text{stab}_G(x)}{\times}A$.
\begin{example}\label{transitivesheafex}
Let $X$ be a transitive $G$-space and $A$ be a discrete $\mathbb{Q}\left[\text{stab}_G(x)\right]$-module for some $x\in X$. Then we have a $G$-sheaf given by:
\begin{align*}
G\underset{\text{stab}_G(x)}{\times}A.
\end{align*}
The $G$-action is given by the map:
\begin{align*}
G\times \left(G\underset{\text{stab}_G(x)}{\times}A\right)&\rightarrow G\underset{\text{stab}_G(x)}{\times}A\\
\left(g,[g^{\prime},a]\right)&\mapsto [gg^{\prime},a].
\end{align*}
Notice we have a non-trivial action on the stalks. If $h\in\text{stab}_G(x)$ then we have the following:
\begin{align*}
ghg^{-1}*[g,a]=[ghg^{-1}g,a]=[gh,a]=[g,ha].
\end{align*} 
\end{example}
We use that a typical element of $\text{stab}_G(gx)$ is of the form $ghg^{-1}$. We next prove a useful lemma which will provide us with an important adjoint pair of functors that will play a pivotal role later on.
\begin{lemma}\label{restconst}
Let $O$ be an orbit in a $G$-space $X$, $x\in O$ be a point and $E$ be a $G$-equivariant sheaf over $X$. Then:
\begin{align*}
E_{|_{O}}\cong G\underset{\stabgx(x)}{\times}E_x
\end{align*}
as $G$-sheaf spaces over $O$.
\end{lemma}
\begin{proof}
We begin by considering the map:
\begin{align*}
\psi:G\times E_x &\rightarrow E_{|_{O}}\\(g,e)&\mapsto ge.
\end{align*}
This is a continuous $G$-map since it is the restriction of the $G$-action map on the $G$-sheaf space $E$. Also we can see that this is surjective since if $e$ belongs to $E_{|_{O}}$ where $e$ is a germ over $gx$ for some $g\in G$, then $(g,g^{-1}e)$ is in the preimage. Therefore $\psi$ is a continuous surjection.

Since $\psi(gh,e)=(gh)e=g(he)=\psi(g,he)$ we have the following factorisation:
\begin{align*}
\xymatrix{G\times E_x\ar[r]^{\psi}\ar@{->>}[d]&E_{|_{O}}\\G\underset{\text{stab}_G(x)}{\times}E_x\ar[ur]_{\overline{\psi}}}
\end{align*}
We know that $\overline{\psi}$ satisfies the same properties that we have shown $\psi$ to possess. $\overline{\psi}$ is also a map of sheaf spaces since it is compatible with the local homeomorphisms and induces $\mathbb{Q}$-module maps on stalks. It therefore is an open map too since we can apply \cite[Lemma 2.3.5]{Tennison}. We will now prove that in addition it is also injective. Suppose that $\left[g_1,e_1\right],\left[g_2,e_2\right]$ satisfy that $g_1e_1=g_2e_2$. It therefore follows that $g_1x=g_2x$ and hence ${g_1}^{-1}g_2x=x$ so that ${g_1}^{-1}g_2\in \text{stab}_G(x)$. Therefore:
\begin{align*}
[g_2,e_2]=[g_1{g_1}^{-1}g_2,e_2]=[g_1,{g_1}^{-1}g_2e_2]=[g_1,e_1]
\end{align*}
as required.
\end{proof}
As a consequence of Proposition \ref{orbclose} we can apply extension by zero to end up with a $G$-equivariant sheaf $\overline{E_{|_{O}}}$. Notice that the stalks of $\overline{E_{|_{O}}}$ or equivalently of $\overline{G\underset{\text{stab}_G(x)}{\times}E_x}$ for $x\in O$ are isomorphic to $E_y$ for any $y\in O$ and zero otherwise. 
\begin{example}\label{Gskyscraper}
If $O$ is an orbit of a $G$-space $X$, then for any $x\in O$ we have a sheaf over $X$ given by $\overline{G\underset{\text{stab}_G(x)}{\times}A}$ for any discrete $\mathbb{Q}\left[\text{stab}_G(x)\right]$-module $A$. This is extension by zero of the $G$-sheaf from Example \ref{transitivesheafex}.
\end{example}
If $X$ is a $G$-space and $x\in X$, then there are functors defined by:
\begin{align*}
(-)_x:G\text{-Sheaf}(X)&\rightarrow \mathbb{Q}\left[\text{stab}_G(x)\right]\text{-Mod}\\E&\mapsto E_x
\end{align*}
where a morphism $f:E\rightarrow F$ is mapped to $f_x$, and
\begin{align*}
R:\mathbb{Q}\left[\text{stab}_G(x)\right]\text{-Mod}&\rightarrow G\text{-Sheaf}(X)\\A&\mapsto \overline{G\underset{\text{stab}_G(x)}{\times}A}
\end{align*}
where a morphism $\alpha:A\rightarrow B$ is mapped to the morphism of $G$-sheaves induced by $\id{\times}\alpha$. We now verify that these functors provide an adjunction.
\begin{proposition}\label{equiadjunct}
The pair of functors $\left((-)_x,R\right)$ are an adjoint pair.
\end{proposition}
\begin{proof}
Let $E,F$ be $G$-equivariant sheaves, $A,B$ be $\mathbb{Q}\left[\text{stab}_G(x)\right]$-modules and $f:E\rightarrow F$ and $\alpha:A\rightarrow B$ be morphisms in their respective categories. We will define a map of sets as follows:
\begin{align*}
\phi_{E,A}:\hom\left(E,\overline{G\underset{\text{stab}_G(x)}{\times}A}\right)&\rightarrow \hom\left(E_x,A\right)\\
h&\mapsto h_x. 
\end{align*}
We show that this is a bijection by proving that this is a composition of bijective maps. First note that the following map is a bijection:
\begin{align*}
\hom\left(E,\overline{G\underset{\text{stab}_G(x)}{\times}A}\right)&\rightarrow \hom\left(E_{|_{O}},G\underset{\text{stab}_G(x)}{\times}A\right)\\h&\mapsto h_{|_{O}}
\end{align*}
as a result of the direct image-inverse image functor as seen in \cite[Theorem 3.7.13]{Tennison}, and using the fact that both $O$ and the complement of $O$ are closed under taking orbits so that $h$ is $G$-equivariant implies that $h_{|_{O}}$ is. Also $h$ applied to the stalks of the complement of $O$ must be zero. Therefore, using that the complement of $O$ is $G$-invariant, $h$ is $G$-equivariant if $h_{|_{O}}$ is.

As a consequence of Lemma \ref{restconst}, $h_{|_{O}}$ in turn is uniquely determined by $h_x$. This is because we can define a morphism of $G$-sheaves from $h_x$ as follows:
\begin{align*}
G\underset{\text{stab}_G(x)}{\times}E_x&\rightarrow G\underset{\text{stab}_G(x)}{\times}A\\ \left[g,e\right]&\mapsto \left[g,h_x(e)\right] 
\end{align*}

We now seek to show that these bijections are natural. Let $\alpha^{\prime}$ denote the map induced by $\id\times \alpha$. Then the commutativity of the two necessary squares  follows immediately from the equalities:
\begin{enumerate}
\item $\left(\alpha^{\prime}\circ h\right)_x=\alpha\circ h_x$,
\item $\left(h\circ f\right)_x=h_x\circ f_x$.
\end{enumerate}
\end{proof}
\begin{corollary}
If $X$ is a profinite space of the form $SG$ where $G$ is a profinite group, then the adjoint pair defined in Proposition \ref{equiadjunct} give an adjoint pair $\left((-)_x,R\right)$:
\begin{align*}
\xymatrix{\Weyl(SG)\ar@<-4pt>[r]|-a
					&\mathbb{Q}\left[N_G(K)/K\right]\Mod\ar@<-4pt>[l]|-b},
\end{align*} 
where $a$ is the functor $(-)_x$ and $b$ is the functor $R$ from Proposition \ref{equiadjunct}, for any $K\in SG$.
\end{corollary}

This adjoint pair has a useful application as the following proposition will illustrate.
\begin{proposition}\label{equivinjshf}
If $X$ is a $G$-space then the $G$-equivariant sheaves of the form $\overline{G\underset{\stabgx(x)}{\times}A}$ are injective in the category of $G$-equivariant sheaves. This result holds in the category of Weyl-$G$-sheaves over $SG$ also.
\end{proposition}
\begin{proof}
First notice that we can apply the result from \cite[Proposition 3.1]{Castellano} which says that the category $\mathbb{Q}\left[\text{stab}_G(x)\right]\text{-Mod}$ has injective dimension zero. Looking at the adjoint pair of functors in Proposition \ref{equiadjunct} it is clear that the left adjoint of the pair preserves monomorphisms, so we can use the fact that the corresponding right adjoint preserves injective objects.
\end{proof}
The application of \cite{Castellano} in this proof formally completes a proof in \cite{BarZp} which left out the verification that $\mathbb{Q}\left[\mathbb{Z}_p\right]$-modules are injective.

\chapter{Correspondence}
In this chapter we will construct a correspondence between rational Mackey functors for $G$ profinite, and Weyl-$G$-sheaves over $SG$. The following diagram will illustrate our plan:
\begin{align*}
\xymatrix{\text{Weyl-}G\text{-sheaf}(SG)\ar@<-4pt>[r]|-a
					&\text{Mackey}_{\mathbb{Q}}(G)\ar@<-4pt>[l]|-b}
\end{align*} 
where we shall explicitly construct functors $a$ and $b$, (Theorems \ref{sheaftomack} and \ref{macktosheaf}). In Theorem \ref{equivalencemain} we will see that these functors are equivalences.
\section{Weyl-G-sheaves determine Mackey Functors}
We now seek to demonstrate a correspondence between Mackey functors of profinite groups $G$ and $\text{Weyl-}G$-sheaves over the $G$-space $SG$. 

We begin by proving a lemma which loosely says that the product of sums of two transversals is another transversal, and this is useful in proving the transitivty property of Mackey functors for the construction we will derive from any $\text{Weyl-}G$-sheaf.

\begin{lemma}\label{transverse}
Let $G$ be a profinite group with $L\leq J\leq H\leq G$, all open. Take sums over a left transversal for $L$ with respect to $H$, for $L$ with respect to $J$ and for $J$ with respect to $H$ given by $\underset{1\leq k\leq n}{\sum} h_kL$, $\underset{1\leq a\leq m}{\sum} j_a L$ and $\underset{1\leq b\leq p}{\sum} \overline{h}_bJ$ respectively. Then $\underset{1\leq b\leq p}{\sum} \,\,\underset{1\leq a\leq m}{\sum}\overline{h}_b j_a L=\underset{1\leq k\leq n}{\sum} h_kL$. 
\end{lemma}
\begin{proof}
It is sufficient to show that the set $\left\lbrace \overline{h}_bj_aL\right\rbrace_{a,b}$ is a transversal for $L$ with respect to $H$. We do this by showing that this set is in one to one correspondence with the set $\left\lbrace h_kL\right\rbrace_{k}$. Given $h_kL$, we will show that this equals some $\overline{h}_bj_aL$. First $h_kJ=\overline{h}_bJ$ for some $b$, so $h_k=\overline{h}_bj$ for $j\in J$. Similarly $jL=j_aL$ for some $a$ so $h_kL=\overline{h}_bjL=\overline{h}_bj_aL$. This is clearly a surjective correspondence since any $\overline{h}_bj_aL$ equals some $h_kL$ as the cosets partition $H$ and we are working with transversals. Also this is injective since if $h_kL$ and $h_{k^{\prime}}L$ equal the same $\overline{h}_bj_aL$ then they are equal by transitivity of equality.
\end{proof}
We begin constructing the correspondence by defining and proving the existence of a functor:
\begin{align*}
\text{Mackey}:\text{Weyl-}G\text{-Sheaf}_{\mathbb{Q}}(SG)\rightarrow \mackey_{\mathbb{Q}}(G)
\end{align*}
where $\text{Weyl-}G\text{-Sheaf}_{\mathbb{Q}}(SG)$ is the category of $\text{Weyl-}G$-sheaves over $SG$ and \\$\mackey_{\mathbb{Q}}(G)$ is the category of Mackey functors over $G$. We will achieve this by first constructing a functor from $G$-sheaves over $SG$ to $\mackey_{\mathbb{Q}}(G)$ and later we will show that this restricts to a functor from $\text{Weyl-}G\text{-Sheaf}_{\mathbb{Q}}(SG)$ to $\mackey_{\mathbb{Q}}(G)$ which is an equivalence of categories.
\begin{construction}\label{mackey_construct}
Let $F$ be a $G$-sheaf over $SG$, we can define a Mackey functor\index{Mackey functor} $\text{Mackey}(F)$ as follows:

If $H\leq G$ is open then we define $\text{Mackey}(F)(G/H)=F(SH)^H$. Explicitly this is the set $\left\lbrace \alpha:SH\rightarrow LF\mid p\circ \alpha=\id,\alpha\,\text{is continuous}\right\rbrace^H$, where $LF$ is the sheaf space for $F$ and $p$ is the local homeomorphism for $LF$. For the maps, if $H,J\leq G$ are both open with $J\leq H$ then we have the following restriction map:
\begin{center}
$R^H_J:F(SH)^H\rightarrow F(SJ)^J$ given by $\alpha\mapsto \alpha|_{SJ}$.
\end{center}
Note that if $\alpha$ is $H$-fixed then it is $J$-fixed so the restriction has image in $F(SJ)^J$. For the induction maps we proceed as follows:
\begin{center}
$I^H_J:F(SJ)^J\rightarrow F(SH)^H$, given by $\beta\mapsto \underset{1\leq i\leq n}{\sum}h_i \overline{\beta}$,
\end{center}
where $n=|H/J|$, $\left\lbrace h_iJ\mid 1\leq i\leq n\right\rbrace$ is a left transversal, and $\overline{\beta}$ is the extension of $\beta$ by $0$, which is continuous since $SJ$ is closed-open. The conjugation maps $C_g$ are given by the $G$-action on the sheaf $F$.
\end{construction} 

\begin{lemma}\label{well_define}
The maps defined in Construction \ref{mackey_construct} are well defined.
\end{lemma}
\begin{proof}
This is clearly true for the conjugation maps, and it is true for the restriction for the reason stated in Construction \ref{mackey_construct}. For induction, first notice that if $L$ is an open subgroup of $J$ and we have an $L$-equivariant map from $SL$ to some sheaf space $E$ denoted $\beta$, the extension of $\beta$ by zero $\overline{\beta}$ gives a map from $SJ$ to $E$. Also $j\overline{\beta}$ defined by $x\mapsto j\beta(j^{-1}x)$ is independent of the choice of left coset representative in $J/L$. To see this suppose we have two representatives $j_1L=j_2L$. Then we have $j_1\beta=j_1\left(j_1^{-1}j_2\right)\beta$ since $j_1^{-1}j_2\in L$ and $\beta $ is $L$-fixed. By associativity of the group action this equals $\left(j_1j_1^{-1}\right)j_2\beta=j_2\beta$. This shows that $I_L^J(\beta)$ as defined above is independent of the choice of transversal.

To show that the section $\underset{J/L}{\sum}j\overline{\beta}$ is $J$-fixed, note that for any $a\in J$ and any point $x$ of $SJ$ we have:
\begin{align*}
\left(a\underset{J/L}{\sum}\left(j\overline{\beta}\right)\right)(x)&=a\underset{J/L}{\sum}\left(j\overline{\beta}\right)(a^{-1}xa)=\underset{J/L}{\sum}a\left(j\overline{\beta}\right)(a^{-1}xa)\\&=\underset{J/L}{\sum}aj\overline{\beta}(j^{-1}a^{-1}xaj)=\underset{J/L}{\sum}\left(aj\overline{\beta}\right)(x)
\end{align*}
Since we have already proven that the definition of induction above is independent of the choice of transversal we have that this equals $\underset{J/L}{\sum}\left(j\overline{\beta}\right)(x)$.
\end{proof}
\begin{lemma}\label{fix}
The construction $\mackey(F)$ satisfies that $\mackey(F)(G/H)$ is $H$-fixed, and both restriction and induction from $H$ to itself is the identity.
\end{lemma}
\begin{proof}
We next show that for each $H\leq G$ open we have that $C_h$ for each $h \in H$, $I^H_H$, and $R^H_H$ are all the identity. The conjugation by each $h\in H$ follows from the fact that $F(SH)^H$ is $H$-fixed. Clearly $R^H_H$ is the identity. We see that $I^H_H$ is the identity by observing the following two facts. We know that extension by $0$ of a section over $SH$ to itself is clearly the identity. Also there is only one element in the transversal for $H$ over itself, namely $eH$, so using that we are interested in $H$-fixed sections the left transversal $eH$ changes nothing.
\end{proof}
\begin{lemma}\label{associative}
The structure maps of $\mackey(F)$ are associative and transitive.
\end{lemma}
\begin{proof}
For associativity of the conjugation maps observe that we have $C_gC_h=C_{gh}$ by the definition of the action on the $G$-sheaf. It is clear that $R^J_L\circ R^H_J=R^H_L$ by definition of restriction of functions. For induction we have:
\begin{align*}
I^H_JI^J_L(\beta)&=\underset{h}{\sum} hJ \left(\overline{\underset{j}{\sum} jL\overline{\beta}}\right)=\left(\underset{h}{\sum}hJ\left( \underset{j}{\sum} jL \overline{\overline{\beta}}\right)\right)\\&=\underset{h}{\sum} hL\overline{\beta}=I^H_L(\beta)
\end{align*}
where we apply Lemma \ref{transverse} and $\overline{\overline{\beta}}$ corresponds to extension by $0$ applied twice.
\end{proof}
\begin{lemma}\label{equivariance}
The construction $\mackey(F)$ satisfies the equivariance condition, that is the restriction and induction maps commute with the conjugation maps.
\end{lemma}
\begin{proof}
For the equivariance condition, by definition of restriction we clearly have $C_gR^H_J=R^{gHg^{-1}}_{gJg^{-1}}C_g$. For the induction maps first observe that if we take a sum over a transversal $\underset{1\leq i\leq n}{\sum} h_iJ$ then conjugation gives another $\underset{1\leq i\leq n}{\sum} gh_ig^{-1}gJg^{-1}$. We have the following equalities: 
\begin{align*}
I^{gHg^{-1}}_{gJg^{-1}}(g\beta)(x)&=\underset{1\leq i\leq n}{\sum}\left(gh_ig^{-1}\right)\left(\overline{g\beta}\right)(x)=\underset{1\leq i\leq n}{\sum}g\left(h_ig^{-1}\right)\left(g\overline{\beta}\right)(g^{-1}xg)\\&=\underset{1\leq i\leq n}{\sum}gh_i\left(g^{-1}g\overline{\beta}\right)(h_i^{-1}g^{-1}xgh_i)\\&=\underset{1\leq i\leq n}{\sum}\left(gh_i\overline{\beta}\right)(x)\\&=g\underset{1\leq i\leq n}{\sum}\left(h_i\overline{\beta}(x)\right)=C_gI^H_J(\beta)(x)
\end{align*}
which proves the equality as needed. 
\end{proof}
\begin{lemma}\label{Mackey_axiom}
The construction $\mackey(F)$ satisfies the Mackey axiom.
\end{lemma}
\begin{proof}
We start with $J,L\leq H$ where $H,J,L\leq G$ are open. First note that we have the following:

\begin{align*}
\underset{hL\in H/L}{\coprod} hL&=H=\underset{x\in\left[J\diagdown H\diagup L\right]}{\coprod}JxL=\underset{x\in \left[J\diagdown H\diagup L\right]}{\coprod}\,\,\,\underset{j\in J}{\bigcup} jxL\\&=\underset{x\in \left[J\diagdown H\diagup L\right]}{\coprod}\,\,\,\underset{J/J\cap xLx^{-1}}{\coprod}j_x xL.
\end{align*}
We therefore have that: 
\begin{align*}
R^H_JI^H_L(\beta)&=(\underset{H/L}{\sum} hL\overline{\beta})|_{SJ}=\left(\underset{x\in\left[J\diagdown H\diagup L\right]}{\sum}\underset{J/J\cap xLx^{-1}}{\sum}j_x x\overline{\beta}\right)|_{SJ}\\&=\underset{x\in\left[J\diagdown H\diagup L\right]}{\sum}\underset{J/J\cap xLx^{-1}}{\sum}\left(j_x x\overline{\beta|_{S\left(x^{-1}Jx\cap L\right)}}\right)
\end{align*}
where the extension by zero in the top line is respect to $H$ and the bottom with respect to $J$. If we start from the other direction we have:
\begin{align*}
\underset{x\in \left[J\diagdown H\diagup L\right]}{\sum}I^J_{J\cap xLx^{-1}}C_xR^L_{x^{-1}Jx\cap L}(\beta)&=\underset{x\in \left[J\diagdown H\diagup L\right]}{\sum}\underset{J/J\cap xLx^{-1}}{\sum}j_x x\left(\overline{\beta|_{S\left(L\cap x^{-1}Jx\right)}}\right)\\&= \underset{x\in \left[J\diagdown H\diagup L\right]}{\sum}\underset{J/J\cap xLx^{-1}}{\sum}\left(j_x x\overline{\beta|_{S\left(L\cap x^{-1}Jx\right)}}\right)
\end{align*}
proving that the two sides coincide.
\end{proof}
\begin{theorem}\label{sheaftomack}
Let $G$ be a profinite group. If $F$ is a $G$-sheaf of $\mathbb{Q}$-modules over $SG$ then $\mackey(F)$ as constructed in Construction \ref{mackey_construct} is a Mackey functor.
\end{theorem}
\begin{proof}
This follows by combining Lemmas \ref{well_define}, \ref{fix}, \ref{associative}, \ref{equivariance} and \ref{Mackey_axiom}.
\end{proof}
We will see that the assignment $\mackey$ is functorial.
\begin{theorem}\label{MackisFunct}
The assignment $\mackey(-)$ is a functor.
\end{theorem}
\begin{proof}
We will first define what $\mackey$ does to morphisms of Weyl-$G$-sheaves. Let $f:F\rightarrow F^{\prime}$ be a morphism of Weyl-$G$-sheaves. By Lemma \ref{subgroup_open} we know that $SH$ is open for any open subgroup $H$. This means we have a morphism of $\mathbb{Q}$-modules, $f_{SH}:F(SH)\rightarrow F^{\prime}(SH)$, for every open subgroup $H$. By Proposition \ref{sectionquot}, both $\mathbb{Q}$-modules are $N_G(H)$-modules and since $f$ is a $G$-equivariant map of spaces we know that $f_{SH}$ is $N_G(H)$-equivariant. It follows that a $H$-fixed section $s$ in $F(SH)$ is $H$-fixed in $F^{\prime}(SH)$. This proves that for each $H$ we have a map of $\mathbb{Q}$-modules, $f_H:\mackey(F)(G/H)\rightarrow \mackey(F^{\prime})(G/H)$. 

To prove that this is a morphism of Mackey functors we need to show that we get a commuting square with conjugation, restriction and induction. This is true for conjugation as a consequence of the $G$-equivariance of $f$. For restriction, if $L\leq H$, $K\in SL$ and $s\in \mackey(F)(G/H)$, then this compatibility holds as follows:
\begin{align*}
f(s_{|_{SL}})_{K}=f_K(s_K)=(f(s)_{|_{SL}})_K.
\end{align*}
Here we use the equality $(f(s))_K=f_K(s_K)$. In order to prove compatibility with the induction maps we have to show that:
\begin{align*}
\underset{H/L}{\sum}h\overline{f(s)}=f\left(\underset{H/L}{\sum}h\overline{s}\right).
\end{align*}
We can again use the equality $(f(s))_K=f_K(s_K)$ to prove this since $f$ is additive on stalks and $G$-equivariant. On the one hand we have $\overline{f(s)}_x=0$ if $x\notin SL$ and $f_x(s_x)$ otherwise. On the other hand $f(\overline{s})_x=f_x(\overline{s}_x)$, which equals $0$ if $x\notin SL$ and $f_x(s_x)$ otherwise. The functoriality conditions are clear from the construction.
\end{proof}
\section{Mackey Functors determine Weyl-G-sheaves}
We construct a functor in the opposite direction, from rational $G$-Mackey functors to $\text{Weyl-}G$-sheaves over $SG$, for a profinite group $G$. We use the idea found in \cite{Gre98}, where a Mackey functor $M$ determines sheaves $M(G/H)$ over $SH/H$ for varying $H$ and these fit together to give a Weyl-$G$-sheaf over $SG$. In our setting, we know from Corollary \ref{MHsheaf} that each $M(G/H)$ determines a sheaf over $SH/H$. In this section we will piece these together to make a Weyl-$G$-sheaf over $SG$. We begin by understanding some of the key parts of the construction needed. We start by constructing the stalks of the Weyl-$G$-sheaf.
\begin{lemma}\label{stalkwelldefined}
If $M$ is a Mackey functor for $G$, then the $\mathbb{Q}$-module:
\begin{align*}
F_K=\underset{J\underset{\text{open}}{\geq}K}{\colim}M(G/J)_{(K)}
\end{align*}
is well-defined, where $M(G/J)_{(K)}$ refers to the stalk of the conjugacy class of $K$ in $SJ/J$. Here we are using that $M(G/J)$ determines a sheaf over $SJ/J$.
\end{lemma}
\begin{proof}
This boils down to showing that if $J\geq H\geq K$ are such that $H$ and $J$ are open in $G$ and $K$ closed, then we have a map of the form:
\begin{align*}
\left({R^J_H}\right)_{(K)}:M(G/J)_{(K)}\rightarrow M(G/H)_{K)}. 
\end{align*}
Take $s\in M(G/J)$ and $N_2\leq N_1$ both open and normal in $G,J$ and $H$. Then $O(N_2,N_2K)\subseteq O(N_1,N_1K)$ are open neighbourhoods of $K$, hence their images with respect to the quotient map in the orbit spaces $SJ/J$ and $SH/H$ are open neighbourhoods of the respective class of $K$.

By Proposition \ref{restrictprop} we have the following two equalities:
\begin{align*}
R^J_H\left(e^J_{O\left(N_1,N_1K\right)}s\right)&=\underset{x\in\mathfrak{I}^H_{N_1K}}{\sum}e^H_{O(N_1,xN_1Kx^{-1})}R^J_H(s)+e_VR^J_H(s)\\
R^J_H\left(e^J_{O\left(N_2,N_2K\right)}s\right)&=\underset{x\in\mathfrak{I}^H_{N_2K}}{\sum}e^H_{O(N_2,xN_2Kx^{-1})}R^J_H(s)+e_VR^J_H(s).
\end{align*}
Notice that in the sums of the right hand side, the $O(N_1,xN_1Kx^{-1})$ are disjoint since if $A$ belonged to more than one we would have the following contradiction: 
\begin{align*}
xN_1Kx^{-1}=N_1A=yN_1Ky^{-1}
\end{align*}
and $yN_1Ky^{-1}\neq xN_1Kx^{-1}$. Therefore the right hand side is an idempotent with support given by a finite disjoint union of closed-open subsets.

We need to show that the two right hand side expressions represent the same germ in the stalk of $(K)$. First notice that the class $x=e$ belongs to both $\mathfrak{I}^H_{N_1K}$ and $\mathfrak{I}^H_{N_2K}$. Secondly notice that $(K)$ belongs to both $O(N_1,N_1K)$ and $O(N_2,N_2K)$. We can therefore restrict the top equality to $e^H_{O(N_1,N_1K)}R^J_H(s)$ and the bottom one to $e^H_{O(N_2,N_2K)}R^J_H(s)$, since we are interested in the diagram of neighbourhoods of $(K)$. To conclude, we observe that the support of the second is contained in that of the first and both contain $(K)$.
\end{proof}
In order to define the topology on the sheaf space of the Weyl-$G$-sheaf, we consider the next construction as seen in \cite[Construction 3.7.11]{Tennison}.
\begin{construction}\label{pullbacksheaf}
Let $F$ be a sheaf over a topological space $X$ and $\phi:Y\rightarrow X$ be a continuous surjection. Then we construct a sheaf $\phi^*(F)$ using a pullback diagram as follows:
\begin{align*}
\xymatrix{L\phi^*(F)\ar[d]^{p^{\prime}}\ar[r]&LF\ar[d]^p\\Y\ar[r]^{\phi}&X}
\end{align*}
The sections for $\phi^*(F)$ are given by the continuous sections over $L\phi^*(F)$.
\end{construction}
\begin{proposition}\label{pullbacktop}
The topology for $L\phi^*(F)$ has a basis given by the sets of the form:
\begin{align*}
U_s=\left\lbrace (y,s_{\phi(y)})\mid y\in U\right\rbrace
\end{align*}
where $U$ is an open neighbourhood of $y$ and $s$ is a section over $\phi(U)$ with respect to $LF$.
\end{proposition}
\begin{proof}
By definition, $L\phi^*(F)$ is given the subspace topology of the open box topology. This means that a typical basis subset of $L\phi^*(F)$ is of the form $\left(V\times s(W)\right)\cap L\phi^*(F)$, for basis subsets $V$ and $W$ of $Y$ and $X$ respectively. To see that sets of the form $U_s$ cover $L\phi^*(F)$, take any element $(y,s_{\phi(y)})\in L\phi^*(F)$. Then for any open neighbourhood $U$ of $y$ we have that $(y,s_{\phi(y)})\in U_s$.

Suppose $(y,s_{\phi(y)})\in V_t\cap U_s$, so in particular we have $s_{\phi(y)}=t_{\phi(y)}$. Then there exists $W\subseteq \phi(U)\cap \phi(V)$ such that $s_{|_{W}}$ and $t_{|_{W}}$ agree as sections over $LF$. Since $\phi$ is surjective and continuous we have that $\phi^{-1}(W)$ is open in $Y$ and $\phi(\phi^{-1}(W))=W$. Set $W^{\prime}=\phi^{-1}(W)$, then $W^{\prime}_s\subseteq U_s\cap V_t$ as required. 
\end{proof}
The following two lemmas and corollary focus on Corollary \ref{MHsheaf}, which says that if $L$ is an open subgroup of $G$ profinite then $M(G/L)$ is a sheaf over $SL/L$. Specifically it builds on the previous lemma and tells us how we can relate sections over sheaves $M(G/L)$ for varying $L$. We will consider the sheaf space of each sheaf $M(G/L)$ over $SL/L$, which has underlying set $\underset{y\in SL/L}{\coprod}M(G/L)_{(y)}$ and whose topology is determined by the sheaf structure of $M(G/L)$. In particular, if $s$ belongs to $M(G/L)$ then it is a global section over $SL/L$ by:
\begin{align*}
s\colon SL/L&\rightarrow \underset{y\in SL/L}{\coprod}M(G/L)_{(y)}\\(A)&\mapsto s((A)),
\end{align*} 
where $s((A))$ is the germ of $s$ in the stalk of $(A)$ with respect to $M(G/L)$. We will be using this idea during the next two lemmas and corollary. 
\begin{lemma}\label{contruct1}
Let $x$ be a closed subgroup of $G$ with $N$ open and normal and $L$ open containing $Nx$. Then every section of the form:
\begin{align*}
s:O(N,Nx)/L\rightarrow\underset{y\in O(N,Nx)/L}{\coprod}M(G/L)_{(y)}
\end{align*}
yields a section of the form 
\begin{align*}
s^{\prime}:O(N,Nx)/Nx\rightarrow\underset{y\in O(N,Nx)/Nx}{\coprod}M(G/Nx)_{(y)},
\end{align*}
such that $s^{\prime}(y)=s(y)$ in $F_y=\underset{J\underset{\text{open}}{\geq}y}{\colim}M(G/J)_{(y)}$.
\end{lemma}
\begin{proof}
Suppose we have a section of the form:
\begin{align*}
s:O(N,Nx)/L\rightarrow\underset{y\in O(N,Nx)/L}{\coprod}M(G/L)_{(y)}
\end{align*}
We then define $s^{\prime}$ as follows:
\begin{align*}
s^{\prime}:O(N,Nx)/Nx&\rightarrow \underset{y\in O(N,Nx)/Nx}{\coprod}M(G/Nx)_{(y)}\\(A)&\mapsto R^L_{Nx}(s)((A)),
\end{align*}
which is a continuous section since $R^L_{Nx}(s)$ is a global section over $SNx/Nx$. We then restrict this to the open subset $O(N,Nx)/Nx$. Here we are using that a section $s$ over $O(N,Nx)/L$ is of the form $e^L_{O(N,Nx)}s$ for $s\in M(G/L)$ by definition. We can also use the proof of Lemma \ref{stalkwelldefined} to apply the restriction map $R^L_{Nx}$, to $e^L_{O(N,Nx)}s$ in order to restrict to a section $e^{Nx}_{O(N,Nx)}R^L_{Nx}(s)$. This clearly satisfies $R^L_{Nx}(s)(y)=s(y)$ in $F_y$.
\end{proof}
In the following lemma, we will apply Construction \ref{pullbacksheaf} to the sheaf $M(G/H)$ over $SH/H$ with respect to the projection map $p\colon SH\rightarrow SH/H$. Namely, we will see how to extend a section $s$ in $M(G/H)$ with domain in $SH/H$ to a section with domain $SH$. 
\begin{lemma}\label{contruct2}
A section of the form: 
\begin{align*}
s:O(N,Nx)/Nx\rightarrow\underset{y\in O(N,Nx)/NK}{\coprod}M(G/Nx)_{(y)}
\end{align*}
yields a section of the form:
\begin{align*}
s^{\prime}:O(N,Nx)\rightarrow \underset{y\in O(N,Nx)}{\coprod}M(G/Nx)_{(y)}
\end{align*}
where the codomain is topologised as in Construction \ref{pullbacksheaf}.
\end{lemma}
\begin{proof}
We define $s^{\prime}$ as follows:
\begin{align*}
s^{\prime}:O(N,Nx)&\rightarrow \underset{y\in O(N,Nx)}{\coprod}M(G/Nx)_{(y)}\\A&\mapsto s((A))
\end{align*}
This is continuous as a result of Proposition \ref{pullbacktop}, by considering the continuous surjection $p$ and the sheaf $M(G/Nx)$ over $SNx/Nx$.
\end{proof}
Combining the first two lemmas we have the following corollary.
\begin{corollary}\label{corsec}
A section of the form:
\begin{align*}
s:O(N,NK)/H\rightarrow \underset{L\in O(N,NK)/H}{\coprod}M(G/H)_{(L)}
\end{align*}
yields a section of the form:
\begin{align*}
s^{\prime}:O(N,NK)\rightarrow \underset{L\in O(N,NK)}{\coprod}M(G/NK)_{(L)}
\end{align*}
such that $s^{\prime}(y)=s(y)$ in $F_y=\underset{J\underset{\text{open}}{\geq}y}{\colim}M(G/J)_{(y)}$ and the codomain is topologised as in Construction \ref{pullbacksheaf}.
\end{corollary}
\begin{proof}
Given $s:O(N,NK)/H\rightarrow \underset{L\in O(N,NK)/H}{\coprod}M(G/H)_{(L)}$, we define $s^{\prime}$ as follows:
\begin{align*}
s^{\prime}:O(N,NK)&\rightarrow \underset{L\in O(N,NK)}{\coprod}M(G/NK)_{(L)}\\
L&\mapsto R^H_{NK}(s)((L))
\end{align*}
as done in Lemmas \ref{contruct1} and \ref{contruct2}, by using Proposition \ref{pullbacktop}.
\end{proof}
\begin{construction}\label{Sheaf_mack}
We construct the underlying set of the sheaf space by setting $E=\underset{K\in SG}{\coprod}F_K$ where: 
\begin{align*}
F_K=\underset{J\geq_O K}{\colim}M(G/J)_{(K)}
\end{align*}
and $J\geq_O K$ means $K$ is a subgroup of $J$ where $J$ is open in $G$. We define the topology by giving the generating sets of neighbourhoods of a point as follows.

Let $\overline{s}_x \in F_x=\underset{J\geq_O x}{\colim}M(G/J)_{(x)}$. Then this has some representative $s_x\in M(G/L)_{(x)}$ and this gives a section:
\begin{align*}
s:O(N,Nx)/L\rightarrow \underset{(y)\in O(N,Nx)/L}{\coprod}M(G/L)_{(y)}.
\end{align*}
for $N$ open and normal in $G$. An application of Corollary \ref{corsec} gives a section:
\begin{align*}
s:O(N,Nx)\rightarrow \underset{y\in O(N,Nx)}{\coprod}M(G/Nx)_{(y)}.
\end{align*}
We define the sets:
\begin{align*}
s(O(N,Nx))=\left\lbrace \overline{s}_y\mid y \in O(N,Nx)\right\rbrace
\end{align*}
to be the generating collection of neighbourhoods of $\overline{s}_x$. Here $\overline{s}_x$ represents the equivalence class of ${s}_x$ in $F_x$.

The sets of the form $s(O(N,Nx))$ deliver a sub-basis for the topology. The projection map required in the definition of $G$-sheaf space sends $\overline{s}_x\in F_x$ to $x$. 
\end{construction}
The above construction is well-defined by Lemma \ref{stalkwelldefined}. The following lemma will refine our understanding of the topology defined in Construction \ref{Sheaf_mack}.
\begin{lemma} \label{basis_sheaf}
The sets of the form $s(O(N,NK))$ form a basis for the topology in Construction \ref{Sheaf_mack}.
\end{lemma}
\begin{proof}
To prove the statement we need to show that if:
\begin{align*}
s_x\in s_1(O(N_1,N_1K_1))\cap s_2(O(N_2,N_2K_2))
\end{align*}
then there is a $t(O(N,NK))$ satisfying:
\begin{align*}
s_x\in t(O(N,NK)) \subseteq s_1(O(N_1,N_1K_1))\cap s_2(O(N_2,N_2K_2)). 
\end{align*}
Given any element
\begin{align*}
s_x\in s_1(O(N_1,N_1K_1))\cap s_2(O(N_2,N_2K_2)),
\end{align*}
it follows that $x\in O(N_1,N_1K_1)\cap O(N_2,N_2K_2)$ and ${s_1}_x={s_2}_x$ in $F_x$. It therefore follows that they are equal as germs in some sheaf $M(G/Nx)$ for some \\$N\leq N_1\cap N_2$. By definition of the stalk of a sheaf, there exists an open set $W$ upon which $e_Ws_1$ and $e_Ws_2$ agree. We can assume $W$ to be of the form $O(\tilde{N},\tilde{N}x)/Nx$ since these sets form a basis of $SNx/Nx$. By taking $s_1=t$ we have that 
\begin{align*}
t(O(\tilde{N},\tilde{N}x))\subseteq s_1(O(N_1,N_1K_1))\cap s_2(O(N_2,N_2K_2)).
\end{align*}
\end{proof}
We now verify that Construction \ref{Sheaf_mack} satisfies the desired conditions.
\begin{lemma}\label{Gspace}
If $M$ is a rational $G$-Mackey funtor then the construction $E$ as given in Construction \ref{Sheaf_mack} is a $G$-set.
\end{lemma}
\begin{proof}
Clearly $E$ is a topological space as we have a set where the topology is generated by a given sub-basis. We begin by showing that $E$ has a $G$-action. If $g\in G$ then we have a map $C_g\colon M(G/J)\rightarrow M(G/gJg^{-1})$ using the conjugacy map for the Mackey functor. This further induces $g\colon M(G/J)_{(K)}\rightarrow M(G/gJg^{-1})_{(gKg^{-1})}$, and using the universal properties of colimits this gives a map $F_K\rightarrow F_{gKg^{-1}}$ for each $K\in SG$. This gives a map $g:E\rightarrow E$. The associativity of this $G$-action follows from the associativity condition of the conjugation maps for Mackey functors. The action is clearly unital since the conjugation map $C_e$ is always the identity.
\end{proof}
\begin{lemma}\label{ctsGact}
The $G$-action of the $G$-set given in Construction \ref{Sheaf_mack} is continuous. 
\end{lemma}
\begin{proof}
We next want to check that the $G$-action constructed is continuous, which involves showing that the map $\psi:G\times E\rightarrow E$ given by $(g,e)\mapsto ge$ is continuous. We need only prove that the preimage of the basis subsets of $E$ are open, so take any $s(O(N,NK))$ for $s\in M(G/NK)$. Take any point $(g,t_y)$ in $\psi^{-1}(s(O(N,NK)))$. By definition it follows that $(gt)_{gyg^{-1}}=s_{gyg^{-1}}$ in $F_{gyg^{-1}}$ and that they are therefore equal as germs in some sheaf $M(G/N_1gyg^{-1})$ for some $N_1\leq N$. Similar to Lemma \ref{basis_sheaf} we can find $N_2\leq N_1$ and representatives such that $(gt)_{|_{O(N_2,N_2gyg^{-1})}}=s_{|_{O(N_2,N_2gyg^{-1})}}$. Therefore we also have $t_{|_{O(N_2,N_2y)}}=(g^{-1}s)_{|_{O(N_2,N_2y)}}$. Observe also that these sections have $N_2$-invariant domains and are obtained from an $N_2$-fixed module $M(G/N_1y)$. We can conclude that the set:
\begin{align*}
W=gN_2\times t(O(N_2,N_2y))
\end{align*}
is contained in $\psi^{-1}(s(O(N,NK)))$. Take any point $(gn,t_z)$ and observe that it has image $(gt)_{gnzn^{-1}g^{-1}}$. Since $gnzn^{-1}g^{-1}$ belongs to $O(N_2,N_2gyg^{-1})$ and $s$ and $(gt)$ are equal on this, it follows that:
\begin{align*}
(gt)_{gnzn^{-1}g^{-1}}=s_{gnzn^{-1}g^{-1}}\in s(O(N_2,N_2gyg^{-1}))\subseteq s(O(N,NK)).
\end{align*}
Therefore the open set $W$ is contained in the preimage as required.
\end{proof}
\begin{lemma}\label{ctsGmap}
The projection map in Construction \ref{Sheaf_mack} is a continuous $G$-map.
\end{lemma}
\begin{proof}
We next need to verify certain properties of the projection map $p\colon E\rightarrow SG$. We first show that $p$ is a $G$-equivariant map. To see this take any $g\in G$ and $s_K\in E$ where $s_K\in F_K$. Then $g(s_K)$ belongs to $F_{gKg^{-1}}$ and hence $p(g(s_K))=gKg^{-1}$. On the other hand $g(p(s_K))=g(K)=gKg^{-1}$ which proves the required equality.

We now prove that $p$ is continuous. Given a sub-basic open subset of $SG$ of the form $O(N,NK)$, we have that $p^{-1}(O(N,NK))=\underset{y\in O(N,NK)}{\coprod}F_y$, so to prove continuity we need to show that $\underset{y\in O(N,NK)}{\coprod}F_y$ is open. We do this by showing it equals its interior and this requires showing that each point is contained in an open subset in $\underset{y\in O(N,NK)}{\coprod}F_y$. Take any $\left[\overline{t}_x\right]$ in $F_x\subseteq \underset{y\in O(N,NK)}{\coprod}F_y$. As seen earlier when constructing the neighbourhood generating sets in Construction \ref{Sheaf_mack}, we can assume $\overline{t}_x$ gives a continuous section $t:A\rightarrow \underset{z\in A}{\coprod}F_z$. We can assume that $A$ is of the form $O(N^{\prime},N^{\prime}x)$ for $N^{\prime}\leq N$ since we can restrict to a basic open subset of this form. It follows by definition that $t(A)\subseteq \underset{z\in A}{\coprod}F_z\subseteq \underset{y\in O(N,NK)}{\coprod}F_z$ is open as required.
\end{proof}
We now prove that the projection map $p$ is a local homeomorphism.
\begin{lemma}\label{localhomeo}
The projection map in Construction \ref{Sheaf_mack} is a local homeomorphism.
\end{lemma}
\begin{proof}
Given any point $\overline{s}_x$ in $F_x$, this has a neighbourhood of the form \\$s(O(N,NK))$ as seen previously. Let $f=p_{|s(O(N,NK))}$, we will prove that 
\begin{align*}
f:s(O(N,NK))\rightarrow O(N,NK)
\end{align*}
is a homeomorphism, and since it is clearly bijective and continuous we need only show that it is open.

Take any open set of the form $\left(\underset{\alpha}{\bigcup} U_{\alpha}\right)\cap s(O(N,NK))$ where each $U_{\alpha}=s^{\alpha}(A_{\alpha})$, which is a typical open subset of the domain of $f$ under the subspace topology. Since $f$ is bijective we have that:
\begin{align*}
f\left(\left(\underset{\alpha}{\bigcup} U_{\alpha}\right)\bigcap s(O(N,NK))\right)=f\left(\underset{\alpha}{\bigcup} U_{\alpha}\right)\bigcap f\left(s(O(N,NK))\right), 
\end{align*}
so we need only show that $f\left(\underset{\alpha}{\bigcup} U_{\alpha}\right)\cap O(N,NK)$ is open. This equals:
\begin{align*}
\left(\underset{\alpha}{\bigcup} f\left(U_{\alpha}\right)\right)\cap O(N,NK)=\left(\underset{\alpha}{\bigcup}A_{\alpha}\right)\,\cap O(N,NK)
\end{align*}
which is open in $O(N,NK)$ as required.
\end{proof}
\begin{lemma}\label{Qmod}
If $U$ is any open subset of $SG$, then the set of sections over $U$ into the space $E$ from Construction \ref{Sheaf_mack} is a $\mathbb{Q}$-module.
\end{lemma}
\begin{proof}
We will show that there is a $\mathbb{Q}$-module structure on the sections over the basic open subsets, then this will be true for more general subsets since the set of sections over a more general $U$ are a limit of the sections over the basic subsets contained in $U$. Suppose we have two sections $\overline{s}$ and $\overline{t}$ over some $O(N,NK)$.

Since $NK$ is maximal in $O(N,NK)$ and we know that representatives $s$ and $t$ each belong to some $M(G/J)$ for $J$ containing $NK$, we can restrict to have them both belong to $M(G/NK)$. Therefore we use additivity of sections over the sheaf $M(G/NK)$ to define $\overline{s+t}$.
\end{proof}
\begin{lemma}\label{Weyl}
The $G$-space given in Construction \ref{Sheaf_mack} satisfies that $p^{-1}(K)$ is $K$-fixed for each $K \in SG$.
\end{lemma}
\begin{proof}
If $K\in SG$, then $p^{-1}(K)=F_K$. Notice that $F_K=\underset{J}{\colim} M(G/J)_{(K)}$ where $J$ ranges over all of the open subgroups containing $K$. Take any $\left[s_K\right]\in F_K$ which has representative $s_K$ in some $M(G/J)_{(K)}$. This germ is fixed by $K$ since $M(G/J)$ is $J$-fixed, $K\leq J$ and the neighbourhood basis for $K$ is $K$-invariant.
\end{proof}
\begin{proposition}\label{G_sheaf}
Every Mackey functor $M$ over a profinite group $G$ determines a $\text{Weyl-}G$-sheaf\index{$\text{Weyl-}G$-sheaf} space. 
\end{proposition}
\begin{proof}
The $G$-space given in Construction \ref{Sheaf_mack} is a $G$-sheaf space, and it follows from Lemmas \ref{Gspace}, \ref{ctsGact}, \ref{ctsGmap}, \ref{localhomeo}, \ref{Qmod} and \ref{Weyl} that it is a Weyl-$G$-sheaf space.
\end{proof}
\begin{theorem}\label{macktosheaf}
There is a functor:
\begin{align*}
\weyl:\mackey_{\mathbb{Q}}(G)\rightarrow \Weyl_{\mathbb{Q}}(SG)
\end{align*}
which sends a Mackey functor $M$ over $G$ to a $\text{Weyl-}G$-sheaf denoted $\weyl(M)$, where $\weyl(M)$ is defined using Proposition \ref{G_sheaf}. We define the sheaf at each open neighbourhood $U$ to be the collection of continuous sections from $U$ to $E$, where the explicit construction for $E$ is given in Construction \ref{Sheaf_mack}.  
\end{theorem}
\begin{proof}
This will follow immediately from Proposition \ref{G_sheaf} when we define what this does on morphisms of Mackey functors. Let $f:M\rightarrow \overline{M}$ be a morphism of Mackey functors. For each open subgroup $NK$, where $K$ is a closed subgroup and $N$ open and normal, we have a morphism of $\mathbb{Q}$-modules, \\$f_{NK}:M(G/NK)\rightarrow \overline{M}(G/NK)$. Let $s\in M(G/NK)$, then:
\begin{align*}
f_{NK}(e_{O(N,NK)}(s))&=f_{NK}\left(\underset{NA\leq NK}{\sum}q_AI^{NK}_{NA}R^{NK}_{NA}(s)\right)\\
&=\underset{NA\leq NK}{\sum}q_AI^{NK}_{NA}R^{NK}_{NA}f_{NK}(s)=e_{O(N,NK)}f_{NK}(s).
\end{align*}
Therefore $f$ induces a map of stalks:
\begin{align*}
f_K:\weyl(M)_K\rightarrow \weyl(\overline{M})_K.
\end{align*}
This shows that we have a map of sets which are compatible with the local homeomorphisms. To show that this is a morphism of sheaf spaces it is sufficient to prove that it is continuous by \cite[Lemma 2.3.5]{Tennison}.

Let $s(O(N,NK))$ be a typical open subset of $\weyl(\overline{M})$ and $t_y$ an element in the preimage of $s(O(N,NK))$ with respect to $\weyl(f)$. It follows that $f_y(t_y)=s_y$ in $\weyl(\overline{M})_y$ and $y\in O(N,NK)$. Similar to Lemma \ref{basis_sheaf}, this implies that there exists some open normal subgroup $N_1$ and open subset $V$ of $O(N,NK)$ such that $e_VR^{NK}_{N_1K}(f(t))=e_VR^{NK}_{N_1K}(s)$. This proves that $f(t)(V)=s(V)\subseteq s(O(N,NK))$ in $\weyl(\overline{M})$. In particular, $t(V)$ is contained in the preimage as required.

For functoriality of this construction, if $f$ is the identity morphism of Mackey functors, then the induced map of stalks will be the identity also. Therefore $\weyl$ maps the identity to the identity. The transitivity property follows from the identity $(h\circ f)_x=h_x\circ f_x$. 
\end{proof}

\chapter{Equivalence}
In this chapter we shall focus on proving that the two functors $\text{Mackey}(-)$ and $\text{Weyl}(-)$ are inverse equivalences (Theorem \ref{equivalencemain}). We will then look at some useful consequences of this equivalence at the end of the chapter. These consequences include allowing us to calculate products of Weyl-$G$-sheaves (Construction \ref{prodweyl}), and defining an adjunction $(L,R)$ between Weyl-$G$-sheaves and $G$-sheaves of $\mathbb{Q}$-modules (Proposition \ref{Weyladjunct}). 
\section{Equivalence of Categories}
In order to prove that the two functors are inverse equivalences we will begin by looking at how we can interpret the action of the Burnside ring on Mackey functors in a sheaf theoretic language.
\begin{lemma}\label{burnsheafrest}
Let $F$ be $G$-sheaf of $\mathbb{Q}$-modules, then the Mackey functor \\$\mackey(F)$ satisfies that $\left(\left[G/NA\right](s)\right)(x)=\left[G/NA\right](x)s(x)$.
\end{lemma}
\begin{proof}
First note that by definition $\left[G/NA\right](s)=\underset{gNA\in G/NA}{\sum}C_g\overline{s_{|_{S(NA)}}}$. Next note that $C_g\overline{s_{|_{S(NA)}}}(x)$ is defined to be $g\overline{s_{|_{S(NA)}}}(g^{-1}xg)$ which is zero for $x$ outside of $S(gNAg^{-1})$. Next notice that $C_g\overline{s_{|_{S(NA)}}}(x)=g\overline{s_{|_{S(NA)}}}(g^{-1}xg)=gs(g^{-1}xg)=s(x)$ since $s$ started out $G$-equivariant, as $s$ belongs to $M(G/G)$. Therefore the above boils down to $\underset{gNA\in G/NA}{\sum}\overline{s_{|_{S(gNAg^{-1})}}}$.

Now each summand of $\underset{gNA\in G/NA}{\sum}\overline{s_{|_{S(gNAg^{-1})}}}(x)$ is equal to $s(x)$ if and only if $x\in S(gNAg^{-1})$, otherwise it is zero. But $x\in S(gNAg^{-1})$ if and only if $g^{-1}xg\leq NA$. This in turn happens if and only if $gNA\in \left(G/NA\right)^x$. Hence the sum becomes $\underset{gNA\in\left(G/NA\right)^x}{\sum}\overline{s_{|_{S(gNAg^{-1})}}}(x)$ which equals $|\left(G/NA\right)^x|s(x)$. This by definition is $\left[G/NA\right](x)s(x)$.
\end{proof}
\begin{proposition}
If $F$ is a $G$-sheaf then $e_U\mackey(F)(G/H)$ is equal to the set $\left\lbrace \overline{s_{|_{U}}}\mid s \in \mackey(F)(G/H)\right\rbrace$.
\end{proposition}  
\begin{proof}
By definition of $\text{Mackey}(F)$ as a Mackey functor, we have:
\begin{align*}
\text{Mackey}(F)(G/H)=F(SH)^H.
\end{align*}
Since the sets of the form $O(N,NK)$ form a basis for $SH$ where $N$ is open and normal in $H$ we can assume that $U=O(N,NK)$. We then have:
\begin{align*}
e_{O(N,NK)}s(x)&=\underset{NA\leq NK}{\sum}q_{NA}\left(\left[G/NA\right](s)\right)(x)\\&=\left(\underset{NA\leq NK}{\sum}q_{NA}\left[G/NA\right](x)\right)s(x)
\end{align*}
where the first equality comes from Proposition \ref{idempotentform} and the second from Lemma \ref{burnsheafrest}. This is either $s(x)$ if $x\in O(N,NK)$ or $0$ otherwise proving that $e_{O(N,NK)}s=\overline{s_{|_{O(N,NK)}}}$ as required.
\end{proof}
\begin{proposition}\label{stalkeq}
We have the following equality of colimits:
\begin{align*}
\underset{H}{\colim}\,\underset{U\backepsilon K}{\colim}\, e_U F(SH)^H=\underset{H\geq K}{\colim}\,\underset{U\backepsilon K}{\colim}\,F(q_H^{-1}(U))^H,
\end{align*}
where $H$ ranges over the open subgroups containing closed subgroup $K$ of $G$ and $U$ is a basic open neighbourhood of the conjugacy class of $K$ in $SH/H$.
\end{proposition}
\begin{proof}
Let the left hand side equal $A$ and the right hand side equal $B$. Take a representative in $A$ of the form $e_{O(N,NK)}s$ for $s\in F(SH)^H$. Then since $O(N,NK)$ is $NK$-invariant, $NK\leq H$ and since $e_{O(N,NK)}s=\overline{s_{|_{O(N,NK)}}}$, we can restrict to $S(NK)$ and this representative gives a class in $B$. This is $NK$-equivariant since $s$ was $H$-equivariant by assumption.

On the other hand if we take a representative section $s$ over some $H$-invariant subset $U\subseteq SH$, then extending by zero gives a $H$-equivariant section over $SH$. The extension by zero is continuous since $U$ is a closed-open subset. Therefore the class given by the representative $e_U\overline{s}$ gives an element of $A$.
\end{proof}
\begin{lemma}\label{stalk}
If $F$ is a $\text{Weyl-}G$-Sheaf over $SG$, then for each $K\in SG$ we have $F_K\cong \weyl\circ\mackey(F)_K$.
\end{lemma}
\begin{proof}
Let $\overline{F_K}=\underset{H\geq K}{\colim}\,\underset{U\backepsilon K}{\colim}\,e_UF(SH)^H$ where is $U$ open in $SH/H$ and $H$ is open. Then by Proposition \ref{stalkeq} this equals 
\begin{center}
$\underset{H\geq K}{\colim}\,\underset{U\backepsilon K}{\colim}\,F(q_H^{-1}(U))^H$,
\end{center}
we now show that this is isomorphic to $\underset{U\backepsilon K}{\colim}F(U)$. We define a map of abelian groups:
\begin{center}
$\theta:F_K\rightarrow \overline{F_K}$ given by $\left[s\right]\mapsto \left[s|_{U_J}\right]$
\end{center}
where $K$ is a subgroup of an open subgroup $J$, $s|_{U_J}$ is a $J$-equivariant variant section that $s$ restricts to from Proposition \ref{Weylequi}.

We prove this is well defined. If $\left[s\right]=\left[t\right]$ then $s=t:U\rightarrow LF$ for some $U$. If $\theta\left[s\right]=\left[s|_{U_J}\right]$ and $\theta\left[t\right]=\left[t|_{V_L}\right]$, then since $s=t$ on the larger domain $U$ we must have that $s|_{U_J}=t|_{U_J}$ and $s|_{V_L}=t|_{V_L}$. It follows that we need only show that $\left[t|_{U_J}\right]=\left[t|_{V_L}\right]$, but this is true since they have common descendent $\left[t|_{{U\cap V}_{J\cap L}}\right]$. 

This is surjective since any element $\left[s\right]$ in the codomain has itself in the preimage. For injectivity, if we have $\left[s\right],\left[t\right]$ with $\theta\left[s\right]=\theta\left[t\right]$, then $\left[s|_{U_J}\right]=\left[t|_{V_L}\right]$. By definition of equality of equivalence classes there is an open subgroup $H$ and a $H$-invariant set $W_H$ such that $s$ and $t$ have the same restriction to $W_H$ and are $H$-equivariant. This then means that $s$ and $t$ belong to the same equivalence class of $F_K$.
\end{proof}
Notice that this result requires $F$ to be a Weyl-$G$-sheaf. This is because in this case we know that each germ at $K$ is represented by a $J$-equivariant section where $J$ is an open subgroup containing $K$. For more general $G$-sheaves we don't necessarily know that $J$ contains $K$.  
\begin{lemma}\label{equivequi}
If $K\in SG$ and $g\in G$ arbitrary, then the following square commutes:
\begin{align*}
\xymatrix{F_K\ar[r]^{\theta_K}\ar[d]^{C_g}&\overline{F}_K\ar[d]^{C_g}\\F_{gKg^{-1}}\ar[r]^{\theta_{gKg^{-1}}}&\overline{F}_{gKg^{-1}}}
\end{align*}
where $\theta_K$ is the map defined in the proof of Lemma \ref{stalk}.
\end{lemma}
\begin{proof}
If $s_K\in F_K$ then this maps to a corresponding $\left[{s_|}_{U_J}\right]$ where $J$ and $U_J$ exist as in Proposition \ref{Weylequi}. We first show that if $t={s_|}_{U_J}$, then $g*t$ is a $gJg^{-1}$ equivariant section over $gU_Jg^{-1}$. Clearly $gKg^{-1}\leq gJg^{-1}$, and we now show that $gU_Jg^{-1}$ is $gJg^{-1}$-invariant. Take $gjg^{-1}\in gJg^{-1}$ and $gug^{-1}\in gU_Jg^{-1}$. Then:
\begin{align*}
(gjg^{-1})(gug^{-1})=gjug^{-1}\in gU_Jg^{-1}
\end{align*}
using that $U_J$ is $J$-invariant. Next if $gjg^{-1}\in gJg^{-1}$ and $gug^{-1}\in gU_Jg^{-1}$ then we have:
\begin{align*}
gjg^{-1}(g*s)(gyg^{-1})&=gjg^{-1}gs(y)=gjs(y)=gs(jyj^{-1})\\&=gs(g^{-1}gjg^{-1}gyg^{-1}gj^{-1}g^{-1}g)=(g*s)(gjg^{-1}*gyg^{-1})
\end{align*}
as required. Set $t={(g*s)_|}_{gU_Jg^{-1}}$, $J_1=gJg^{-1}$ and $V_{J_1}=gU_Jg^{-1}$. Let $J_2$ be the open subgroup for $C_g(s_K)$ with subset $V_{J_{V_2}}$ so that ${(g*s)_|}_{V_{J_2}}$ is $J_2$-equivariant. Then $V_{J_1}\cap V_{J_2}$ is $J_1\cap J_2$-invariant and ${(g*s)_|}_{V_{J_1}\cap V_{J_2}}$ is $J_1\cap J_2$-equivariant, so in either direction the sections restrict to the same equivariant sections. 
\end{proof}
\begin{theorem}\label{weylmack}
If $F$ is any \index{$\text{Weyl-}G$-sheaf} Weyl-$G$-sheaf of $\mathbb{Q}$-modules then
\begin{align*}
\weyl\,\circ\mackey(F)\cong F
\end{align*}
in the category of $\text{Weyl-}G$-Sheaves.
\end{theorem}
\begin{proof}
We will prove the result by showing that $LF$ is isomorphic to the space $L\text{Weyl}\,\circ\text{Mackey}(F)$ as $\text{Weyl-}G$-Sheaf spaces. They are equal as sets, by Lemma \ref{stalk}.

We now need to prove that they are topologically equal, so we let $\overline{LF}$ denote the sheaf space with the topology given by applying $\text{Weyl}\circ\text{Mackey}(F)$. By \cite[Lemma 2.3.5]{Tennison} we know that the set theoretic identity map from $LF$ to $\overline{LF}$ is continuous if and only if open, so since it is already bijective we need only prove that it is open. Take any $U$ open in $LF$. If $a\in U$, then by definition of local homeomorphism, there exists a neighbourhood $V_a\subseteq U$ such that $p|_{V_a}$ is a homeomorphism with inverse $s_a$ which has codomain $W_a=p|_{V_a}(V_a)$. Then $U$ is a union of open sets of the form $s_a(W_a)$, so to prove the result we need only show that the sets $s(W)$ are open in $\overline{LF}$.

Take any $s_x\in s(W)$ for $x\in W$. Then $s_x\in F_x=\underset{H\geq K}{\colim}\,\underset{U\backepsilon K}{\colim}\,e_UF(SH)^H$, where $H$ and $U$ range over the open subgroups and subsets respectively. This germ has a representative $s^{\prime}:W_x\rightarrow \overline{LF}$ which is continuous and coincides with the restriction of $s:W\rightarrow LF$ to $W_x$ since they originate from the same germ. Therefore $s(W_x)$ is open in $\overline{LF}$ and $s(W)$, therefore $s(W)$ is a union of these sets which are open in $\overline{LF}$ as required. This is a $G$-equivariant homeomorphism by Lemma \ref{equivequi}.
\end{proof}
We now work towards proving that there is an equivalence in the other direction, namely that $\text{Mackey}\circ \text{Weyl}(M)\cong M$. We begin by proving some necessary lemmas.
\begin{lemma}\label{Normrestrict}
Let $G$ be a profinite group and $M$ a Mackey functor over $G$. Then if $H\leq K$ are open subgroups of $G$ then:
\begin{align*}
R^{N_K(H)}_H\colon M(G/N_K(H))\rightarrow M(G/H)
\end{align*}
is surjective onto $M(G/H)^{N_K(H)}$. 
\end{lemma}
\begin{proof}
Let $m\in M(G/H)^{N_K(H)}$, then we will show that:
\begin{align*}
R^{N_K(H)}_H(qI^{N_K(H)}_H(m))=m
\end{align*}
for some $q\in\mathbb{Q}$. We begin by observing the following equalities:
\begin{align*}
R^{N_K(H)}_H(I^{N_K(H)}_H(m))&=\underset{H\diagdown N_K(H)\diagup H}{\sum}I^H_{H\cap xHx^{-1}}C_xR^H_{H\cap x^{-1}Hx}(m)\\&=\underset{H\diagdown N_K(H)\diagup H}{\sum}I^H_{H\cap xHx^{-1}}R^H_{H\cap xHx^{-1}}C_x(m)\\&=\underset{H\diagdown N_K(H)\diagup H}{\sum}I^H_{H}R^H_{H}(m)=\underset{H\diagdown N_K(H)\diagup H}{\sum}m\\&=|H\diagdown N_K(H)\diagup H|m.
\end{align*}
We therefore set $q=\frac{1}{|H\diagdown N_K(H)\diagup H|}$ and use the fact that restriction, induction and conjugation maps are $\mathbb{Q}$-equivariant.
\end{proof}
We now state a proposition for finite groups which will be proven in a more general profinite context in Proposition \ref{profinitefixed}.
\begin{proposition}\label{fixcommte}
If $G$ is a finite group and $K\leq H\leq G$ then:
\begin{enumerate}
\item $e_K\left(M(G/K)^{N_H(K)}\right)$ is $N_H(K)$-fixed.
\item $e_K\left(M(G/K)^{N_H(K)}\right)=\left(e_KM(G/K)\right)^{N_H(K)}$
\end{enumerate}
\end{proposition}
We now give a lemma which holds in the case where $G$ is finite. It will be used to prove the analogous result for the profinite case in Lemma \ref{fixinflate}. 
\begin{lemma}\label{Fixptfin}
Let $M$ be a Mackey functor on a finite discrete group $G$ and $K\leq H\leq G$. Then: 
\begin{align*}
e_KM(G/H)\cong e_K\left(M(G/K)^{N_H(K)}\right)
\end{align*}
\end{lemma}
\begin{proof}
We prove that $e_KR^H_K$ gives an isomorphism onto $e_K\left(M(G/K)^{N_H(K)}\right)$. Starting with surjectivity, suppose $e_Km\in e_K\left(M(G/K)^{N_H(K)}\right)$. Begin by choosing $a\in M(G/N_H(K))$ with $R^{N_H(K)}_K(a)=m$ which we can do by Lemma \ref{Normrestrict}. Then we have:
\begin{align*}
R^H_K(I^H_{N_H(K)}(e_Ka))=\underset{K\diagdown H\diagup N_H(K)}{\sum}I^K_{K\cap xN_H(K)x^{-1}}C_xR^{N_H(K)}_{N_H(K)\cap x^{-1}Kx}(e_Ka).
\end{align*} 
We then apply Corollary \ref{burn0} to $R^{N_H(K)}_{N_H(K)\bigcap x^{-1}Kx}(e_Ka)$, in order to deduce that the only non-zero summand is given by the class $KeN_H(K)$. Therefore the above sum boils down to the following:
\begin{align*}
R^{N_H(K)}_K(e_Ka)=\underset{\mathfrak{I}^K_K}{\sum} e_KR^{N_H(K)}_K(a)=e_KR^{N_H(K)}_K(a)=e_Km,
\end{align*}
where the indexing set $\mathfrak{I}^K_K$ is given by:
\begin{align*}
\left\lbrace gKg^{-1}\mid g\in \left[K\diagdown N_H(K)\diagup K\right]\,\text{and}\, gKg^{-1}\leq K\right\rbrace.
\end{align*}
Here we use the fact that the indexing set $\mathfrak{I}^K_K$ is equal to $\left\lbrace K\right\rbrace$ as seen in Corollary \ref{burncommute}. For injectivity if $s\in \ker R^H_K$ then:
\begin{align*}
e_Ks&=\underset{D\leq K}{\sum}\frac{|K|}{|N_H(K)|}\mu(D,K)\left[H/D\right](s)\\&=\underset{D\leq K}{\sum}\frac{|K|}{|N_H(K)|}\mu(D,K)I^H_DR^H_D(s)\\&=\underset{D\leq K}{\sum}\frac{|K|}{|N_H(K)|}\mu(D,K)I^H_DR^K_DR^H_K(s)=0.
\end{align*} 
This proves that $e_K\ker R^H_K$ is trivial.
\end{proof}
We next prove a useful group theoretic result which holds for any group $G$. However we shall see that this is helpful when $G$ is profinite since a neighbourhood basis for the identity is given by open normal subgroups. 
\begin{lemma}\label{normaliseinclude}
If $N_1,N_2\unlhd G$ are open with $N_1\leq N_2$ and $K\leq G$, then:
\begin{align*}
N_G(N_1K)\leq N_G(N_2K). 
\end{align*}
\end{lemma}
\begin{proof}
If $g \in N_G(N_1K)$ then for any $k\in K$ we have that:
\begin{align*}
gkg^{-1}\in N_1gKg^{-1}=N_1K\leq N_2K
\end{align*}
so $gKg^{-1}\leq N_2K$ and hence $N_2gKg^{-1}=N_2K$. It follows that $g\in N_G(N_2K)$.
\end{proof}
\begin{lemma}\label{colimitfixed}
Let $G$ be a profinite group. For $K,H\leq G$, $H$ open and $K$ closed the colimit 
\begin{align*}
\underset{N\unlhd G}{\colim}\left(M(G/NK)^{N_H(NK)}\right)
\end{align*}
over the diagram of open normal subgroups $N$ of $H$ exists. 
\end{lemma}
\begin{proof}
We need to prove that if $N_1\leq N_2$ are open and normal in $H$ then:
\begin{align*}
A=\im {R^{N_2K}_{N_1K}}_{|_{M(G/N_2K)^{N_H(N_2K)}}}\subseteq M(G/N_1K)^{N_H(N_1K)}.
\end{align*}
Let $a\in A$ then $a=R^{N_2K}_{N_1K}(b)$ for $b\in M(G/N_2K)^{N_H(N_2K)}$. By Lemma \ref{normaliseinclude} we have $N_H(N_1K)\leq N_H(N_2K)$, so if $g\in N_H(N_1K)$ we can use the equivariance condition as follows:
\begin{align*}
ga=gR^{N_2K}_{N_1K}(b)=R^{N_2K}_{N_1K}(gb)=R^{N_2K}_{N_1K}(b)=a
\end{align*}
as required.
\end{proof}
We now prove a lemma aimed at establishing Proposition \ref{fixcommte} in the profinite case more generally.
\begin{lemma}\label{fixcolimitlemma}
Let $G$ be a profinite group, $H$ an open subgroup of $G$, $N,\overline{N}$ open normal subgroups of $G$ such that $\overline{N}\unlhd N\unlhd H$ and $K$ a closed subgroup of $H$. Then $e_{O(\overline{N},\overline{N}NK)}\left(M(G/NK)^{N_H(NK)}\right)$ is $N_H(NK)$-fixed.
\end{lemma}
\begin{proof}
Let $h\in N_H(NK)$ and $e_{O(\overline{N},\overline{N}NK)}s\in e_{O(\overline{N},\overline{N}NK)}M(G/NK)^{N_H(NK)}$. Bearing in mind that $NK=\overline{N}NK$ we have:
\begin{align*}
he_{O(\overline{N},\overline{N}NK)}(s)&=h\left(\underset{\overline{N}A\leq NK}{\sum}q_{\overline{N}A}I^{NK}_{\overline{N}A}R^{NK}_{\overline{N}A}(s)\right)\\&=\underset{\overline{N}A\leq NK}{\sum}q_{\overline{N}A}I^{NK}_{h\overline{N}Ah^{-1}}R^{NK}_{h\overline{N}Ah^{-1}}(s)
\end{align*}
For the last term we use Proposition \ref{rationcoeff} to deduce that this is the same sum which determines $e_{O(\overline{N},\overline{N}NK)}s$. 
\end{proof}
This in particular allows us to consider the following $\mathbb{Q}$-module:
\begin{align*}
\left(M(G/NK)^{N_H(NK)}\right)_{(NK)}=\underset{\overline{N}}{\colim}\,e_{O(\overline{N},NK)}\left(M(G/NK)^{N_H(NK)}\right).
\end{align*}
We now verify that an analogous colimit exists when $K$ is closed but not open.
\begin{proposition}
The following colimit exists: 
\begin{align*}
\tilde{F}_K=\underset{N}{\colim}\,e_{O(N,NK)}\left(M(G/NK)^{N_H(NK)}\right)
\end{align*}
as $N$ ranges over the open normal subgroups of $G$ contained in $H$ exists. Let $N_1\leq N_2$ be open and normal in $G$. Then the image of $N_H(N_2K)$-fixed section $e_{O(N_2,N_2K)}s$ under the morphisms of the colimit diagram is $N_H(N_1K)$-fixed.
\end{proposition}
\begin{proof}
Let $N_1\leq N_2$. The morphisms of the diagram we are trying to take the colimit of can be seen in the following diagram:
\begin{align*}
\xymatrix{e_{O(N_2,N_2K)}M(G/N_2K)^{N_H(N_2K)}\ar[d]\ar@{-->}[dr]\\e_{O(N_1,N_1K)}M(G/N_2K)^{N_H(N_1K)}\ar[r]_(.45){R^{N_2K}_{N_1K}}&e_{O(N_1,N_1K)}M(G/N_1K)^{N_H(N_1K)}}
\end{align*}
where the vertical arrow is given by inclusion into $e_{O(N_2,N_2K)}M(G/N_2K)^{N_H(N_1K)}$ followed by multiplication by $e_{O(N_1,N_1K)}$. This inclusion is necessary since a section with support $O(N_1,N_1K)$ can only be $N_H(N_1K)$-equivariant as the set is only $N_H(N_1K)$-invariant. The idempotent $e_{O(N_1,N_1K)}$ therefore doesn't act on $N_H(N_2K)$-fixed points as in Lemma \ref{fixcolimitlemma}. However, $e_{O(N_2,N_2K)}$ can act on $N_H(N_1K)$-fixed points since this support is $N_H(N_1K)$-invariant. The proof uses Lemma \ref{normaliseinclude} to deduce that a similar argument to that of Lemma \ref{fixcolimitlemma} is valid.

We now verify that these maps are compatible. If we take $e_{O(N_2,N_2K)}s$ in $e_{O(N_2,N_2K)}M(G/N_2K)^{N_H(N_2K)}$, then using Lemma \ref{fixcolimitlemma} we know that $e_{O(N_2,N_2K)}s$ is $N_H(N_2K)$-fixed. Since $O(N_1,N_1K)$ is $N_H(N_1K)$-invariant, Lemma \ref{normaliseinclude} implies that $e_{O(N_1,N_1K)}s$ is $N_H(N_1K)$-fixed. A similar argument to Lemma \ref{colimitfixed} shows that that the image of $e_{O(N_1,N_1K)}s$ under $R^{N_2K}_{N_1K}$ is $N_H(N_1K)$-fixed. This has image given by:
\begin{align*}
e_{O(N_1,N_1K)}R^{N_2K}_{N_1K}(s)
\end{align*}
as seen in Corollary \ref{burncommute}, where we can project onto $e_{O(N,NK)}$ eliminating $e_V$ as explained in the paragraph after the corollary. If $h\in N_H(N_1K)$ then we have:
\begin{align*}
he_{O(N_1,N_1K)}R^{N_2K}_{N_1K}(s)=e_{O(N_1,N_1K)}R^{N_2K}_{N_1K}(s).
\end{align*}
Therefore we have a well defined colimit diagram since this holds for any $h\in N_H(N_1K)$ and the result follows.
\end{proof}
\begin{remark}\label{opencasestalk}
Notice that in the case where $K$ is open:
\begin{align*}
\tilde{F}_K=\left(M(G/K)^{N_H(K)}\right)_{(K)}.
\end{align*}
This is because $K$ is the cofinal term in the colimit diagram. To see this take any term $NK$ for $N$ open and normal in $G$ and contained in $H$. Then $N\cap K^C$ is an open normal subgroup inside $H$ with $NK$ mapping down to the term $\left(N\cap K^C\right)K=K$.
\end{remark}
\begin{proposition}\label{profinitefixed}
If $G$ is a profinite group then the following holds:
\begin{enumerate}
\item $\left(M(G/NK)^{N_H(NK)}\right)_{(NK)}$ is $N_H(NK)$-fixed.
\item $\left(M(G/NK)^{N_H(NK)}\right)_{(NK)}=\left(M(G/NK)_{(NK)}\right)^{N_H(NK)}$.
\end{enumerate}
\end{proposition}
\begin{proof}
The first is immediate from Lemma \ref{fixcolimitlemma}. For the second we begin by proving that the left hand side is contained in the right hand side. Take $s_{(NK)}$ in $\left(M(G/NK)^{N_H(NK)}\right)_{(NK)}$ which has representative $e_{O(\overline{N},NK)}s$ for:
\begin{align*}
s\in M(G/NK)^{N_H(NK)}. 
\end{align*}
We know from Lemma \ref{fixcolimitlemma} that $e_{O(\overline{N},NK)}s$ is $N_H(NK)$-fixed, since $O(\overline{N},NK)$ is $N_H(NK)$-invariant. We therefore know that $s_{(NK)}$ belongs to:
\begin{align*}
\left(M(G/NK)_{(NK)}\right)^{N_H(NK)}.
\end{align*}
To see that the right hand side is contained in the left take a representative $s_{(NK)}\in \left(M(G/NK)_{(NK)}\right)^{N_H(NK)}$. For each $g\in W_H(NK)$ there exists an $N_g$ such that $\left(g*s\right)_{|_{O(N_g,N_gNK)}}=s_{|_{O(N_g,N_gNK)}}$ where $s$ is a representative of $s_{(NK)}$.

Let $N^{\prime}=\underset{g\in W_H(NK)/NK}{\bigcap}N_g$ which is a finite intersection and observe that $O(N^{\prime},N^{\prime}NK)$ is $W_H(NK)$-invariant. By construction we also have: 
\begin{align*}
\left(g*s\right)_{|_{O(N^{\prime},N^{\prime}NK)}}=s_{|_{O(N^{\prime},N^{\prime}NK)}}
\end{align*}
for each $g\in W_H(NK)$. This tells us that $s_{(NK)}$ belongs to:
\begin{align*}
\left(M(G/NK)^{N_H(NK)}\right)_{(NK)}.
\end{align*}
\end{proof}
We seek to extend the fixed point relation given in the finite case to the profinite case and we begin by proving a basic group theoretic lemma.
\begin{lemma}\label{normalisequot}
If $G$ is a profinite group, $H$ an open subgroup in $G$ containing open normal subgroup $N$ of $G$ then $N_{G/N}(H/N)\cong N_G(H)/N$. 
\end{lemma}
\begin{proof}
Firstly if $gN\in N_{G/N}(H/N)$ then $gN\left(H/N\right)g^{-1}N=H/N$ and hence $gHg^{-1}/N=H/N$. For any $h\in H$, we know that $ghg^{-1}N\in gHg^{-1}/N=H/N$ and so there exists an $h^{\prime}N\in H/N$ with $ghg^{-1}N=h^{\prime}N$. This in particular means that there is an $n\in N$ with $ghg^{-1}=h^{\prime}n$. This belongs to $H$ since $N\subseteq H$, and so $g\in N_G(H)$ and $gN\in N_G(H)/N$. On the other hand if $gN\in N_G(H)/N$ then:
\begin{align*}
H/N=gHg^{-1}/N=gN\left(H/N\right)g^{-1}N,
\end{align*}
and so $gN\in N_{G/N}(H/N)$.
\end{proof}
We now extend Lemma \ref{Fixptfin} to the profinite case using Construction \ref{Mackquot}.
\begin{lemma}\label{fixinflate}
For $N$ an open normal subgroup of $H$ and $K\leq H$ both open, we have that $e_{O(N,K)}M(G/H)$ is isomorphic to $e_{O(N,K)}\left(M(G/K)^{N_H(K)}\right)$.
\end{lemma}
\begin{proof}
We begin by using Lemma \ref{Fixptfin} on the Mackey functor $\overline{M}$ over $G/N$ from Construction \ref{Mackquot} to deduce that $e_{K/N}\left(\overline{M}\left(\left(G/N\right)\left(K/N\right)\right)^{N_{H/N}(K/N)}\right)$ is isomorphic to $e_{K/N}\overline{M}\left(\left(G/N\right)\left(H/N\right)\right)$ via $e_{K/N}R^{H/N}_{K/N}$. By definition of Construction \ref{Mackquot} and Lemma \ref{normalisequot} this boils down to:
\begin{align*}
e_{K/N}\left(\overline{M}\left(\left(G/N\right)\left(K/N\right)\right)^{N_{H/N}(K/N)}\right)&=e_{O(N,K)}\left(M(G/K)^{N_H(K)/N}\right)\\&=e_{O(N,K)}\left(M(G/K)^{N_H(K)}\right)
\end{align*}
since $N\subseteq K$ implies that the $N$-action on $M(G/K)$ is trivial. This is then isomorphic to
\begin{align*}
e_{K/N}\overline{M}\left(\left(G/N\right)\left(H/N\right)\right)=e_{O(N,K)}M(G/H)
\end{align*}
as required. We finally note that $e_{K/N}R^{H/N}_{K/N}=e_{O(N,K)}R^H_K$.
\end{proof}
For the following proposition we recall that $M(G/H)_{(K)}$ is the stalk of \\$M(G/H)$ at $(K)$. We now prove a result which is analogous to Lemma \ref{Fixptfin}.
\begin{remark}
For a closed subgroup $K$ we can write $\underset{J\geq K}{\colim} M(G/J)_{(K)}$ whose colimit is taken over every open subgroup $J$ containing $K$ as
\begin{align*}
\underset{N\unlhd G}{\colim} M(G/NK)_{(K)}
\end{align*}
whose colimit is taken over every open normal subgroup $N$ of $G$. To see this, observe that each $NK$ is equal to some $J$ in the first diagram, but for each $J$ open and containing $K$ we know that the term $J$ maps down to the term $J^CK$ in the first colimit diagram so we can apply a cofinality argument. 
\end{remark}

\begin{remark}\label{include}
Since $e_{O(N,NK)}M(G/NK)^{N_H(NK)}$ includes into:
\begin{align*}
e_{O(N,NK)}M(G/NK)
\end{align*}
and filtered colimits preserve inclusions we know that $\tilde{F}_K$ includes into:
\begin{align*}
F_K=\underset{N}{\colim} M(G/NK)_{(K)}\cong \underset{J}{\colim} M(G/J)_{(K)}.
\end{align*}  
\end{remark}
\begin{proposition}\label{fixfinal}
For $K\leq H\leq G$, $H$ open in $G$ and $K$ closed, there is an isomorphism of $\mathbb{Q}$-modules:
\begin{align*}
\theta: M(G/H)_{(K)}&\rightarrow \tilde{F}_K\\s(K)&\mapsto \left[e_{O(N,NK)}R^H_{NK}(s)\right]
\end{align*}
where $N=H^C$.
\end{proposition}
\begin{proof}
The fact this is well defined is immediate from the definition of $\theta$ and the colimit diagram of $\tilde{F}_K$. Namely, if $N_1\leq N_2$ then $e_{O(N_1,N_1K)}s$ has support contained in that of $e_{O(N_2,N_2K)}s$. The images under $\theta$ of these are $e_{O(N_1,N_1K)}R^H_{N_1K}(s)$ and $e_{O(N_2,N_2K)}R^H_{N_2K}(s)$ respectively. These are equal in $\tilde{F}_K$ since the latter maps to the former via the colimit diagram. For surjectivity suppose $\left[e_{O(N,NK)}t\right]$ belongs to $\tilde{F}_K$ and has representative:
\begin{align*}
e_{O(N,NK)}t\in e_{O(N,NK)}\left(M(G/NK)^{N_H(NK)}\right) \,\text{for}\, t\in M(G/NK)^{N_H(NK)}. 
\end{align*}
We can now apply Lemma \ref{fixinflate} to find $s\in e_{O(N,NK)}M(G/H)$ with $R^H_{NK}(s)=e_{O(N,NK)}t$. Therefore $\left[e_{O(N,NK)}R^H_{NK}(s)\right]=\left[e_{O(N,NK)}t\right]$.

For injectivity, if $s(K)$ maps to zero in $\tilde{F}_K$, then by definition there exists some open normal subgroup $N$ such that $e_{O(N,NK)}R^H_{NK}(s)=0$. Since $e_{O(N,NK)}R^H_{NK}(s)$ belongs to $e_{O(N,NK)}M(G/NK)^{H_H(NK)}$ we can apply Lemma \ref{fixinflate} to deduce that $e_{O(N,NK)}s=0$ and hence $s(K)=0$.
\end{proof} 
In the setting where $K$ is open we know from Remark \ref{opencasestalk} that $\tilde{F}_K$ is isomorphic to $\left(M(G/K)^{N_H(K)}\right)_{(K)}$. Therefore Proposition \ref{fixfinal} restricts to give the relation we would expect from the finite case. We now begin to relate these lemmas towards the overall goal of giving an equivalence between $M$ and $\text{Mackey}\circ\text{Weyl}(M)$ for any Mackey functor $M$.
\begin{remark}
We know that any $G$-sheaf of $\mathbb{Q}$-modules is in particular a sheaf of $\mathbb{Q}$-modules. Also any sheaf $F$ satisfies that its sheaf space $LF$ is homeomorphic to the sheaf space of its sheafification $L\Gamma F$. From this, we deduce that since the images of the sections of $F$ form a basis for $L\Gamma F$, these sets are also open in $LF$.
\end{remark}
\begin{proposition}\label{fixtogether}
If $G$ is a profinite group and $H$ an open subgroup of $G$, then $\mackey\circ\weyl(M)(G/H)$ can be written as
\begin{align*}
A=\left\lbrace s:SH\rightarrow \coprod_{L\in SH}\tilde{F}_L\mid \text{s continuous},\,p\circ s=\id\right\rbrace ^H
\end{align*}
where $\underset{L\in SH}{\coprod}\tilde{F}_L$ has the subspace topology from $\underset{L\in SH}{\coprod}F_L$ as defined in Construction \ref{Sheaf_mack}.
\end{proposition}
\begin{proof}
Clearly we have the following:
\begin{align*}
A&\subseteq \text{Mackey}\circ\text{Weyl}(M)(G/H)\\&=\left\lbrace s:SH\rightarrow \coprod_{L\in SH}F_L\mid\, s\, \text{continuous},\,p\circ s=\id\right\rbrace ^H
\end{align*}
using the sheaf space language.

On the other hand if $s\in \text{Mackey}\circ\text{Weyl}(M)(G/H)$ and $K\leq H$ closed, first note that $s(K)$ belongs to $F_K^{N_H(K)}$ since for any $h\in N_H(K)$:
\begin{align*}
s(K)=(hs)(K)=hs(h^{-1}Kh)=hs(K).
\end{align*}
In particular each $s(y)$ belongs to $F_y^{N_H(y)}$. Since $s$ is a section we have that $s(SH)$ is open in $\underset{L\in SH}{\coprod}F_L$ by the proceeding remark, and $s(K)$ is a point belonging to it. By definition of the basis for the sheaf space choose $N$ open and normal and contained in $H$ with $t\in M(G/NK)$ and: 
\begin{align*}
s(K)\in t(O(N,NK))\subseteq s(SH). 
\end{align*}
By Proposition \ref{fixfinal} we can assume that $t$ belongs to $M(G/H)$, since \\$\left(M(G/NK)_{(NK)}\right)^{N_H(NK)}\cong\tilde{F}_{NK}\cong M(G/H)_{(NK)}$. As $t$ belongs to $M(G/H)$ it is $H$-fixed and it follows that for each $y\in O(N,NK)$, $t_y$ belongs to $\tilde{F}_y$.

Therefore $e_{O(N,NK)}t=s_{|_{O(N,NK)}}$ and it follows:
\begin{align*}
s(K)=(e_{O(N,NK)}t)(K)\in M(G/H)_{(K)}.
\end{align*}
We can find a section $e_{O(N,NK)}t$ for every $K\in SH$, so we have a collection of these in $M(G/H)$ for each $K\in SH$ and these agree on intersections by construction using Proposition \ref{fixfinal}. To see this observe that if $t^1_y=s_y=t^2_y$ in $\tilde{F}_y$ for $y$ in the intersection of the domain subsets, then Proposition \ref{fixfinal} says that they must be determined by the same germ in $M(G/H)_{(y)}$. We can therefore apply the gluing axiom of $M(G/H)$ to get the required section $t$ section over $SH$, by using Construction \ref{Sheaf_mack} to make a section from an element of $M(G/H)$.
\end{proof}
We can apply Proposition \ref{fixfinal} to deduce the following proposition.
\begin{proposition}\label{M(G/H)set}
There is a bijection of sets between:
\begin{align*}
A=\left\lbrace s:SH\rightarrow \coprod_{L\in SH}\tilde{F}_L\mid \text{s continuous},\,p\circ s=\id\right\rbrace ^H
\end{align*}
and $M(G/H)$.
\end{proposition}
\begin{proof}
First observe that $M(G/H)$ can be interpreted as global sections over $SH/H$, with respect to the sheaf space of $M(G/H)$. As seen in Construction \ref{Sheaf_mack}, each $s\in M(G/H)$ determines an element of $A$. If we take any $a\in A$, we can look at the proof of Proposition \ref{fixtogether} to deduce that $a$ is determined by some $s$ in $M(G/H)$, by applying Proposition \ref{fixfinal}. This tells us that the correspondence is surjective. For injectivity, suppose that $s$ and $t$ are elements of $M(G/H)$ whose corresponding global sections in $A$ are equal. Then it follows that for every $y\in SH$ that $\overline{t}_y=\overline{s}_y$ in $\tilde{F}_y$. We use the isomorphism from Proposition \ref{fixfinal}:
\begin{align*}
\tilde{F}_y\cong M(G/H)_{(y)}
\end{align*}
to deduce that $s_y=t_y$ for every class $y$. This means that $s=t$. 
\end{proof}
\begin{corollary}\label{homeo}
If $H$ is an open subgroup of a profinite group $G$ and $M$ a \index{Mackey functor} Mackey functor for $G$, then there is an isomorphism between $M(G/H)$ and $\mackey\circ\weyl(M)(G/H)$.
\end{corollary}
\begin{proof}
By Proposition \ref{fixtogether} $\text{Mackey}\circ\text{Weyl}(M)(G/H)$ is equal to:
\begin{align*}
\left\lbrace s:SH\rightarrow \underset{L\in SH}{\coprod}\tilde{F}_L\mid \text{$s$ continuous},\,p\circ s=\id\right\rbrace ^H
\end{align*}
which is equal to $M(G/H)$ by Proposition \ref{M(G/H)set}. The fact that this is an isomorphism of $\mathbb{Q}$-modules is immediate from the fact that the correspondence in Proposition \ref{fixfinal} is an isomorphism of $\mathbb{Q}$-modules.
\end{proof}
We have succeeded in showing that there is a level-wise equivalence between the Mackey functors $M$ and $\text{Mackey}\circ\text{Weyl}(M)$, but to finish we need to show that this is a map of Mackey functors and hence that the correspondence commutes with the three types of structure maps of Mackey functors. We begin by proving some useful Lemmas.
\begin{lemma}\label{Ndef}
If $G$ is a profinite group, $L,K\leq G$ with $L$ closed and $K$ open, then $\left[K^C\diagdown H\diagup K\right]=H/K$, where the core is with respect to $H$. 
\end{lemma}
\begin{proof}
Let $N$ denote the core of $K$ with respect to $H$. Given $NhK$, then this equals the orbit $\left(N/\text{stab}_N(hK)\right)hK$. Now:
\begin{align*}
\text{stab}_N(hK)&=\left\lbrace n\in N\mid nhK=hK\right\rbrace=\left\lbrace n \in N\mid h^{-1}nh\in K\right\rbrace\\&=\left\lbrace n\in N\mid n\in hKh^{-1}\right\rbrace=N\cap hKh^{-1}=N.
\end{align*}
Therefore $NhK=\left(N/\text{stab}_N(hK)\right)hK=\left(N/N\right)hK=hK$.
\end{proof}
\begin{lemma}
If $N,L,K$ are defined to be the same as in the proceeding lemma, then $\left[NL\diagdown H\diagup K\right]$ is equal to $\left[L\diagdown H\diagup K\right]$.
\end{lemma}
\begin{proof}
If $x,y\in H$ with $x\sim y$ in $\left[L\diagdown H\diagup K\right]$, then the fact that $x\sim y$ in $\left[NL\diagdown H\diagup K\right]$ follows directly from the fact that $L\leq NL$. 

On the other hand if $x\sim y$ in $\left[NL\diagdown H\diagup K\right]$, then $x=layk$ for $l\in L,a\in N,k\in K$. By Lemma \ref{Ndef} $ayk=yk^{\prime}$ for $k^{\prime}\in K$. It follows that $x=lyk^{\prime}$ and that $x\sim y$ in $\left[L\setminus H/K\right]$.
\end{proof}
\begin{lemma}\label{inductcharac}
We can write $(I^H_K(s))(L)$ in terms of conjugation and restriction in $F_L$, where $L,K\leq H$ with $L$ closed and $K$ open.
\end{lemma}
\begin{proof}
By definition of the colimit $F_L$, we have that $(I^H_K(s))(L)=(R^H_{NL}I^H_K(s))(L)$ in $F_L$, where $N$ is taken as in the proof of Lemma \ref{Ndef}. Applying the Mackey axiom gives:
\begin{align*}
\underset{NL\diagdown H\diagup K}{\sum}(I^{NL}_{NL\cap xKx^{-1}}C_xR^K_{K\cap x^{-1}NLx}(s))(L)
\end{align*}
which further brings us to:
\begin{align*}
\underset{NL\diagdown H\diagup K}{\sum}(I^{NL}_{NL\cap xKx^{-1}}R^{xKx^{-1}}_{xKx^{-1}\cap NL}C_x(s))(L).
\end{align*}
We know that $NL$ is $NL$-subconjugate to $NL\cap xKx^{-1}$ if and only if it is $NL$-conjugate to $xKx^{-1}$, which happens if and only if $L\leq xKx^{-1}$ by our choice of $N$. Therefore an application of Corollary \ref{burn0} says that $e_{O(N,NL)}I^{NL}_{NL\cap xKx^{-1}}(s)$ is not supported on $L$ if $L$ is not a subgroup of $xKx^{-1}$. Hence in this case the stalk at $L$ is zero for this particular summand. Let $I$ denote the following indexing set:
\begin{align*}
\left\lbrace x\in\left[NL\diagdown H\diagup K\right]\mid L\leq xKx^{-1}\right\rbrace.
\end{align*}
Notice that this choice is independent of the choice of representative in the sense that if $x\sim y$ in the double coset, then $L\leq xKx^{-1}$ if and only if $L\leq yKy^{-1}$. To see this, suppose $x=ayk$ and $L\leq xKx^{-1}$. By substitution we have \\$L\leq ayKy^{-1}a^{-1}$. Since both $N$ and $L$ are subgroups of $ayKy^{-1}a^{-1}$ it follows that $NL$ is and therefore $a$ belongs to $ayKy^{-1}a^{-1}$. This means that: 
\begin{align*}
ayKy^{-1}a^{-1}=yKy^{-1}
\end{align*}
as required.

However if a summand for $x$ satisfies that $L\leq xKx^{-1}$ (and hence $x\in I$) then $NL\leq xKx^{-1}$ and $NL\cap xKx^{-1}=NL$, and hence the corresponding term is 
\begin{align*}
I^{NL}_{NL}R^{xKx^{-1}}_{NL}C_x(s)(L)=R^{xKx^{-1}}_{NL}C_x(s)(L).
\end{align*}
Therefore we have shown that:
\begin{align*}
(I^H_K(s))(L)=\underset{I}{\sum}R^{xKx^{-1}}_{NL}C_x(s)(L),
\end{align*}
in $F_L$.
\end{proof}
\begin{proposition}\label{morphism}
The level-wise equivalence between Mackey functors $M$ and $\mackey\circ\weyl(M)$ is a morphism in the category of Mackey funtors.
\end{proposition}
\begin{proof}
We need to show that the correspondence between $M$ and \\$\text{Mackey}\circ\text{Weyl}(M)$ commutes with the three maps; restriction, induction and conjugation. Beginning with conjugation suppose $g\in G,H\leq G$ is open and $s\in M(G/H)$. We need to show that the following square commutes:
\begin{align*}
\xymatrix{M(G/H)\ar[d]^{C_g}\ar[r]_(0.4){\theta_H}&\text{Mackey}\circ\text{Weyl}(M)(G/H)\ar[d]^{\overline{C}_g}\\M(G/gHg^{-1})\ar[r]_(0.4){\theta_{gHg^{-1}}}&\ \text{Mackey}\circ\ \text{Weyl}(M)(G/gHg^{-1})}
\end{align*}
where $\theta$ represents the level wise equivalence and $\overline{C}_g$ is conjugation for \\$\text{Mackey}\circ\text{Weyl}(M)$. On the one hand we have:
\begin{align*}
(\overline{C}_g\theta_H(s))(gKg^{-1})=g\theta_H(s)(K)=g(s(K)).
\end{align*}  
On the other hand we have:
\begin{align*}
(\theta_{gHg^{-1}}C_g(s))(gKg^{-1})=(gs)(gKg^{-1})=g(s(K))
\end{align*}
proving equality. Here we are using that $C_g(s(K))=(C_g(s))(gKg^{-1})$, since from the formula of the idempotents we have that:
\begin{align*}
C_g(e_{O(N,NK)}(s))=e_{O(N,NgKg^{-1})}(C_g(s)). 
\end{align*}
For restriction suppose $K\leq H$ are open subgroups of $G$ and $s\in M(G/H)$, then we need to see if the following square commutes:
\begin{align*}
\xymatrix{M(G/H)\ar[d]^{R^H_K}\ar[r]_(0.3){\theta_H}&\text{Mackey}\circ\text{Weyl}(M)(G/H)\ar[d]^{\overline{R^H_K}}\\M(G/K)\ar[r]_(0.3){\theta_{K}}&\text{Mackey}\circ\text{Weyl}(M)(G/K)}
\end{align*}
where $\overline{R^H_K}$ denotes restriction for $\text{Mackey}\circ\text{Weyl}(M)$. On the one hand we have:
\begin{align*}
\overline{R^H_K}\theta_H(s)(L)=\theta_H(s)_{|_{SK}}(L)=\theta_H(s)(L)=s(L).
\end{align*}
On the other hand we have:
\begin{align*}
\theta_K(R^H_K(s))(L)=R^H_K(s)(L).
\end{align*}
We observe that $s(L)=R^H_K(s)(L)$ with respect to the colimit relation in $F_L$. For induction suppose $K\leq H$ are open and $s\in M(G/K)$. Then we must show that the following square commutes:
\begin{align*}
\xymatrix{M(G/H)\ar[r]_(0.3){\theta_H}&\text{Mackey}\circ\text{Weyl}(M)(G/H)\\M(G/K)\ar[u]^{I^H_K}\ar[r]_(0.3){\theta_{K}}&\text{Mackey}\circ\text{Weyl}(M)(G/K)\ar[u]^{\overline{I^H_K}}}
\end{align*}
where $\overline{I^H_K}$ is the induction map for $\text{Mackey}\circ\text{Weyl}(M)$. We begin by noticing that:
\begin{align*}
(\theta_HI^H_K(s))(L)=I^H_K(s)(L).
\end{align*}
By Proposition \ref{inductcharac} this is equivalent to $\underset{I}{\sum}(R^{xKx^{-1}}_{NL}C_x(s))(L)$, where
\begin{align*}
I=\left\lbrace x\in\left[NL\diagdown H \diagup K\right] \mid L\leq xKx^{-1}\right\rbrace.
\end{align*}
Using the previous arguments established in the proof this is equivalent in $F_L$ to $\underset{I}{\sum}(xs)(L)$. On the other hand:
\begin{align*}
\overline{R^H_{NL}}\overline{I^H_K}(s)(L)&=\left(\underset{H/K}{\sum}h\overline{s}\right)_{|_{S(NL)}}(L)\\&=\underset{x\in NL\diagdown H\diagup K}{\sum}\,\underset{a\in NL/NL\bigcap xKx^{-1}}{\sum}\overline{(ax)s_{|_{S(NL\bigcap xKx^{-1})}}}(L).
\end{align*}
If we call the final sum $\underset{x\in NL\diagdown H\diagup K}{\sum}A_x(L)$, then $A_x(L)=0$ if $L$ is not a subgroup of $xKx^{-1}$. Therefore the only values for $A_x$ which contribute anything are for those $x$ which satisfy that $L\leq xKx^{-1}$. However, if this happens then $NL\cap xKx^{-1}=NL$ and it follows that:
\begin{align*}
\underset{a\in NL/\left(NL\cap xKx^{-1}\right)}{\sum}(axs)(L)=\underset{a\in NL/NL}{\sum}(axs)(L)=(xs)(L).
\end{align*}
Hence summing over $\left[NL\diagdown H\diagup K\right]$ gives us $\underset{I}{\sum}(xs)(L)$ using a similar argument to Lemma \ref{inductcharac}. Therefore we have shown that:
\begin{align*}
\theta_H(I^H_K(s))(L)=\underset{I}{\sum}(xs)(L)=\overline{I^H_K}\theta_K(s)(L)
\end{align*}
as required.
\end{proof}
\begin{theorem}\label{mackeyweyl}
If $M$ is a Mackey functor for $G$ then $\mackey\circ\weyl(M)\cong M$ in the category of Mackey functors.
\end{theorem}
\begin{proof}
This is a direct consequence of Propositions \ref{homeo} and \ref{morphism}.
\end{proof}
\begin{remark}\label{NatTrans}
If $t$ is a natural transformation from $F$ to $G$ which are functors from $\mathfrak{C}$ to $\mathfrak{D}$, then functors $H$ from $\mathfrak{D}$ to $\mathfrak{E}$ and $J$ from $\mathfrak{B}$ to $\mathfrak{C}$ give natural transformations:
\begin{align*}
H\circ t: H\circ F&\rightarrow H\circ G\,\text{ given by maps}\\
H(t_X): H\circ F(X)&\mapsto H\circ G(X),
\end{align*} 
and
\begin{align*}
t\circ J:F\circ J&\rightarrow G\circ J\,\text{ given by maps}\\
t_{J(a)}:F\circ J(a)&\mapsto G\circ J(a).
\end{align*}
\end{remark}
\begin{theorem}\label{equivalencemain}
If $G$ is a profinite group then the category of rational $G$-Mackey functors is equivalent to the category of Weyl-$G$-sheaves over $SG$.
\end{theorem}
\begin{proof}
This is an application of Theorems \ref{weylmack}, \ref{mackeyweyl}. We need only prove that these functors induce a bijective correspondence on morphism sets. Let $M$ and $\overline{M}$ be Mackey functors, we will show that the following map is bijective:
\begin{align*}
\text{Mackey}(M,\overline{M})&\rightarrow \Weyl(\weyl(M),\weyl(\overline{M}))\\
f&\mapsto \weyl(f).
\end{align*}
Let $f,h\in\text{Mackey}(M,\overline{M})$ with $\weyl(f)=\weyl(h)$. It follows that \\$\weyl(f)(SH)^H=\weyl(h)(SH)^H$ as morphisms:
\begin{align*}
\weyl(M)(SH)^H\rightarrow \weyl(\overline{M})(SH)^H.
\end{align*}
We now apply Proposition \ref{M(G/H)set} to deduce that $f=h$. Therefore the correspondence is injective.

Suppose we have $f\in \Weyl(\weyl(M),\weyl(\overline{M}))$. We will see that $\weyl(\mackey(f))=f$. We begin by considering the map $f(SH)^H$ which is precisely $\mackey(f)(G/H)$. In order to apply $\weyl(-)$ to this we take the colimit of the diagram of morphisms $e_Uf(SH)^H$ for varying $U$ and $H$ at each stalk $K$. An application of Proposition \ref{stalkeq} shows that this is simply the stalk map $f_K$. Therefore $\weyl(\mackey(f))=f$ as required.

We have shown that the functors $\text{Mackey}\circ\text{Weyl}$ and $\text{Id}_{\text{Mackey}}$ induce an object-wise isomorphism and a bijection on hom sets. We have also shown the same for $\text{Weyl}\circ\text{Mackey}$ and $\text{Id}_{\text{Weyl}}$. We need to prove these are natural to finish. To simplify notation for this proof, we shall denote $\text{Mackey}$ by $M$ and $\text{Weyl}$ by $W$. If we prove that we have defined natural transformation from $\text{Id}_{\text{Mack}}$ to $M\circ W$ this will be sufficient by Remark \ref{NatTrans} and the fact that $W$ and $M$ are inverse.

We will show that the following two squares commute:
\begin{align*}
\xymatrix{\text{Mack}(N,\overline{N})\ar[r]\ar[d]^{f^*}&\text{Mack}(M\circ W(N),M\circ W(\overline{N}))\ar[d]^{M\circ W(f)^*}\text{Mack}(N,\overline{N})\\
\text{Mack}(\overline{M},\overline{N})\ar[r]&\text{Mack}(M\circ W(\overline{M}),M\circ W(\overline{N}))}
\end{align*}
and
\begin{align*}
\xymatrix{\text{Mack}(N,\overline{N})\ar[r]\ar[d]^{h_*}&\text{Mack}(M\circ W(N),M\circ W(\overline{N}))\ar[d]^{M\circ W(h)_*} \\ \text{Mack}(N,\overline{M})\ar[r]&\text{Mack}(M\circ W(\overline{M}),M\circ W(\overline{M}))}
\end{align*}
for $f:\overline{M}\rightarrow N$ and $h:\overline{N}\rightarrow \overline{M}$ in the category of Mackey functors. In order to see that the squares commute we need to prove the equalities:
\begin{enumerate}
\item $M\circ W(\alpha\circ f)= M\circ W(\alpha)\circ M\circ W(f)$,
\item $M\circ W(h\circ\beta)= M\circ W(h)\circ M\circ W(\beta)$,
\end{enumerate}
where $\alpha\in \text{Mack}(N,\overline{N})$ and $\beta \in \text{Mack}(N,\overline{N})$. These equalities are immediate from the functoriality of both $M$ and $W$ as seen in Theorems \ref{macktosheaf} and \ref{MackisFunct}.
\end{proof}
We can apply this equivalence to arrive at the following theorem.
\begin{theorem}
If $G$ is a profinite group, then the category of rational $G$-Spectra is modelled by the category:
\begin{align*}
\ch\left(\Weyl_{\mathbb{Q}}(SG)\right).
\end{align*}
Here we use the projective model structure, where weak equivalences are the homology isomorphisms and the fibrations are the level-wise epimorphisms,.
\end{theorem} 
\begin{proof}
See Theorems \ref{equivalencemain} and \ref{equicohomchar}.
\end{proof}
In combination with Chapter $3$ we also have the following diagram which illustrates the characterisation of rational $G$-Spectra for $G$ profinite.
\begin{align*}
\xymatrix{\text{Rational $G$-spectra}\ar[d]^{\simeq}\\ \text{Ch}\left(\text{Fun}_{\text{Ab}}\left(\pi_0(\mathfrak{O}^{\mathbb{Q}}_G),\mathbb{Q}\text{-Mod}\right)\right)\ar[d]^{\cong}\\\text{Ch}\left(\text{Mackey}_{\mathbb{Q}}(G)\right)\ar[d]^{\cong}\\ \text{Ch}\left(\text{Weyl-}G\text{-Sheaf}_{\mathbb{Q}}(SG)\right)}
\end{align*}
where $\simeq$ represents a Quillen equivalence and $\cong$ is an equivalence of categories.
\section{G-Sheaf to Weyl-G-Sheaf Adjunction}
We will now construct an adjunction between $G$-sheaves and Weyl-$G$-sheaves of $\mathbb{Q}$-modules over $SG$. The following construction will show what the two functors will be.
\begin{construction}  
We know that every Weyl-$G$-sheaf over $SG$ is in particular a $G$-sheaf over $SG$ hence we have an inclusion functor which we denote by $L$. In the other direction if we begin with a $G$-sheaf over $SG$ say $F$ we can define a Weyl-$G$-sheaf over $SG$ to be $\text{Weyl}\circ \text{Mackey}(F)$ and we denote this functor by $R$. 
\end{construction}
We will prove that these functors are an adjoint pair. We begin with the following observations.
\begin{proposition}
If $F$ is a $G$-sheaf over $SG$ then $RF$ is a subsheaf of $F$.
\end{proposition}
\begin{proof}
If $s_K\in RF_K$ then it has a representative $s\in e_UF(SJ)^J$ which is contained in $F(U)$ and so we have an inclusion from $RF_K$ into $F_K$, and hence $RF$ includes into $F$.
\end{proof}
\begin{remark}
If $F$ is a $G$-sheaf then we can apply Corollary \ref{RFdefn} to deduce that the underlying set of the sheaf space $RF$ is given by:
\begin{align*}
\underset{K\in SG}{\coprod}F_K^K.
\end{align*} 
\end{remark}
A vital component of the proof of the adjunction is given in the following lemma.
\begin{lemma}\label{factor}
If $A$ is a Weyl-$G$-sheaf of $\mathbb{Q}$-modules over $SG$, $F$ is a $G$-sheaf of $\mathbb{Q}$-modules over $SG$ and $\alpha:A\rightarrow F$ is a morphism of $G$-sheaves, then $\alpha$ factors through $RF$.
\end{lemma}
\begin{proof}
We begin by showing that for each $K\in SG$ we know that $\alpha_K$ has image in $RF_K$. If $a_K\in A_K$ for $K\in SG$ then $\alpha_K(a_K)\in F_K$ is $K$-fixed. To see this take $k\in K$ then using the equivariance of the map $\alpha$ we have:
\begin{align*}
k\alpha_K(a_K)=\alpha_K(ka_K)=\alpha_K(a_K).
\end{align*}
Since $\alpha_K(a_K)\in F_K^K$ we can apply Proposition \ref{Weylequi} to deduce that there exists an open subgroup $J$ of $G$ containing $K$ and a neighbourhood $U$ of $K$ so that $\alpha_K(a_K)$ is represented by a section in $e_UF(SJ)^J$, and so $\alpha_K(a_K)$ belongs to $RF_K$.
\end{proof}
\begin{proposition}\label{restr}
If $f:F\rightarrow H$ is a morphism of $G$-sheaves over $SG$, then $R(f)$ is given by restriction to the sub-object $RF$ of $F$.
\end{proposition}
\begin{proof}
First note that the corresponding map of Mackey functors is given by $f(SJ)$ restricted to the $J$-fixed sections for $J$ open, and the map from $RF_K$ is given by taking colimits of these restrictions of $f(SJ)$ to the $J$-fixed sections. But this is just $f_K$ applied to the germ in $RF_K$. The image of this map is contained in $RH$ by Lemma \ref{factor}.
\end{proof} 

\begin{lemma}\label{iso}
If $A$ is a Weyl-$G$-sheaf of $\mathbb{Q}$-modules and $F$ a $G$-sheaf of $\mathbb{Q}$-modules over $SG$ then we have an isomorphism:
\begin{align*}
\phi:\Weyl\,(A,RF)&\rightarrow \gheaf\,(LA,F)\\\alpha&\mapsto \iota\circ\alpha
\end{align*}
where $\iota$ is the inclusion from $RF$ to $F$.
\end{lemma}
\begin{proof}
For surjectivity if we take any $\beta:LA\rightarrow F$, then by Lemma \ref{factor} this factors through $\overline{\beta}:LA\rightarrow RF$, hence $\beta=\phi(\overline{\beta})$.
For injectivity if we take $\alpha_1,\alpha_2:A\rightarrow RF$ which satisfy $\iota\circ \alpha_1=\iota\circ \alpha_2$. Then since $\iota$ is a monomorphism we have $\alpha_1=\alpha_2$.
\end{proof}
\begin{proposition}\label{Weyladjunct}
The pair of functors $L$ and $R$ are an adjoint pair.
\end{proposition}
\begin{proof}
We already have an isomorphism of Hom-sets as provided in Lemma \ref{iso}, so it is left to prove that this isomorphism is natural. Set $\mathfrak{C}=G\text{-Sheaf}(SG)$ and $\mathfrak{D}=\text{Weyl-}G\text{-Sheaf}(SG)$. Take any morphism $f:X\rightarrow Y$ in $\mathfrak{C}$ and $h:A\rightarrow B$ in $\mathfrak{D}$. Then we need to show that the following square commutes:
\begin{align*}
\xymatrix{\mathfrak{D}(B,RX)\ar[r]^{\phi}\ar[d]^{h^*}&\mathfrak{C}(LB,X)\ar[d]^{h^*}\\ \mathfrak{D}(A,RX)\ar[r]^{\phi}&\mathfrak{C}(LA,X)}
\end{align*} 
If $\alpha\in \mathfrak{D}(B,RX)$ then we have $h^*(\phi(\alpha))=(\iota\circ\alpha)\circ h$. On the other hand we have that $\phi(h^*(\alpha))=\phi(\alpha\circ h)=\iota\circ(\alpha\circ h)$. These two are equal by associativity of composition. 

In the other direction we show that the following square commutes:
\begin{align*}
\xymatrix{\mathfrak{D}(A,RX)\ar[r]^{\phi}\ar[d]^{\overline{f}_*}&\mathfrak{C}(LA,X)\ar[d]^{f_*}\\ \mathfrak{D}(A,RY)\ar[r]^{\phi}&\mathfrak{C}(LA,Y)}
\end{align*}
Notice that if we have a morphism $f:X\rightarrow Y$, then by Proposition \ref{restr} and Lemma \ref{factor} $f$ induces a map $\overline{f}:RX\rightarrow RY$ where the following square commutes:
\begin{align*}
\xymatrix{X\ar[r]^f&Y\\RX\ar[u]^{\iota_X}\ar[r]^{\overline{f}}&RY\ar[u]^{\iota_Y}}
\end{align*}
If $\beta\in \mathfrak{D}(A,RX)$ then we have $f_*(\phi(\beta))=f\circ\iota_X\circ\beta$. On the other hand we have $\phi(\overline{f}_*(\beta))=\iota_Y\circ\overline{f}\circ\beta$. These are equal by commutativity of the above square.
\end{proof}
We now observe the following diagram to illustrate the picture we have so far:
\begin{align*}
\xymatrix{G\text{-sheaf}(SG)\ar@{<->}[dd]^{(L,R)}\\ \\ \text{Weyl-}G\text{-sheaf}(SG)\ar@{<->}[rr]^{\cong}&& \text{Mackey}_{\mathbb{Q}}(G)\ar@{<-->}[uull]_{(L^{\prime},R^{\prime})}}
\end{align*} 
The pair of functors $(L^{\prime},R^{\prime})$ is obtained by composing the equivalence between $G$-Mackey functors and Weyl-$G$-sheaves with the adjunction between Weyl-$G$-sheaves and $G$-sheaves. 
\section{Products of Weyl-G-sheaves}
The equivalence between Weyl-$G$-sheaves over $SG$ and rational $G$-Mackey functors can be used to characterise the infinite products in the category of Weyl-$G$-sheaves.
\begin{construction}\label{mackeyprod}
If $M^i$ is a family of rational $G$-Mackey functors then we define the product so that $\left(\underset{i\in I}{\prod}M^i\right)\left(G/H\right)=\underset{i\in I}{\prod}\left(M^i\left(G/H\right)\right)$ for any open subgroup $H$. The restriction, induction and conjugation maps are done component-wise.
\end{construction}
\begin{proposition}\label{mackprod1}
The product of Mackey functors defined in Construction \ref{mackeyprod} is the categorical product.
\end{proposition}
\begin{proof}
The construction provided in Construction \ref{mackeyprod}, which we shall denote by $M$, is a Mackey functor by definition since each $M^i$ is. This is because the restriction, induction and conjugation maps are taken component-wise, as is the $\mathbb{Q}$-module structure. To see that this is the categorical product take any Mackey functor $N$ with morphisms of Mackey functors $\psi_i:N\rightarrow M^i$. If $p_i$ is the projection map from $M$ to $M^i$, we have a map of Mackey functors:
\begin{align*}
\psi:N(G/H)&\rightarrow M(G/H)\\s&\mapsto \left(\psi_i(s)\right)_{i\in I}
\end{align*}
on each open subgroup $H$, which satisfies that $p_i\circ \psi=\psi_i$ for each $i\in I$. Notice that $\psi$ is a morphism of Mackey functors since each $\psi_i$ is and since the Mackey functor compatibility maps are taken component-wise. Therefore we have shown that a morphism exists making the diagram commute. 

For uniqueness suppose that we have another morphism $\phi$ which satisfies $p_i\circ \phi=\psi_i$ for each $i\in I$. If $H$ is any open subgroup then:
\begin{align*}
\phi(s)=\left(p_i\circ\phi(s)\right)_{i\in I}=\left(\psi_i(s)\right)_{i\in I}=\psi(s),
\end{align*}
so $\psi=\phi$.
\end{proof}
Since infinite products are shown to exist in the category of rational $G$-Mackey functors, the equivalence from Theorem \ref{equivalencemain} shows that these exist in the category of Weyl-$G$-sheaves. The following construction describes what these products look like in the category of Weyl-$G$-sheaves.
\begin{construction}\label{prodweyl}
If $M$ is the product of Mackey functors $M^i$ and $U$ is any basic open neighbourhood $U$ of $SG$, then for $s=(s^i)_{i\in I}\in M(G/H)$:
\begin{align*}
e_Us&=\underset{A}{\sum}q_AI^H_A\circ R^H_A(s)=\underset{A}{\sum}q_A\left({I^H_A}^i\circ {R^H_A}^i(s^i)\right)_{i\in I}\\&=\left(\underset{A}{\sum}q_A{I^H_A}^i\circ {R^H_A}^i(s^i)\right)_{i\in I}=\left(e_Us^i\right)_{i\in I}
\end{align*}
as required. We know that $U$ is of the form $O(N,NK)$ and from the constructed equivalence we know that the basic open subsets of the sheaf space of $\text{Weyl}\left(\underset{i\in I}{\prod}M^i\right)$ are of the form:
\begin{align*}
\left\lbrace \left[(s^i)_{i\in I}\right]_y\mid y\in U \right\rbrace
\end{align*}
for $(s^i)_{i\in I}\in \underset{i\in I}{\prod}e_UM^i(G/NK)$. This describes the form of the sections over the product Weyl-$G$-sheaf.
\end{construction}
\begin{proposition}
The description of the product of Weyl-$G$-sheaves\index{product of Weyl-$G$-sheaves} in Construction \ref{prodweyl} coincides with the adjunction description in Proposition \ref{Weyladjunct} where we calculate the product in the category of $G$-sheaves of $\mathbb{Q}$-modules and apply the functor $R$.
\end{proposition}
\begin{proof}
Let $F^i$ be a family of Weyl-$G$-sheaves of $\mathbb{Q}$-modules and $M^i$ be the corresponding family of Mackey functors. If $M$ is the product of Mackey functors, we need to show that $\text{Weyl}(M)$ is equal to $F=R\,\text{disc}\left(\underset{i\in I}{\prod}F^i\right)$.

A basic open subset of the sheaf space of $\text{Weyl}(M)$ is of the form:
\begin{align*}
\left\lbrace \left[(s^i)_{i\in I}\right]_y\mid y\in U \right\rbrace
\end{align*}
where $U=O(N,NK)$ and $(s^i)_{i\in I}\in \underset{i\in I}{\prod}e_UM^i(G/NK)$. On the other hand a basic open subset of the sheaf space of $F$ is of the form:
\begin{align*}
\left\lbrace \left[(s^i)_{i\in I}\right]_y\mid y\in U \right\rbrace
\end{align*}
where $s^i\in F^i(U)^{NK}$ and $U=O(N,NK)$. We know from the definition of $F^i$ as the Mackey functor $M^i$ that $s^i$ is an element of $e_UM^i(G/NK)$. Since both sheaves have the same underlying basic sections it follows that they have equal stalks and that their sheaf spaces have equivalent topologies.
\end{proof}
\chapter{Injective dimension of sheaves of rational vector spaces}\label{chapterIDCB}
\chaptermark{Injective dimension of sheaves}
Given a profinite space $X$, the aim of this chapter is to calculate the injective dimension of the category of sheaves of $\mathbb{Q}$-modules over $X$ in terms of the Cantor-Bendixson rank of $X$, denoted by ${\rank}_{CB}(X)$ (Theorem \ref{Summary}). This work will be completed for general spaces but we are specifically interested in profinite spaces since we will apply this work to the space of closed subgroups of a profinite group. This chapter is similar to \cite{Sugrue}. In Section \ref{Sheafexample} additional examples have been included.
\section{Cantor-Bendixson rank}
We begin by defining and stating known properties of the Cantor-Bendixson rank. We will see in Proposition \ref{prodcant} that the Cantor-Bendixson rank of a product of two spaces can be calculated in terms of the Cantor-Bendixson rank of the individual spaces. Recall that an \textbf{isolated point} of a topological space $X$ is a point $x$ which satisfies that $\left\lbrace x\right\rbrace$ is open in $X$.
\begin{definition}
For a topological space $X$ we define the \index{Cantor-Bendixson process}\textbf{Cantor-Bendixson} \textbf{process} on $X$. Denote by $X^{\prime}$ the set of all isolated points of $X$. We define:
\begin{enumerate}
\item Let $X^{(0)}=X$ and $X^{(1)}=X\setminus X^{\prime}$ have the subspace topology with respect to $X$.
\item For successor ordinals suppose we have $X^{(\alpha)}$ for an ordinal $\alpha$, we define $X^{(\alpha+1)}=X^{(\alpha)}\setminus {X^{(\alpha)}}^{\prime}$.
\item If $\lambda$ is a limit ordinal we define $X^{(\lambda)}=\underset{\alpha<\lambda}{\cap}X^{(\alpha)}$.
\end{enumerate}
\end{definition}
Every Hausdorff topological space $X$ has a minimal ordinal $\alpha$ such that $X^{(\alpha)}=X^{(\lambda)}$ for all $\lambda\geq \alpha$, see \cite[Lemma 2.7]{Gartside}.
\begin{definition}
Let $X$ be a Hausdorff topological space. Then we define the \textbf{Cantor-Bendixson rank}\index{Cantor-Bendixson rank} of $X$ denoted $\rank_{CB}(X)$ to be the minimal ordinal $\alpha$ such that $X^{(\alpha)}=X^{(\lambda)}$ for all $\lambda\geq \alpha$.
\end{definition}
A topological space $X$ is called \textbf{perfect} if it has no isolated points.
\begin{definition}
If $X$ is a Hausdorff space with Cantor-Bendixson rank $\lambda$, then we define the \textbf{perfect hull}\index{perfect hull} of $X$ to be the subspace $X^{(\lambda)}$.
\end{definition}
There are two ways that the Cantor-Bendixson process can stabilise. The first way is where the perfect hull is the empty set and the second is where it is a non-trivial subspace.
\begin{definition}
A compact Hausdorff space $X$ of Cantor-Bendixson rank $\alpha$ is called \textbf{scattered}\index{scattered space} if the space $X^{(\alpha)}$ obtained by the definition above is equal to the empty set. 
\end{definition}
\begin{example}
If $X$ is perfect or if $X=\emptyset$ then $\rank_{CB}(X)=0$.
\end{example}
\begin{example}\label{padic}
Consider $S\left(\mathbb{Z}_p\right)$ which by Proposition \ref{padicspace} is equivalent to $P=\left\lbrace \frac{1}{n}\mid n\in\mathbb{N}\right\rbrace\bigcup\left\lbrace 0\right\rbrace$ with the subspace topology of $\mathbb{R}$. Applying the first stage of the Cantor-Bendixson process to $P$ results in removing the accumulation points of the form $\frac{1}{n}$ leaving only the limit point $0$. A second application of this process leaves us with the empty set since the singleton space consisting only of $0$ is discrete. The process is stable from this point onwards so we therefore know that $\rank_{CB}(P)=2$.
\end{example}
We will see more interesting examples after Proposition \ref{prodcant}. The following proposition and theorem will explain how the perfect hull of a space $X$ relates to $X$ as a subspace. We shall allow $X_H$ to denote the perfect hull of $X$ and $X_S$ to denote its complement, which we call the scattered part of $X$. Both $X_S$ and $X_H$ will be considered with the subspace topology with respect to $X$.

\begin{proposition}
If $X$ is a Hausdorff space, then $X_H$ is always closed and $X_S$ is always open.
\end{proposition}
\begin{proof}
We shall prove that $X_S$ is an open subset of $X$. If $x$ is any point of $X_S$ we will find an open subset of $X$ containing $x$ and contained in $X_S$. If $x$ belongs to $X_S$, then by definition it is in the complement of $X_H$ and hence there exist some ordinal $\kappa$ such that $x$ is isolated in $X^{(\kappa)}$. This in turn means that there exists some open subset $U$ of $X$ such that $U\cap X^{(\kappa)}=\left\lbrace x\right\rbrace$. This proves that each point belonging to $U$ is eliminated in the Cantor-Bendixson process atleast before the stage of any ordinal strictly larger than $\kappa$. This is another way of saying that $U$ is contained in $X_S$ which proves the result. 
\end{proof}

In general $X_H$ may not be open. The next theorem shows that we can compute the cardinality of $X_S$ in certain cases.
\begin{theorem}[Cantor-Bendixson Theorem]\label{CantThm}\index{Cantor-Bendixson Theorem}
Given a countably based Hausdorff topological space $X$, we can write $X$ as a disjoint union of a countable scattered subset $X_S$ with the perfect hull $X_H$ of $X$. 
\end{theorem}
It is important to note that this theorem does not say that $X$ can be written as the coproduct of $X_H$ and $X_S$ in the category of spaces. That is, this theorem does not claim that $X_S$ and $X_H$ provide a topological disconnection of $X$.
\begin{definition}
If $X$ is a space and $x\in X_S$, we define the \textbf{height}\index{height of a point} of $x$ denoted $\text{ht}(X,x)$, to be the ordinal $\kappa$ such that $x\in X^{(\kappa)}$ but $x\notin X^{(\kappa+1)}$. We sometimes denote this by $\text{ht}(x)$ when the background space $X$ is understood.
\end{definition}

In the following proposition we will see how to calculate the Cantor-Bendixson rank of a product of two spaces. It is important to notice that this only works when both spaces have Cantor-Bendixson rank bigger than zero.

\begin{proposition}\label{prodcant}
Let $X$ be a space with $\rank_{\cb}(X)=n+1$ and $Y$ be a space with $\rank_{\cb}(Y)=m+1$, where $m,n\in\mathbb{N}_0$. Then $\rank_{\cb}(X\prod Y)=m+n+1=\rank_{\cb}(X)+\rank_{\cb}(Y)-1$.
\end{proposition}
\begin{proof}
We first prove this in the case where both $X$ and $Y$ are scattered. Let $X_k$ denote the set of isolated points in $X^{(k)}$, and $Y_k$ denote the set of isolated points in $Y^{(k)}$. First note that the isolated points of $\left(X\prod Y\right)$ are given by $X_0\prod Y_0$, and so:
\begin{align*}
\left(X\prod Y\right)^{(1)}=\left(X\prod Y\right)\setminus \left(X_0\prod Y_0\right).
\end{align*}
To see this, first observe that $X_0\prod Y_0$ consists of isolated points. On the other hand, take any point outside this set, say $(x,y)$, where either $x$ or $y$ has height bigger than or equal to $1$. Assume without loss of generality that $x$ is the point with non-trivial height. Then points of the form $(x^{\prime},y)$, where $x^{\prime}$ represents points which converge to $x$, belong to $X\prod Y$ and converge to $(x,y)$. Therefore $(x,y)$ cannot be isolated in $X\prod Y$.

The set of isolated points of $\left(X\prod Y\right)^{(1)}$ are equal to the set
\begin{align*}
\left(X_0\prod Y_1\right)\coprod\left(X_1\prod Y_0\right).
\end{align*}
To see that this is true, first observe that points in this set are isolated. On the other hand take a point $(x,y)$ such that $x$ has height greater than or equal to $2$ and $y$ is isolated. Then points of the form $(x^{\prime},y)$ would converge to $(x,y)$, where $x^{\prime}$ has height between $1$ and the height of $x$. The points of the form $(x^{\prime},y)$ therefore belong to $\left(X\prod Y\right)^{(1)}$. Similarly if we take $(x,y)$ in $X_1\prod Y_1$ we will have points in $X_0\prod Y_1$ and $X_1\prod Y_0$ accumulating at $(x,y)$, and these points belong to $\left(X\prod Y\right)^{(1)}$. It follows that $(x,y)$ cannot be isolated in $\left(X\prod Y\right)^{(1)}$. We therefore have:
\begin{align*}
\left(X\prod Y\right)^{(2)}=\left(X\prod Y\right)\setminus \left[(X_0\prod Y_0)\coprod\left(X_0\prod Y_1\right)\coprod \left(X_1\prod Y_0\right)\right].
\end{align*}

\textbf{Claim:} The isolated points in $\left(X\prod Y\right)^{(i)}$ are of the form
\begin{align*}
\underset{p+q=i}{\coprod}\left(X_p\prod Y_q\right)
\end{align*}
where $0\leq p\leq n$ and $0\leq q \leq m$.

We have shown this holds for $i=0$ and $i=1$ so let the above claim be our inductive hypothesis, and suppose it holds for $i$ and that $\gamma$ is an isolated point of $\left(X\prod Y\right)^{(i+1)}$. 

Then if $i$ is even and hence $i+1$ is odd, since all of the points accumulating at $\gamma$ were eliminated in the previous stage of the Cantor-Bendixson process, and by our hypothesis each of these accumulation points which were isolated in $\left(X\prod Y\right)^{(i)}$ belong to some $X_p\prod Y_q$ where $p+q=i$, so $\gamma$ must belong to $X_{p+1}\prod Y_q$ or $X_p \prod Y_{q+1}$. The reasoning behind this is similar to the first case. If $\gamma$ belonged to $X_p\prod Y_q$ for $p+q>i+1$, then we can find points of the form $(x,y^{\prime})$ converging to $\gamma$ where $y^{\prime}\in Y_{q-1}$ which belongs to $\left(X\prod Y\right)^{(i+1)}$. This contradicts that $\gamma$ is isolated. If $p+q<i+1$ holds then our inductive hypothesis informs us that $\gamma$ is not an element of $\left(X\prod Y\right)^{(i+1)}$, which is a contradiction.

In the case where $i$ is odd and $i+1$ is even we have the same possibilities plus the additional possibility where $\gamma$ is in $X_{\frac{i+1}{2}}\prod Y_{\frac{i+1}{2}}$. Therefore we have shown by induction that the isolated points are of the form $\underset{p+q=i+1}{\coprod}\left(X_p\prod Y_q\right)$.

From this we can see that:
\begin{align*}
\left(X\prod Y\right)^{(i)}=\left(X\prod Y\right)\setminus \left[\underset{0\leq k\leq i-1}{\coprod}\left(\underset{p+q=k}{\coprod}X_p\prod Y_q\right)\right]. 
\end{align*}
We then have:
\begin{align*}
\left(X\prod Y\right)^{(n+m-1)}&=\left(X_n\prod Y_{m-1}\right)\coprod \left(X_{n-1}\prod Y_m\right)\coprod \left(X_n\prod Y_m\right)\\
\left(X\prod Y\right)^{(n+m)}&=\left(X_n\prod Y_m\right)\\
\left(X\prod Y\right)^{(n+m+1)}&=\emptyset.
\end{align*}
This proves that 
\begin{align*}
\rank_{CB}(X\prod Y)=m+n+1=\rank_{CB}(X)+\rank_{CB}(Y)-1. 
\end{align*}
The case where at least one of $X$ and $Y$ is non-scattered is similar except we observe that we end up with
\begin{align*}
\left(X\prod Y\right)^{(n+m+1)}=\left(X\prod Y\right)_H
\end{align*}
which may not be empty. 
\end{proof}
The following example shows that the condition that both $X$ and $Y$ need to have non-zero Cantor-Bendixson rank in order for Proposition \ref{prodcant} to hold.
\begin{example}
If $X=\emptyset$ and $Y$ is any space with Cantor-Bendixson rank bigger than $1$ then Proposition \ref{prodcant} fails. We know that $X\prod Y=\emptyset$ and therefore has Cantor-Bendixson rank $0$. On the other hand: 
\begin{align*}
\rank_{\cb}(X)+\rank_{\cb}(Y)-1=\rank_{\cb}(Y)-1\neq 0.
\end{align*}
Furthermore if $X$ is perfect this fails. To see this take any point $(x,y)$ in $X\prod Y$ where $y$ is isolated. Since $X$ is perfect we can find a net $x_{\gamma}$ converging to $x$ which is not constant. Therefore $(x_{\gamma},y)$ provides a non-constant net converging to $(x,y)$ in $X\prod Y$. Therefore in this case the Cantor-Bendixson rank of $X\prod Y$ is $0$. We can see that Proposition \ref{prodcant} fails in this case similar to how it failed when $X=\emptyset$. 
\end{example}
We now have the following two examples of Cantor-Bendixson rank calculations. 
\begin{example}
From \cite[Proposition 2.5]{Gartside1} we know that if $q_1,q_2,\ldots,q_n$ are a finite collection of distinct primes then there is an isomorphism:
\begin{align*}
S\left(\underset{1\leq i\leq n}{\prod}\mathbb{Z}_{q_i}\right)\cong \underset{1\leq i\leq n}{\prod}S\left(\mathbb{Z}_{q_i}\right)
\end{align*}
By Example \ref{padic} there is a homeomorphism of spaces:
\begin{align*}
S\left(\underset{1\leq i\leq n}{\prod}\mathbb{Z}_{q_i}\right)\cong P^n.
\end{align*}
An application of Proposition \ref{prodcant} shows that:
\begin{align*}
\rank_{\cb}\left(S\left(\underset{1\leq i\leq n}{\prod}\mathbb{Z}_{q_i}\right)\right)=n+1.
\end{align*}
\end{example}
\begin{example}
The space $\underset{n\in\mathbb{N}}{\coprod}{P^n}$ whilst not being profinite since it is not compact, gives an example of a space which has infinite Cantor-Bendixson rank. This is because we can set $x_n$ to be the point in $P^n$ with height equal to $n+1$ and we therefore have a sequence of points with unbounded height.
\end{example}
\section{Injective Resolutions of Sheaves}
In this section we construct an injective resolution of sheaves of $\mathbb{Q}$-modules over a space $X$. We do this by defining the Godement resolution of a sheaf and outlining why this is injective. In order to achieve this we recall Definition \ref{sheafspacedefn} from \cite[Definition 2.5.4]{Tennison}. 

Every sheaf space $E$ determines a sheaf by considering the assignment \\$U\mapsto E(U)$. On the other hand a sheaf $F$ determines a sheaf space $(LF,\pi)$. This space has underlying set $\underset{x\in X}{\coprod}F_x$ and is topologised as in \cite[Construction 2.3.8]{Tennison}. The map of spaces $\pi$ assigns a germ $s_x$ to $x$.
\begin{definition}\label{serr}
Let $F$ be a sheaf of $\mathbb{Q}$-modules over a topological space $X$. Then we define the sheaf $C^0(F)$ on the open sets $U$ by taking $C^0(F)(U)$ to be the collection of serrations, i.e. the set of not necessarily continuous functions $\left\lbrace f:U\rightarrow LF\,\mid \pi\circ f=\id\right\rbrace$ which equates to $\underset{x \in U}{\prod} F_x$.

For the restriction maps, if $V\subseteq U$ are open subsets of $X$, we restrict a serration over $U$ to a serration over $V$ by restricting the serration as a function to $V$. This means that our restriction maps are projections of the following form:
\begin{align*}
\underset{x\in U}{\prod}F_x&\rightarrow \underset{x\in V}{\prod}F_x\\
(s_x)_{x\in U}&\mapsto (s_x)_{x\in V}.
\end{align*}
\end{definition}
Note that every section is a serration so we have a natural inclusion: 
\begin{align*}
\delta_0:F\rightarrow C^0(F)
\end{align*}
which is a monomorphism.
\begin{remark}\label{sermapdef}
The map from a sheaf $F$ into $C^0(F)$ is given as follows:
\begin{align*}
F(U)&\rightarrow \underset{y\in U}{\prod}F_y\rightarrow \underset{V\backepsilon \,x}{\colim}\underset{y\in V}{\prod}F_y\\s&\mapsto (s_y)_{y\in U}\mapsto \left((s_y)_{y\in U}\right)_x
\end{align*}
where $U$ is an open neighbourhood of a point $x$ in $X$ and $\left(-\right)_x$ is the germ at $x$. This morphism of sheaves induces the following map ${\delta_0}_x$ on stalks:
\begin{align*}
F_x&\rightarrow \underset{V\backepsilon x}{\colim}\underset{y\in V}{\prod}F_y\\s_x&\mapsto \left((s_y)_{y\in U}\right)_x 
\end{align*}
We call ${\delta_0}_x$ the serration map and denote it by $S$ throughout to simplify notation.
\end{remark}
Notice that if a map $f$ belongs to the set of serrations in Definition \ref{serr} then $f$ does not have to be continuous. We can now define the Godemont resolution using Definition \ref{serr} and \cite[pp. 36, 37]{Bredon}.
\begin{definition}\label{Godement}
Let $F$ be a sheaf of $\mathbb{Q}$-modules over a topological space $X$. Then as in Definition \ref{serr} we have $C^0(F)$ and a monomorphism $\delta_0:F\rightarrow C^0(F)$.

Consider $\coker\delta_0$, if we replace $F$ in the construction above with $\coker{\delta_0}$ and set $C^1(F)= C^0(\coker{\delta_0})$ from Definition \ref{serr} we will get the following diagram:
\begin{center}
$\xymatrix{0\ar[r]&F\ar[r]^{\delta_0}&C^0(F)\ar[d]\ar[r]^{\delta_1}&C^1(F)\\
&&\coker{\delta_0}\ar[ur]_{{\delta_1}^{\prime}}}$
\end{center}
where ${\delta_1}^{\prime}$ is the monomorphism from $\coker\delta_0$ into $C^1(F)$. We can then continue to build the resolution inductively using this idea. This resolution which we have constructed is called the \index{Godement resolution}\textbf{Godement resolution}.
\end{definition}
We consider the following example of a sheaf which will give us a more complete understanding of the Godement resolution.
\begin{example}\label{Examplesky}\index{skyscraper sheaf}
If $x$ is any point of $X$ and $M$ any $\mathbb{Q}$-module, then we can define a sheaf $\iota_x(M)$ over $X$. This takes value $M$ at an open subset $U$ of $X$ if $x$ belongs to $U$ and $0$ otherwise. We call this the skyscraper sheaf and the assignment $M\mapsto \iota_x(M)$ defines a functor from the category of $\mathbb{Q}$-modules to the category of sheaves of $\mathbb{Q}$-modules. This functor is right adjoint to the functor from the category of sheaves to the category of $\mathbb{Q}$-modules which assigns $F$ to the stalk $F_x$. We see this by considering the closed subset $\left\lbrace x\right\rbrace$ of $X$ and applying \cite[Theorem 3.7.13]{Tennison}.
\end{example} 
\begin{remark}\label{skyprod}
Each $C^0(F)$ can be written as $\underset{y\in X}{\prod}\iota_y(F_y)$. To see this, if $U\subseteq X$ is open then $C^0(F)(U)=\underset{y\in U}{\prod}F_y$. On the other hand:
\begin{align*}
\left(\underset{y\in X}{\prod}\iota_y(F_y)\right)(U)&=\underset{y\in X}{\prod}\left(\iota_y(F_y)(U)\right)=\underset{y\in U}{\prod}F_y.
\end{align*}
\end{remark}
The following lemma relates the Cantor-Bendixson process to the Godement resolution. It shows that the $k^{\text{th}}$ term of the Godement resolution is concentrated over $X^{(k)}$. This will ultimately provide an upper bound for the injective dimension of sheaves.
\begin{lemma}\label{lem0}
Let $X$ be a topological space and $F$ be a sheaf of $\mathbb{Q}$-modules over $X$. Then for every $k\in\mathbb{N}_0$ we have that $C^k(F)_x=0$ for every $x \in X\setminus X^{(k)}$.
\end{lemma}
\begin{proof}
We prove this using mathematical induction. For $k=0$ we will first calculate $C^0(F)_x$ for $x$ isolated. Firstly we have ${C^0(F)}_x=\underset{U\backepsilon\, x}{\colim}\underset{y\in U}{\prod} F_y$ where $U$ ranges across all neighbourhoods of $x$. Since $x$ is isolated it is clear that $\left\lbrace x\right\rbrace$ is the minimal neighbourhood of $x$, so we have that $C^0(F)_x=F_x$ as well as the fact that the monomorphism ${\delta_0}_x$ is an isomorphism. In particular this says that $\coker{\delta_0}_x=0$. It therefore follows that $C^1(F)_x=0$ since 
\begin{align*}
C^1(F)_x=C^0(\coker{\delta_0})_x={\coker{\delta_0}}_x=0.
\end{align*}
Suppose ${\coker\delta_{k-1}}_x=0$ and hence $C^k(F)_x=0$ on $X\setminus X^{(k)}$, and take any $x \in X\setminus X^{(k+1)}$ for some $k\in\mathbb{N}$. First observe that $X\setminus X^{(k+1)}\supseteq X\setminus X^{(k)}$. If it happens that $x \in X\setminus X^{(k)}$ and hence has height less than $k$, then by hypothesis ${\coker\delta_{k-1}}_x=0$ and hence $C^k(F)_x=0$. Therefore since ${\coker{\delta_{k}}}_x$ is a quotient of $C^k(F)_x$ which is zero, it follows that ${\coker{\delta_{k}}}_x=0$. Since any point $y$ which accumulates at $x$ has scattered height less than that of $x$, and hence less than $k$, it follows that ${\coker{\delta_{k}}}_y=0$. We can then see immediately that $C^{k+1}(F)_x=\underset{V\backepsilon\, x}{\colim} \underset{y\in V}{\prod}{\coker{\delta_{k}}}_y=0$.

The final case is the one where $x$ is isolated in $X^{(k)}$ and hence the scattered height of $x$ is equal to $k$. This case yields the following diagram:
\begin{center}
$\xymatrix{&C^k(F)\ar[dr]&&C^{k+1}(F)\\\coker{\delta_{k-1}^{\prime}}\ar[ur]^{\delta_{k}^{\prime}}&&\coker{\delta_{k}^{\prime}}\ar[ur]^{\delta_{k+1}^{\prime}}}$
\end{center}
Observe that all of the points $y$ accumulating at $x$ satisfy that the scattered height of $y$ is less than that of $x$ and hence less than $k$. It follows that ${\coker{\delta_{k-1}^{\prime}}}_y=0$ by the inductive hypothesis for every such $y$. Similar to the $k=0$ step above we have: 
\begin{align*}
C^{k}(F)_x=\underset{U\backepsilon\, x}{\colim}\underset{y\in U}{\prod}\,{\coker{\delta_{k-1}}}_y={\coker{\delta_{k-1}}}_x
\end{align*}
and that ${\delta_{k}^{\prime}}_x$ is an isomorphism so ${\coker{\delta_{k}^{\prime}}}_x=0$. All of the points which accumulate at $x$ must belong to $X\setminus X^{(k)}$ and we have already shown that these points $y$ satisfy ${\coker{\delta_{k}^{\prime}}}_y=0$. This information combined proves that $C^{k+1}(F)_x=\underset{U\backepsilon x}{\colim} \underset{y\in U}{\prod}{\coker{\delta_{k}}}_y=0$.
\end{proof}

Recall the following well-known proposition from category theory which will prove useful and is seen in \cite[Proposition 2.3.10]{Weibel}.
\begin{proposition}\label{Adjointinj}
If $\left(F,G\right)$ is an adjoint pair where
\begin{align*}
F\colon\mathfrak{C}\rightarrow \mathfrak{D}\,\, \text{and}\,\, G\colon\mathfrak{D}\rightarrow \mathfrak{C}
\end{align*}
are functors of abelian categories satisfying that $F$ preserves monomorphisms, then $G$ preserves injective objects.
\end{proposition}
\begin{proposition}\label{Godinj}
If $F$ is a sheaf of $\mathbb{Q}$-modules over $X$ then $C^k(F)$ is injective in the category of sheaves of $\mathbb{Q}$-modules.
\end{proposition}
\begin{proof}
From the inductive way that $C^k(F)$ is defined it is sufficient to prove that $C^0(F)$ is injective. Remark \ref{skyprod} suggests that it is sufficient to prove that each $\iota_x(F_x)$ is injective since products of injective objects are injective. 

This follows from Proposition \ref{Adjointinj} applied to the adjoint pair of functors in Example \ref{Examplesky}, $\left(Ev_x(-),\iota_x(-)\right)$, where $Ev_x(F)=F_x$ for a sheaf $F$.  

Note that the left adjoint preserves monomorphisms since a monomorphism of sheaves is a morphism of sheaves such that the map at each stalk is a monomorphism of $\mathbb{Q}$-modules. Using that each $F_x$ is a $\mathbb{Q}$-module, we apply the fact that every object in the category of $\mathbb{Q}$-modules is injective to deduce that $\iota_x(F_x)$ is an injective sheaf.
\end{proof}

We next look at a lemma which helps us with our injective dimension calculations since it will ultimately enable us to calculate $\text{Ext}$ groups. 

Recall from Definition \ref{Godement} that if $x\in X$ and $k<\text{ht}(X,x)$ then:
\begin{align*}
{\coker\delta_k}_x=\left[\underset{U\backepsilon\,x}{\colim}\underset{y\in U}{\prod}{\coker\delta_{k-1}}_y\right]/\text{Im}S
\end{align*}
where $S$ is the serration map from Remark \ref{sermapdef}. Explicitly if $a\in\coker{\delta_{k-1}}_x$ we define $(a,\underline{0})_x$ to be the element in $\coker{\delta_k}_x$ which is the germ at $x$ of the family which is $a$ in place $x$ and zero elsewhere. 
\begin{lemma}\label{Homologychar}
Suppose $X$ is a space with $\rank_{CB}(X)=n$ for $n\in \mathbb{N}$ such that $X^{(n)}=\emptyset$. Then for $j\leq n-1$, $x\in X^{(j)}$ and $F$ a sheaf over $X$, we have an isomorphism $\hom(\iota_x(\mathbb{Q}),C^k(F))\cong \coker{\delta_{k-1}}_x$ for $k<j$, and the map:
\begin{align*}
{\delta_{k+1}}_*:\hom(\iota_x(\mathbb{Q}),C^k(F))\rightarrow \hom(\iota_x(\mathbb{Q}),C^{k+1}(F))
\end{align*}
is given by the map:
\begin{align*}
\alpha_{k+1}:\coker{\delta_{k-1}}_x&\rightarrow \coker{\delta_{k}}_x\\a&\mapsto \left(a,\underline{0}\right)_x.
\end{align*}
\end{lemma}
\begin{proof}
Firstly notice that $C^k(F)$ is defined to be $C^0(\coker\delta_{k-1})$, so we begin by proving that $\hom(\iota_x(\mathbb{Q}),C^0(F))\cong F_x$. Observe that $C^0(F)=\underset{y\in X}{\prod}\iota_y(F_y)$ so we can write:
\begin{align*}
\hom(\iota_x(\mathbb{Q}),C^0(F))&=\hom(\iota_x(\mathbb{Q}),\underset{y\in X}{\prod}\iota_y(F_y))=\underset{y\in X}{\prod}\hom(\iota_x(\mathbb{Q}),\iota_y(F_y))\\&=\underset{y\in X}{\prod}\hom(\iota_x(\mathbb{Q})_y,F_y)=F_x.
\end{align*}
In particular if $f\in \hom(\iota_x(\mathbb{Q}),C^k(F))$, then this is determined by a map in $\hom(\iota_x(\mathbb{Q}),\iota_x(\coker{\delta_{k-1}}_x))$. This corresponds to a point $f_x\in \coker{\delta_{k-1}}_x$, so $f$ is given by the element: 
\begin{align*}
[f_x,\underline{0}]_x\in \underset{V\backepsilon\, x}{\colim}\left(\underset{y\in V}{\prod}\coker{\delta_{k-1}}_y\right)=C^0(\coker{\delta_{k-1}})_x,
\end{align*}
with this germ at $x$ of the family taking value $f_x$ in position $x$ and $0$ elsewhere. It follows that ${\delta_{k+1}}_*(f)$ corresponds to $\delta_{k+1}([f_x,\underline{0}]_x)$.

But we therefore have:
\begin{align*}
\delta_{k+1}([f_x,\underline{0}])=\left(\left[\left(f_x,\underline{0}\right)_x\right]^S,\left(\left[\left(f_x,\underline{0}\right)_y\right]^S\right)_{y\in U}\right)_x
\end{align*}
in $\underset{V\backepsilon\, x}{\colim}\left(\underset{y\in V}{\prod}\coker{\delta_{k}}_y\right)=C^0(\coker{\delta_{k}})_x$, where $\left[-\right]^S$ represents the class in the cokernel of the map $S$. This follows from the definition of the maps $\delta$ in Definition \ref{Godement}.

However if $x\neq y$, then since $X$ is Hausdorff there is a neighbourhood of $y$ not containing $x$ so $\left[\left(f_x,\underline{0}\right)_y\right]^S=\left[\left(\underline{0}\right)_y\right]^S$. Therefore $\delta_{k+1}\left(\left(f_x,\underline{0}\right)_x\right)$ can be written as $\left(\left[\left(f_x,\underline{0}\right)_x\right]^S,\left[\left(\underline{0}\right)_y\right]^S\right)_{y\in U}$.

We therefore have that ${\delta_{k+1}}_*$ is defined as follows:
\begin{align*}
{\delta_{k+1}}_*:\hom(\iota_x(\mathbb{Q}),C^k(F))&\rightarrow \hom(\iota_x(\mathbb{Q}),C^{k+1}(F))\\ [f_x,\underline{0}]_x &\mapsto \left(\left[\left(f_x,\underline{0}\right)_x\right]^S,\left[\left(\underline{0}\right)_y\right]^S\right)_{y\in U}
\end{align*} 

Since we are interested in what the maps correspond to as maps between the $x$ components of the products of $C^0(\coker\delta_{k-1})$ and $C^0(\coker\delta_k)$, we observe that it sends $f_x$ to $\left[\left(f_x,\underline{0}\right)_x\right]^S$.
\end{proof}
Now we consider the preceeding lemma for points in $X$ which either have infinite height or belong to the hull of $X$.
\begin{lemma}\label{homoloinf}
Suppose $X$ is a space with infinite Cantor-Bendixson rank. Then for $x\in X^{(n)}$ for any $n\in \mathbb{N}$, and $F$ a sheaf over $X$, we have an isomorphism $\hom(\iota_x(\mathbb{Q}),C^k(F))\cong \coker{\delta_{k-1}}_x$ for $k<n$, and the map:
\begin{align*}
{\delta_{k+1}}_*:\hom(\iota_x(\mathbb{Q}),C^k(F))\rightarrow \hom(\iota_x(\mathbb{Q}),C^{k+1}(F))
\end{align*}
is given by the map:
\begin{align*}
\alpha_{k+1}:\coker{\delta_{k-1}}_x&\rightarrow \coker{\delta_{k}}_x\\a&\mapsto \left(a,\underline{0}\right)_x.
\end{align*}
This also holds for points of infinite height and points in the hull.
\end{lemma}
\begin{proof}
The proof of this result is almost the same as the proof of Lemma \ref{Homologychar}. The only differences arise from the fact that in Lemma \ref{Homologychar} $X$ has finite Cantor-Bendixson rank, say $n$, and so $\coker{\delta_{k}}$ is zero if $k$ is bigger than $n$. Therefore the argument in Lemma \ref{Homologychar} is only interesting provided we are applying it to $\coker{\delta_{k}}$ for $k$ small enough. If $X$ has infinite Cantor-Bendixson rank then we know that for each $n\in \mathbb{N}$ there exists a point $x_n$ with height $n$, as well as points of infinite height. Therefore the argument on $\coker{\delta_{k}}$ doesn't become trivial eventually. The same observation is true if $X$ has any point in the perfect hull of $X$.
\end{proof}
The previous two lemmas indicate how the calculations in this thesis differ from sheaf cohomology. In this setting we are applying the functor $\hom\left(\iota_x\left(\mathbb{Q},-\right)\right)$ to the injective resolution whereas in the case of sheaf cohomology we apply the functor $\hom\left(A,-\right)$ for a constant sheaf $A$.  
\section{Injective dimension calculation}\label{Sheafexample}
We split this section into two parts. Firstly we give an example and show explicitly the calculations in a specific case to illustrate the overall idea. For the second, we seek to generalise the examples to a proof for more general cases.

We now formally give the definition of the injective dimension of sheaves of $\mathbb{Q}$-modules over a space $X$ as seen in \cite[Definition 4.1.1, Definition 10.5.10]{Weibel}.
\begin{definition}\index{injective dimension of a sheaf}
The \textbf{injective dimension} of a sheaf $F$ over $X$ denoted by $\ID(F)$ is the minimum positive integer $n$ (if it exists) such that there is an injective resolution of the form
\begin{center}
$\xymatrix{0\ar[r]&X\ar[r]^{\varepsilon}&I_0\ar[r]^{f_0}&I_1\ar[r]^{f_1}&\ldots\ar[r]^{f_{n-1}} &I_n\ar[r]&0.}$
\end{center}
It is infinite if such a value doesn't exist.
\end{definition}
From \cite[Theorem 4.1.2]{Weibel} we define the injective dimension of the category of sheaves of $\mathbb{Q}$-modules to be:
\begin{align*}
\sup\left\lbrace \ID(F)\mid\,F\,\in\text{Sheaf}_{\mathbb{Q}}(X)\right\rbrace, 
\end{align*}
where $\text{Sheaf}_{\mathbb{Q}}(X)$ denotes the category of sheaves of $\mathbb{Q}$-modules over $X$. We can now verify that the injective dimension of sheaves of $\mathbb{Q}$-modules is bounded above for a particular class of profinite space.
\begin{proposition}
If $X$ is a scattered space with $\rank_{\cb}(X)=n$ for $n\in\mathbb{N}$ then the injective dimension of sheaves of $\mathbb{Q}$-modules over $X$ is bounded above by $n-1$.
\end{proposition}
\begin{proof} 
By Proposition \ref{Godinj} the Godement resolution is an injective resolution. An application of Lemma \ref{lem0} shows that the terms of the Godement resolution are zero after term $n-1$. Therefore the injective dimension of each sheaf is less than or equal to $n-1$.
\end{proof}
In order to get equality it is sufficient to find a particular sheaf $F$ for which $\ID(F)\geq \rank_{CB}(X)-1$. To achieve this we look at \cite[Lemma 4.1.8, Exercise 10.7.2]{Weibel} which says $\ID(F)\leq \rank_{CB}(X)-2$ if and only if \\$\text{Ext}^{\rank_{CB}(X)-1}(A,F)=0$ for every sheaf $A$. In particular if we can find sheaves $A$ and $F$ such that $\text{Ext}^{\rank_{CB}(X)-1}(A,F)\neq 0$ then we must have that \\$\ID(F)>\rank_{CB}(X)-2$. This then forces $\ID(F)=\rank_{CB}(X)-1$.

Notice that the concept of an $\text{Ext}$ group in the category of $R$-modules for a ring $R$ stated above in \cite[Chapter 4]{Weibel} holds for general abelian categories by \cite[Corollary 10.7.5]{Weibel}. This is formally stated in Theorem \ref{extcharacter}.
\subsection{Example case}
Consider the space $P=\left\lbrace \frac{1}{n}\mid n\in \mathbb{N}\right\rbrace\bigcup \left\lbrace 0\right\rbrace$, where the topology is given by the subspace topology with respect to $\mathbb{R}$. As seen in Proposition \ref{padicspace} this space is homeomorphic to the space of closed subgroups of $\mathbb{Z}_p$. For convenience we shall denote the point $0$ by $\infty$.
\begin{proposition}\label{godelengthP1}
If $c\mathbb{Q}$ is the constant sheaf of rational numbers over $P$ then the Godement resolution for $c\mathbb{Q}$ is of length $1$.
\end{proposition} 
\begin{proof}
Let $F=c\mathbb{Q}$, then for any $x\in X$ we have:
\begin{align*}
\underset{U\backepsilon\, x}{\colim}\, c\mathbb{Q}(U)=\mathbb{Q}_x=\underset{U\backepsilon\, x}{\colim}\,Pc\mathbb{Q}(U),
\end{align*}
where $Pc\mathbb{Q}$ is the presheaf and $\mathbb{Q}_x$ is the stalk of $c\mathbb{Q}$ at $x$. Therefore any $s_x\in F_x$ can be represented by a section in $Pc\mathbb{Q}(U)=\mathbb{Q}$. Given $q_x\in \mathbb{Q}=F_x$ we have a map:
\begin{align*}
\mathbb{Q}&\rightarrow \underset{y\in U}{\prod}\mathbb{Q}\rightarrow \underset{V\backepsilon x}{\colim}\underset{y\in V}{\prod}\mathbb{Q}\\q&\mapsto (q_y)_{y\in U}\mapsto \left((q_y)_{y\in V}\right)_x
\end{align*}
which induces a map:
\begin{align*}
S:F_x&\rightarrow \underset{V\backepsilon x}{\colim}\underset{y\in V}{\prod}\mathbb{Q}\\q_x&\mapsto \left((q_y)_{y\in U}\right)_x
\end{align*}
which we call the serration map and we denote it by $S$. The codomain of this map is $C^0(F)_x$ so the map of sheaves induced in this way is denoted $\delta$. For the isolated points $x=\left\lbrace\frac{1}{n}\right\rbrace$ we have ${\delta}_x$ is the identity so that $\coker{\delta}_x=0$. For the point $\infty$ we have:
\begin{align*}
\coker{\delta}_x=\left(\underset{n}{\colim}\underset{\infty\geq a\geq n}{\prod}\mathbb{Q}\right)/\text{Im}S.
\end{align*}
This is non-zero since for example we can take the element represented by taking $0$ at point $\infty$ and $1$ elsewhere. We next observe that $C^1(F)\cong\iota_{\infty}(\coker\delta_{\infty})$, hence $C^2(F)=0$.
\end{proof}
\begin{proposition}
The injective dimension of $c\mathbb{Q}$ in the category of sheaves over $P$ is $1$. 
\end{proposition}
\begin{proof}
Let $F=c\mathbb{Q}$, from Proposition \ref{godelengthP1} we have an injective resolution for $F$ of the form:
\begin{align*}
\xymatrix{0\ar[r]&F\ar[r]^{\delta_0}&I^0(F)\ar[r]^{\delta_1}&I^1(F)\ar[r]&0}
\end{align*}
We know that $x=\infty$ is the point with maximal height in $P$ so we apply the functor $\hom(\iota_x(\mathbb{Q}),-)$ to the above resolution to get:
\begin{align*}
\xymatrix{0\ar[r]&\hom(\iota_x(\mathbb{Q}),I^0(F))\ar[r]^{{\delta_1}_*}&\hom(\iota_x(\mathbb{Q}),I^1(F))\ar[r]&0}
\end{align*}
where we forget the $F$ term. Applying Lemma \ref{Homologychar} gives:
\begin{align*}
\xymatrix{0\ar[r]&\mathbb{Q}\ar[r]^-{\alpha_1}&\coker{\delta_0}_x\ar[r]&0}
\end{align*}
Therefore the problem boils down to verifying that the following map is not surjective:
\begin{align*}
\alpha_1:\mathbb{Q}_x&\rightarrow \coker{\delta_0}_x\\s_x&\mapsto \left[\left(s_x,\underline{0}\right)_x\right]^S
\end{align*} 
where $\left[-\right]^S$ is the cokernel class of an element in the cokernel of the map $S$. We consider the element in $\coker{\delta_0}_x$: 
\begin{align*}
\left[\left(0_x,\underline{1}\right)_x\right]^S
\end{align*}
This cannot be in the image of $\alpha_1$ and therefore $\text{Ext}^1\left(\iota_x(\mathbb{Q}),c\mathbb{Q}\right)\neq 0$, hence $F$ has injective dimension $1$.
\end{proof}
The following corollary is now immediate.
\begin{corollary}
The injective dimension of the category of sheaves of $\mathbb{Q}$-modules over $P$ is $1$.
\end{corollary}

The calculations of the space $P$ is too simple to get an idea of the general case so we proceed to consider a more difficult example. In this space each of the points $(\frac{1}{m},\frac{1}{n})$ for $m,n< \infty$ are isolated points. The neighbourhoods of the points of the form $(0,\frac{1}{m})$ for $m< \infty$ are of the form $U_{(n,-)}=\left\lbrace \left(\frac{1}{j},\frac{1}{m}\right)\mid \infty\geq j\geq n\right\rbrace$. We include infinity as an index so that $\frac{1}{\infty}$ can denote the limit point $0$. Similarly a neighbourhood of $(\frac{1}{m},0)$ is the corresponding set of the form $U_{(-,n)}$. For $(0,0)$ we consider sets of the form $U_{(n,m)}=\left\lbrace (\frac{1}{a},\frac{1}{b})\mid \infty\geq a\geq n,\infty\geq b\geq m\right\rbrace$. This topology coincides with the subspace topology with respect to $\mathbb{R}^2$.  For the next calculation we are interested in the space $P^2$ and for convenience we will write $\infty$ in place of $0$.
\begin{proposition}\label{Godeg}
If $c\mathbb{Q}$ is the constant sheaf of rational numbers over $P^2$ then the Godement resolution for $c\mathbb{Q}$ is of length $2$.
\end{proposition}
\begin{proof}
As in the case for $P$, let $F=c\mathbb{Q}$. Then for any $x\in X$ we have 
\begin{align*}
\underset{U\backepsilon\, x}{\colim}\, c\mathbb{Q}(U)=\mathbb{Q}_x=\underset{U\backepsilon\, x}{\colim}\,Pc\mathbb{Q}(U), 
\end{align*}
where $Pc\mathbb{Q}$ is the presheaf and $\mathbb{Q}_x$ is the stalk of $c\mathbb{Q}$ at $x$. Therefore any $s_x\in F_x$ can be represented by a section in $Pc\mathbb{Q}(U)=\mathbb{Q}$. Given $q_x\in \mathbb{Q}=F_x$ we have a map:
\begin{align*}
\mathbb{Q}&\rightarrow \underset{y\in U}{\prod}\mathbb{Q}\rightarrow \underset{V\backepsilon \,x}{\colim}\underset{y\in V}{\prod}\mathbb{Q}\\q&\mapsto (q_y)_{y\in U}\mapsto \left((q_y)_{y\in V}\right)_x
\end{align*}
which induces a map:
\begin{align*}
S:F_x&\rightarrow \underset{V\backepsilon\, x}{\colim}\underset{y\in V}{\prod}\mathbb{Q}\\q_x&\mapsto \left((q_y)_{y\in U}\right)_x
\end{align*}
which we call the serration map and we denote it by $S$. Here $\left(-\right)_x$ represents the germ at $x$. The codomain of this map is $C^0(F)_x$ so the map of sheaves induced in this way is denoted $\delta_0$. For the isolated points we have ${\delta_0}_x$ is the identity so that $\coker{\delta_0}_x=0$. For the points of the form $x=(\infty,m)$ we have:
\begin{align*}
\coker{\delta_0}_x=\left(\underset{n\in\mathbb{N}}{\colim}\underset{\infty\geq a\geq n}{\prod}\mathbb{Q}\right)/\text{Im}S.
\end{align*}
We have similar for a point of the form $(m,\infty)$ and for $x=(\infty,\infty)$ we have:
\begin{align*}
\coker{\delta_0}_x=\left(\underset{m,n\in\mathbb{N}}{\colim}\underset{\infty\geq a\geq m, \infty\geq b \geq n}{\prod}\mathbb{Q}\right)/\text{Im}S. 
\end{align*}
These are non-zero for the three types of limit point $x$ since we can consider $\left[\left(0_x,\underline{1}\right)_x\right]^S$, where $0_x$ denotes that $0$ is in component $x$, $\underline{1}$ is the constant $1$ sequence and $\left[-\right]^S$ is the class of an element in the cokernel of $S$. Here we use that a germ over $x$ is zero if and only if there exists a section representative with domain an open neighbourhood of $x$ which is zero. Namely if $S(a)=\left[\left(0_x,\underline{1}\right)_x\right]^S$ then $a_x=0$ and so there is an open neighbourhood $U$ of $x$ such that $a_y=0$ for $y\in U$. However the definition of the serration map shows that $a_y=1$ also for $y\neq x$ which is a contradiction.  

We now construct $C^1(F)=C^0(\coker\delta_0)$ and construct a morphism $\delta_1$ from $\coker\delta_0$ to $C^0(\coker\delta_0)$ in a similar manner. Similar to before
\begin{align*}
\underset{U\backepsilon x}{\colim} \coker\delta_0(U)=\underset{U\backepsilon\,x}{\colim}P\coker\delta_0(U)
\end{align*}
so any $s_x\in \coker{\delta_0}_x$ can be represented by a section in $P\coker{\delta_0}(U)=\underset{y\in U}{\prod}\mathbb{Q}/S$, so as seen previously this gives a map:
\begin{align*}
\coker{\delta_0}_x&\rightarrow C^1(F)_x\\ \left[\left((q_y)_{y\in U}\right)_x\right]^S&\mapsto \left(\left(\left[\left((q_y)_{y\in U}\right)_z\right]^S\right)_{z\in U}\right)_x.
\end{align*}
The map ${\delta_1}_x$ for $x=(\infty,m)$ is the identity by constructing $\delta_1$ using the serration map and from the fact that: 
\begin{align*}
C^1(F)_x=\underset{n\in \mathbb{N}}{\colim}\underset{\infty\geq a\geq n}{\prod}\coker{\delta_0}_{(a,m)}=\coker{\delta_0}_x.
\end{align*}
Similarly for points of the form $(m,\infty)$. This shows in particular that \\$\coker{\delta_1}_x=0$ for points of that form. If $x=(\infty,\infty)$ then: 
\begin{align*}
\coker{\delta_1}_x=\left(\underset{m,n\in\mathbb{N}}{\colim}\underset{\infty\geq a\geq m, \infty\geq b \geq n}{\prod}\coker{\delta_0}_{(a,b)}\right)/\text{Im}S.
\end{align*}
This is non-zero since $0\neq\coker{\delta_0}_y$ for any limit point $y$ in $P^2$ accumulating at $x$ so we can apply the same argument as mentioned above in the case of $\coker{\delta_0}_x$. If $U$ is any neighbourhood of $x$ then $U$ contains infinitely many limit points, so for every such $y\in U$ take $s^y\neq 0$ in $\coker{\delta_0}_y$. Then $\coker{\delta_1}_x$ has a non-zero element of the form:
\begin{align*}
\left[\left((0_x,s^y)_{y\in U\setminus \left\lbrace x\right\rbrace}\right)_x\right]^S
\end{align*}
which is non-zero by the same argument as above involving the characterisation of zero germs. 

At the next stage we see that $C^2(F)$ is only concentrated at $x=(\infty,\infty)$ so ${\delta_2}_x$ is the identity map, hence $\coker\delta_2=0$ and $C^2(F)\neq 0$. 
\end{proof}
\begin{proposition}\label{ID2}
The injective dimension of $c\mathbb{Q}$ in the category of sheaves over $P^2$ is $2$. 
\end{proposition}
\begin{proof}
Let $F=c\mathbb{Q}$, from Proposition \ref{Godeg} we have an injective resolution for $F$ of the form:
\begin{align*}
\xymatrix{0\ar[r]&F\ar[r]^{\delta_0}&I^0(F)\ar[r]^{\delta_1}&I^1(F)\ar[r]^{\delta_2}&I^2(F)\ar[r]&0.}
\end{align*}
We know that $x=(\infty,\infty)$ is the point with maximal height in $P^2$ so we apply the functor $\hom(\iota_x(\mathbb{Q}),-)$ to the above resolution to get:
\begin{align*}
\xymatrix{0\ar[r]&\hom(\iota_x(\mathbb{Q}),I^0(F))\ar[r]^{{\delta_1}_*}&\hom(\iota_x(\mathbb{Q}),I^1(F))\ar[d]^{{\delta_2}_*}\\&&\hom(\iota_x(\mathbb{Q}),I^2(F))\ar[r]&0}
\end{align*}
where we forget the $F$ term. Applying Lemma \ref{Homologychar} gives:
\begin{align*}
\xymatrix{0\ar[r]&\mathbb{Q}\ar[r]^-{\alpha_1}&\coker{\delta_0}_x\ar[r]^{\alpha_2}&\coker{\delta_1}_x\ar[r]&0.}
\end{align*}
Therefore the problem boils down to verifying that the following map is not surjective:
\begin{align*}
\alpha_2:\coker{\delta_0}_x&\rightarrow \coker{\delta_1}_x\\s_x&\mapsto \left[\left(s_x,\underline{0}\right)_x\right]^S.
\end{align*} 
We can define a family $s=(s^y)_{y\in U}$ in $\underset{y\in U}{\prod}{\coker\delta_0}_y$ where $s^y\neq 0$ for $y$ a limit point and $U$ a neighbourhood of $x$. We consider the element in $\coker{\delta_1}_x$: 
\begin{align*}
\left[\left(0_x,s\right)_x\right]^S.
\end{align*}
This cannot be in the image of $\alpha_2$ and therefore $\text{Ext}^2\left(\iota_x(\mathbb{Q}),c\mathbb{Q}\right)\neq 0$, hence $F$ has injective dimension $2$.
\end{proof}
\begin{corollary}
The injective dimension of the category of sheaves of $\mathbb{Q}$-modules over $P^2$ is $2$.
\end{corollary}
\begin{proof}
This follows from Proposition \ref{ID2}.
\end{proof}
\subsection{General case}
We now work towards generalising this argument to abstract spaces. We will look at the following lemma which will illustrate some interesting and useful properties of the Godement resolution which will help us to generalise the calculations seen in the previous subsection. We will look at the following lemma which will illustrate that the Godement resolution is non-zero at term $k$ provided $k$ is less than $\rank_{\cb}(X)$.

\begin{lemma}\label{godementnon0}
Let $X$ be a non-empty scattered space and $k\in \mathbb{N}$ be less than or equal to $\rank_{\cb}(X)$. For each $x \in X^{(k)}$, $\coker{\delta_{k-1}}_x \neq 0$ in the Godement resolution of $c\mathbb{Q}$ the constant sheaf at $\mathbb{Q}$. 

Furthermore if $X$ is any space with a non-empty perfect hull or infinite Cantor-Bendixson rank and $k \in \mathbb{N}$, then for each $x \in X^{(k)}$ we have $\coker{\delta_{k-1}}_x \neq 0$ in the Godement resolution of $c\mathbb{Q}$.
\end{lemma}
\begin{proof}
We prove this using an induction argument. Since 
\begin{align*}
\underset{U\backepsilon\, x}{\colim}\,c\mathbb{Q}(U)=c\mathbb{Q}_x=\underset{U\backepsilon\,x}{\colim}\,Pc\mathbb{Q}(U)=\mathbb{Q}
\end{align*}
where $Pc\mathbb{Q}$ represents the presheaf, any $q_x\in \mathbb{Q}_x$ is represented by some $q\in \mathbb{Q}$. We therefore have the following diagram by \cite[pp. 36, 37]{Bredon}:
\begin{align*}
\mathbb{Q}&\rightarrow\underset{y\in U}{\prod}\mathbb{Q}\rightarrow\underset{V\backepsilon\, x}{\colim} \underset{y\in V}{\prod}\mathbb{Q}\\q&\mapsto(q_y)_{y\in U}\mapsto\left((q_y)_{y\in U}\right)_x
\end{align*} 
which induces a map:
\begin{align*}
\mathbb{Q}&\rightarrow \underset{V\backepsilon\, x}{\colim}\underset{y \in V}{\prod}\mathbb{Q}\\q_x&\mapsto \left((q_y)_{y\in U}\right)_x.
\end{align*}
We call this map the serration map and denote it by $S$. This is not surjective since we have a point $\left(0_x,\underline{1}\right)_x$ not in the image of $S$. This point was defined in Proposition \ref{Godeg}. The proof that this is non-zero is similar to the argument given in Proposition \ref{Godeg}. Namely if $S(a)=\left[\left(0_x,\underline{1}\right)_x\right]^S$ then $a_x=0$ and so there is an open neighbourhood $U$ of $x$ such that $a_y=0$ for $y\in U$. However the definition of the serration map shows that $a_y=1$ also for $y\neq x$ which is a contradiction. Therefore $\coker{\delta_0}_x\neq 0$ for $x\in X^{(1)}$.

Suppose this holds up to some $n\in\mathbb{N}$ and for any $x\in X^{(n+1)}$. By assumption we have: 
\begin{align*}
0\neq \coker{\delta_n}_x=\underset{U\backepsilon \,x}{\colim}\left(\prod_{y\in U}\coker{\delta_{n-1}}_y\right)/\text{Im}S.
\end{align*}
Using the fact that sheafification preserves stalks of presheaves we have a map using \cite[pp. 36-37]{Bredon} as follows:
\begin{align*}
\left(\underset{y\in U}{\prod}\coker{\delta_{n-1}}_y\right)/S&\rightarrow \underset{z\in U}{\prod}\coker{\delta_n}_z\rightarrow\underset{V\backepsilon \,x}{\colim}\underset{z\in V}{\prod}\coker{\delta_n}_z\\ \left[(a_y)_{y\in U}\right]^S&\mapsto \left(\left[\left((a_y)_{y\in U}\right]^S\right)_z\right)_{z\in U}\mapsto \left(\left(\left(\left[(a_y)_{y\in U}\right]^S\right)_z\right)_{z\in U}\right)_x 
\end{align*}
which induces a map:
\begin{align*}
\coker{\delta_n}_x&\rightarrow \underset{V\backepsilon\, x}{\colim}\underset{y\in V}{\prod}\coker{\delta_n}_y\\ \left(\left[(a_y)_{y\in U}\right]^S\right)_x&\mapsto \left(\left(\left(\left[(a_y)_{y\in U}\right]^S\right)_z\right)_{z\in U}\right)_x.
\end{align*}
Let $U$ be any open neighbourhood of $x$, which we shall assume to have height $n+2$. Then for each $y\in U$ such that $y\in X^{(n)},X^{(n+1)}$ or $X^{(n+2)}$ we can choose $0\neq a_y\in \coker{\delta_{n-1}}_y$ by the inductive hypothesis. Set $s^y=\left(0_y,a_z\right)_{z\in U\setminus\left\lbrace y\right\rbrace}$, then $\left[\left(s^y\right)_y\right]^S\neq 0$ in $\coker{\delta_n}_y$ for $y\in X^{(n+1)},X^{(n+2)}$ and we denote this by $b_y$. This is shown to be non-zero by following a similar argument to that seen in Proposition \ref{Godeg} and earlier in this proof. We can therefore consider for any $x\in X^{(n+2)}$:
\begin{align*}
\left(\left[\left(0_x,b_y\right)_{y\in U\setminus\left\lbrace x\right\rbrace}\right]^S\right)_x
\end{align*}
which is not in the image of the serration map so $\coker{\delta_{n+1}}_x\neq 0$. This is also seen by referring to the previous argument earlier in this proof and by that in Proposition \ref{Godeg}.

Note if $\rank_{\cb}(X)$ is infinite then for each $k\in \mathbb{N}$ we know that each $X^{(k)}$ has isolated points to remove, so this is true for every $k$. If $\rank_{CB}(X)=n$ and $X^{(n)}=\emptyset$ then $X^{(n-1)}$ is discrete and therefore satisfies that $C^n(F)=0$ by Lemma \ref{lem0}. It follows that the argument therefore only results in non-zero stalks for $k \leq n-1$. If $x$ belongs to the perfect hull of $X$ then this argument also holds for each $k\in\mathbb{N}$ since the hull is contained in each $X^{(k)}$.
\end{proof}

We now use the above calculations to verify the injective dimension of sheaves using the Cantor-Bendixson rank.
\begin{theorem}\label{ID_CB}
Suppose $X$ is a space with $\rank_{\cb}(X)=n$ such that $X^{(n)}=\emptyset$. Then the category of sheaves over $X$ has injective dimension equal to $n-1$.
\end{theorem}
\begin{proof}
To see this, we need to find an object in the category of sheaves over $X$ whose injective dimension is $n-1$, we will show that $c\mathbb{Q}$ satisfies $\ID(c\mathbb{Q})=n-1$. Firstly by Lemma \ref{godementnon0} we know that $\ID(c\mathbb{Q})\leq n-1$ since the Godement resolution gives an injective resolution of the form:
\begin{align*}
\xymatrix{0\ar[r]&c\mathbb{Q}\ar[r]^{\delta_0}&I^0\ar[r]^{\delta_1}&\ldots\ar[r]^{\delta_{n-2}}&I^{n-2}\ar[r]^{\delta_{n-1}}&I^{n-1}\ar[r]&0.}
\end{align*} 
We will show that the $\text{Ext}^{n-1}\left(\iota_x(\mathbb{Q}),c\mathbb{Q}\right)$ group calculated by the above injective resolution is non-zero. Let $x$ be an element of $X$ with $\text{ht}(X,x)=n-1$ (any point of $X$ with maximal height). 

We apply the functor $\Hom(\iota_x(\mathbb{Q},-)$ and forget the $c\mathbb{Q}$ term to get:
\begin{align*}
\xymatrix{\Hom(\iota_x(\mathbb{Q}),I^0)\ar[r]^-{{\delta_1}_*}&\Hom(\iota_x(\mathbb{Q}),I^1)\ar[r]^-{{\delta_2}_*}&\ldots\ar[r]^-{{\delta_{n-2}}_*}&\Hom(\iota_x(\mathbb{Q}),I^{n-2})\ar[d]^{{\delta_{n-2}}_*}\\&&0&\Hom(\iota_x(\mathbb{Q}),I^{n-1})\ar[l]}
\end{align*}
which we can no longer assume to be exact. This is equal to the following sequence:
\begin{align*}
\xymatrix{\mathbb{Q}\ar[r]^-{\alpha_1}&\coker{\delta_0}_x\ar[r]^-{\alpha_2}&\ldots\ar[r]^-{\alpha_{n-3}}&\coker{\delta_{n-4}}_x\ar[d]^{\alpha_{n-2}}\\&0&\coker{\delta_{n-2}}_x\ar[l]^{\alpha_{n}}&\coker{\delta_{n-3}}_x\ar[l]^{\alpha_{n-1}}}
\end{align*}
We want to show that $\text{Ext}^{(n-1)}\left(\iota_x(\mathbb{Q}),c\mathbb{Q}\right)=\ker\alpha_n/\im\alpha_{n-1}\neq 0$ and\\ $\text{Ext}^n\left(\iota_x(\mathbb{Q}),c\mathbb{Q}\right)=\ker \alpha_{n+1}/\im \alpha_n=0$. It is clear that $\text{Ext}^n\left(\iota_x(\mathbb{Q}),c\mathbb{Q}\right)=0$. For the other we need to prove that the map:
\begin{align*}
\alpha_{n-1}:\coker {\delta_{n-3}}_x&\rightarrow \coker {\delta_{n-2}}_x\\a&\mapsto\left(a,\underline{0}\right)_x
\end{align*}
is not surjective. This is done in a similar fashion to Propositions \ref{Godeg} and \ref{godementnon0}.

For any open neighbourhood $U$ of $x$ there are infinitely many points $z$ of $U$ such that $z\in X^{(n-2)}$ and $\coker{\delta_{n-3}}_z\neq 0$ so we can choose such a point $a_z$ for each $z$. Consider:
\begin{align*}
a=\left[\left((0_x,a_z)_{z\in U\setminus\left\lbrace x\right\rbrace}\right)_x\right]^S\in \coker{\delta_{n-2}}_x. 
\end{align*}
If $t_x\in \coker{\delta_{n-3}}_x$ is in the preimage of $a$ with respect to $\alpha_{n-1}$ we would have:
\begin{align*}
\left[\left(t_x,0\right)_x\right]^S=\alpha_{n-1}(t_x)=\left[\left((0_x,a_z)_{z\in U\setminus\left\lbrace x\right\rbrace}\right)_x\right]^S
\end{align*}
so $t_x=0$ which implies that $\left[\left((0_x,a_z)_{z\in U\setminus\left\lbrace x\right\rbrace}\right)_x\right]^S=0$. But this cannot be the case since there are infinitely many $z$ satisfying that $a_z\neq \underline{0}$ by construction, so we have a contradiction and $\alpha_{n-1}$ cannot be surjective.
\end{proof}
We now deal with the case where the Cantor-Bendixson dimension is infinite.
\begin{theorem}
If $X$ is a space with infinite Cantor-Bendixson rank, then the injective dimension of sheaves of $\mathbb{Q}$-modules over $X$ is infinite.
\end{theorem}
\begin{proof}
Since the Cantor-Bendixson rank of $X$ is infinite there exists a sequence of points $x_n$ each having height $n$. As a consequence of Theorem \ref{ID_CB} for each $x_n$ we know that $\text{Ext}^n\left(\iota_{x_n}(\mathbb{Q}),c\mathbb{Q}\right)\neq 0$, and this happens for each $n$ since we do not have a maximal height. This gives the result. 
\end{proof}
We are left to deal with the case where $X$ is not scattered but has finite Cantor-Bendixson rank. In the above cases when we resolve with respect to $\iota_x\left(\mathbb{Q}\right)$, the resolved sequence becomes zero after term $\text{ht}(x)-1$, so the kernel of the map $\alpha_{\text{ht}(x)+1}$ is $\coker{\delta_{\text{ht}(x)-1}}_x$. This is advantageous since we can choose any point of $\coker{\delta_{\text{ht}(x)-1}}$ not in the image of $\alpha_{\text{ht}(x)}$. This changes when we are working in the case where $X$ has a perfect hull.

We know that the following Godement resolution is infinite:
\begin{align*}
\xymatrix{0\ar[r]&c\mathbb{Q}\ar[r]^{\delta_0}&I^0\ar[r]^{\delta_1}&\ldots\ar[r]^{\delta_{n-1}}&I^{n-1}\ar[r]^{\delta_{n}}&I^n\ar[r]^{\delta_n}&\ldots}
\end{align*} 
Therefore after resolving like above for a point $x$ in the perfect hull we obtain the following infinite sequence:
\begin{align*}
\xymatrix{0\ar[r]&\mathbb{Q}\ar[r]^-{\alpha_1}&\coker{\delta_0}_x\ar[r]^-{\alpha_2}&\ldots\ar[r]^-{\alpha_{n-2}}&\coker{\delta_{n-3}}_x\ar[d]^{\alpha_{n-1}}\\&&\ldots&\coker{\delta_{n-1}}_x\ar[l]^{\alpha_{n+1}}&\coker{\delta_{n-2}}_x\ar[l]^{\alpha_{n}}}
\end{align*}
The important thing to notice is that since this is non-zero at infinitely many places, when calculating the group $\text{Ext}^n$ we can't just chose any representative of $\coker{\delta_{n-1}}_x$ since the kernel is not everything.

In order to choose something in the kernel we need to adjust our argument above, namely instead of choosing a representative $(0_x,s^y)_{y\in U\setminus\left\lbrace x\right\rbrace}$ with $0\neq s^y\in\coker{\delta_{n-2}}_y$ arbitrary, we need $(s^y)_{y\in U\setminus\left\lbrace x\right\rbrace}$ to be determined by a section $s$ over $\coker{\delta_{n-2}}$. That is, we want each $s^y$ to be of the form $s_y$ for that section $s$, and such that each open neighbourhood $U$ of $x$ contains infinitely many $y$ such that $s^y\neq 0$.

Recall that a section $s$ over an open neighbourhood $U$ of $x$ has germ $s_x=0$ if and only if $s$ restricts to some smaller neighbourhood to give the zero section. Also recall that we can build a section in $\coker\delta_{n-2}(U)$ by considering $\underset{y\in U}{\prod}\coker\delta_{n-3}/S$. Using these two facts, if we could construct a family $\left[(a^y)_{y\in U\setminus\left\lbrace x\right\rbrace}\right]_S$ to be an alternating family where infinitely many $a^y$ are non-zero in ${\coker\delta_{n-3}}_y$ and infinitely many do equal zero, then we may have a suitable section $s$ to proceed with the proof. This approach needs the following condition to proceed:

If $a^y=0$, then any neighbourhood $U$ of $y$ contains infinitely many points $z$ such that $a^z\neq 0$.

If we can construct given any net converging to $x$, two term-wise disjoint subnets then we can do the above construction to show that the injective dimension of sheaves in this case is infinite, provided the sequence is set up to satisfy the condition. With this in mind we have the following conjecture.
\begin{conjecture}\label{Conject}
If $X$ has finite Cantor-Bendixson rank and non-empty perfect hull then the injective dimension of sheaves of $\mathbb{Q}$-modules over $X$ is infinite.
\end{conjecture}

If we consider a point $K$ of $SG$ the space of closed subgroups of a profinite group $G$, we have a net of the form $\left\lbrace NK\right\rbrace$ converging to $K$ where $N$ ranges through the open normal subgroups of $G$. Let us first assume that $SG$ is perfect so that the points $NK$ are not isolated. We are in the situation where as $NK$ gets closer to $K$ the index of $NK$ gets larger, so we could perhaps alternate the net $NK$ by considering the increasing sequence of indexes of subgroups $NK$ and using this sequence to give binary labels to the terms of the net. However it is difficult to see why any neighbourhood of a point $NK$ would contain infinitely many terms of both type of label. 

In the situation where $SG$ is not perfect it is well known by \cite{Gartside} that the open subgroups are isolated and hence get removed in the Cantor-Bendixson process immediately so we can't even use the above net. If we take $\left\lbrace A_N\right\rbrace$ where a closed subgroup $A_N$ is chosen in $O(N,NK)$, this converges to $K$ since sets of this form give a neighbourhood basis for $K$. However the subgroups $A_N$ have infinite index so the previous idea is not applicable and further ideas are needed to construct the alternating sequence.

We summarise what we have established thus far in the following theorem.
\begin{theorem}\label{Summary}
If $X$ is a space which is scattered and of Cantor-Bendixson rank $n$, then the injective dimension of sheaves of $\mathbb{Q}$-modules over $X$ is $n-1$. If $X$ is any space with infinite Cantor-Bendixson rank then the injective dimension is also infinite.

\end{theorem}
We now look at examples of profinite spaces and the application of the result relating injective dimension of sheaves of $\mathbb{Q}$-modules over $X$ to the Cantor-Bendixson rank of $X$.
\begin{example}
If $G$ is a discrete group then $SG$ is a finite discrete space and hence has Cantor-Bendixson dimension $1$. Therefore Theorem \ref{ID_CB} implies that the injective dimension of sheaves of $\mathbb{Q}$-modules over $SG$ is $0$.
\end{example} 
The above example works equally for any discrete space. Another way of seeing that the injective dimension of sheaves of $\mathbb{Q}$-modules over a discrete space is zero is by observing that such a sheaf is equivalent to a product of $\mathbb{Q}$-modules. We can see this by noticing that since each point $x$ in $X$ is isolated, the stalk of a sheaf $F$ at $x$ is determined by evaluating the sheaf at the open subset $\left\lbrace x\right\rbrace$. This becomes even clearer by considering a sheaf $F$ from the point of view of sheaf spaces. Let $(LF,\pi)$ be the sheaf space for $F$ over a discrete space $X$. Then since every map of sets from $X$ to $LF$ is continuous, it follows that $F(U)=\underset{x\in U}{\prod}F_x$ for any open subset $U$. This is an alternative way of saying that $F$ is equivalent to $C^0(F)$ from Definition \ref{serr}, which we know to be injective by Proposition \ref{Godinj}.
\begin{example}
If $G=\mathbb{Z}_p$ for any prime number $p$, then $S\left(\mathbb{Z}_p\right)$ is homeomorphic to the scattered space $P$ from Proposition \ref{padicspace}. This space has Cantor-Bendixson rank $2$ as seen in Example \ref{padic}. Therefore the category of sheaves of $\mathbb{Q}$-modules over $S\left(\mathbb{Z}_p\right)$ has injective dimension $1$ by Theorem \ref{ID_CB}. 
\end{example}
\begin{example}\label{sumpadic}
Consider distinct primes $p_1,p_2,\ldots,p_n$, then we have a profinite group $\underset{1\leq i\leq n}{\prod}\mathbb{Z}_{p_i}$ with corresponding profinite space $S\left(\underset{1\leq i\leq n}{\prod}\mathbb{Z}_{p_i}\right)$. This space is homeomorphic to $P^n$ by \cite[Proposition 2.5]{Gartside}. Then by Proposition \ref{prodcant} we have that $\dim_{CB}(P^n)=n+1$ and $\left(P^n\right)^{(n+1)}=\emptyset$. We can now apply Theorem \ref{ID_CB} to deduce that the injective dimension of sheaves over $S\left(\underset{1\leq i\leq n}{\prod}\mathbb{Z}_{p_i}\right)$ is exactly $n$. 
\end{example}
Furthermore if we consider Conjecture \ref{Conject} we can see the possible implications.
\begin{example}
The profinite completion of $\mathbb{Z}$ is defined to be:
\begin{align*}
\hat{\mathbb{Z}}=\underset{p}{\prod}\mathbb{Z}_p
\end{align*}
where the product runs over the collection of prime numbers $p$. This is a profinite group under the product topology and we can see that $S\left(\hat{\mathbb{Z}}\right)$ is perfect. If proven to be correct, Conjecture \ref{Conject} would imply that the injective dimension of sheaves over this space is infinite.
\end{example}
Another important example of a space is defined in \cite[Definition 2.8]{Gartside}, and this construction is similar to the Cantor space.
\begin{definition}\label{Pel}
Let $F_0=P_0=[0,1]$, the closed unit interval. We set $F_1=F_0\setminus (\frac{1}{3},\frac{2}{3})$ and $B_1=F_1\bigcup \left\lbrace\frac{1}{2}\right\rbrace$. That is, to form $F_1$ we remove the middle third of the interval of $F_0$ and to form $B_1$ we reinsert the midpoint of the deleted interval to $F_1$.

Given $F_{i-1}$ we define $F_i$ by deleting the middle third intervals of the remaining segments of $F_{i-1}$ and we define $B_i$ by reinserting midpoints of the deleted intervals to $F_i$. We set $F=\underset{n\in\mathbb{N}}{\bigcap}F_n$ and $B=\underset{n\in\mathbb{N}}{\bigcap}B_n$.
\end{definition}
The main focus of \cite{Gartside} is on proving that the algebraic structure of a profinite group $G$ can tell us about $SG$. Specifically, throughout \cite{Gartside} there are many assumptions on the algebraic structure of $G$ which lead to the conclusion that $SG$ is homeomorphic to $B$ from Definition \ref{Pel}. 
\begin{example}
Consider the spaces $B$ and $F$ defined in Definition \ref{Pel}. From \cite[Definition 2.8]{Gartside} we know that the space $B$ has perfect hull given by the Cantor space $F$ and that $\rank_{\cb}(B)=1$. If Conjecture \ref{Conject} were true then it would follow that the injective dimension of sheaves over $B$ is infinite.
\end{example}
\chapter{Injective dimension of G-equivariant sheaves of vector spaces}\label{finalchap}
\chaptermark{Injective dimension of $G$-sheaves}
In this chapter we seek to generalise the relationship of Chapter \ref{chapterIDCB} between the Cantor-Bendixson rank of a profinite space and the injective dimension of rational sheaves on $X$ to the $G$-equivariant setting. Here we consider $G$-sheaves over a $G$-space $X$. We know that if $x\in X$ is an isolated point of a $G$-space then each of the translates of $x$ must also be isolated. Therefore if $x\in X$ has height $n$ then so does $gx$ for any $g\in G$. Another way of saying this is that the Cantor-Bendixson height of a point is invariant under taking orbits. In the first section we will construct an equivariant Godement resolution and show that it is injective. In the final section we will carry out the injective dimension calculations on the equivariant Godement resolution to show that the analogous result to Chapter \ref{chapterIDCB} holds, Theorem \ref{GSummary}.
\section{G-equivariant Godement resolution}
We begin by making explicit the injective resolution we will be using throughout. We will prove that this satisfies the required properties.
\begin{construction}\label{equires}\index{equivariant Godement resolution}
Let $E$ be a $G$-equivariant sheaf over a profinite $G$-space $X$. By Proposition \ref{equivinjshf}, $\underset{A\in \text{orb}(X)}{\prod}\overline{E_{|_{A}}}$ is injective since the product of injective objects is an injective object. Here the product is in the category of $G$-sheaves over $X$. We denote this by $I^0(E)$.

There is also a monomorphism given by:
\begin{align*}
\delta:E\left(U\right)&\rightarrow \underset{A\in \text{orb}(X)}{\prod}\overline{E_{|_{A}}}\left(U\right)\\ s&\mapsto \left(s_{|_{A\cap U}}\right)_{A\in \text{orb}(X)} 
\end{align*}
where $U$ is an open subset of $X$. We calculate the cokernel of $\delta$ and continue this process inductively to obtain an injective resolution for $E$.
\end{construction} 
We now prove that the above map $\delta$ is a morphism of $G$-sheaves of $\mathbb{Q}$-modules over $X$.
\begin{proof}
Let $\iota_A$ be the closed inclusion of the orbit $A$ into $X$. We know that any orbit $A$ is closed by Proposition \ref{orbclose}. Then as in \cite[Theorem 3.7.13]{Tennison} ${\iota_A}^*$ represents the restriction functor and ${\iota_A}_*$ represents extension by zero, which are left and right adjoint to each other respectively.

We first show that we have a morphism:
\begin{align*}
\phi_A:E(U)&\rightarrow \overline{E_{|_{A}}}\left(U\right)\\ s&\mapsto s^{\prime},\,\text{where}\, s^{\prime}(x)=\left(s(\iota_A(x)),x\right),
\end{align*}
$U$ is open in $X$ and $\overline{E_{|_{A}}}\left(U\right)=E_{|_{A}}\left({\iota_A}^{-1}(U)\right)={\iota_A}_*{\iota_A}^*E(U)$. The map $\delta$ is taken to be the map:
\begin{align*}
(\phi_A)_A :E\rightarrow \underset{A\in \text{orb}(X)}{\prod}\overline{E_{|_{A}}},
\end{align*}
where the codomain is the product of non-equivariant sheaves. We will then show that $\delta$ maps into the discretisation of the product of sections over $U$, so that $\delta$ is a morphism whose codomain is the categorical product with respect to the category of $G$-sheaves over $X$.

We shall prove that the map $\phi_A$ is the adjoint of the identity morphism on $E_{|_{A}}$ which equals ${\iota_A}^*E$. By definition, the adjoint of this morphism is given by the following composite:
\begin{align*}
\xymatrix{E\ar[r]^{\nu}&{\iota_A}_*{\iota_A}^*E\ar[r]^{{\iota_A}_*(\id)}&{\iota_A}_*{\iota_A}^*E}
\end{align*}
where $\nu$ is the unit map and ${\iota_A}_*(\id_{{\iota_A}^*E})=\id_{{\iota_A}_*{\iota_A}^*E}$. By \cite[pp. 59]{Tennison} this is precisely $\phi_A$, so this is a map of sheaves. For $G$-equivariance, if $x\in A$ we will show that if $s_x\in E_x$ and $g\in G$ then ${\phi_A}_{gx}(g(s_x))=g({\phi_A}_x(s_x))$.

First note that this by definition reduces to showing that $g(s^{\prime}(x))=(g*s)^{\prime}(gx)$. But by definition of $s^{\prime}$ this holds immediately, since:
\begin{align*}
g(s_x,x)=(g(s_x),gx).
\end{align*}
For the discrete condition notice that if $s\in E(U)$ is $N$-equivariant for an open normal subgroup $N$ of $G$ and $U$ being $N$-invariant, then $\left(s_{|_{U\cap A}}\right)_{A\in orb(X)}$ must be $N$-equivariant. 

To see that $\delta$ is a monomorphism, let $s$ and $t$ be sections over $U$ which satisfy that 
\begin{align*}
\left(s_{|_{A\cap U}}\right)_{A\in \text{orb}(X)}=\left(t_{|_{A\cap U}}\right)_{A\in \text{orb}(X)}.
\end{align*}
We can use the general fact that if $a,b:Y\rightarrow Z$ are maps of sets and $\left\lbrace K_i\mid i\in I\right\rbrace$ is a partition of $Y$ with $a_{|_{K_i}}=b_{|_{K_i}}$ for each $i\in I$, then $a=b$. Therefore by setting $a=s_{|_{U}}$, $b=t_{|_{U}}$ and the partition to be
\begin{align*}
\left\lbrace U\cap A\mid A\in \text{orb}(X)\right\rbrace
\end{align*}
we obtain $s_{|_{U}}=t_{|_{U}}$ as required.
\end{proof}
\begin{corollary}
If $E$ is a Weyl-$G$-sheaf of $\mathbb{Q}$-modules then the $G$-sheaf constructed from $E$ in Construction \ref{equires} is a Weyl-$G$-sheaf. Let $X$ be a $G$-space such that each $x\in X$ has a neighbourhood basis $\mathfrak{B}_x$ satisfying that each $U\in \mathfrak{B}_x$ is $\stabgx$-invariant. Then if $E$ is a stalk-wise fixed sheaf then the $G$-sheaf constructed from $E$ in Construction \ref{equires} is a stalk-wise fixed sheaf. 
\end{corollary}
\begin{proof}
The reason the proof of Construction \ref{equires} holds for these subcategories can be seen in light of Propositions \ref{stalkfixchar} and \ref{Weylequi}. If a section $s$ is $NK$-equivariant (or $\text{stab}_G(x)$-equivariant) then the corresponding section under the image of $\delta$ is equivariant with respect to this subgroup. This is analogous to the part of the proof in Construction \ref{equires} where we see that a $N$-equivariant section $s$ satisfies that $\delta(s)$ is $N$-equivariant.
\end{proof}
\begin{definition}\label{godeequi}\index{equivariant Godement resolution}
If $E$ is a $G$-sheaf of $\mathbb{Q}$-modules over a profinite $G$-space we define the equivariant Godement resolution inductively as follows:
\begin{align*}
\xymatrix{0\ar[r]&E\ar[r]^{\delta_0}&I^0(E)\ar[d]^p\ar[r]&I^0(\coker\delta_0)=I^1(E)\ar[r]&\ldots\\&&\coker\delta_0\ar[ur]_{\delta_1}}
\end{align*}
This comes from Construction \ref{equires}.
\end{definition}
As a consequence of Proposition \ref{equivinjshf} it follows that for any $G$-sheaf of $\mathbb{Q}$-modules (or Weyl-$G$-sheaf of $\mathbb{Q}$-modules) $E$, $I^0(E)$ is injective, since products of injective objects are injective. This demonstrates that as in the non-equivariant case, the analogous equivariant Godement resolution in Definition \ref{godeequi} is an injective resolution.
\begin{proposition}\label{equidisc}
Let $x$ be an isolated point in a $G$-space $X$, then $G$-sheaves of the form $I^0(E)$ satisfy that $I^0(E)_x=E_x$ for any $G$-sheaf $E$. The same is true in the category of Weyl-$G$-sheaves.
\end{proposition}
\begin{proof}
Since $\left\lbrace x\right\rbrace$ is open we know that: 
\begin{align*}
\underset{A\in \text{orb}(X)}{\prod}\overline{E_{|_{A}}}\left\lbrace x\right\rbrace&=E_{|_{A}}\left(\left\lbrace x\right\rbrace\right)=E_x
\end{align*}
where $A$ is the orbit of $x$.
\end{proof}
This gives the following useful corollary which is immediate from Proposition \ref{equidisc}.
\begin{corollary}\label{corequidisc}
If $E$ is any Weyl-$G$-sheaf over a profinite $G$-space $X$ and $x$ is an isolated point of $X$, then $\delta_x$ is surjective so that ${\coker\delta}_x=0$.
\end{corollary}
\begin{remark}
This demonstrates an important subtlety, in the definition of $I^0(E)$ we consider a product which is indexed over $\text{orb}(X)$ as opposed to $X$. The previous corollary shows that $\delta_x$ is an isomorphism with $\text{orb}(X)$ indexing. If we index the product over $X$ then $I^0(E)_x=\underset{y\in A}{\prod} E_y$ where $A$ is the orbit of $x$. This is not $E_x$ unless $x$ is $G$-fixed.
\end{remark}
\section{Injective dimension calculation}
This section will complete the chapter by calculating the injective dimension of the category of $G$-sheaves in a way analogous to Chapter \ref{chapterIDCB}. We begin with the following lemma, which will be useful in proving Propositions \ref{spillorb} and \ref{constfill}, which help us to generalise what is known in the non-equivariant case.
\begin{lemma}\label{alternateheight}
Let $X$ be a $G$-space with Cantor-Bendixson rank $n$. If $x$ is any point of height $k$, where $k>0$, then any neighbourhood $U$ of $x$ contains points of height $j$ for every $j\leq k$.
\end{lemma}
\begin{proof}
We can assume without loss of generality that $U\cap X^{(k)}=\left\lbrace x\right\rbrace$, since if it does not we can choose $V\subseteq U$ which does satisfy this. Clearly $U$ contains a point of height $k$, namely $x$, so we now show that there is a point of height $k-1$. We know the following two facts:
\begin{itemize}
\item $U\cap X^{(k)}=\left\lbrace x\right\rbrace$,
\item $U\cap X^{(k-1)}\neq\left\lbrace x\right\rbrace$.
\end{itemize}
It therefore follows there exists infinitely many points of the form $y$ in $U$ which are isolated in $X^{(k-1)}$ and hence have height $k-1$. For height $k-2$, we know we can find an open neighbourhood $V\subseteq U$ of a point $y$ of height $k-1$ satisfying the following two properties:
\begin{itemize}
\item $V\cap X^{(k-1)}=\left\lbrace y\right\rbrace$,
\item $V\cap X^{(k-2)}\neq\left\lbrace y\right\rbrace$.
\end{itemize}
This means that there are infinitely many points $z$ in $V\subseteq U$ of height $k-2$. The proof continues inductively.
\end{proof}
\begin{proposition}\label{spillorb}
If $x$ is a non-isolated point belonging to the scattered part of a profinite $G$-space $X$, then any open neighbourhood of $x$ contains infinitely many points in a different orbit. If $X$ is of the form $SG$ for some profinite group $G$ this holds for any non-isolated point.
\end{proposition} 
\begin{proof}
For the first part, observe that this statement is true since the following facts hold:
\begin{enumerate}
\item every neighbourhood of $x$ contains infinitely many points of \emph{every} Cantor-Bendixson height less than $x$, 
\item points in the same orbit of $x$ have the same height as $x$,
\end{enumerate}
where the first fact follows from Lemma \ref{alternateheight}. For the second part if $K$ is not isolated then take any $O(N,NK)$ for $N$ open and normal in $G$. If $K\neq NK$ then $K$ is a proper subgroup of $NK$ so this does not belong in the orbit of $K$. If they are equal then since $K$ is not isolated there must exist $H\in O(N,NK)$ where $HN=K$ and $H\neq K$, hence $H$ is not in the orbit of $K$ as it is a proper subgroup. 

Note we can assume that there are infinitely many points in this case since each $O(N,NK)$ contains infinitely many $O(N^{\prime},N^{\prime}K)$ and we can apply the argument to each of these individually. 
\end{proof}
This is significant since it proves that if $x$ satisfies the assumptions in the proceeding proposition, then $I^0(E)_x$ is not determined solely by the orbit of the $x$ component.

In the non-equivariant case in Chapter \ref{chapterIDCB}, when constructing an element of the $\text{Ext}^n$ group which was non-zero we considered $(0_x,s^y)_{y\in U\setminus\left\lbrace x\right\rbrace}$ where $s^y$ was a chosen non-zero element of $\coker{\delta_i}_y$. In this setting our resolution is constructed by taking products of sheaves which are not just concentrated over a point but over a closed orbit which is probably infinite. Therefore the elements we construct have to be non-zero sections over each orbit, which explains why the following lemma is analogous to the non-equivariant case despite being somewhat different. 
\begin{proposition}\label{constfill}
Suppose that $X$ is a profinite $G$-space such that each $x\in X$ has a neighbourhood basis $\mathfrak{B}_x$ satisfying that each $U\in \mathfrak{B}_x$ is $\stabgx(x)$-invariant. If $F=c\mathbb{Q}$ then each $\coker\delta_i(U)$ from Definition \ref{godeequi} has a $\stabgx$-equivariant section which is non-zero, where $U\in \mathfrak{B}_x$ and $i$ smaller than the height of $x$.
\end{proposition}
In this proof we shall construct an equivariant section which is non-zero across many orbits.
\begin{proof}
We will prove this result inductively. If we take any such $U$ then we can choose a non-zero rational $p$ belonging to $F(U)$ and define $t^{A_x}=\left(s_A\right)_{A\in orb(X)}$ in $\underset{A\in \text{orb}(X)}{\prod}\overline{F_{|_{A}}}\left(U\right)$ by:
\begin{align*}
s_A=0\,\text{if height$(A)$ equals height$(x)-2k$ for $k\in \mathbb{N}_0$, and}\\
s_A=p_{|_{A}}\,\text{if height$(A)$ equals height$(x)-(2k-1)$ for $k\in \mathbb{N}_0$.}
\end{align*}
To see why this would work, assume that $\delta(a)=\left(s_A\right)_{A\in orb(X)}$ for some $a$ in $F(U)$. Then $a_{|_{A}}=s_A$ which is either $p$ or $0$. There exists some orbit $A$ with $a_{|_{A}}=0$ and so $a_y=0$ for all $y\in A$. This means that there is some open neighbourhood $V$ such that $a_{|_{V}}=0$. However we have also seen in Lemma \ref{alternateheight} that $V$ contains infinitely many points of each height less than $A$. This leads to a contradiction since there must exist orbits $A^{\prime}$ such that $a_{|_{A^{\prime}\bigcap V}}=s_{A^{\prime}}=p\neq 0$. 

For example, the first four examples are given as follows:
\begin{align*}
s_{A_x}=0\,\text{if $A_x$ is the orbit of $x$}\\s_A=p_{|_{A}}\,\text{if $A$ has height $1$ less than that of $x$}\\s_A=0\,\text{if $A$ has height $2$ less than that of $x$}\\s_A=p_{|_{A}} \,\text{if $A$ has height $3$ less than that of $x$.}
\end{align*}
Therefore this family we have constructed cannot be in the image of $\delta$ since orbits have empty interior by Proposition \ref{spillorb}. Inductively, we continue by setting $t$ equal to this family and using $t\in \coker\delta(U)$ in place of $p$ and applying the same alternating argument i.e. we have $r^{A_x}=\left(r_A\right)_{A\in \text{orb}(X)}$:
\begin{align*}
r_{A}=0\,\text{if height$(A)$ equals height$(x)-2k$ for $k\in \mathbb{N}_0$, and}\\r_A={t^A}_{|_{A}}\,\text{if height$(A)$ equals height$(x)-(2k-1)$ for $k\in \mathbb{N}_0$}
\end{align*}
where ${t^A}_{|_{A}}$ represents the section constructed analogously to $t^{A_x}$ by replacing $A_x$ with $A$. The section maps $z\mapsto {t^A}_z$ for $z\in A$, which is non-zero by Proposition \ref{spillorb}. Namely if $z\in A$ and $\delta(a)_z={t^A}_z$, then there is an open neighbourhood $V$ of $z$ such that $(a_{|_{V\cap A}})_A=(s_A)_A$. The proof is similar to above:

Since $s_{{A}}=0$, every neighbourhood of $z$ contains infinitely many orbits $B$ such that $s_{{B}}=p_{|_{B}}$ which is non-zero. Then $a_z=0$ and hence $a$ is zero on some neighbourhood of $z$ contradicting that $a$ restricts to non-zero $p_{|_{B}}$ for some subsets of this neighbourhood. This means that ${t^A}_z$ is not in the image of $\delta$. It is therefore non-zero in $\coker\delta_z$.

We will now see that $r^{A_x}$ is a non-zero section with respect to the next cokernel, which illustrates the inductive step. Namely, there must exit an orbit $B$ such that $r_B=0$. If $y\in B$ with $\delta(a)_y={r^{A_x}}_y$, then there is an open neighbourhood $V$ of $y$ such that $(a_{|_{V\cap A}})_A=(r_A)_A$. 

Since $a_{|_{V\cap B}}=r_{{B}}=0$, every neighbourhood of $y$ contains infinitely many orbits $A$ such that $r_{|_{A}}=t^A_{|_{A}}$ which is non-zero. On the other hand $a_y=0$ so there exists some open neighbourhood which $a$ is zero on. This contradicts the fact that $a$ restricts to non-zero $t^A_{|_{A}}$ for some subsets of this neighbourhood. This means that ${r^{A_x}}_y$ is not in the image of the $a_y$ with respect to $\delta_0$. Here $\delta_0$ is the monomorphism in the next stage of the resolution. Therefore ${\coker\delta_0}_y\neq 0$.

If $x$ has height atleast three bigger than the index of $\coker \delta_i$ then the section $r^{A_x}$ constructed will be zero over all orbits of two less than the height of $x$.

Notice that these sections are $\text{stab}_G(x)$-equivariant since the original section $p$ is and we are taking families of restrictions to entire orbits, therefore to $\text{stab}_G(x)$-invariant subsets. Also, viewing $p$ as a global section we know that it is $G$-equivariant. Since each point has a neighbourhood basis where each element $U$ is $\text{stab}_G(x)$-invariant, it follows that $p_{|_{U}}$ is $\text{stab}_G(x)$-equivariant.
\end{proof}
There is a reason why we consider an alternating sequence in the above proof. If we tried to consider the family which was defined by $s_{A_x}=0$ and $s_A=p$ otherwise, we could use that $A_x$ is closed to deduce that $U\setminus A_x$ is open and the restriction of $U$ to this family is in the image of $p$. This approach could show that $\coker\delta_x\neq 0$, however the section which arrives at this conclusion using this approach would be zero for orbits outside of that of $x$. This is why we try to construct a section which is non-zero across many orbits.

Notice in particular that this holds for $X=SG$ since for each closed subgroup $K$, the neighbourhood basis sets of the form $O(N,NK)$ are $N_G(K)$-invariant. We now have the following consequence to Proposition \ref{constfill}.
\begin{corollary}\label{equi}
Suppose that $X$ is a profinite $G$-space such that each $x\in X$ has a neighbourhood basis $\mathfrak{B}_x$ satisfying that each $U\in \mathfrak{B}_x$ is $\stabgx$-invariant. Then any point $x$ of $X$ and $i$ less than the height of $x$ satisfies ${{\coker\delta_i}_x}^{\stabgx(x)}\neq 0$.
\end{corollary}
We now observe that key lemmas proven in the non-equivariant setting now hold even when we consider a $G$-action.
\begin{lemma}
Lemmas \ref{lem0} and \ref{godementnon0} hold in the Weyl-$G$-sheaf setting.
\end{lemma}
\begin{proof}
Lemma \ref{lem0} follows immediately from Proposition \ref{equidisc} combined with the original proof of Lemma \ref{lem0}. The $G$-equivariant analogue of Lemma \ref{godementnon0} follows from Proposition \ref{constfill} and Corollary \ref{equi}.
\end{proof}
Furthermore we have the following lemma which also holds in the category of Weyl-$G$-sheaves when we introduce a $G$-action.
\begin{lemma}
Lemmas \ref{Homologychar} and \ref{homoloinf} hold in the equivariant setting except we end up with $\stabgx(x)$-fixed points, where we consider $\iota_{A_x}(\mathbb{Q})$, the extension by zero of the constant sheaf at $\mathbb{Q}$ over the orbit $A_x$.
\end{lemma}
\begin{proof}
We have the following equalities which are similar to those in the proof of the original non-equivariant case. For the following we will consider the $G$-sheaf $E$.
\begin{align*}
\hom_{G\text{-Sh}(X)}\left(\iota_{A_x}(\mathbb{Q}),\underset{A\in \text{orb}(X)}{\prod}\iota_A(E_{|_{A}})\right)&=\underset{A\in \text{orb}(X)}{\prod}\hom_{G\text{-Sh}(X)}\left(\iota_{A_x}(\mathbb{Q}),\iota_A(E_{|_{A}})\right)\\&=\hom_{G\text{-Sh}(A)}\left(G\underset{\text{stab}_G(x)}{\times}\mathbb{Q},G\underset{\text{stab}_G(x)}{\times}E_x\right)\\&={E_x}^{\text{stab}_G(x)}
\end{align*}
where the second last equality follows from Proposition \ref{restconst}, and the final term is non-zero at any stage of the resolution by Corollary \ref{equi} when $E=c\mathbb{Q}$.

Looking at the original non-equivariant proof, for the maps induced by the hom-functor if:
\begin{align*}
f\in \hom_{G\text{-Sh}(A_x)}\left(G\underset{\text{stab}_G(x)}{\times}\mathbb{Q},G\underset{\text{stab}_G(x)}{\times}{\coker\delta_k}_x\right)
\end{align*}
where ${\coker\delta_k}_x\neq 0$, then we have a map of basic open subsets similar to the proof of Proposition \ref{equiadjunct} as follows:
\begin{align*}
f:G\underset{\text{stab}_G(x)}{\times}\mathbb{Q}&\rightarrow G\underset{\text{stab}_G(x)}{\times}{\coker\delta_k}_x\\ [g,q]&\mapsto [g,qf_x(1)]
\end{align*}
where necessarily $f_x(1)$ is $\text{stab}_G(x)$-fixed. As in the proof of the original case, from the definition of $\delta$ the morphism $\delta\circ f$ is induced as follows:
\begin{align*}
\delta\circ f:G\underset{\text{stab}_G(x)}{\times}\mathbb{Q}&\rightarrow G\underset{\text{stab}_G(x)}{\times}{\coker\delta_{k+1}}_x\\ [g,q] &\mapsto [g,\left[qf_x(1),0,0,\ldots\right]_S].
\end{align*}
It follows as in the original case, the homomorphisms in the chain complex that we calculate the homology of are given by:
\begin{align*}
\delta_*:{{\coker\delta_k}_x}^{\text{stab}_G(x)}&\rightarrow {{\coker\delta_{k+1}}_x}^{\text{stab}_G(x)}\\a&\mapsto\left[\left(a_{A_x},0,0,\ldots\right)_x\right]_S
\end{align*}
where instead of using the Hausdorff property we use Proposition \ref{orbclose} to deduce that the complement of $A_x$ is open, which shows that if $y\notin A_x$ then $\left[\left(a_{A_x},0,0,\ldots\right)_y\right]_S=\left[\left(0\right)_y\right]_S$. This is analogous to the non-equivariant case.
\end{proof}
By applying these equivariant alterations to the argument given in the non-equivariant case we have the following theorem.
\begin{theorem}
Suppose that $X$ is a profinite $G$-space such that each $x\in X$ has a neighbourhood basis $\mathfrak{B}_x$ such that each $U\in \mathfrak{B}_x$ is $\stabgx(x)$-invariant. If $\rank_{\cb}(X)=\infty$ then the injective dimension of $G$-sheaves of $\mathbb{Q}$-modules over $X$ is infinite. If $\rank_{\cb}(X)<\infty$ and $X$ is scattered then the injective dimension $G$-sheaves of $\mathbb{Q}$-modules over $X$ is equal to $\rank_{\cb}(X)-1$.
\end{theorem}

Since $SG$ satisfies the neighbourhood basis assumptions needed, Theorem \ref{Summary} holds for $G$-equivariant sheaves of $\mathbb{Q}$-modules over $SG$ where $G$ is a profinite group. Since the construction of the $G$-equivariant Godement resolution holds for the category of Weyl-$G$-sheaves over $SG$ and since this is a full subcategory of $G$-equivariant sheaves we also have the following theorem.
\begin{theorem}\label{GSummary}
If $X$ is equal to $SG$ for some profinite group $G$ such that $X$ is scattered and of Cantor-Bendixson rank $n$, then the injective dimension of Weyl-$G$-sheaves of $\mathbb{Q}$-modules over $X$ is $n-1$. If $X$ is a profinite space with infinite Cantor-Bendixson rank then the injective dimension is also infinite.
\end{theorem}
This theorem has the following useful consequence.
\begin{corollary}
If $G$ is a profinite group, then the category of rational $G$-Mackey functors has infinite injective dimension if the Cantor-Bendixson rank of $SG$ is infinite. It is equal to the Cantor-Bendixson rank $-1$ provided $SG$ is scattered with finite Cantor-Bendixson rank. 
\end{corollary}
\begin{proof}
This is a combination of Theorem \ref{GSummary} and \ref{equivalencemain}.
\end{proof}
It follows from this result that examples in the previous chapter give examples of various injective dimension calculations of categories of rational $G$-Mackey functors for various profinite groups $G$. We now consider the following examples.
\begin{example}
Let $G$ be a finite group. Then since $SG$ is a finite discrete $G$-space it follows that this has Cantor-Bendixson rank $1$. An application of Theorem \ref{GSummary} shows that the injective dimension of Weyl-$G$-sheaves over $SG$ is zero. Consequently the algebraic model for rational $G$-spectra has injective dimension $0$. This example coincides with \cite[Appendix A]{Tate}.
\end{example}
The following application coincides with the work in \cite{BarZp}.
\begin{example}
Let $G$ be the $p$-adic integers, $\mathbb{Z}_p$. We know from Example \ref{padic} that the Cantor-Bendixson rank of $S\mathbb{Z}_p$ is $2$. An application of Theorem \ref{GSummary} shows that the injective dimension of Weyl-$G$-sheaves over $SG$ is $1$. Consequently, the algebraic model for rational $\mathbb{Z}_p$-spectra has injective dimension $1$. 
\end{example}
\begin{example}
Consider distinct primes $p_1,p_2,\ldots,p_n$, then we have a profinite group $\underset{1\leq i\leq n}{\prod}\mathbb{Z}_{p_i}$ with corresponding profinite space $S\left(\underset{1\leq i\leq n}{\prod}\mathbb{Z}_{p_i}\right)$. We have seen in Example \ref{sumpadic} that this space has Cantor-Bendixson rank $n+1$. An application of Theorem \ref{GSummary} shows that in this case the injective dimension of Weyl-$G$-sheaves over $SG$ is $n$. Consequently, the algebraic model for $G$-spectra has injective dimension $n$.
\end{example}
\bibliography{ref2}

\begin{thebibliography}{LMSM86}

\bibitem[Bar11]{BarZp}
David Barnes.
\newblock Rational {$\mathbb{Z}_p$-equivariant spectra}.
\newblock {\em Algebr. Geom. Topol.}, 11(4):2107--2135, 2011.

\bibitem[Bar17]{BarnesO2}
David Barnes.
\newblock Rational {$O(2)$}-equivariant spectra.
\newblock {\em Homology Homotopy Appl.}, 19(1):225--252, 2017.

\bibitem[BR14]{BousfieldBarnes}
David Barnes and Constanze Roitzheim.
\newblock Stable left and right {B}ousfield localisations.
\newblock {\em Glasg. Math. J.}, 56(1):13--42, 2014.

\bibitem[Bre97]{Bredon}
Glen~E. Bredon.
\newblock {\em Sheaf theory}, volume 170 of {\em Graduate Texts in
  Mathematics}.
\newblock Springer-Verlag, New York, second edition, 1997.

\bibitem[CW16]{Castellano}
I.~Castellano and Th. Weigel.
\newblock Rational discrete cohomology for totally disconnected locally compact
  groups.
\newblock {\em J. Algebra}, 453:101--159, 2016.

\bibitem[Dre71]{Dress}
Andreas W.~M. Dress.
\newblock {\em Notes on the theory of representations of finite groups. {P}art
  {I}: {T}he {B}urnside ring of a finite group and some {AGN}-applications}.
\newblock Universit\"at Bielefeld, Fakult\"at f\"ur Mathematik, Bielefeld,
  1971.
\newblock With the aid of lecture notes, taken by Manfred K\"uchler.

\bibitem[Fau08]{Fausk}
Halvard Fausk.
\newblock Equivariant homotopy theory for pro-spectra.
\newblock {\em Geom. Topol.}, 12(1):103--176, 2008.

\bibitem[Glu81]{Gluck}
David Gluck.
\newblock Idempotent formula for the {B}urnside algebra with applications to
  the {$p$}-subgroup simplicial complex.
\newblock {\em Illinois J. Math.}, 25(1):63--67, 1981.

\bibitem[GM92]{Structure}
J.~P.~C. Greenlees and J.~P. May.
\newblock Some remarks on the structure of {M}ackey functors.
\newblock {\em Proc. Amer. Math. Soc.}, 115(1):237--243, 1992.

\bibitem[GM95]{Tate}
J.~P.~C. Greenlees and J.~P. May.
\newblock Generalized {T}ate cohomology.
\newblock {\em Mem. Amer. Math. Soc.}, 113(543):viii+178, 1995.

\bibitem[Gre98]{Gre98}
J.~P.~C. Greenlees.
\newblock Rational {M}ackey functors for compact {L}ie groups. {I}.
\newblock {\em Proc. London Math. Soc. (3)}, 76(3):549--578, 1998.

\bibitem[Gre08]{conjecture}
J.~P.~C. Greenlees.
\newblock Rational torus-equivariant stable homotopy. {I}. {C}alculating groups
  of stable maps.
\newblock {\em J. Pure Appl. Algebra}, 212(1):72--98, 2008.

\bibitem[GS10a]{Gartside}
Paul Gartside and Michael Smith.
\newblock Classifying spaces of subgroups of profinite groups.
\newblock {\em J. Group Theory}, 13(3):315--336, 2010.

\bibitem[GS10b]{Gartside1}
Paul Gartside and Michael Smith.
\newblock Counting the closed subgroups of profinite groups.
\newblock {\em J. Group Theory}, 13(1):41--61, 2010.

\bibitem[Hir03]{Hirschhorn}
Philip~S. Hirschhorn.
\newblock {\em Model categories and their localizations}, volume~99 of {\em
  Mathematical Surveys and Monographs}.
\newblock American Mathematical Society, Providence, RI, 2003.

\bibitem[Hov99]{Hovey}
Mark Hovey.
\newblock {\em Model categories}, volume~63 of {\em Mathematical Surveys and
  Monographs}.
\newblock American Mathematical Society, Providence, RI, 1999.

\bibitem[K\k17]{MagSO3}
Magdalena K\k{e}dziorek.
\newblock An algebraic model for rational {${\rm SO}(3)$}-spectra.
\newblock {\em Algebr. Geom. Topol.}, 17(5):3095--3136, 2017.

\bibitem[Lin76]{Linder}
Harald Lindner.
\newblock A remark on {M}ackey-functors.
\newblock {\em Manuscripta Math.}, 18(3):273--278, 1976.

\bibitem[LMSM86]{LMSMcC}
L.~G. Lewis, Jr., J.~P. May, M.~Steinberger, and J.~E. McClure.
\newblock {\em Equivariant stable homotopy theory}, volume 1213 of {\em Lecture
  Notes in Mathematics}.
\newblock Springer-Verlag, Berlin, 1986.
\newblock With contributions by J. E. McClure.

\bibitem[May96]{alaska}
J.~P. May.
\newblock {\em Equivariant homotopy and cohomology theory}, volume~91 of {\em
  CBMS Regional Conference Series in Mathematics}.
\newblock Published for the Conference Board of the Mathematical Sciences,
  Washington, DC; by the American Mathematical Society, Providence, RI, 1996.
\newblock With contributions by M. Cole, G. Comeza\~na, S. Costenoble, A. D.
  Elmendorf, J. P. C. Greenlees, L. G. Lewis, Jr., R. J. Piacenza, G.
  Triantafillou, and S. Waner.

\bibitem[MM02]{MandelMayequivspectra}
M.~A. Mandell and J.~P. May.
\newblock Equivariant orthogonal spectra and {$S$}-modules.
\newblock {\em Mem. Amer. Math. Soc.}, 159(755):x+108, 2002.

\bibitem[MMSS01]{SchwedeMayShipleyMandell}
M.~A. Mandell, J.~P. May, S.~Schwede, and B.~Shipley.
\newblock Model categories of diagram spectra.
\newblock {\em Proc. London Math. Soc. (3)}, 82(2):441--512, 2001.

\bibitem[Rot64]{Rota}
Gian-Carlo Rota.
\newblock On the foundations of combinatorial theory. {I}. {T}heory of
  {M}\"obius functions.
\newblock {\em Z. Wahrscheinlichkeitstheorie und Verw. Gebiete}, 2:340--368
  (1964), 1964.

\bibitem[RZ00]{Ribes}
Luis Ribes and Pavel Zalesskii.
\newblock {\em Profinite groups}, volume~40 of {\em Ergebnisse der Mathematik
  und ihrer Grenzgebiete. 3. Folge. A Series of Modern Surveys in Mathematics
  [Results in Mathematics and Related Areas. 3rd Series. A Series of Modern
  Surveys in Mathematics]}.
\newblock Springer-Verlag, Berlin, 2000.

\bibitem[SS03]{StableModel}
Stefan Schwede and Brooke Shipley.
\newblock Stable model categories are categories of modules.
\newblock {\em Topology}, 42(1):103--153, 2003.

\bibitem[{Sug}19]{Sugrue}
D.~{Sugrue}.
\newblock {Injective dimension of sheaves of rational vector spaces}.
\newblock {\em J. Pure Appl. Algebra}, 223(7):3112--3128, 2019.

\bibitem[Ten75]{Tennison}
B.~R. Tennison.
\newblock {\em Sheaf theory}.
\newblock Cambridge University Press, Cambridge, England-New York-Melbourne,
  1975.
\newblock London Mathematical Society Lecture Note Series, No. 20.

\bibitem[{Thi}11]{Thiel}
U.~{Thiel}.
\newblock {Mackey functors and abelian class field theories}.
\newblock {\em ArXiv e-prints}, November 2011.

\bibitem[TW95]{Webb}
Jacques Th\'evenaz and Peter Webb.
\newblock The structure of {M}ackey functors.
\newblock {\em Trans. Amer. Math. Soc.}, 347(6):1865--1961, 1995.

\bibitem[Wei94]{Weibel}
Charles~A. Weibel.
\newblock {\em An introduction to homological algebra}, volume~38 of {\em
  Cambridge Studies in Advanced Mathematics}.
\newblock Cambridge University Press, Cambridge, 1994.

\bibitem[Yos83]{Yosh}
Tomoyuki Yoshida.
\newblock Idempotents of {B}urnside rings and {D}ress induction theorem.
\newblock {\em J. Algebra}, 80(1):90--105, 1983.

\end{thebibliography}
\printindex
\end{document}